\documentclass[a4paper, reqno]{amsart}

\usepackage[all]{xy}
\usepackage[english]{babel}
\usepackage[utf8]{inputenc}
\usepackage{amssymb,amsmath,amsthm}
\usepackage{amsfonts}
\usepackage{graphicx}
\usepackage{psfrag}
\usepackage[dvipsnames]{xcolor}
\usepackage[hidelinks]{hyperref}
\usepackage{csquotes}
\usepackage[alphabetic]{amsrefs}
\usepackage{tikz-cd}

\allowdisplaybreaks

\newcommand{\Rm}{\mathrm{Rm}}
\newcommand{\id}{\mathrm{id}}
\newcommand{\Ric}{\mathrm{Ric}}

\newcommand{\scal}{\mathrm{scal}}
\newcommand{\e}{\epsilon}

\newcommand{\n}{\nabla}
\newcommand{\im}{\mathrm{im}}
\renewcommand{\L}{\mathcal{L}}

\renewcommand{\o}{\omega}

\newcommand{\tr}{\mathrm{tr}}
\renewcommand{\a}{\alpha}
\newcommand{\de}{\delta}
\renewcommand{\b}{\beta}
\renewcommand{\d}{\partial}

\newcommand{\abs}[1]{\left\lvert#1\right\rvert}
\newcommand{\norm}[1]{\left\lVert#1\right\rVert}

\renewcommand{\div}{\mathrm{div}}

\renewcommand{\O}{\mathcal O}

\newcommand{\pr}{\mathrm{pr}}

\theoremstyle{plain}
\newtheorem{thm}{Theorem}[section]
\newtheorem{prop}[thm]{Proposition}
\newtheorem{lemma}[thm]{Lemma}
\newtheorem{cor}[thm]{Corollary}
\theoremstyle{definition}
\newtheorem{definition}[thm]{Definition}

\newtheorem{remark}[thm]{Remark}
\newtheorem{example}[thm]{Example}

\newtheorem{prob}[thm]{Problem}

\newcommand{\R}[0]{\mathbb{R}}							% Real numbers
\newcommand{\N}[0]{\mathbb{N}}							% Natural numbers
\newcommand{\Z}[0]{\mathbb{Z}}							% Integers
\newcommand{\M}[0]{\mathcal M}
\newcommand{\U}[0]{\mathcal U}
\newcommand{\V}[0]{\mathcal V}
\newcommand{\F}[0]{\mathcal F}

\makeatletter
\@namedef{subjclassname@1991}{Subject}
\@namedef{subjclassname@2000}{Subject}
\@namedef{subjclassname@2010}{2020 Mathematics Subject Classification}
\makeatother

\newcommand{\dv}{\text{ }dV}
\newtheorem{lem}[thm]{Lemma}
\newtheorem{rem}[thm]{Remark}
\usepackage{geometry}

\title[Convergence of the Ricci flow to Ricci-flat ALE manifolds]{Convergence of the Ricci flow to Ricci-flat ALE manifolds and positive scalar curvature rigidity}

\author{Klaus Kröncke}
\address{Department of Mathematics, KTH Royal Institute of Technology,
Lindstedtsvägen 25,
10044 Stockholm, Sweden}
\email{kroncke@kth.se}

\author{Oliver Petersen}
\address{Department of Mathematics, Stockholm University,
Albanovägen 28,
10691 Stockholm, Sweden}
\email{oliver.petersen@math.su.se}

\begin{document}
\hbadness=100000
\vbadness=100000

\begin{abstract}
We prove stability of integrable ALE manifolds with a parallel spinor under Ricci flow, with respect to perturbations in $L^p\cap L^{\infty}$ for any $p \in (1, n)$, improving a result by Deruelle and the first author \cite{DK17}. Our result applies to all ALE gravitational instantons.

The theorem is proved by a fixed point argument, based on novel estimates for the heat kernel of the Lichnerowicz Laplacian. It allows us to give a precise description of the convergence behaviour of the Ricci flow. Our decay rates are strong enough to prove positive scalar curvature rigidity in $L^p$, for each $p\in [1,\frac{n}{n-2})$, generalizing a result by Appleton. 
\end{abstract}

\maketitle

\tableofcontents
\begin{sloppypar}

\section{Introduction}
A one-parameter family $\left\{g_t\right\}_{t\in I}$ of Riemannian metrics on a manifold $M^n$, $n\geq2$, is called a Ricci flow if
\begin{align*}
	\partial_tg_t=-2\Ric_{g_t}.
\end{align*}
The Ricci flow was introduced in the eighties by Hamilton \cite{Ham82}
and it has become an important tool in Riemannian geometry ever since. Its success culminated in Perelman's proof
of the Poincar\'{e} and Geometrization Conjectures about the classification of closed three-dimensional
manifolds \cite{Per02}.
A natural question in geometric analysis is the stability of stationary points of the Ricci flow on the space of metrics (modulo homotheties), which we call Ricci solitons. This problem is relevant for the formation of singularities under the Ricci flow. 
Any type I singularity admits a blowup limit which is a Ricci soliton, see \cite{EMT2011}, and its instability would exclude it as a possible singularity model for generic initial data \cite{IKS19}.

On compact manifolds, the stability problem is by now well understood in terms of Perelman's entropies due to work by Haslhofer-M\"{u}ller and the first author \cites{HM14,Kro15,Kro20}. 
In this paper, we are interested in the stability problem on \emph{non-compact manifolds}.
As singularity models for the $4$-dimensional Ricci flow on compact manifolds, only non-compact Ricci-flat \emph{ALE spaces} can appear if the scalar curvature is bounded along the flow, see \cite{Bam18}.
In fact, such singularities have recently been shown to exist by Appleton \cite{App19} (however, on non-compact manifolds with unbounded scalar curvature). 
We expect that stability questions of ALE spaces are deeply connected to the formation of singularities under 4-dimensional Ricci flow.

The main result of this paper states that integrable ALE spaces with a parallel spinor (hence Ricci-flat) are dynamically stable to perturbations in a small $L^p \cap L^\infty$-neighbourhood, for any $p < n$.
In terms of fall-off conditions on the perturbations, our result requires a fall-off of order $\O \left(r^{-1-\e}\right)$.
See Theorem \ref{thm : mainthm 1 introduction} for the precise statement.
Our main result applies in particular to all known $4$-dimensional Ricci-flat ALE spaces (which all belong to the Kronheimer classification of gravitational instantons).
An interesting application of our result is the scalar curvature rigidity result of integrable ALE spaces with a parallel spinor with respect to perturbations in $L^p \cap L^\infty$ for $p < \frac n{n-2}$, see Theorem \ref{thm : psc rigidity introduction}.
We construct counterexamples to scalar curvature rigidity for $p > \frac n{n-2}$, showing that our scalar curvature rigidity result is at least almost sharp, see Theorem \ref{thm: counterexample scalar curvature rigidity intro}.

Our result is a significant improvement over the $L^2$-stability result of Ricci-flat ALE spaces, by Deruelle and the first author \cite{DK17}, where the initial data was assumed to be $L^2 \cap L^\infty$-close and no convergence rate was established. 
As a consequence, only convergence of the Ricci-de Turck flow was shown in \cite{DK17}, not convergence of the Ricci flow.
In this paper, we establish sharp convergence rates for the Ricci-de Turck flow, allowing us to conclude convergence of the actual Ricci flow.
Analogous to these results, the stability of $\R^n$ was proven by Schulze, Schn\"{u}rer and Simon \cite{SSS08}.
Their proof relies heavily on the explicit geometry of $\R^n$ and cannot be generalized to the ALE setting.

There are several further results in the literature on stability of Ricci flow on certain non-compact manifolds, including the stability of hyperbolic space \cite{SSS11}, hyperbolic spaces with cusps \cite{Bam14}, symmetric spaces of non-compact type \cite{Bam15}, complex hyperbolic space \cite{Wu13}, $\cosh$-cylinders by the first author \cite{Kro18} and further non-trivial non-compact expanding Ricci solitons \cites{Der15,DL17,WW16}.
None of these appear as blowup limits of Ricci flows, hence these results are not relevant for Ricci flow singularities.
%\textcolor{red}
{Moreover, the present geometric situation has the following additional complication: The continuous spectrum of the linearized operator in the ALE case is the entire \emph{non-negative real axis}, whereas in the above mentioned results (apart from $\R^n$), the operator is  \emph{strictly positive}.}
%Moreover, the main technical difference is that 
%the continuous spectrum of the linearized operator in the ALE case is the \emph{non-negative real axis}, whereas the continuous spectrum in the above mentioned results (apart from $\R^n$) is \emph{bounded away from zero}.

\subsection{Geometric setup}
Before we explain the main results of this paper, let us introduce the geometric setting we are working in.
\begin{definition}[ALE manifold] \label{def: ALE}
	A complete Riemannian manifold $(M^n, g)$ is called \emph{asymptotically locally Euclidean} with one end of order $\tau > 0$ if there is a compact set $K \subset M$ and a diffeomorphism $\phi: M_\infty := M \backslash K \to (\R^n \backslash \overline{B_1})/\Gamma$, where $\Gamma$ is a finite subgroup of $SO(n)$ acting freely on $\R^n \backslash \{0\}$, such that
	\[
	\abs{(\n^{eucl})^k(\phi_* g - g_{eucl})}_{eucl} = O(r^{-\tau-k})
	\]
	holds on $(\R^n \backslash B_1)/\Gamma$ for all $k \in \N$.
	The diffeomorphism $\phi$ will also be called \enquote{coordinate system at infinity}.
\end{definition}

We are particularly interested in \emph{Ricci-flat} ALE manifolds, of which many examples do exist:

\begin{example}[Ricci-flat ALE manifolds]
	The simplest example of a Ricci-flat ALE manifold (different from $\R^n$) is the Eguchi-Hanson manifold.
	Let $\a_1, \a_2, \a_3$ be the standard left-invariant one-forms on $S^3$.
	For each $\e > 0$, define the Eguchi-Hanson metric
	\[
	g_{eh, \e} := \frac{r^2}{\left(r^4 + \e^4\right)^{\frac12}}\left(dr \otimes dr + r^2 \a_1 \otimes \a_1\right) + \left(r^4 + \e^4\right)^{\frac12}\left(\a_2 \otimes \a_2 + \a_3 \otimes \a_3\right),
	\]
	for $r >0$.
	After we quotient by $\Z_2$, we can smoothly glue in an $S^2$ at $r = 0$ to get the (complete) Eguchi-Hanson manifold $(T^*S^2,g_{eh,\e})$, which is ALE with $\Gamma=\Z_2$ and hyperk\"{a}hler, hence Ricci-flat. 
	This is an example in Kronheimer's classification of hyperk\"{a}hler ALE manifolds \cite{Kro89}: 
	Each $4$-dimensional hyperk\"{a}hler ALE manifold is diffeomorphic to a minimal resolution of $(\R^4\setminus\left\{0\right\})/\Gamma$, where $\Gamma\subset \mathrm{SU}(2)$ be a discrete subgroup acting freely on $S^3$.
\end{example}
These examples satisfy an important assumption, which can be defined under the following condition: 
\begin{definition}[Spin ALE manifold]
	We say that an ALE manifold is \emph{spin} if it carries a spin structure which is compatible with the Euclidean spin structure on $M_\infty$, for which the flat metric admits parallel spinors.
\end{definition}
The main assumption is now that $(M,h)$ is an ALE spin manifold with a \emph{parallel spinor}.
This assumption has various consequences, c.f. also the discussion in \cite{KP2020}*{Section 6}:
\begin{itemize}
	\item $(M,h)$ is Ricci-flat.
	\item $(M,h)$ has irreducible holonomy, unless it is flat. Consequently,
	\begin{align}\label{holonomy}
		\mathrm{Hol}(M,h)\in\left\{\mathrm{SU}(n/2),\mathrm{Sp}(n/4),\mathrm{Spin}(7)\right\}.
	\end{align}
	\item $M$ is even-dimensional (we therefore excluded the case of holonomy $G_2$).
	\item $(M,h)$ has at most finite fundamental group.
\end{itemize}
\begin{remark}
	All known Ricci-flat ALE manifolds satisfy \eqref{holonomy} and thus carry a parallel spinor. Moreover, all these groups actually appear as holonomy groups of Ricci-flat ALE manifolds, see \cites{Kro89,Joy99,Joy00,Joy01}. 
	It is an open question whether there are other examples, cf. \cite{BKN89}*{p.\ 315}.
\end{remark}
Up to a gauge term, the linearization of the Ricci curvature is given by an elliptic operator, called the Lichnerowicz Laplacian:
\begin{definition}
	Let $S^2M$ denote the vector bundle of symmetric $(0,2)$-tensors on $M$.
	The \emph{Lichnerowicz Laplacian} on a Ricci-flat manifold $(M,h)$ is defined as
	\[
	\Delta_L := \n^*\n - 2 \Rm: C^\infty(M, S^2M) \to C^\infty(M, S^2M),
	\]
	where
	\[
	\Rm (k)_{ij} = R_{imnj}k^{mn},
	\]
	for any $k \in C^\infty(M, S^2M)$. The manifold $(M,h)$ is called \emph{linearly stable} if $\Delta_L\geq 0$.
\end{definition}
It is well known that a Ricci-flat manifold with a parallel spinor is \emph{linearly stable} \cites{DWW05, Wang91}. 
The reason is that there is an isometric parallel bundle endomorphism
\begin{align*}
	\Phi:C^{\infty}(S^2M)\to C^{\infty}(S\otimes T^*M),
\end{align*}
such that 
\begin{align}\label{eq : Dirac}
	\Phi\circ \Delta_L
		= \left( D_{T^*M} \right)^2\circ \Phi,
\end{align}
where $D_{T^*M}$ is the twisted Dirac operator on vector spinors, i.e.\ the natural Dirac operator on
\[
	S \otimes T^*M,
\]
with respect to the Clifford multiplication ${X \cdot (s \otimes \o) := (X \cdot s) \otimes \o}$, where ${X \in TM}$ and ${s \otimes \o \in S \otimes T^*M}$, and $\cdot$ denotes the Clifford multiplication.
Given \eqref{eq : Dirac}, one can compute
\begin{align*}
	(\Delta_L h, h)_{L^2}
		&= (\Phi \circ \Delta_L h, \Phi(h))_{L^2}
		= (\left( D_{T^*M} \right)^2\circ \Phi(h), \Phi(h))_{L^2}
		= \norm{D_{T^*M} \Phi(h)}_{L^2}^2,
\end{align*}
showing the desired positivity.

Another necessary notion we need is \emph{integrability} which we define in the following.
For $1\leq p< q\leq\infty$, we use the notation $L^{[p,q]}:=L^{p}\cap L^{q}=\cap_{r\in [p,q]}L^r$.
Furthermore, for a fixed metric $\hat{h}$, we define $\mathcal{M}^{[p,q]}$ as the set of metrics $g$ such that $g-\hat{h}\in L^{[p,q]}(S^2M)$.
\begin{definition}
	A spin ALE manifold $(M,\hat{h})$ with a parallel spinor is called \emph{$L^{[p,\infty]}$-integrable} if there exists an $L^{[p,\infty]}$-neighbourhood $\U\subset \mathcal{M}^{[p,\infty]}$ such that the set
	\[
	\F_{\U}:=\left\{h\in \U\mid \Ric_h=0,\quad 2\div_h \hat{h} -  d (\tr_h \hat{h}) =0 \right\}
	\]
	is a finite-dimensional submanifold of $\mathcal{M}^{[p,\infty]}$ only containing metrics with a parallel spinor and satisfying 
	\begin{align*}
		T_{\hat{h}}\mathcal{F}_{\U}=\ker_{L^{[p,\infty]}}(\Delta_{L,\hat{h}}):=
		\left\{k\in \ker(\Delta_{L,\hat h})\mid k\in L^{[p,\infty]}(S^2M)\right\}.
	\end{align*}
	We call it \emph{integrable}, if it is $L^{[p,\infty]}$-integrable for all $p\in (1,\infty)$.
\end{definition}
\begin{rem}
	The additional condition $2\div_h \hat{h} -  d (\tr_h \hat{h}) =0$ serves as a gauge condition. In suitable weighted Sobolev spaces, it defines a slice of the action of the diffeomorphism group on the space of metrics, see \cite{DK17}*{Proposition 2.11}.
\end{rem}
The integrability condition has been shown to hold for K\"{a}hler and hence also for hyperk\"{a}hler manifolds, see \cite{DK17}. 
Therefore, the integrability is in fact known to be \emph{automatic}, given a parallel spinor, unless the holonomy is $\mathrm{Spin}(7)$. 
However, also in the $\mathrm{Spin}(7)$ case integrability is widely expected to be true.
\subsection{Main results}
We formulate the main theorem of this paper. 
The appearing norms and covariant derivatives are taken with respect to $\hat{h}$.
Here and throughout the paper, whenever we write ``ALE manifold, which carries a parallel spinor'', we assume that the manifold is spin.
\begin{thm}\label{thm : mainthm 1 introduction}
	Let $(M^n,\hat{h})$ be an ALE manifold, which carries a parallel spinor and is integrable. 
	Then for each $q\in (1,n)$ and each $L^{[q,\infty]}$-neighbourhood $\mathcal{U}\subset\mathcal{M}$ of $\hat{h}$ in the space of metrics, there exists another $L^{[q,\infty]}$-neighbourhood $\mathcal{V}\subset\mathcal{U}$ of $\hat{h}$  with the following property:\\ For each metric $g_0\in \mathcal{V}$ on $M$,
	 the Ricci flow $\left\{g_t\right\}_{t\geq0}$ starting at $g_0$ exists for all time and there is a family of diffeomorphisms $\left\{\phi_t\right\}_{t\geq0}$ such that $\phi_t^*g_t\in\mathcal{U}$
	for all $t\geq0$ and $\phi_t^*g_t$ converges to a Ricci-flat metric $h_{\infty}$ as $t\to\infty$.\\
	 Moreover, if $g_0-\hat{h}\in L^p$ for some $p\in (1,q]$, there exists a smooth family of Ricci-flat metrics $h_t$, such that following convergence rates hold: 
\begin{itemize}
	\item[(i)] For each $k\in\N_0$ and $\tau>0$, there exists a constant $C=C(\tau, k, \hat h)$ such that for all $t\geq 1$, we have
	\begin{align}\label{eq: decay rate h}
		\left\|h_t-h_{\infty}\right\|_{C^k}\leq C\cdot t^{1-\frac{n}{p}+\tau}.
	\end{align}
	\item[(ii)] For $r\in [p,\infty]$ and $k \in\N_0$ such that $\frac{n}{2}\left(\frac{1}{p}-\frac{1}{r}\right)+\frac{k}{2}<\frac{n}{2p}$, there exists a constant $C = C(p, q, k, \hat h)$ such that for all $t\geq 1$, we have
	\begin{align}\label{eq : decay rate k 1}
	\norm{\nabla^k (\phi_t^*g_t-h_t)}_{L^{r}} \leq C\cdot t^{-\frac n2 \left(\frac1p - \frac{1}{r}\right)-\frac{k}{2}}.
	\end{align}
	\item[(iii)] For $r\in [p,\infty]$ and $k\in\N_0$ such that $\frac{n}{2}\left(\frac{1}{p}-\frac{1}{r}\right)+\frac{k}{2}\geq\frac{n}{2p}$ and for each $\tau>0$ there exists a constant $C = C(p, q,k,\tau, \hat h)$ such that for all $t\geq 1$, we have
	\begin{align}\label{eq : decay rate k 2}
	\norm{\nabla^k (\phi_t^*g_t-h_t)}_{L^{r}} \leq C\cdot t^{-\frac{n}{2p}+\tau}.
	\end{align}
\end{itemize}
\end{thm}
The technical difficulty is already apparent from spectral properties of the Lichnerowicz Laplacian. 
All the other results on non-compact manifolds (except $\R^n$) mentioned in the final paragraph of the introduction use the property \emph{strict linear stability}, which means
\begin{align}\label{eq : strict positivity}
	P\geq c>0
\end{align}
for some constant $c>0$ and the associated linear operator $P$. This property gives nice decay estimates for the heat kernel and causes exponential convergence of the flow. In contrast, the continuous spectrum of $\Delta_L$ on Ricci-flat ALE manifolds is always $[0,\infty)$. Thus \eqref{eq : strict positivity} can never hold in this setting, not even on $\ker_{L^2}(\Delta_L)^{\perp}$. Instead, one can prove the weaker inequality
\begin{align*}
	\Delta_L|_{\ker_{L^2}(\Delta_L)^{\perp}}\geq c\nabla^*\nabla>0,
\end{align*}
for some constant $c>0$,
 which was the central ingredient for proving the aforementioned $L^2$-stability result in \cite{DK17}. 
%In this paper,
%\textcolor{red}
{Here,} we use novel estimates for the heat kernel of $\Delta_L$ and its derivatives, which the authors developed in a recent paper \cite{KP2020}. 
%Here, 
Inspired by an approach 
%\textcolor{red}{inspired }
by Koch and Lamm \cite{KL12}, we establish the Ricci-de Turck flow as the fixed point of a contraction map. 
In the present geometric situation, we had to overcome some technical obstacles, which we explain in Subsection \ref{subsection : technical obstacles} below.
\begin{rem}
	The diffeomorphisms $\phi_t$ are coming from the de Turck vector field: The family $\phi_t^*g_t$ is a Ricci-de Turck flow. For $t\in [0,1]$, we take $\hat{h}$ as the reference metric. 
	For $t\geq1$, $\phi_t^*g_t$ is a Ricci-de Turck flow with \emph{moving} reference metric $h_t$. 
	This choice of gauge turned out to be more convenient in our setting.
\end{rem}
The assumption of having a parallel spinor is pivotal for the following two reasons:
\begin{itemize}
	\item We showed in \cite{KP2020} that under this assumption, $\ker_{L^2}(\Delta_{L,\hat{h}})\subset\O_{\infty}(r^{-n})$, which improves the result in \cite{DK17}, where we only showed $\ker_{L^2}(\Delta_{L,\hat{h}})\subset\O_{\infty}(r^{1-n})$. 
	This allows us to have a better control on the reference metrics $h_t$, as we will then have $h_t-\hat{h}\in \O_{\infty}(r^{-n})$ as well.
	\item In \cite{KP2020}, we computed (with the help of \eqref{eq : Dirac}) optimal estimates on the heat kernel of the Lichnerowicz Laplacian and its derivatives. These estimates are strong enough to establish the Ricci flow as the fixed point of an iteration map.
\end{itemize}
\begin{rem}
	The decay rates for $\phi_t^*g_t-h_t$ in \eqref{eq : decay rate k 1} and \eqref{eq : decay rate k 2} coincide with the decay rates of the norm of the map
	\begin{align*}
		e^{-t \Delta_{L,\hat{h}}} 
			: L^p(S^2M)\cap \ker_{L^2}(\Delta_{L,\hat{h}})^{\perp}
			\to L^r(S^2M).
	\end{align*}
The convergence rate of $h_t$ in \eqref{eq: decay rate h} comes from integrating the inequality
\begin{align}\label{eq : explaining decay rate h}
	\left\|\partial_t h\right\|_{C^k}\leq C \left\|\phi_t^*g_t-h_t\right\|_{W^{2,r}}^2\leq Ct^{-n\left(\frac{1}{p}-\frac{1}{r}\right)},
\end{align}
where we can pick $r \in (p, \infty]$ as large as we want. 
Here, the first estimate follows from the construction of $h_t$ and the second one follows from \eqref{eq : decay rate k 2}. We could also replace the $C^k$-norm by any $W^{k,s}$-norm with $s \in (1,\infty]$.
Note that the right hand side of \eqref{eq : explaining decay rate h} is not integrable, for any $r \in (p,\infty]$, if $p \in [n, \infty]$.
This rate therefore explains why we cannot take $p\in [n,\infty]$.
This is in sharp contrast to the Euclidean case, where one can take $h_t \equiv h_{\R^n}$. 
There, one also expects the rate in \eqref{eq : decay rate k 1} to hold for all $p\in [1,\infty]$, see \cite{App18} for partial results.
\end{rem}
If we restrict to $p<\frac{3n}{4}$, we can get rid of the diffeomorphisms: 
\begin{thm}\label{thm : mainthm 2 introduction}
Let $(M^n,\hat{h})$ be an ALE manifold, which carries a parallel spinor and is integrable. 
Then for each $q\in (1,n)$, there exists an $L^{[q,\infty]}$-neighbourhood $\V$	with the following property: For each metric $g_0\in\V$ satisfying $g_0-\hat{h}\in L^p$ for some $p\in (1,\frac{3n}{4})$, the Ricci flow $g_t$ starting at $g_0$ exists for all time and converges to a Ricci-flat limit metric $h_{\infty}$ as $t\to\infty$. Moreover,
 there exists a smooth family of Ricci-flat metrics $h_t$ such that for each $\tau>0$, we have a constant $C=C(\tau)$ such that for all $t\geq 1$,
	\begin{align*}
		\left\|h_t-h_{\infty}\right\|_{C^k}\leq C\cdot \begin{cases} t^{\frac{1}{2}-\frac{n}{2p}+\tau},& \text{ if } p\in \left(1,\frac{2n}{3}\right),\\
			t^{2-\frac{3n}{2p}+\tau},& \text{ if } p\in \left[\frac{2n}{3},\frac{3n}{4}\right),
		\end{cases}
	\qquad \left\|g_t-h_t\right\|_{C^k}\leq C\cdot   t^{-\frac{n}{2p}+\tau},
	\end{align*}
with the norms taken with respect to the limit metric $h_{\infty}$.
\end{thm}

Our third main result is an application of the previous ones and reads as follows:
\begin{thm}\label{thm : psc rigidity introduction}
	Let $(M^n,\hat{h})$ be an ALE Ricci-flat spin manifold which is integrable and carries a parallel spinor. Then for each $q\in (1,n)$, there exists a $L^{[q,\infty]}$-neighbourhood $\U$ of $\hat{h}$ in the space of metrics such that each smooth metric $g\in\U$  on $M$ satisfying
	\begin{align*}
	 \scal_g\geq0,
	 	 \qquad\text{ and }\qquad\left\|g-\hat{h}\right\|_{L^p}<\infty
	\end{align*}
	for some $p\in \left[1,\frac{n}{n-2}\right)$ is Ricci-flat.
\end{thm}
This theorem generalizes a corresponding result for Euclidean space \cite{App18} which also holds for $q=\infty$.
It is related to the rigidity part of the positive mass theorem, for which there exists also a version on ALE spin manifolds \cite{Dahl97}. The ADM mass of an ALE manifold $(M^n,g)$ is
\begin{align*}
	m(g)
		=\lim_{r\to\infty}\int_{S^{n-1}(r)/\Gamma}(\partial_jg_{ij}-\partial_ig_{jj})\dv_{S^{n-1}(r)},
\end{align*}
where the components of $g$ are taken with respect to an asymptotic coordinate system. Now if $g-\hat{h}\in L^p$, we heuristically expect
\begin{align*}
	g-\hat{h}
		\in o\left(r^{-\frac{n}{p}}\right),
		\qquad \partial(g-\hat{h})\in o\left(r^{-\frac{n}{p}-1}\right).
	\end{align*}
Due to \cite{BKN89}, we know that $\hat{h}-h_{\R^n}\in \O_{\infty}\left(r^{-n+1}\right)$ in suitable coordinates. We get
\begin{align*}
	\partial g
		\in o\left(r^{-\frac{n}{p}-1}\right),
			\text{ if }p>\frac{n}{n-1},
	\qquad \partial g
		\in O\left(r^{-n}\right),
			\text{ if }p\in \left(1,\frac{n}{n-1}\right).
\end{align*}
Thus for $p<\frac{n}{n-2}$ we expect that $m(g)=0$ and $g$ has to be Ricci-flat by the rigidity statement of the positive mass theorem.

For $p=\frac{n}{n-2}$, it is unclear what happens. In \cite{App18}, a partial result was shown for Euclidean space, which we are not able to reproduce here. 
For the remaining cases for $p$, the converse holds under even milder assumptions on the background metric.

\begin{thm} \label{thm: counterexample scalar curvature rigidity intro}
Let $(M,\hat{h})$ be a Ricci-flat manifold. Then for every $p>\frac{n}{n-2}$ there exists a sequence $(g_i)_{i\in\N}$ with $\scal_{g_i}>0$ such that $g_i\to\hat{h}$ in $L^{[p,\infty]}$ as $i\to\infty$. 
Moreover, the $g_i$ can be chosen to be conformal to $\hat{h}$.
\end{thm}
This assertion is a simple consequence of the implicit function theorem. Nevertheless, we could not find it in this form in the literature. Therefore we state it here to complement Theorem \ref{thm : psc rigidity introduction}.

\subsection{Outline of the proof of stability}\label{subsection : technical obstacles}
For proving stability of a given Ricci-flat metric $h$, it is more convenient to use the Ricci-de Turck flow
\begin{align}\label{eq : RiccideTurck-1}
	\partial_tg_t=-2\Ric_{g_t}+\mathcal{L}_{V(g_t,h)}g_t,\qquad V(g,h)^k=g^{ij}(\Gamma(g)_{ij}^k-\Gamma(h)_{ij}^k).
\end{align}
 instead of the Ricci flow as it has the advantage of being strictly parabolic. More precisely, it can be written in terms of the difference $k=g-h$ as
\begin{align}\label{eq : RiccideTurck0}
	\partial_tk+\Delta_{L,h}k=g^{-1}*R_h*k*k+g^{-1}*g^{-1}*\nabla k*\nabla k+\nabla ((g^{-1}-h^{-1})*\nabla k).
\end{align}
In integral form, the latter equation reads
\begin{equation}\begin{split}\label{eq : RiccideTurck0_integral}
k_t&=e^{-t\Delta_{L,h}}k_0+\int_0^t e^{-(t-s)\Delta_{L,h}}[g_s^{-1}*R_h*k_s*k_s+g_s^{-1}*g_s^{-1}*\nabla k_s*\nabla k_s]ds\\
	&\qquad+\int_0^t e^{-(t-s)\Delta_{L,h}}[\nabla ((g_s^{-1}-h_s^{-1})*\nabla k_s)]ds.	
\end{split}
\end{equation}
Given an initial $k_0$, the linear part of the equation is expected to dominate the behavior of the solution $k_t$. 
It is therefore natural to set
\begin{align*}
	k_t^{(1)}=e^{-t\Delta_{L,h}}k_0
\end{align*}
and inductively define
\begin{align*}
	k_t^{(i+1)}&=e^{-t\Delta_{L,h}}k_0+\int_0^t e^{-(t-s)\Delta_{L,h}}[(g_s^{(i)})^{-1}*R_h*k^{(i)}_s*k^{(i)}_s+g_s^{-1}*(g_s^{(i)})^{-1}*\nabla k^{(i)}_s*\nabla k^{(i)}_s]ds\\
	&\qquad+\int_0^t e^{-(t-s)\Delta_{L,h}}[\nabla (((g_s^{(i)})^{-1}-h^{-1})*\nabla k^{(i)}_s)]ds	
\end{align*}
for $i\in\N$. One would now hope that as $i\to\infty$, $k_t^{(i)}$ converges in a suitable Banach space to a solution of \eqref{eq : RiccideTurck0_integral} and hence of \eqref{eq : RiccideTurck0}. In fact, this strategy was successfully carried out in \cite{KL12}, to prove $L^{\infty}$-stability of $\R^n$ under Ricci flow.

Carrying out these steps in the general ALE case is far more complicated. Here, we will explain the main technical issues and outline the ideas how to overcome these problems.
\subsubsection{Controlling the linear part}
\begin{prob}
The operator $\Delta_{L,h}$ will in general have a nontrivial $(L^2)$-kernel and hence, $e^{-t\Delta_{L,h}}$ admits stationary points. However, we need some decay for $k$ in order to bound the convolution integral.	
\end{prob}
In \cite{KP2020}, we were able to derive optimal polynomial decay rates of the heat kernel on the orthogonal complement of the kernel, which means we have to assume $k\perp \ker_{L^2}(\Delta_{L,h})$. These estimates can be used if we find a projection map $\Phi$ which maps from a neighbourhood $\mathcal{U}$ of a given Ricci-flat metric $\hat{h}$ onto a set of Ricci-flat metrics $\mathcal{F}$, and comes with the property that 
$g-\Phi(g)\perp_{L^2}\ker_{L^2}(\Delta_{L,\Phi(g)})$ for $g\in \mathcal{U}$. 
We therefore gauge the Ricci flow with respect to the Ricci-flat ``shadow metric'' $\Phi(g_t)$ instead of a fixed metric.
The Ricci-de Turck flow is then
\begin{align}\label{eq : movingRiccideTurck-1}
	\partial_tg_t=-2\Ric_{g_t}+\mathcal{L}_{V(g_t,\Phi(g_t))}g_t	.
\end{align}
The evolution equation on $k_t=g_t-\Phi(g_t)$ now looks slightly different than \eqref{eq : RiccideTurck0}:
\begin{align}\label{eq : movingRiccideTurck0}
	\partial_tk+\Delta_{L,h}k=(1-D_g\Phi)(g^{-1}*R_h*k*k+g^{-1}*g^{-1}*\nabla k*\nabla k+\nabla ((g^{-1}-h^{-1})*\nabla k)),
\end{align}
where $h=h_t=\Phi(g_t)$ and the additional $D_g\Phi$-term describes the evolution of $h$. In view of \eqref{eq : RiccideTurck0_integral}, the next problem arises:\\
\begin{prob} Compute the heat kernel of $\Delta_{L,h}$ for a time-dependent family $h_t$ of Ricci-flat metrics.
	\end{prob}
We will solve this problem by assuming that $h_t$ converges to a limit $h_{\infty}$. We are then rewriting \eqref{eq : movingRiccideTurck0} as an equation on $\overline{k}_t$ (the part of $k_t$ orthogonal to $\ker_{L^2}(\Delta_{L,h_{\infty}})$), with left hand side given by $\partial_t\overline{k}+\Delta_{L,h_{\infty}}\overline{k}$ and an appropriate modified right hand side containing  $(\Delta_{L,h_{t}}-\Delta_{L,h_{\infty}})(k_t)$. Controlling $\overline{k}_t$ is already good enough to control $k_t$: The orthogonal projection $\Pi^{\perp}_t:\ker_{L^2}(\Delta_{L,h_t})^{\perp}\to \ker_{L^2}(\Delta_{L,h_{\infty}})^{\perp}$ is an isomorphism for $h_t$ close enough to $h_{\infty}$.

\subsubsection{Finding the right Banach space}
The heat kernel of $\Delta_{L,h}$ admits mapping properties of the form
\begin{align*}
	\left\|\nabla^i\circ  e^{-t\Delta_{L,h}}|_{\ker_{L^2}^{\perp}}\right\|_{L^p,L^q}\leq C t^{-\alpha(p,q,i)},
\end{align*}
for some $\alpha(p,q,i)\geq0$.
This suggests  a suitable linear combination of terms
\begin{align*}
	\sup_{t\geq0}\left( t^{\alpha(p,q,i)}\cdot \left\|\nabla^i k_t\right\|_{L^q}\right),
\end{align*}
to define a norm controlling $k_t$.
Since at most second derivatives appear in the evolution equation, it is not necessary to use terms with $i\geq3$.
We also need a norm controlling $h_t$. The evolution equation
\begin{align}\label{eq : shadow metric evolution}
	\partial_th=D_g\Phi(g^{-1}*R_h*k*k+g^{-1}*g^{-1}*\nabla k*\nabla k+\nabla ((g^{-1}-h^{-1})*\nabla k))
\end{align}
suggests a combination of
\begin{align*}
		\sup_{t\geq0}\left\|h_t-\hat{h}\right\|_{L^p},\qquad 	\sup_{t\geq0}\left( t^{\beta(p,q)}\cdot \left\|\partial_t h_t\right\|_{L^q}\right).
\end{align*}
The first norm controls the distance to a Ricci-flat reference metric $\hat{h}$. The second part determines a possible convergence to a limit metric $h_{\infty}$. Here, $\beta(p,q)\geq0$ is suggested from the expected polynomial decay rate of right hand side of \eqref{eq : shadow metric evolution}, coming from the decay of $k_t$ and its derivatives.

\subsubsection{Controlling the inhomogeneous part in the iteration process}
In the iteration process, we will have triples $(h_t^{(i)},h_{\infty}^{(i)},k_t^{(i)})$ consisting of an evolving family of Ricci-flat shadow metrics $h_t^{(i)}$ with a limit $h_{\infty}^{(i)}$ and an evolving family of symmetric $2$-tensors $k_t^{(i)}$ orthogonal to the respective kernels of $h_t^{(i)}$, which form a family of evolving metrics $g_t^{(i)}:=k_t^{(i)}+h_t^{(i)}$ that should eventually converge to a Ricci-de Turck flow.
 In the iteration process, we will have to control terms of the form
\begin{align*}
	\int_0^t e^{-(t-s)\Delta_{L,h^{(i)}_{\infty}}}H(k^{(i)}_s,h^{(i)}_s) ds,
\end{align*}
where $H(k,h)$ here denotes some non-linear expressions in $h$ and $k$.
The polynomial decay rates appearing so far suggest to control integrals of the form
\begin{align}\label{eq : convolution of rational functions}
		\int_0^{t}s^{-\alpha}(t-s)^{-\beta}ds.
\end{align}
If $\min\left\{\alpha,\beta\right\}\geq1$, this integral is not finite. However, we can at least control the interior part of the integral by an elementary lemma. 
 \begin{lem}\label{important_technical_lemma}
 	Let $\alpha,\beta>0$, and define $\gamma=\min\left\{\alpha,\beta\right\}$, $\delta:=\max\left\{\alpha,\beta\right\}$. Then there exists a constant $C=C(\alpha,\beta)$ such that for all $t\geq 2$, we have
 	\begin{align}\label{eq : central technical lemma}
 		\int_1^{t-1}s^{-\alpha}(t-s)^{-\beta}ds\leq C\cdot
 		\begin{cases}	t^{-\gamma}&  \text{ if }\delta>1,\\
 			t^{-\gamma}\log(t)& \text{ if }\delta=1,\\
 			t^{1-\alpha-\beta}& \text{ if }\delta<1,
 		\end{cases}
 	\end{align}
 and the rates on the right hand side are optimal.
 	In particular, we have
 	\begin{align*}
 		\int_1^{t-1}s^{-\alpha}(t-s)^{-\beta}ds\leq C\cdot t^{-\theta}
 	\end{align*}
 	for every $\theta< \min\left\{\alpha,\beta,\alpha+\beta-1\right\}$ and also for $\theta= \min\left\{\alpha,\beta,\alpha+\beta-1\right\}$, if $\max\left\{\alpha,\beta\right\}\neq1$.
 \end{lem}
 \begin{proof}
 	If we substitute $s=r+t/2$, the left hand side of the inequality can be written and estimated from above and below as
 	\begin{align*}
 		&(t-1)^{-\beta}\int_{1-t/2}^{0}(r+t/2)^{-\alpha}dr+  (t-1)^{-\alpha}\int_{0}^{t/2-1}(t/2-r)^{-\beta}dr\\
 		&\qquad\leq \int_{1-t/2}^{0}(r+t/2)^{-\alpha}(t/2-r)^{-\beta}dr	+\int_{0}^{t/2-1}(r+t/2)^{-\alpha}(t/2-r)^{-\beta}dr\\
 		&\qquad\leq (t/2)^{-\beta}\int_{1-t/2}^{0}(r+t/2)^{-\alpha}dr+  (t/2)^{-\alpha}\int_{0}^{t/2-1}(t/2-r)^{-\beta}dr.
 	\end{align*}
 	The rest of the proof follows from elementary calculus and a case by case analysis.
 \end{proof}
\subsubsection{Treating boundary terms of the time integral}\label{subsubseq : boundary terms}
In our proof, we will let an initial metric $g_0$ (which is $L^{p}\cap L^{\infty}$-close to $\hat{h}$) evolve under the Ricci-de Turck flow (with gauge metric $\hat{h}$) up to time $t=1$. The metric $g_1$ and the tensors $h_1=\Phi(g_1)$, $k_1:=g_1-h_1$ are smooth and we can bound all derivatives in terms of the initial data. 
 By starting the iteration argument from time $1$ instead of $0$, we get sequences of metrics and tensors $k^{(i)}_t,h^{(i)}_t$, whose norms do not blow up as $t\searrow 1$. In this way, we can just get rid of the integral in \eqref{eq : convolution of rational functions} from $0$ to $1$.

For the integral from $t-1$ to $t$, there is one term that causes troubles:
\begin{prob}
	We need to control the term
	\begin{align}\label{eq : critical term}
\int_{t-1}^t e^{-(t-s)\Delta_{L,h}}[\nabla ((g_s^{-1}-h_s^{-1})*\nabla k_s)]ds.	
\end{align}	
	\end{prob}
In the iteration process, we need to control up to second derivatives of $k_s^{(i+1)}$ by using only up to second derivatives of $k_s^{(i)}$.
Short-time estimates for parabolic equations show that
\begin{align*}
	\left\| \nabla^2 e^{-(t-s)\Delta_{L,h}}\Theta\right\|_{L^p}&\leq C (t-s)^{-\frac{1}{2}}\left\|\Theta\right\|_{W^{1,p}} \in L^1([t-1,t])
	\\
		  	\left\| \nabla^2 e^{-(t-s)\Delta_{L,h}}\Theta\right\|_{L^p}&\leq C (t-s)^{-1}\left\|\Theta\right\|_{L^p}
		  	\notin L^1([t-1,t]).
\end{align*}
Therefore, we can only estimate
\begin{align*}
\left\|\nabla^2 \int_{t-1}^t e^{-(t-s)\Delta_{L,h}}[\nabla ((g_s^{-1}-h_s^{-1})*\nabla k_s)]ds\right\|_{L^p}
\leq C\sup_{s\in[t-1,t]}\left\| \nabla ((g_s^{-1}-h_s^{-1})*\nabla k_s)\right\|_{W^{1,p}},
\end{align*}
but the right hand side contains third derivatives of $k$ and thus, the iteration argument can not be closed. Instead, we put this part of the integral to the left hand side of the equation as follows: For an initial tensor $k_1$, and a fixed time $t>1$, we solve the equation
\begin{align*}
	\partial_s k+\Delta_{L,h^{(i)}_{\infty}}k=0,\qquad \text{ for }s\in [1,\max\left\{t-1,1\right\}]
\end{align*}
and afterwards the equation
\begin{align*}
\partial_t k+\Delta_{L,g^{(i)},h^{(i)}}k=0	,\qquad \text{ for }s\in [\max\left\{t-1,1\right\},t].
\end{align*}
Here, $\Delta_{L,g^{(i)},h^{(i)}}$ is a slightly modified Lichnerowicz Laplacian which captures exactly the critical terms just discussed.
It turns out that the associated evolution operator admits the same mapping properties as the Lichnerwicz Laplacian. 
For large $t$, the short-time estimates for $s\in [t-1,t]$ do not destroy the decay rates generated by the Lichnerowicz Laplacian for $s\in [1,t-1]$.
 \subsection{Structure of the paper}
 In Section \ref{sec : Ricci-flat metrics}, we describe the space of gauged Ricci-flat metrics $\F$ in detail, study the asymptotics of its elements and derive sharp estimates for projection maps defined by elements in $\F$. In Section \ref{sec : short-time estimates}, we derive novel Shi-type estimates for $L^p$-norms under parabolic equations. In Section \ref{sec : RdT mixed evolution}, a suitable integral expression for solutions of the Ricci-de Turck flow is derived. 
Section \ref{sec : iteration map} is the technical core of the paper, in which we study the precise mapping properties of the iteration map which comes from the aforementioned integral expression. 
These estimates are used to establish the Ricci-de Turck flow as a fixed point of this map in Section \ref{sec : long-time existence convergence} and to conclude the main Theorem \ref{thm : mainthm 1 introduction}. 
We conclude by proving the remaining theorems of the introduction at the end of Section \ref{sec : long-time existence convergence}.
 
\vspace{3mm}
\noindent\textbf{Acknowledgments.}
Part of this work was carried out while the authors were visiting the Institut Mittag-Leffler during the program \emph{General Relativity, Geometry and Analysis} in Fall 2019, supported by the Swedish Research Council under grant no.\ 2016-06596.
We wish to thank the institute for their hospitality and for the excellent working conditions provided. We furthermore want to thank the referees for careful reading and many helpful comments which helped to improve the paper.
The work of the authors is supported by the German Research Foundation through the grant KR 4978/1-1 in the framework of the priority program 2026: \textit{Geometry at infinity}.
\section{The space of gauged Ricci-flat metrics}\label{sec : Ricci-flat metrics}

In order to set up the stability problem, we introduce the space of metrics we are considering and the space of gauged Ricci-flat metrics, which we show convergence to.

\subsection{Asymptotic structure}
Let from now on $\M$ denote the space of smooth Riemannian metrics on $M$.
As explained in the introduction, we will prove \emph{dynamical stability}, i.e.\ that any Ricci flow \emph{starting close to} $\hat{h}$ converges to a Ricci-flat metric \emph{near} $\hat{h}$. 
Due to the diffeomorphism invariance, the space of Ricci-flat metrics near $\hat{h}$ is infinite dimensional within $\M$.
In order to get a finite dimensional space of Ricci-flat metrics near $\hat{h}$, we therefore need to impose the de Turck gauge. 
This corresponds to considering the following set
\[
\F := \{h \in \M \mid 2 \Ric_h = \L_{V(h, \hat{h})}h \},
\]
where $V(h, \hat{h})$ is the de Turck vector field, defined by
\[
\hat{h}(V(h, \hat{h}), \cdot) := - \div_h \hat{h} + \frac12 d (\tr_h \hat{h})
\]
or locally by
\[
V(h,\hat{h})^l=h^{ij}(\Gamma(h)_{ij}^l-\Gamma (\hat{h})_{ij}^l).
\]
In \cite{DK17}*{Section 2}, it was shown that for any Ricci-flat ALE manifold $(M,\hat{h})$ and any $p\in (1,\infty)$, there exists a $L^{[p,\infty]}$-neighbourhood $\U$ such that 
\[
\F\cap\U= \{h \in \U \mid  \Ric_h =0, V(h, \hat{h})=0 \}=\F_{\U}.
\]
In particular, if $\hat{h}$ is integrable, then $\F\cap\U$ is a smooth manifold.

For the analysis performed in this subsection we need weighted Sobolev spaces which are defined as follows: Fix a point $x \in M$ and pick a smooth function such that 
\[
\rho(y) = \sqrt{1 + d(x, y)^2}
\]
for all $y\in M$ outside a compact set, where $d$ is the Riemannian distance.
\begin{definition} \label{def: Weighted ALE}
	Let $V$ be a Riemannian vector bundle over $M$.
	For any ${k \in \N_0}$, ${p \in [1, \infty]}$ and any ${\delta \in \R}$, the weighted Sobolev space $W_\delta^{k,p}(V)$ is the space of $V$-valued sections $u \in W^{k,p}_{loc}(V)$ such that
	\[
	\norm{u}_{W^{k, p}_\delta} := \begin{cases}\sum_{j = 0}^k \left(\int_{M}\abs{(\rho\n)^ju}^p\rho^{-\delta p - n}\dv\right)^{1/p}, & \text{ if }p<\infty,\\
	\sum_{j = 0}^k\mathrm{ess}\sup_M |\rho^{-\delta}(\rho\n)^ju| & \text{ if }p=\infty
	\end{cases}
	\]
	is finite.
	We also use the notation $L_\delta^p := W_\delta^{0,p}$ and $H_\delta^k := W_\delta^{k,2}$. 
\end{definition}
\begin{prop} \label{prop: near metrics}
	Let $(M,\hat{h})$ be an {integrable Ricci-flat ALE manifold with a parallel spinor} and $p\in (1,\infty)$.
	Then, there is an open $L^{[p,\infty]}$-neighbourhood $\U$ of $\hat{h}$, such that if $h \in \U { \cap \F}$, then
	\[
	h - \hat{h} \in \O_{\infty}\left(\rho^{-n}\right).
	\]
\end{prop}
Here and henceforth the notation $a \in \O_{\infty}\left(r^\a\right)$ means that for each $k \in \N_0$, there are constants $C_k > 0$ such that $\abs{\n^k a} \leq C_k \rho^{\a - k}$ for large enough $\rho$.
Proposition \ref{prop: near metrics} is an improvement of Theorem 2.7 and Remark 2.8 in \cite{DK17}, where it was proven that
\[
h - \hat{h} \in \O_\infty \left(\rho^{1-n}\right).
\]
The tool to improve this result is a result of our companion paper \cite{KP2020}, which says that {under the assumption of a parallel spinor}, elements in the $L^2$-kernel of $\Delta_L$ decay as $\rho^{-n}$, i.e.\
\begin{align}\label{eq : decay kernel}
	\ker_{L^2}\left(\Delta_L\right) 
		\subset \O_\infty\left(\rho^{-n}\right).
\end{align}
By linearizing the defining equation in $\F$, we note that any $k\in T_h\mathcal{F}$ satisfies
\begin{align*}
0
	&= \Delta_{L,h}k+\mathcal{L}_{\langle k,\Gamma(h)-\Gamma(\hat{h})\rangle}h, \text{ where}\\ 
	\langle k,\Gamma(h)-\Gamma(\hat{h})\rangle^m
	&:=h^{iq}h^{jl}k_{ij}(\Gamma(h)_{ql}^m-\Gamma(\hat{h})_{ql}^m)
\end{align*}
and from the proof of  \cite{DK17}*{Thm.\ 2.7}, it follows that
\begin{equation} \label{eq: tangent space decay}
	T_h\mathcal{F} \subset \O_\infty\left(\rho^{1-n}\right).
\end{equation}
The key to prove Proposition \ref{prop: near metrics} is to first improve \eqref{eq: tangent space decay} as follows:
\begin{prop}\label{prop : kernel properties}
	Let $(M,\hat{h})$ be an {integrable Ricci-flat ALE manifold with a parallel spinor}. 
	Then there is a small $L^{[p,\infty]}$-neighbourhood $\mathcal{U}$ with the following properties:
	\begin{itemize}
		\item[(i)] $\dim\ker_{L^2}(\Delta_L)$ is constant for all $h\in \mathcal{U}\cap\mathcal{F}$. In particular, we can choose for each $h\in \mathcal{U}\cap\mathcal{F}$ a set of tensors $\left\{e_1(h),\ldots e_m(h)\right\}$ smoothly depending on $h$ that forms an $L^2(h)$-orthonormal basis of $\ker_{L^2}(\Delta_L)$.
		\item[(ii)] For all $h\in \mathcal{U}\cap\mathcal{F}$, we have
		\[
		T_h\mathcal{F} \subset \O_\infty\left(\rho^{-n}\right).
		\]
	\end{itemize}
\end{prop}

The proof relies on the following standard fact, which we prove for convenience of the reader:

\begin{lemma} \label{le: Fredholm theory}
For given Banach spaces $X, Y$, assume that $A: X \to Y$ is a Fredholm operator.
Then there is an $\e > 0$, such that if $B: X \to Y$ is a bounded operator with $\norm{B}_{\mathrm{op}} < \e$, then
\[
	A + B : X \to Y
\]
is a Fredholm operator and $\dim\left(\ker\left(A + B\right)\right) \leq \dim(\ker(A))$.
\end{lemma}
\begin{proof}[Proof of Lemma \ref{le: Fredholm theory}]
Let us first assume that $A$ is invertible, i.e.~has a bounded inverse $A^{-1}: Y \to X$.
In that case, formally,
\begin{align*}
	\left(A + B\right)^{-1}
		&= \left((1 - (-BA^{-1})A )\right)^{-1} \\
		&= A^{-1} \left(1 - (-BA^{-1})\right)^{-1} \\
		&= A^{-1} \sum_{j = 0}^\infty (-BA^{-1})^j,
\end{align*}
which converges if $\norm{B}_{\mathrm{op}} < \norm{A^{-1}}_{\mathrm{op}}^{-1}$.

Assume now that $A$ is a Fredholm operator and write $X = X^c \oplus \ker(A)$ and $Y = \im(A) \oplus Y^c$, where $X^c$ and $Y^c$ denote closed complements.
We get induced continuous maps
\begin{align*}
	\pr : X 
		&\to \ker(A), \\
	\iota : Y^c
		&\to Y.
\end{align*}
We may thus define the invertible operator
\[
	\hat A: X \oplus Y^c \to Y \oplus \ker(A)
\]
by $\hat A(u, v) = (A u + \iota(v), \pr(u))$.
Given $B: X \to Y$, define now 
\[
	\hat B: X \oplus Y^c \to Y \oplus \ker(A),
\]
by $\hat B(u, v) := (Bu, 0)$.
By what we showed above, $\hat A + \hat B$ is invertible if $\norm{B}_{\mathrm{op}}$ is small enough.
Moreover, note that
\[
	\left( \hat A + \hat B \right) (u, v)
		= \left( (A + B)(u) + \iota(v), \pr(u) \right).
\]
Because $\hat A + \hat B $ is surjective and $Y^c$ is finite-dimensional, it follows that $\im(A+B)\subset Y$ has finite codimension.
It also follows that the map
\begin{align*}
	X
		&\to Y \oplus \ker(A) \\
	u
		&\mapsto \left( (A + B)(u), \pr(u) \right)
\end{align*}
is injective.
In particular, the restriction of this map to the kernel of $A + B$ maps injectively to $\ker(A) = \im(\pr)$.
We therefore conclude that
\[
	\dim \left( \ker\left(A + B \right) \right)
		\leq \dim\left( \im(\pr) \right) 
		= \dim \left( \ker\left(A \right) \right).
\]
This completes the proof of the lemma.
\end{proof}

\begin{proof}[Proof of Proposition \ref{prop: near metrics}]
Let us start with the proof of (i):
	Since $\hat{h}$ is integrable, $\mathcal{U}\cap\mathcal{F}$ is a smooth manifold and all tangent spaces $T_{h}\mathcal{F}$ have the same dimension for $h\in \mathcal{U}\cap\mathcal{F}$. 
	We first construct an injection $\iota:T_{h}\mathcal{F}\to \ker_{L^2}(\Delta_{L,h})$. 
	Let $k\in T_h\mathcal{F}$. 
	Then $k$ satisfies
	\begin{align*}
		0
			=\Delta_{L,h}k+\mathcal{L}_{\langle k,\Gamma(h)-\Gamma(\hat{h})\rangle}h.
	\end{align*}
	We now add a gauge term to $k$ to get an element in $\ker_{L^2}(\Delta_{L,h})$. 
	More precisely, let $X$ be a vector field and consider the tensor $\overline{k}=k+\mathcal{L}_Xh$. 
	Then we have, due to standard commutation formulas for operators on Ricci-flat metrics that
	\begin{align*}
		\Delta_{L,h}\overline{k}
			=\Delta_{L,h}(\mathcal{L}_Xh)-\mathcal{L}_{\langle k,\Gamma(h)-\Gamma(\hat{h})\rangle}h=\mathcal{L}_{\Delta X}h-\mathcal{L}_{\langle k,\Gamma(h)-\Gamma(\hat{h})\rangle}h.
	\end{align*}
	Thus, it suffices to solve the equation 
	\begin{align*}
		\Delta X
			=\langle k,\Gamma(h)-\Gamma(\hat{h})\rangle
	\end{align*}
	in a suitable function space.
	Observe that $\Gamma(h)-\Gamma(\hat{h})\in H^i_{\delta_1}$ for $\delta_1>-n$ and $k\in H^i_{\delta_2}$ for $\delta_2>1-n$ (for any $i\in\N_0$ in both cases). 
	Therefore, $\langle k,\Gamma(h)-\Gamma(\hat{h})\rangle\in H^i_{\delta}$ for $\delta>1-2n$ and $i>\frac{n}{2}+1$ due to the estimate
	\begin{align*}
		\left\|\langle k,\Gamma(h)-\Gamma(\hat{h})\rangle\right\|_{H^i_{\delta_1+\delta_2}}
			\leq C\left\|\Gamma(h)-\Gamma(\hat{h})\right\|_{H^i_{\delta_1}}\left\|k\right\|_{H^i_{\delta_2}}.
	\end{align*}
	Now the connection Laplacian $\Delta :H^i_{\delta}(TM)\to H^{i-2}_{\delta-2}(TM)$ is a Fredholm operator for the nonexceptional values $\delta\in \R\setminus \left( \left\{0,1,2,\ldots\right\}\cup\left\{2-n,1-n,-n\ldots\right\} \right)$ (see e.g.\ \cite{pacini}*{Corollary 9.4}). 
	Due to the identity $\Delta |X|^2=-2|\nabla X|^2$ for harmonic vector fields, the maximum principle implies that every bounded harmonic vector field is parallel. 
	On the other hand, since $(M,g)$ is ALE and Ricci-flat, it cannot contain parallel vector fields, due to the Cheeger-Gromoll splitting theorem.
	Therefore, we even have  $\ker(\Delta)\cap H^i_{\delta}(TM)=\left\{0\right\}$ for $\delta<1$.
	As a consequence, duality arguments (see e.g.\ \cite{pacini}*{Section 10}) imply that $\Delta:H^i_{\delta}(TM)\to H^{i-2}_{\delta-2}(TM)$ is an isomorphism for all $\delta\in(1-n,1)\setminus \left\{2-n,0\right\}$. 
	Therefore, we find for each $\delta\in (1-n,1)$ and $i> \frac{n}{2}+1$ a unique solution $X\in H^{i+2}_{\delta}(TM)$ of the equation  
	\begin{align*}
		\Delta X=\langle k,\Gamma(h)-\Gamma(\hat{h})\rangle.
	\end{align*}
	Moreover, the vector field $X$ is the same for all possible choices for $i$ and $\delta$.
	
	Now, $\Delta_L\overline{k}=0$ and $\overline{k}\in \ker_{L^2}(\Delta_L)$ because $\mathcal{L}_Xh\in H^{i-1}_{\delta-1}(S^2M)\subset H^{i-1}_{-\frac{n}{2}}(S^2M) \subset L^2(S^2M)$. 
	Therefore we can define the desired map by $\iota: k\mapsto k+\mathcal{L}_Xh$ where $X\in  H^i_{1-\frac{n}{2}}(TM)$ is defined as the unique solution of the above equation. 
	This map is injective because
	\begin{align*}
		\left\|\mathcal{L}_Xh\right\|_{H^{i-1}_{-\frac{n}{2}}}
			\leq C\left\|X\right\|_{H^{i}_{1-\frac{n}{2}}}
			\leq C\left\|\Delta X\right\|_{H^{i-2}_{-1-\frac{n}{2}}}
			\leq C\left\|\Gamma(h)-\Gamma(\hat{h})\right\|_{H^{i-2}_{-1}}\left\|k\right\|_{H^{i-1}_{-\frac{n}{2}}}
	\end{align*}
	and therefore,
	\begin{align*}
		\left\| i(k)\right\|_{H^{i-1}_{-\frac{n}{2}}}\geq \left\|k\right\|_{H^{i-1}_{-\frac{n}{2}}}-\left\|\mathcal{L}_Xh\right\|_{H^{i-1}_{-\frac{n}{2}}}\geq \frac{1}{2}\left\|k\right\|_{H^{i-1}_{-\frac{n}{2}}},
	\end{align*} 
	provided that the neighbourhood $\mathcal{U}$ is chosen small enough.
	We get
	\begin{align*}
		\dim\ker_{L^2}(\Delta_{L,h})\geq \dim  T_h\mathcal{F}=\dim T_{\hat{h}}\mathcal{F}=\dim\ker_{L^2}(\Delta_{L,\hat{h}}).
	\end{align*}
	Because of \eqref{eq : decay kernel},
	$\dim\ker_{L^2}(\Delta_{L,h}) = \dim \left( \ker(\Delta_{L,h})\cap H^i_{\delta}(S^2M) \right)$ for all $i\in\N_{0}$ and $\delta\in (-n,0)$. Choose $\delta$ nonexceptional.
	By Lemma \ref{le: Fredholm theory},
	\begin{align*}
		\dim\ker_{L^2}(\Delta_{L,h})=\dim\ker_{H^i_{\delta}}(\Delta_{L,h})\leq \dim\ker_{H^i_{\delta}}(\Delta_{L, {\hat{h}}})=\dim\ker_{L^2}(\Delta_{L,{\hat{h}}})
	\end{align*}
	as well. We conclude (i).
	
	For the proof of (ii), let $k,X$ and $\overline{k}$ as above. 
	By \eqref{eq : decay kernel} and \eqref{eq: tangent space decay}, we know that
	\begin{align*}
		k\in T_h\mathcal{F} \subset\O_\infty\left(\rho^{1-n}\right), k+\mathcal{L}_Xh=\overline{k} \in \ker_{L^2}(\Delta_{L})\subset \O_\infty\left(\rho^{-n}\right).
	\end{align*}
	However, we also have seen that $X\in H^{i}_{\delta}(TM)$ for all $\delta>1-n$ and $i>\frac{n}{2}+3$ so that Sobolev embedding implies 
	\begin{align*}
		X\in \O_\infty\left(\rho^{\delta}\right)
	\end{align*}
	for all $\delta>1-n$. Standard arguments (cf. \cite{KP2020}*{Proposition 4.3}), using the equation
	\begin{align*}
		\Delta X=\langle k,\Gamma(h)-\Gamma(\hat{h})\rangle\in \O_\infty\left(\rho^{1-2n}\right)
	\end{align*}
	show that we actually have
	\begin{align*}
		X\in \O_\infty\left(\rho^{1-n}\right),
	\end{align*}
	hence
	\begin{align*}
		\mathcal{L}_Xh\in \O_\infty\left(\rho^{-n}\right).
	\end{align*}
	Therefore, $k\in T_h\mathcal{F}$ satisfies
	\begin{align*}
		k=\overline{k}-\mathcal{L}_Xh\in \O_\infty\left(\rho^{-n}\right),
	\end{align*}
	which finishes the proof.
\end{proof}

\begin{proof}[Proof of Proposition \ref{prop: near metrics}]
	This is now an immediate consequence of Proposition \ref{prop : kernel properties} (ii).
\end{proof}

\begin{remark}
Due to \eqref{eq : decay kernel}, $\ker_{L^2}(\Delta_{L, \hat{h}})\subset L^{[p,\infty]}$ for every $p\in (1,\infty)$. 
Therefore
\begin{equation} \label{eq: Lp infinity decomposition}
	\ker_{L^2}(\Delta_{L, \hat{h}}) 
		\oplus \left(\ker_{L^2}(\Delta_{L, \hat{h}})^{\perp_{\hat h}} \cap L^{[p,\infty]}\right)
		\cong L^{[p,\infty]}(S^2M).
\end{equation}
By the identification $T_h \F \cong \ker_{L^2}(\Delta_{L, h})$, we also have
\begin{equation} \label{eq: Lp infinity decomposition 2}
	T_h \F 
		\oplus \left(\ker_{L^2}(\Delta_{L, \hat{h}})^{\perp_{\hat h}} \cap L^{[p,\infty]}\right)
		\cong L^{[p,\infty]}(S^2M).
\end{equation}
\end{remark}

\subsection{A projection map onto \texorpdfstring{$\F$}{F}}\label{subsec: proj map}
Throughout the rest of this section, let $(M,\hat{h})$ be an integrable Ricci-flat ALE manifold with a parallel spinor. 
In the previous subsection, we developed some understanding of the space $\F$ of gauged Ricci-flat metrics.
We would like to construct a smooth map
\begin{align*}
	\Phi: \U \subset \M \to \F
\end{align*}
on some open neighbourhood $\U$ of $\hat{h}$, provided that $\F$ is a smooth manifold. The projection should have the property that
\begin{align*}
	g-\Phi(g)\in \ker_{L^2}(\Delta_{L, \Phi(g)})^{\perp_{\Phi(g)}},
\end{align*}
where $\perp_{\Phi(g)}$ denotes the $L^2$-orthogonal complement with respect to $\Phi(g)$.
The map we want to define is illustrated schematically in Figure \ref{fig: tangent spaces in F}.
\begin{figure*}
  \begin{center}
    \includegraphics[scale = 2]{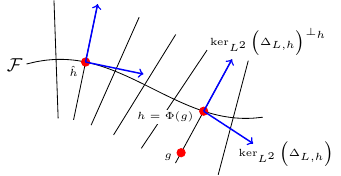}
  \end{center}
  \caption{
	Since $\F$ is defined with respect to $\hat h$, the tangent space of $\F$ is the $L^2$-kernel of the linearized operator precisely at $\hat h$.
  	}
  	\label{fig: tangent spaces in F}
\end{figure*}
The construction goes as follows.
First define the map 
\begin{align*}
	\Psi^{\hat{h}}: \F \times \ker_{L^2}(\Delta_{L, \hat{h}})^{\perp_{\hat{h}}}
	&\to \M, \\
	(h, k)
	&\mapsto h + k - \sum_{j = 1}^m (k, e_j(h))_{L^2(h)} e_j(h),
\end{align*}
where $\left(e_j(h)\right)_{j = 1}^m$ is an $L^2(h)$-orthonormal basis of $\ker_{L^2}(\Delta_{L, h})$ (depending smoothly on $h$).
\begin{remark} \label{rmk: Psi hat h computations}
	Note that $\Psi^{\hat{h}}(h, 0) = h$, for all $h \in \F$, and that
	\[
	\Psi^{\hat{h}}(h, k) \in h + \ker_{L^2}\left(\Delta_{L, h}\right)^{\perp_h}.
	\]
\end{remark}

In fact, we have the following lemma:
\begin{lemma} \label{le: Banach space isomor}
For every $p\in (1,\infty)$, there is an open $L^{[p,\infty]}$-neighbourhood $\V$ of $(\hat{h}, 0)$ such that
	\[
	\Psi^{\hat{h}}|_{\V}: \V \to \M
	\]
	is a diffeomorphism of Banach manifolds onto its image.
\end{lemma}
\begin{proof}
	Given an $\tilde h \in \ker_{L^2}(\Delta_{L, \hat{h}})$,
the first assertion in Remark \ref{rmk: Psi hat h computations} implies that 
	\[
		D \Psi^{\hat h}_{(\hat h, 0)}(\tilde h, 0)
			= \tilde h,
	\]
	where we have identified
	\[
		T_{\hat{h}}\F = \ker(\Delta_{L, \hat{h}}).
	\]
	Moreover, given a $\tilde k \in (\ker_{L^2}(\Delta_{L, \hat{h}})^{\perp_{\hat h}} \cap L^{[p,\infty]})$, note that
	\begin{align*}
		D \Psi^{\hat h}_{(\hat h, 0)}(0, \tilde k)
			= \tilde k - \sum_{j = 1}^m (\tilde k, e_j(\hat h))_{L^2(\hat h)} e_j(\hat h)
			= \tilde k.
	\end{align*}
	Hence $D\Psi^{\hat{h}}_{(\hat{h}, 0)}=\id$ with respect to the decomposition \eqref{eq: Lp infinity decomposition}.
	The inverse function theorem now proves the claim.
\end{proof}
\noindent Using the (smooth) projection map
\[
	\pi: 
		\F \times \ker(\Delta_{L, \hat{h}})^{\perp_{\hat h}} 
			\to \F,
\]
and the neighbourhood
\[
\U:=\Psi^{\hat{h}}(\V),
\]
we may use the previous lemma to define the smooth map
\begin{align*}
	\Phi: \U 
	&\to \F \\
	g
	&\mapsto \pi \circ (\Psi^{\hat{h}})^{-1}(g).
\end{align*}
\begin{rem}\label{rem : proj indep ref metric}
Note that we used $\hat{h}$ in the definition of $\Psi^{\hat{h}}$ in order to fix the reference domain $\F \times \ker_{L^2}(\Delta_{L, \hat{h}})^{\perp_{\hat{h}}}$. We could also choose a different reference metric $\tilde{h}\in \mathcal{F}$ and define $\Psi^{\tilde{h}}$ by exactly the same formula as $\Psi^{\hat{h}}$, but now on $\F \times \ker_{L^2}(\Delta_{L, \tilde{h}})^{\perp_{\tilde{h}}}$, and analogously define $\Phi^{\tilde{h}}(g)=\pi \circ (\Psi^{\tilde{h}})^{-1}(g)$.
We claim that $\Phi = \Phi^{\tilde h}$ in their common domain of definition $\U \cap \tilde \U$.
The reason is that there is a geometrically invariant interpretation of $\Phi$.
Lemma \ref{le: Banach space isomor} implies that for each $g \in \U$, there is a unique $h \in \F \cap \U$ and a unique $k$, such that 
\[
	g = h + k,
\]
with $k \in T_h\F^\perp$. 
Indeed, here $h = \Phi(g)$ and $k = g - \Phi(g)$.
The same geometric interpretation is true for $\Phi^{\tilde h}$, so by uniqueness it follows that $\Phi(g) = \Phi^{\tilde h}(g)$ for all $g \in \U \cap \tilde \U$, as claimed.
\end{rem}

We will use $\Phi$ to construct a curve of Ricci-flat metrics $h_t:=\Phi(g_t)$ that `shadow' $g_t$ and satisfy
\[
	k_t
		:=g_t-h_t\perp_{h_t}\ker_{L^2}(\Delta_{L,h_t}).
\]
The goal will then be to show that $k_t$ converges (fast enough) to $0$ and that $h_t$ converges to a limit metric, i.e.\ that the Ricci flow converges to a metric $g_\infty \in \F$.

\subsection{Properties of projection maps}
Let further $p\in (1,\infty)$ and $\U$ be an $L^{[p,\infty]}$-neighbourhood which is so small that the projection map $\Phi:\U\to \U\cap\F$ from the previous subsection is defined. 
Since $h-\hat{h}\in \O_{\infty}(r^{-n})$, the appearing norms and covariant derivatives can be taken with respect to any $h\in\U\cap\F$.

Here, we collect a few properties of projection maps.
For $h \in \U\cap\F$, let $\left\{e_i(h)\right\}_{1\leq i\leq m}$ be an orthonormal basis of $\ker_{L^2}(\Delta_{L,h})$ which depends smoothly on $h$. For $p\in (1,\infty)$, we define the natural projection maps
\begin{align*}
	\Pi^{\parallel}_{h}:L^{[p,\infty]}(S^2M)&\to \ker_{L^2}(\Delta_{L,h}),\qquad
	&k\mapsto \sum_{i=1}^m(k,e_i(h))_{L^2(h)}e_i(h),\\
	\Pi^{\perp}_{h}:L^{[p,\infty]}({S^2M})&\to \ker_{L^2}(\Delta_{L,h})^{\perp}\cap L^{[p,\infty]}(S^2M),\qquad
	&k\mapsto k-\sum_{i=1}^m(k,e_i(h))_{L^2(h)}e_i(h).
\end{align*}
If we choose $\U$ small enough, the matrix
\[
A_{ij}=(e_i(h),e_j(\bar{h}))_{L^2(\bar{h})}
\]
is invertible for every pair $h,\bar{h}\in\U\cap\F$. In other words, the map
\[
\Pi^{\parallel}_{h,\bar{h}}:=\Pi^{\parallel}_{\bar{h}}|_{\ker_{L^2}(\Delta_{L,h})}:
\ker_{L^2}(\Delta_{L,h})\to \ker_{L^2}(\Delta_{L,\bar{h}})
\]
is invertible.
\begin{lem}\label{lem : invert projection}
	If $\Pi^{\parallel}_{\bar{h}, h}$ is invertible, the map
	\[
	\Pi^{\perp}_{h,\bar{h}}:=\Pi^{\perp}_{\bar{h}}|_{\ker_{L^2}(\Delta_{L,h})^{\perp}}:
	\ker_{L^2}(\Delta_{L,h})^{\perp}\cap L^{[p,\infty]}(S^2M)\to \ker_{L^2}(\Delta_{L,\bar{h}})^{\perp}\cap L^{[p,\infty]}(S^2M)
	\]
	is also invertible for every $p\in (1,\infty)$ and its inverse is given by
	\begin{align}\label{eq : invert projection}
		(\Pi^{\perp}_{h,\bar{h}})^{-1}=\id-(\Pi^{\parallel}_{\bar{h},h})^{-1}\circ\Pi^{\parallel}_h.
	\end{align}
\end{lem}
\begin{proof}
	For $\bar{k}\in \ker_{L^2}(\Delta_{L,\bar{h}})^{\perp}\cap L^{[p,\infty]}(S^2M)$, let
	\[
	k:=\bar{k}-(\Pi^{\parallel}_{\bar{h},h})^{-1}\circ(\Pi^{\parallel}_h(\bar{k})).
	\]
	Then, $k\in \ker_{L^2}(\Delta_{L,h})^{\perp}\cap L^{[p,\infty]}(S^2M)$, because
	\[
	\Pi^{\parallel}_{h}(k)=\Pi^{\parallel}_{h}(\bar{k})-\Pi^{\parallel}_{h}\circ(\Pi^{\parallel}_{\bar{h},h})^{-1}\circ(\Pi^{\parallel}_h(\bar{k}))=\Pi^{\parallel}_{h}(\bar{k})-\Pi^{\parallel}_{h}(\bar{k})=0.
	\]
	Moreover,
	\begin{align*}
	\Pi^\perp_{\bar{h}}(k)
		&=k-\Pi^{\parallel}_{\bar{h}}(k)\\
		&=\bar{k}-(\Pi^{\parallel}_{\bar{h},h})^{-1}\circ(\Pi^{\parallel}_h(\bar{k}))-\Pi^{\parallel}_{\bar{h}}[\bar{k}-(\Pi^{\parallel}_{\bar{h},h})^{-1}\circ(\Pi^{\parallel}_h(\bar{k})]\\
		&=\bar{k}-\Pi^{\parallel}_{\bar{h}}(\bar{k})-(\Pi^{\parallel}_{\bar{h},h})^{-1}\circ(\Pi^{\parallel}_h(\bar{k}))+\Pi^{\parallel}_{\bar{h}}\circ(\Pi^{\parallel}_{\bar{h},h})^{-1}\circ\Pi^{\parallel}_h(\bar{k})=\bar{k}.
	\end{align*}
	In the last equation, we used that $\Pi^{\parallel}_{\bar{h}}(\bar{k})=0$ and that $(\Pi^{\parallel}_{\bar{h},h})^{-1}\circ\Pi^{\parallel}_h(\bar{k})\in \ker_{L^2}(\Delta_{L,\bar{h}})$. Hence, $\Pi^{\perp}_{h,\bar{h}}$ is invertible because we constructed its inverse explicitly.
\end{proof}
\begin{rem}\label{rem : invert projection}
	Later, we need to obtain estimates on the difference 	$(\Pi^{\perp}_{h_1,\bar{h}_1})^{-1}-(\Pi^{\perp}_{h_2,\bar{h}_2})^{-1}$, which a priori doesn't make sense as the operators are defined on different spaces. 
	We can however make sense of this difference on all of $L^q(S^2M)$ and $q\in (1,\infty)$ by using the right hand side of \eqref{eq : invert projection} as a definition. 
	For convenience, we don't change the notation for this extension.
\end{rem}
\begin{lem} \label{lem: DPhi vanishes}
	For $g\in\U$, $D_g\Phi$ vanishes on $\ker_{L^2}(\Delta_{L,h})^{\perp}\cap L^{[p,\infty]}(S^2M)$, where $h = \Phi(g)$.
\end{lem}
\begin{proof}
	Let $g\in\U$ be fixed and consider a curve $g_t$ in $\U$ with $g_0=g$. 
	We may split $g_t=h_t+k_t$ where $h_t$ is a curve in $\F$ and $k_t\in \ker_{L^2}(\Delta_{L,h_t})^{\perp}$. 
	Let $h=h_0$ and $k=k_0$. 
	By the previous lemma, we can write $k_t=\Pi^{\perp}_{h_t}(\tilde{k}_t)$ with $\tilde{k}_t\in \ker_{L^2}(\Delta_{L,h})^{\perp}$. 
	Differentiating at $t=0$ yields
	\[
	g_0'=h_0'+\Pi^{\perp}_h(\tilde{k}'_t)|_{t=0}+D\Pi_{(h,\tilde{k}_t)}^{\perp}(h_t')|_{t=0}=h_0'+\tilde{k}_0'+D\Pi^\perp_{(h,\tilde{k}_0)}(h_0'),
	\]
	where $D\Pi^{\perp}_{(h,k)}$ is the Fr\'echet derivative of $\Pi^\perp_{(\cdot)}(\cdot):\U\cap\F\times L^{[p,\infty]}(S^2M)\to L^{[p,\infty]}(S^2M)$ in the first component at $(h,k)$. It is a map
	\begin{align*}
		D\Pi_{(h,k)}:T_h\F\to 	L^{[p,\infty]}(S^2M).
	\end{align*}
	With respect to the decomposition \eqref{eq: Lp infinity decomposition 2} and corresponding projection maps
	\begin{align*}
		\mathrm{proj}_{T_h\F}&:L^{[p,\infty]}(S^2M)\to T_{h}\F,\\ \mathrm{proj}_{\ker_{L^2}(\Delta_{L,h})^{\perp}}&:L^{[p,\infty]}(S^2M)\to \ker_{L^2}(\Delta_{L,h})^{\perp}\cap L^{[p,\infty]}(S^2M),
	\end{align*}
	the differential $D_{h,k}\Psi^h$ (where $\Psi^h$ is as in Remark \ref{rem : proj indep ref metric}) reads
	\[
	D_{h,k}\Psi^h
		=\begin{pmatrix}
		\mathrm{id}_{T_h\F}+\mathrm{proj}_{T_h\F}\circ D\Pi^\perp_{(h,k)}(\cdot) & 0\\
		\mathrm{proj}_{\ker_{L^2}(\Delta_{L,h})^{\perp}}\circ D\Pi^\perp_{(h,k)}(\cdot) & \mathrm{id}_{\ker_{L^2}(\Delta_{L,h})^{\perp}}
	\end{pmatrix},
	\]
	where $h=\Phi(g)$ and $k=g-h$. 
	The differential of $\Phi$ is therefore given by
	\begin{align}\label{eq : differential Phi}
		D_g\Phi
			=\mathrm{proj}_{T_h\F}\circ (D_{h,k}\Psi^h)^{-1}=(\mathrm{id}_{T_h\F}+\mathrm{proj}_{T_h\F}\circ D\Pi^\perp_{(h,k)}(.))^{-1}\circ \mathrm{proj}_{T_h\F},
	\end{align}
	where we have used that $(D_{h,k} \Psi^h)^{-1}$ is on the form
	\[
		\left( D_{h,k}\Psi^h\right)^{-1}
		=\begin{pmatrix}
			\left( \mathrm{id}_{T_h\F}+\mathrm{proj}_{T_h\F}\circ D\Pi^\perp_{(h,k)}(\cdot) \right)^{-1} & 0\\
			* & \mathrm{id}_{\ker_{L^2}(\Delta_{L,h})^{\perp}}
	\end{pmatrix}.
	\]
	The assertion is now immediate.
\end{proof}

\begin{lem}\label{lem : projection F}
	If $\U$ was chosen small enough, then for every $h\in\U\cap\F$, the map
	\begin{equation} \label{eq: F h map}
	\Pi_{\F,h}^{\parallel}=\Pi_h^{\parallel}|_{T_h\F}:T_h\F\to \ker_{L^2}(\Delta_{L,h})
	\end{equation}
	is invertible. In this case the projection maps $\mathrm{proj}_{T_h\F}$ and $\mathrm{proj}_{\ker_{L^2}(\Delta_{L,h})^{\perp}}$ are given by
	\begin{align*}
		\mathrm{proj}_{T_h\F}(k)
			&=(\Pi_{\F,h}^{\parallel})^{-1}(\Pi^{\parallel}_{h}(k)),\\
		\mathrm{proj}_{\ker_{L^2}(\Delta_{L,h})^{\perp}}(k)
			&=\Pi^{\perp}_h \left( k-(\Pi^{\parallel}_{\F,h})^{-1}(\Pi^{\parallel}_{h}(k)) \right).
	\end{align*}
\end{lem}
\begin{proof}
	Since the map
	\[
		\Pi_{\F,\hat h}^{\parallel}:T_{\hat h}\F\to \ker_{L^2}(\Delta_{L,\hat h})
	\]
	is the identity and since \eqref{eq: F h map}, by Proposition \ref{prop : kernel properties} (i), is a smooth family of linear maps between finite dimensional vector spaces, it is invertible for $h$ near $\hat h$.
	At first, we clearly have $(\Pi^{\parallel}_{\F,h})^{-1}(\Pi^{\parallel}_{h}(k))\in T_h\F$ and 
	\[
		\Pi^{\perp}_h \left( k-(\Pi^{\parallel}_{\F,h})^{-1}(\Pi^{\parallel}_{h}(k) ) \right)
			\perp \ker_{L^2}(\Delta_{L,h})
	\]
	by construction. It thus remains that they add up to $k$. 
	We have
\begin{align*}
	(\Pi^{\parallel}_{\F,h})^{-1} \left( \Pi^{\parallel}_{h}(k)\right)
		& +\Pi^{\perp}_h \left( k-(\Pi^{\parallel}_{\F,h})^{-1}(\Pi^{\parallel}_{h}(k)) \right) \\
		&= \Pi^{\perp}_h(k) + (\Pi_{\F, h}^\parallel)^{-1} \left( \Pi^{\parallel}_{h}(k)\right) - \Pi^{\perp}_h \left( (\Pi^{\parallel}_{\F,h})^{-1}\left(\Pi^{\parallel}_{h}(k)\right) \right) \\
		&= \Pi^{\perp}_h(k) + \left( \id - \Pi^{\perp}_h \right) (\Pi_{\F, h}^\parallel)^{-1} \left( \Pi^{\parallel}_{h}(k)\right) \\
		&=\Pi^{\perp}_h(k)+\Pi^{\parallel}_h(\Pi^{\parallel}_{\F,h})^{-1}(\Pi^{\parallel}_{h}(k))\\
		&=\Pi^{\perp}_h(k)+\Pi^{\parallel}_h(k)
		=k,
	\end{align*}
	which finishes the proof.
\end{proof}
For the next lemma, we introduce the norm
\begin{align*}
	\left\|k\right\|_{L^{[p,\infty]}}=\left\|k\right\|_{L^p}+\left\|k\right\|_{L^{\infty}}
	\end{align*}
which is the natural norm on $L^{[p,\infty]}=L^p\cap L^{\infty}$.
\begin{lem}\label{lem : projection lemma}
	Let $p \in (1, \infty)$ be fixed as in the beginning of this section and let $g,\bar{g}\in\U$ and $h,\bar{h},h_1,\bar{h}_1,h_2,\bar{h}_2\in \U\cap\F$.
	\begin{itemize}
		\item[(i)] For all $q\in (1,\infty),r\in (1,\infty]$ and $l\in \N_0$, there exists a constant $C=C(q,r,l,\U \cap \F)$ such that
		\[
		\left\|\nabla^l\circ\Pi^{\parallel}_{h}(k)\right\|_{L^r}\leq C\left\|k\right\|_{L^q}.
		\]
		\item[(ii)]	For all $q,r\in (1,\infty)$ and $l\in \N_0$, there exists a constant $C=C(q,r,l,\U \cap \F)$ such that
		\[
		\left\|\nabla^l\circ\Pi^{\perp}_h(k)\right\|_{L^r}
			\leq C(\left\|\nabla^lk\right\|_{L^r}+\left\|k\right\|_{L^q}).
		\]
		\item[(iii)] For all $q\in (1,\infty),r\in (1,\infty]$ and $l\in \N_0$, there exists a constant $C=C(q,r,l,\U)$ such that
		\[
		\left\|\nabla^l\circ D\Phi_g(k)\right\|_{L^r}\leq C\left\|k\right\|_{L^q}.
		\]
		\item[(iv)]	For all $q,r\in (1,\infty)$ and $l\in \N_0$, there exists a constant $C=C(q,r,l,\U \cap \F)$ such that
		\[
		\left\|\nabla^l\circ (\Pi^{\perp}_{h,\bar{h}})^{-1}(k)\right\|_{L^r}
			\leq C(\left\|\nabla^lk\right\|_{L^r}+\left\|k\right\|_{L^q}).
		\]
		\item[(v)] For all $q\in (1,\infty),r\in (1,\infty]$ and $l\in \N_0$, there exists a constant $C=C(q,r,l,\U \cap \F)$ such that
		\[
		\left\|\nabla^l\circ(\Pi^{\perp}_h-\Pi^{\perp}_{\bar{h}})(k)\right\|_{L^r}\leq C\left\|h-\bar{h}\right\|_{L^{p}}\left\|k\right\|_{L^q}.
		\]
		\item[(vi)] For all $q\in (1,\infty),r\in (1,\infty]$ and $l\in \N_0$, there exists a constant $C=C(q,r,l,\U)$ such that
		\[
		\left\|(D\Phi_g-D\Phi_{\bar{g}})(k)\right\|_{L^r}\leq C\left\|g-\bar{g}\right\|_{L^{[p,\infty]}}\left\|k\right\|_{L^q}.
		\]
		\item[(vii)] For all $q\in (1,\infty),r\in (1,\infty]$ and $l\in \N_0$, there exists a constant $C=C(q,r,l,\U \cap \F)$ such that
		\[ 
		\left\|\nabla^l\circ[(\Pi^{\perp}_{h_1,\bar{h}_1})^{-1}-	(\Pi^{\perp}_{h_2,\bar{h}_2})^{-1}](k) \right\|_{L^r}\leq C(\left\|h_1-h_2 \right\|_{L^p}+\left\|\bar{h}_1-\bar{h}_2\right\|_{L^p})\left\|k\right\|_{L^q}
		\]
	\end{itemize}
\end{lem}
\begin{proof}
	Recall that $\ker_{L^2}(\Delta_{L,h})\subset \O_{\infty}(r^{-n})$, so that $\ker_{L^2}(\Delta_{L,h})\subset W^{l,q}$
	for all $l\in\N_0$ and $q\in (1,\infty]$. 
	Thus, (i) and (ii) follow immediately from the definition of the projection maps $\Pi^{\parallel}_{h}$ and $\Pi^{\perp}_h$. 	
	From Lemma \ref{lem : projection F} and \eqref{eq : differential Phi}, we see that
	\[
	D_g\Phi=A\circ \Pi^{\parallel}_{\Phi(g)},
	\]	
	where $A:\ker_{L^2}(\Delta_{L,\Phi(g)})\to T_{\Phi(g)}\F$ is a linear map between finite dimensional spaces. On both spaces, all elements are in $\O_{\infty}(r^{-n})$ and all $W^{l,q}$-norms are equivalent for $l\in \N_0$ and $q\in (1,\infty]$. Therefore by using (i), we get
	\[
	\left\|\nabla^lD\Phi_g(k)\right\|_{L^r}
		\leq \left\|A\circ \Pi^{\parallel}_{\Phi(g)}(k)\right\|_{W^{l,r}}
		\leq C\left\|\Pi^{\parallel}_{\Phi(g)}(k)\right\|_{L^r}
		\leq C\left\|k\right\|_{L^q},
	\]
	which proves (iii). For (iv), recall from the proof of Lemma \ref{lem : invert projection} that
	\begin{align*}
		(\Pi^{\perp}_{h,\bar{h}})^{-1}=\id-(\Pi^{\parallel}_{\bar{h},h})^{-1}\circ \Pi^{\parallel}_h.	
	\end{align*}
	The map $A=(\Pi^{\parallel}_{\bar{h},h})^{-1}$ is a linear map between finite-dimensional spaces on which all $W^{l,r}$-norms are equivalent for $l\in \N_0$ and $r \in (1,\infty]$. 
	Therefore, again by using (i), we get
	\[
		\left\|\nabla^l(\Pi^{\parallel}_{\bar h,h})^{-1}(\Pi^{\parallel}_h(k))\right\|_{L^r}
			\leq \left\|A\circ \Pi^{\parallel}_{h}(k)\right\|_{W^{l,r}}
			\leq C\left\|\Pi^{\parallel}_{h}(k)\right\|_{L^r}
			\leq C\left\|k\right\|_{L^q},
	\]
	which implies (iv). For (v), observe first that
	\[
	(\Pi^{\perp}_h-\Pi^{\perp}_{\bar{h}})(k)
		=- (\Pi^{\parallel}_h-\Pi^{\parallel}_{\bar{h}})(k).
	\]
	By (i), we have a family of linear bounded maps $\nabla^l\circ\Pi^{\parallel}_h:L^q\to L^r$ which depends smoothly on $h$, and in particular, the dependence is Lipschitz. 
	This implies (v).	
	From the construction of $D\Phi_g$ and (iii), we also have a family of linear bounded maps $D\Phi_g:L^q\to L^r$ which depends smoothly on $g$, and in particular, the dependence is Lipschitz. 
	The estimate in (vi) is immediate. For the final point, we remark that using Remark \ref{rem : invert projection}, we may write 
	\[
	\nabla^l\circ[(\Pi^{\perp}_{h_1,\bar{h}_1})^{-1}-	(\Pi^{\perp}_{h_2,\bar{h}_2})^{-1}]=
	\nabla^l\circ(\Pi^{\parallel}_{\bar{h}_2,h_2})^{-1}\circ \Pi^{\parallel}_{h_2}-\nabla^l\circ(\Pi^{\parallel}_{\bar{h}_1,h_1})^{-1}\circ \Pi^{\parallel}_{h_1}.
	\]
	By construction and the proof of part (iv), we have a family of bounded maps
	\[
	\nabla^l\circ(\Pi^{\parallel}_{\bar h, h})^{-1}\circ\Pi^{\parallel}_h:L^q\to L^r
	\]
	which is smooth in $h$ and $\bar{h}$, in particular Lipschitz in both entries with respect to the $L^p$ norm. 
	This proves part (vii).
\end{proof}
\begin{lem}\label{lem : elliptic regularity}
	Let $(M,\hat{h})$ be an {integrable Ricci-flat ALE manifold with a parallel spinor}. 
	Then there is a small $L^{[p,\infty]}$-neighbourhood $\mathcal{U}$ such that the following holds: For each $k\in\N_0$, $q\in (1,\infty]$, $r\in (1,\infty]$, there exists a constant $C=C(\mathcal{U},k,q,r)$ such that
	\begin{align*}
\left\|h-\tilde{h}\right\|_{W^{k,r}}\leq C 	\left\|h-\tilde{h}\right\|_{L^q}
	\end{align*}
	for all $h,\tilde{h}\in \mathcal{U}\cap\mathcal{F}$.
\end{lem}
\begin{proof}Fix $h\in \mathcal{U}\cap\mathcal{F}$. 
By Proposition \ref{prop : kernel properties}, we have that $T_h\mathcal{F} \subset\O_\infty\left(\rho^{-n}\right)$, and thus $\left\|h-\tilde{h}\right\|_{W^{k,r}}<\infty$ for all $k\in\N_0$ and $r\in (1,\infty]$. Because $T_h\mathcal{F}$ is finite-dimensional, all these norms are pairwise equivalent.
Thus, we have a constant $C=C(\mathcal{U},k,q,r)$ such that
	\begin{align*}
\left\|k\right\|_{W^{k,r}}\leq C 	\left\|k\right\|_{L^q}
	\end{align*}
	for all $k\in T_h\mathcal{F}$. Because $T_h\mathcal{F}$ depends smoothly on the point, we can choose $C$ so that the above estimate holds for all $h\in \mathcal{U}\cap\mathcal{F}$ simultaneously. Now if $h,\tilde{h}$ are given,  we have for every smooth curve $\left\{h_{t}\right\}_{t\in [0,1]}$ in $\mathcal{F}$ joining $h=h_0$ and $\tilde{h}=h_1$ that
	\begin{align*}
	\left\|h-\tilde{h}\right\|_{W^{k,r}}\leq C\int_0^1\left\|\partial_th\right\|_{W^{k,r}}dt\leq C\int_0^1\left\|\partial_th\right\|_{L^q}dt.
	\end{align*}
	Now set $g_t=t\tilde{h}+(1-t)h$ and $h_t:=\Phi(g_t)$.
Then $h_t$ is a smooth curve in $\mathcal{F}$ joining $h=h_0$ and $\tilde{h}=h_1$	
	 By Lemma \ref{lem : projection lemma} (iii), we obtain
	\begin{align*}
	\int_0^1\left\|\partial_th\right\|_{L^q}dt
	\leq \int_0^1\left\|D_{g}\Phi(\partial_tg)\right\|_{L^q}dt
	\leq C\int_0^1\left\|\partial_tg\right\|_{L^q}dt
		\leq C\left\|\tilde{h}-h\right\|_{L^q}
	\end{align*}
	and the desired estimate follows.
\end{proof}
\section{Short-time estimates for parabolic equations}\label{sec : short-time estimates}
\subsection{Various expansions for the Ricci-de Turck flow}\label{subsec : RdTf expansions}
Let $h$ be a fixed Ricci-flat metric and consider $h$-gauged Ricci-de Turck flow, i.e. the evolution equation
\[
\partial_tg=-2\mathrm{Ric}_g+\mathcal{L}_{V(g,h)}g,\qquad V(g,h)^k=g^{ij}(\Gamma(g)_{ij}^k-\Gamma(h)_{ij}^k).
\]
Let $\n$ denote the Levi-Civita connection and $\abs{\cdot}$ the norm with respect to $h$.
\begin{lemma}\label{lem : expression RdT1}
	The Ricci-de Turck flow can be written with respect to the difference $k=g-h$ as
	\begin{align}	\label{eq : expression RdT1}
		\partial_tk+\Delta_{L,g,h}k&=F_1(g^{-1},g^{-1},\nabla k,\nabla k),\\
		\label{eq : expression RdT2}	
		\partial_tk+\Delta_{L,h}k&=F_1(g^{-1},g^{-1},\nabla k,\nabla k)+F_2(g^{-1},R,k,k)+F_3(g^{-1},k,\nabla^2 k),\\
		\label{eq : expression RdT3}	\partial_tk+\Delta_{h}k&=F_4(g^{-1},g^{-1},\nabla k,\nabla k)+F_5(g^{-1},g,R,k)+\nabla_a((g^{ab}-h^{ab})\nabla_bk_{ij}),
	\end{align}
	where
	\begin{align}
		\Delta_{L,g,h}k_{ij}
			&=-g^{ab}\nabla^{2}_{ab}k_{ij}-k_{ab}g^{ka}h^{lb}g_{ip}h^{pq}R_{jklq}-k_{ab}g^{ka}h^{lb}g_{jp}h^{pq}R_{iklq}, \label{eq: Lgh}\\
		\Delta_{L,h}k_{ij}
			&=-h^{ab}\nabla^2_{ab}k_{ij}-2k_{ab}h^{pa}h^{lb}R_{iplj}, \nonumber
	\end{align}
	and the $F_i$ are $h$-parallel maps which are $C^{\infty}(M)$-linear in all entries.
\end{lemma}
\begin{proof}
	According to \cite{Shi-Def}*{Lemma 2.1}, this evolution equation can be rewritten as
	\begin{align*}
		\partial_t g_{ij}&=g^{ab}\nabla^{2}_{ab}g_{ij}-g^{kl}g_{ip}h^{pq}R_{jklq}-g^{kl}g_{jp}h^{pq}R_{iklq}\\
		&\quad+g^{ab}g^{pq}\left(\frac{1}{2}\nabla_ig_{pa}\nabla_jg_{qb}+\nabla_ag_{jp}\nabla_qg_{ib}\right)\\
		&\quad-g^{ab}g^{pq}\left(\nabla_ag_{jp}\nabla_bg_{iq}+\nabla_jg_{pa}\nabla_bg_{iq}+\nabla_ig_{pa}\nabla_bg_{jq}\right),
	\end{align*}
	where the curvature and the covariant derivatives are taken with respect to $h$. If $h$ is Ricci-flat, this equation can be rewritten in terms of the difference $k=g-h$ as
	\begin{align*}
		\partial_t k_{ij}
			&=g^{ab}\nabla^{2}_{ab}k_{ij}+k_{ab}g^{ra}h^{lb}g_{ip}h^{pq}R_{jrlq}+k_{ab}g^{ra}h^{lb}
		g_{jp}h^{pq}R_{irlq}\\
			&\quad+g^{ab}g^{pq}\left(\frac{1}{2}\nabla_ik_{pa}\nabla_jk_{qb}+\nabla_ak_{jp}\nabla_qk_{ib}\right) \\
			&\quad -g^{ab}g^{pq}\left(\nabla_ak_{jp}\nabla_bk_{iq}+\nabla_jk_{pa}\nabla_bk_{iq}+\nabla_ik_{pa}\nabla_bk_{jq}\right).
	\end{align*}
	Then \eqref{eq : expression RdT1} follows from setting
	\begin{align*}
		F_1(g^{-1},g^{-1},\nabla k,\nabla k)&:=g^{ab}g^{pq}\left(\frac{1}{2}\nabla_ik_{pa}\nabla_jk_{qb}+\nabla_ak_{jp}\nabla_qk_{ib}\right)\\&\quad
		-g^{ab}g^{pq}\left(\nabla_ak_{jp}\nabla_bk_{iq}+\nabla_jk_{pa}\nabla_bk_{iq}+\nabla_ik_{pa}\nabla_bk_{jq}\right).
	\end{align*}
	For \eqref{eq : expression RdT2}, we first write the Lichnerowicz Laplacian as
	\begin{align*}
		\Delta_{L,h}k_{ij}
			=-h^{ab}\nabla^2_{ab}k_{ij}-k_{ab}h^{pa}h^{lb}R_{jpli}-k_{ab}h^{pa}h^{lb}R_{iplj}.
	\end{align*}
	Note that the last two terms are equal but their separate treatment allows a better comparison with $\Delta_{L,g,h}$ from the previous lemma. We compute
	\begin{align*}
		g^{ra}g_{ip}h^{pq}R_{jrlq}-h^{ra}R_{jrli}
			&= g^{ra}g_{ip}h^{pq}R_{jrlq}-h^{ra}h_{ip}h^{pq}R_{jrlq}\\
			&=g^{ra}(g_{ip}-h_{ip})h^{pq}R_{jrlp}+(g^{ra}-h^{ra})h_{ip}h^{pq}R_{jrlq}\\
			&=g^{ra}k_{ip}h^{pq}R_{jrlp}-k_{mn}g^{rm}h^{an}h_{ip}h^{pq}R_{jrlq}\\
			&=g^{ra}k_{ip}h^{pq}R_{jrlp}-k_{mn}g^{rm}h^{an}R_{jrli}
	\end{align*}
	and by exchanging $i$ and $j$,
	\begin{align*}
		g^{ra}g_{jp}h^{pq}R_{irlq}-h^{ra}R_{irlj}
			= g^{ra}k_{jp}h^{pq}R_{irlp}-k_{mn}g^{rm}h^{an}R_{irlj}.
	\end{align*}
	By summing up, we obtain
	\begin{align*}
		\Delta_{L,h} k-\Delta_{L,g,h}k
			&= (g^{ab}-h^{ab})\nabla^2_{ab}k_{ij} +k_{ab}h^{lb}(g^{ra}k_{ip}h^{pq}R_{jrlq}-k_{mn}g^{rm}h^{an}R_{jrli})\\
			&\qquad + k_{ab}h^{lb}(g^{ra}k_{jp}h^{pq}R_{irlq}-k_{mn}g^{rm}h^{an}R_{irlj})\\
			&=:(g^{ab}-h^{ab})\nabla^2_{ab}k_{ij}+F_2(g^{-1},R, k,k)
	\end{align*}
	and
	\begin{align*}
		(g^{ab}-h^{ab})\nabla^2_{ab}k_{ij}
			= -k_{pq}g^{ap}h^{bq}\nabla^2_{ab}k_{ij}
			=:F_3(g^{-1},k,\nabla^2k).
	\end{align*}
	Then, \eqref{eq : expression RdT2} follows from \eqref{eq : expression RdT1}. Finally, \eqref{eq : expression RdT3} follows from computing
	\begin{align*}
		(g^{ab}-h^{ab})\nabla^2_{ab}k_{ij}&=
		\nabla_a((g^{ab}-h^{ab})\nabla_bk_{ij})-\nabla_a(g^{ab}-h^{ab})\nabla_bk_{ij}\\&=\nabla_a((g^{ab}-h^{ab})\nabla_bk_{ij})
		+g^{ap}g^{bq}\nabla_a k_{pq}\nabla_bk_{ij},
	\end{align*}
	setting 
	\begin{align*}
		F_4(g^{-1},g^{-1},\nabla k,\nabla k)
			&:=F_1(g^{-1},g^{-1},\nabla k,\nabla k)+g^{ap}g^{bq}\nabla_a k_{pq}\nabla_b k_{ij},\\
		F_5(g^{-1},g,R,k)
			&:=k_{ab}g^{ra}h^{lb}g_{ip}h^{pq}R_{jrlq}-k_{ab}g^{ra}h^{lb}g_{jp}h^{pq}R_{irlq}
	\end{align*}
	and using \eqref{eq : expression RdT1} again.
\end{proof}

\subsection{An \texorpdfstring{$L^p$}{Lp}-maximum principle}

A standard tool for parabolic equations is short-time derivative estimates of the form $\left\|\nabla^ku_t\right\|_{L^{\infty}}\leq  C\cdot t^{-\frac{k}{2}} \left\|u_0\right\|_{L^{\infty}}$. 
The main purpose of this chapter is to develop analogous estimates for the $L^p$ norm. 
The main tool for doing this is the following theorem:
\begin{thm}[$L^p$-maximum principle]\label{Lpmaxprinc}
	Let $(M,h_t)$, $t\geq 0$ be a smooth $1$-parameter family of ALE manifolds such that 
	\begin{align*}
		\frac1\Lambda h_t 
			\leq \hat h 
			\leq \Lambda h_t, \qquad |\partial_th|_{h_t}\leq \Lambda
	\end{align*}
	for all $t\geq0$, a fixed ALE metric $\hat{h}$ and a time-independent constant $\Lambda>0$.\\	
Let $g_t$, $t\geq0$ be another smooth 1-parameter family of complete Riemannian metrics on $M$ such that for all $t\geq0$, we have
\begin{align*}
		&\frac{1}{\Lambda}h_t
			\leq g_t\leq \Lambda\cdot h_t,
			\qquad \left\|\nabla^{h_t} g_t\right\|_{L^{\infty}(h_t)}<\infty.
\end{align*}	
	Let $E,F,G$ be tensor bundles over $M$ equipped with the natural family of Riemannian metrics and connections induced by $h_t$.
	Let $ u(t)\in C^{\infty}(E)$ and
	\begin{align*}
		&H_1(t)\in L^{\infty}(\mathrm{End}(E)), \quad H_2(t)\in L^{\infty}(\mathrm{Hom}(T^*M\otimes E,E)), \\
		&H_3(t)\in L^{\infty}(\mathrm{Hom}(T^*M\otimes E,TM\otimes E)), \quad H_4(t) \in L^{p}(E), \\
		&H_5(t)\in L^{\infty}(\mathrm{Hom}(E,F)), \quad H_6(t)\in L^{\infty}(\mathrm{Hom}(E,G))
	\end{align*}
	be time-dependent sections.	
	\begin{itemize}
		\item[(i)]
		Suppose that $u$ satisfies the evolution inequality
		\begin{align*}
			\partial_t |u|^2
				&\leq g^{ab}\nabla^2_{ab}|u|^2+ 2 \langle H_1(u)+ H_2(\nabla u)+\nabla_a((H_3)^{ab}\nabla_bu) + H_4,u\rangle \\
				&\qquad -2(1-\delta)g^{ab}\langle\nabla_au,\nabla_bu\rangle
		\end{align*}
		for some $\delta\in [0,1)$.
		Then for every $p_0\in (1+\delta,\infty)$, there exists an $\epsilon>0$ such that the following holds:
		If $\left\|H_3\right\|_{L^{\infty}}<\epsilon$ and $u(0)\in L^p$ for some $p\in [p_0,\infty)$, we have $u(t)\in L^p$ for all $t\geq0$ and the estimate
		\begin{align*}
			\left\| u(t)\right\|_{L^p}\leq e^{\int_0^t\psi(s) ds}	\left\| u(0)\right\|_{L^p}
			+\left\| e^{\int_s^t\psi(r) dr}\cdot H_4(s)\right\|_{L^p([0,t]\times M)},
		\end{align*}
		where
		\begin{align*}
			\psi(t)
				=C\left(\left\|H_1\right\|_{L^\infty} + \left\|\nabla g\right\|_{L^{\infty}}^2+\left\|H_2\right\|_{L^{\infty}}^2+1\right)
		\end{align*}
		and $C=C(p_0,\epsilon,\Lambda,\hat{h},n)$ but independent of $p$.
		\item[(ii)]
		Suppose that $u$ satisfies the evolution inequality
		\begin{align*}
			\partial_t |u|^2
				&\leq g^{ab}\nabla^2_{ab}|u|^2 + 2 \langle H_1(u)+ H_2(\nabla u) + \nabla_a((H_3)^{ab}\nabla_bu) + H_4, u\rangle \\
				& \qquad -2(1-\delta)g^{ab}\langle\nabla_au,\nabla_bu\rangle + 2 \langle \nabla*(H_5(u))+\nabla*(H_6(\nabla u)), u \rangle
		\end{align*}
		for some $\delta\in [0,1)$.
		Then for $1+\delta<p_0<p_1<\infty$, there exists an $\epsilon>0$ such that the following holds:
		If $\left\|H_3\right\|_{L^{\infty}}+\left\|H_6\right\|_{L^{\infty}}<\epsilon$ and $u(0)\in L^p$ for some $p\in [p_0,p_1]$, we have $u(t)\in L^p$ for all $t\geq0$ and the estimate
		\begin{align*}
			\left\| u(t)\right\|_{L^p}&\leq e^{\int_0^t\psi(s) ds}	\left\| u(0)\right\|_{L^p}
			+\left\| e^{\int_s^t\psi(r) dr}\cdot H_4(s)\right\|_{L^p([0,t]\times M)},
		\end{align*}
		where
		\begin{align*}
			\psi(t)=
				C\left(\left\|H_1\right\|_{L^\infty} + \left\|\nabla g\right\|_{L^{\infty}}^2+\left\|H_2\right\|_{L^{\infty}}^2+\left\|H_5\right\|_{L^{\infty}}^2+1\right)
		\end{align*}
		and $C=C(p_0,p_1,\epsilon,\Lambda,\hat{h},n)$ but independent of $p$.
	\end{itemize}
\end{thm}
\begin{proof}
	We start with the proof of (i). 
	Let $q=\frac{p}{2}$ and $\rho>0$ a small parameter. 
	Define $F=|u|^2$ and $F_{\rho}=|u|^2+\rho$. Then we get
	\begin{equation}\begin{split}\label{eq : evo inequ times p}
		\partial_tF_{\rho}^q
			&\leq g^{ab}\nabla^2_{ab}F_{\rho}^q-q(q-1)g^{ab}\nabla_aF \nabla_bF F_{\rho}^{q-2}-2(1-\delta)q\cdot g^{ab}\langle\nabla_au,\nabla_bu\rangle F_{\rho}^{q-1}\\
			&\qquad+2q\cdot \langle H_1(u)+ H_2(\nabla u)+\nabla_a((H_3)^{ab}\nabla_bu)
			+H_4,u\rangle F_{\rho}^{q-1}.
		\end{split}
	\end{equation}
	Choose for each $x\in M$ and large $R>0$ a cutoff function $\phi_{R,x}$ such that 
	\begin{align*}
		\phi_{R,x}\equiv1\text{ on }B_R(x),\qquad \phi_{R,x}\equiv0\text{ on }M\setminus B_{2R}(x),\qquad |\nabla \phi_{R,x}|\leq 2\sqrt{\Lambda}/R,
%		\qquad |\nabla^2 \phi_{R,x}|\leq 8/R^2,
	\end{align*}
	where $B_R(x)$ is defined with respect to $\hat h$, making $\phi_{R,x}$ time independent. Let us define the quantity
	\begin{align*}
		A(R,\rho,t)=\sup_{x\in M}\int_M F_{\rho}^q(t)\cdot \phi_{R,x}^2\dv.
	\end{align*}
	We multiply \eqref{eq : evo inequ times p} by $\phi^2:=\phi_{R,x}^2$ and integrate over $M$. 
	Then we get
	\begin{align*}
		\partial_t\int_M F_{\rho}^q\phi^2\dv
			&\leq \int_M g^{ab}\nabla^2_{ab}F_{\rho}^q\phi^2\dv - q(q-1)\int_Mg^{ab}\nabla_aF \nabla_bF F_{\rho}^{q-2}\phi^2\dv\\
			&\qquad - 2(1-\delta)q\int_Mg^{ab}\langle\nabla_au,\nabla_bu\rangle F_{\rho}^{q-1}\phi^2\dv\\
			&\qquad + 2q\int_M  \langle H_1(u)+ H_2(\nabla u), u \rangle F_{\rho}^{q-1}\phi^2\dv\\
			&\qquad + 2q\int_M  \langle\nabla_a((H_3)^{ab}\nabla_bu) + H_4,u\rangle F_{\rho}^{q-1}\phi^2\dv \\
			&\qquad + \frac12 \int_M F_\rho^q \phi^2 \tr_h(\d_t h) \dv.
	\end{align*}
	Performing integration by parts with the first term yields
	\begin{align*}
		\int_M g^{ab}\nabla^2_{ab}F_{\rho}^q\phi^2\dv
			=-q\int_M \nabla_bF\cdot F_{\rho}^{q-1}\nabla_ag^{ab}\cdot\phi^2\dv-2q\int_M\nabla_bF\cdot F_{\rho}^{q-1}g^{ab}\nabla_a\phi\cdot\phi\dv.
	\end{align*}
	We get, using the Peter-Paul inequality
	\begin{align*}
		-q\int_M \nabla_bF\cdot F_{\rho}^{q-1}\nabla_ag^{ab}\cdot\phi^2\dv
			&\leq C\cdot q\int_M|\nabla u||u||\nabla g|F_{\rho}^{q-1}\cdot\phi^2\dv\\
			&\leq \epsilon_1\cdot q\int_M g^{ab}\langle\nabla_au,\nabla_bu\rangle F_{\rho}^{q-1}\phi^2\dv\\
			&\qquad+C(\epsilon_1)\cdot q \int_M F_{\rho}^q\cdot|\nabla g|^2\cdot\phi^2\dv.
	\end{align*}
	Because $g^{ab}\nabla_aF\nabla_b\phi\leq C|\nabla u||u||\nabla\phi|$, another application of the Peter-Paul inequality yields
	\begin{align*}
		-2q\int_M\nabla_bF\cdot F_{\rho}^{q-1}g^{ab}\nabla_a\phi\cdot\phi\dv&\leq
		\epsilon_1\cdot q\int_M g^{ab}\langle\nabla_au,\nabla_bu\rangle F_{\rho}^{q-1}\phi^2\dv\\
		&\qquad+C(\epsilon_1)\cdot q \int_M F_{\rho}^q\cdot|\nabla\phi|^2\dv.
	\end{align*}
	Similarly
	\begin{align*}
		2q\int_M  \langle H_2(\nabla u),u\rangle F_{\rho}^{q-1}\phi^2\dv&\leq
		\epsilon_1\cdot q\int_M g^{ab}\langle\nabla_au,\nabla_bu\rangle F_{\rho}^{q-1}\phi^2\dv\\
		&\qquad +C(\epsilon_1)\left\|H_2\right\|_{L^{\infty}}^2\cdot q\int_M  F_{\rho}^{q}\phi^2\dv.
	\end{align*}
	We easily get
	\begin{align*}
		2q\int_M  \langle H_1(u),u\rangle F_{\rho}^{q-1}\phi^2\dv&\leq 2q \left\|H_1\right\|_{L^{\infty}}\int_MF_{\rho}^{q}\cdot\phi^2\dv.
	\end{align*}
	Using Young's inequality $ab\leq \frac{1}{p'}a^{p'}+\frac{1}{q'}b^{q'}$ for $a=|H_4|$, $b=F_{\rho}^{q-\frac{1}{2}}$, $p'=2q$ and $q'=\frac{2q}{2q-1}$ yields
	\begin{align*}
		2q
			&\int_M  \langle H_4,u\rangle F_{\rho}^{q-1}\phi^2\dv
			\leq \int_M |H_4|^{2q}\phi^2\dv+(2q-1)\int_M  F_{\rho}^{q}\phi^2\dv.
	\end{align*}
	Let us now look at the remaining term. Integration by parts yields
	\begin{align*}
		2q\int_M\langle  \nabla_a((H_3)^{ab}\nabla_bu),u\rangle  F_{\rho}^{q-1}\phi^2\dv
			&= - 2q\int_M \langle (H_3)^{ab}\nabla_bu,\nabla_a u\rangle F_{\rho}^{q-1}\phi^2 \dv \\
			&\qquad - 2q(q-1)\int_M\langle (H_3)^{ab}\nabla_bu,u\rangle\nabla_a F\cdot F^{q-2}_{\rho}\phi^2\dv\\
			&\qquad - 4q\int_M \langle (H_3)^{ab}\nabla_bu,u\rangle F_{\rho}^{q-1}\nabla_a \phi\cdot \phi\dv.
	\end{align*}
	We have
	\begin{align*}
		- 2q\int_M \langle (H_3)^{ab}\nabla_bu,\nabla_a u\rangle F_{\rho}^{q-1}\phi^2 \dv 
			&\leq 2C q \int_M |H_3| g^{ab} \langle \nabla_bu , \nabla_a u \rangle F_{\rho}^{q-1}\phi^2 \dv 
	\end{align*}
	and
	\begin{align*}
		-2q(q-1)\int_M\langle (H_3)^{ab}\nabla_bu,u\rangle\nabla_a F\cdot F^{q-2}_{\rho}\phi^2\dv
			&=- q(q-1)\int_M (H_3)^{ab}\nabla_bF\nabla_a F\cdot F^{q-2}_{\rho}\phi^2\dv\\
			&\leq q|q-1| \int_M |H_3| |\nabla F|^2 F^{q-2}_{\rho}\phi^2\dv.
	\end{align*}
	Using the Peter Paul inequality again, we get
	\begin{align*}
		-4q\int_M \langle (H_3)^{ab}\nabla_bu,u\rangle F_{\rho}^{q-1}\nabla_a \phi\cdot \phi\dv
			&=- 2q\int_M  (H_3)^{ab}\nabla_bF\cdot F_{\rho}^{q-1}\nabla_a \phi\cdot \phi\dv\\ 
			&\leq q\int_M |H_3||\nabla F|^2 F_{\rho}^{q-2}\phi^2\dv+q\int_M |H_3||\nabla\phi|^2\cdot F_{\rho}^{q}\dv.
	\end{align*}
	Summing up and using 
	\begin{align*}
	|\nabla u|^2\leq Cg^{ab}\langle\nabla_au,\nabla_bu\rangle,\qquad
		|\nabla F|^2\leq Cg^{ab}\langle\nabla_aF,\nabla_bF\rangle,
	\end{align*}
	we get
	\begin{align*}
		2q\int_M\langle  \nabla_a((H_3)^{ab}\nabla_bu),u\rangle  F_{\rho}^{q-1}\phi^2\dv
			&\leq 2C q \int_M |H_3| g^{ab} \langle \nabla_bu, \nabla_a u \rangle F_{\rho}^{q-1}\phi^2 \dv \\
			&\qquad + Cq^2\int_M |H_3|g^{ab}\nabla_aF \nabla_bF F_{\rho}^{q-2}\phi^2\dv\\
			&\qquad + q\int_M |H_3||\nabla\phi|^2\cdot F_{\rho}^{q}\dv.
	\end{align*}
	Finally, we also have
	\[
		\int_M F_\rho^q \phi^2 \tr_h(\d_t h) \dv
			\leq Cq \int_M F_\rho^q \phi^2 \dv,
	\]
	since $q \geq \frac12$.
	In summary, we obtain
	\begin{equation} \label{eq: part (i)}
	\begin{split}
		\partial_t\int_M F_{\rho}^q\phi^2\dv
			&\leq -[q(q-1)-Cq^2\left\|H_3\right\|_{L^\infty}]\int_Mg^{ab}\nabla_aF \nabla_bF F_{\rho}^{q-2}\phi^2\dv\\	
			&\qquad -2\left(1-\delta-\frac32 \epsilon_1 - C\norm{H_3}_{L^\infty}\right)q\int_Mg^{ab}\langle\nabla_au,\nabla_bu\rangle F_{\rho}^{q-1}\phi^2\dv\\
			&\qquad + C(\left\|H_1\right\|_{L^\infty} + \left\|\nabla g\right\|_{L^{\infty}}^2+\left\|H_2\right\|_{L^{\infty}}^2+1)\cdot q\int_MF_{\rho}^q\phi^2\dv\\
			&\qquad+C(1+\left\|H_3\right\|_{L^{\infty}})\cdot q\int_M |\nabla \phi|^2F_{\rho}^q\dv+\int_M |H_4|^{2q}\phi^2\dv.
	\end{split}
	\end{equation}
Note that the constants which appear in this estimate only depend on $\Lambda,n$ and $\epsilon_1$.	
	
	We now claim that the sum of the first two terms is nonpositive, provided that $\epsilon_1$ and $\left\|H_3\right\|_{L^{\infty}}$ are chosen small enough in comparison to the constants $C$, but independent of an upper bound on $q$.
	If $q\geq 2$, it is immediate.
	Before proceeding with the other cases, note first that by the Cauchy-Schwarz inequality, we get
	\begin{align*}
		g^{ab}\nabla_aF \nabla_bF
			= 4 g^{ab}\langle\nabla_au,u\rangle\langle\nabla_bu,u\rangle\leq4 g^{ab}\langle\nabla_au,\nabla_bu\rangle F_\rho,
	\end{align*}
	so that
	\begin{align*}
		-[q(q-1)
			&-Cq^2\left\|H_3\right\|_{L^{\infty}}]\int_Mg^{ab}\nabla_aF \nabla_bF F_{\rho}^{q-2}\phi^2\dv \\
			&\qquad -2\left( 1-\delta-\frac32 \epsilon_1 - C\norm{H_3}_{L^\infty} \right) q\int_Mg^{ab}\langle\nabla_au,\nabla_bu\rangle F_{\rho}^{q-1}\phi^2\dv\\
			&\leq -q(q-1)\int_Mg^{ab}\nabla_aF \nabla_bF F_{\rho}^{q-2}\phi^2\dv\\
			&\qquad -2\left(1-\delta-\frac32 \epsilon_1 - C\norm{H_3}_{L^\infty} -2Cq\left\|H_3\right\|_{L^{\infty}}\right)\cdot q \int_Mg^{ab}\langle\nabla_au,\nabla_bu\rangle F_{\rho}^{q-1}\phi^2\dv.
	\end{align*}
	Thus if $q\in [1,2]$, the right hand side is nonpositive, provided that $\epsilon_1$ and $\left\|H_3\right\|_{L^{\infty}}$ are small enough in comparison to the constants $C$, but independent of an upper bound on $q$. If $p_0\in (1,2)$ and $q\in [\frac{p_0}{2},1]$, we use the above inequality again to obtain
	\begin{align*}
		-q&(q-1)\int_Mg^{ab}\nabla_aF \nabla_bF F_{\rho}^{q-2}\phi^2\dv\\
		&\qquad -2\left(1-\delta-\frac32 \epsilon_1 - C \norm{H_3}_{L^\infty} -2Cq\left\|H_3\right\|_{L^{\infty}} \right)\cdot q \int_Mg^{ab}\langle\nabla_au,\nabla_bu\rangle F_{\rho}^{q-1}\phi^2\dv\\
		&\leq -2\left(2q-1-\delta-\frac32 \epsilon_1 - C \norm{H_3}_{L^\infty} - 2C q\left\|H_3\right\|_{L^{\infty}} \right)\cdot q \int_Mg^{ab}\langle\nabla_au,\nabla_bu\rangle F_{\rho}^{q-1}\phi^2\dv.
	\end{align*}
	The right hand side is nonpositive, provided that $\epsilon_1+\left\|H_3\right\|_{L^{\infty}}$ is smaller than a constant which depends on $\Lambda,n$ and $p_0$ but is independent of $q$.
	We arrive at the estimate
	\begin{align*}
		\partial_t\int_M F_{\rho}^q\phi^2\dv
		&\leq C(\left\|H_1\right\|_{L^\infty}+
		\left\|\nabla g\right\|_{L^{\infty}}^2+\left\|H_2\right\|_{L^{\infty}}^2+1)\cdot q\int_MF_{\rho}^q\phi^2\dv\\
		&\qquad+C(1+\left\|H_3\right\|_{L^{\infty}})R^{-2}\cdot q\int_{B_{2R}} F_{\rho}^q\dv+\left\|H_4\right\|_{L^{2q}}^{2q}.
	\end{align*}
	Abbreviate
	\begin{align*}
		\psi(t)=\frac{C}{2}(\left\|H_1\right\|_{L^\infty}+
		\left\|\nabla g\right\|_{L^{\infty}}^2+\left\|H_2\right\|_{L^{\infty}}^2+1).
	\end{align*}
	Integrating this differential inequality in time, we obtain
	\begin{align*}
	&\int_M F_{\rho}^q(t)\phi^2\dv \\*
		&\quad \leq \exp\left(2q\cdot \int_0^t\psi(s)ds\right)\int_M F_{\rho}^q(0)\phi^2\dv\\*
		&\quad \qquad +\int_0^t\exp\left(2q\cdot \int_s^t\psi(r)dr\right)\left(C(1+\left\|H_3\right\|_{L^{\infty}})R^{-2}\cdot q\int_{B_{2R}} F_{\rho}^q\dv+\left\|H_4\right\|_{L^{2q}}^{2q} \right)ds\\
		&\quad \leq \exp\left(2q\cdot \int_0^t\psi(s)ds\right)A(R,\rho,0)+C_{double}\cdot C R^{-2}\cdot q\int_0^t\exp\left(2q\cdot \int_s^t\psi(r)dr\right)A(R,\rho,s)ds\\
		&\quad \qquad+ \int_0^t\exp\left(2q\cdot \int_s^t\psi(r)dr\right)\left\|H_4\right\|_{L^{2q}}^{2q} ds,
	\end{align*}
	where we used the definition of $A$ and the fact that we can cover $B_{2R}(x)$ by $C_{double}$ balls of radius $R$.
	By taking the supremum over all $x\in M$ on the left hand side, we conclude
	\begin{align*}
		A(R,\rho,t)
		&\leq \exp\left(2q\cdot \int_0^t\psi(r)dr\right)A(R,\rho,0)\\&\qquad+C_{double}\cdot C R^{-2}\cdot q\int_0^t\exp\left(2q\cdot \int_s^t\psi(r)dr\right)A(R,\rho,s)ds\\
		&\qquad+ \int_0^t\exp\left(2q\cdot \int_s^t\psi(r)dr\right)\left\|H_4\right\|_{L^{2q}}^{2q} ds.
	\end{align*}
	By a variant of Gronwall's lemma (cf. \cite{Mit91}*{p.~356}), we get
	\begin{align*}
		A(R,\rho,t)
		&\leq \exp\left(2q\cdot \int_0^t\psi(r)dr\right)A(R,\rho,0)+
		\int_0^t\exp\left(2q\cdot \int_s^t\psi(r)dr\right)\left\|H_4\right\|_{L^{2q}}^{2q} ds\\
		&\qquad +\gamma(t)\int_0^t \alpha(t)\beta(s)\exp\left(\int_0^t\beta(r)\gamma(r)dr\right)ds,
	\end{align*}
	where
	\begin{align*}
		\alpha(t)&=\exp\left(2q\cdot \int_0^t\psi(r)dr\right)A(R,\rho,0)+
		\int_0^t\exp\left(2q\cdot \int_s^t\psi(r)dr\right)\left\|H_4\right\|_{L^{2q}}^{2q} ds,\\
		\beta(t)&=C_{double}\cdot C R^{-2}\cdot q\cdot\exp\left(-2q\cdot \int_0^s\psi(r)dr\right),\\
		\gamma(t)&=\exp\left(2q\cdot \int_0^t\psi(r)dr\right).
	\end{align*}
	Letting $\rho\to0$ and $R\to\infty$ and using $p=2q$, we get
	\begin{align*}
		\left\|u(t)\right\|_{L^{p}}^{p}
		&=\left\|F(t)\right\|_{L^{q}}^{q} \\*
		&\leq \exp\left(2q\cdot \int_0^t\psi(r)dr\right)\left\|F(0)\right\|_{L^{q}}^q+
		\int_0^t\exp\left(2q\cdot \int_s^t\psi(r)dr\right)\left\|H_4\right\|_{L^{2q}}^{2q} ds\\
		&\leq \exp\left(p\cdot \int_0^t\psi(r)dr\right)\left\|u(0)\right\|_{L^{p}}^{p}+
		\int_0^t\exp\left(p\cdot \int_s^t\psi(r)dr\right)\left\|H_4\right\|_{L^{p}}^{p} ds.
	\end{align*}
	With
	\begin{align*}
		x(t)=\exp\left(\int_0^t\psi(r)dr\right)\left\|u(0)\right\|_{L^{p}},\qquad
		y(t)=\left(\int_0^t \left[(\exp\left(\int_s^t\psi(r)dr\right)\left\|H_4\right\|_{L^p}\right]^pds\right)^{\frac{1}{p}}
	\end{align*}
	and the elementary inequality
	\begin{align*}
		(x(t)^p+y(t)^p)^{\frac{1}{p}}\leq |x(t)|+|y(t)|,
	\end{align*}
	we finally get
	\begin{align*}
		\left\|u(t)\right\|_{L^{p}}\leq \exp\left(\int_0^t\psi(r)dr\right)\left\|u(0)\right\|_{L^{p}}+\left(\int_0^t\exp\left(p\cdot\int_s^t\psi(r)dr\right) \left\|H_4\right\|_{L^{p}}^pds\right)^{\frac{1}{p}}
	\end{align*}
	for all $p<\infty$ with the function $\psi$ chosen independently of $p$. 
	This finishes the proof of (i).\\
	For the proof of (ii), we proceed as in the first part and we also use the notation from the beginning of the proof. We have to deal with the additional terms
	\begin{align*}
		2q\int_M \langle \nabla*(H_5(u))+\nabla*(H_6(\nabla u)),u\rangle F_{\rho}^{q-1}\phi^2\dv.
	\end{align*}
	For the first term, we proceed as follows:
	\begin{align*}
		2q\int_M 
			&\langle \nabla*(H_5(u)),u\rangle F_{\rho}^{q-1}\phi^2\dv \\
			&\leq C q \abs{\int_M\langle H_5(u),\nabla u\rangle F_{\rho}^{q-1}\phi^2\dv} + C q\abs{q-1}\abs{\int_M \langle H_5(u),u\rangle \nabla F\cdot F_{\rho}^{q-2}\phi^2\dv} \\
			&\qquad +Cq\abs{\int_M \langle H_5(u),u\rangle F_{\rho}^{q-1}\nabla \phi\cdot \phi\dv}\\
			&\leq Cq\epsilon_1\int_M g^{ab}\langle\nabla_au,\nabla_bu\rangle F_{\rho}^{q-1}\phi^2\dv + C(\epsilon_1)q\int_M |H_5|^2F_{\rho}^{q}\phi^2\dv\\
			&\qquad +\epsilon_1\cdot C q\abs{q-1}\int_Mg^{ab}\nabla_aF \nabla_bF F_{\rho}^{q-2}\phi^2\dv+C(\epsilon_1)q\abs{q-1}\int_M|H_5|^2F_{\rho}^{q}\phi^2dv\\
			&\qquad +Cq\int_M |H_5|^2F_{\rho}^q\phi^2\dv+Cq\int_M F_{\rho}^q|\nabla\phi|^2\dv\\
			&=Cq\epsilon_1\int_M g^{ab}\langle\nabla_au,\nabla_bu\rangle F_{\rho}^{q-1}\phi^2\dv+\epsilon_1\cdot Cq\abs{q-1}\int_Mg^{ab}\nabla_aF \nabla_bF F_{\rho}^{q-2}\phi^2\dv\\
			&\qquad+ C(\epsilon_1)q^2\int_M |H_5|^2F_{\rho}^{q}\phi^2\dv +Cq\int_M F_{\rho}^q|\nabla\phi|^2\dv.
	\end{align*}
	The second term is treated as
	\begin{align*}
		2q
			&\int_M\langle\nabla*(H_6(\nabla u)),u\rangle F_{\rho}^{q-1}\phi^2\dv \\
			&= 2q \int_M H_6(\nabla u) * (\nabla u) F_{\rho}^{q-1}\phi^2\dv + 2q(q-1) \int_M H_6 (\nabla u) * u (\nabla F)\cdot F_{\rho}^{q-2}\phi^2\dv \\
			&\qquad +4q\int_M H_6 (\nabla u) * u F_{\rho}^{q-1}\nabla \phi\cdot\phi \dv\\
			&\leq Cq^2 \left\|H_6\right\|_{L^{\infty}}\int_M g^{ab}\langle\nabla_au,\nabla_bu\rangle F^{q-1}_{\rho}\phi^2\dv + Cq \left\|H_6\right\|_{L^{\infty}}\int_M F_{\rho}^{q}|\nabla\phi|^2\dv.
	\end{align*}
	Summing with the terms from part (i), the right hand side of the estimate \eqref{eq: part (i)}, we obtain
	\begin{align*}
	\partial_t
		&\int_M F_{\rho}^q\phi^2\dv \\
		&\leq -[q(q-1) - \epsilon_1 C q \abs{q - 1} -Cq^2\left\|H_3\right\|_{L^\infty}]\int_Mg^{ab}\nabla_aF \nabla_bF F_{\rho}^{q-2}\phi^2\dv
		\\&\qquad
		-2 q \left(1-\delta-\frac{3}{2} \epsilon_1 - C \norm{H_3}_{L^\infty}-Cq\left\|H_6\right\|_{L^{\infty}}\right)
			\int_Mg^{ab}\langle\nabla_au,\nabla_bu\rangle F_{\rho}^{q-1}\phi^2\dv\\
		&\qquad + C(\left\|H_1\right\|_{L^\infty}+
		\left\|\nabla g\right\|_{L^{\infty}}^2+\left\|H_2\right\|_{L^{\infty}}^2+q\cdot\left\|H_5\right\|_{L^{\infty}}^2+1)\cdot q\int_MF_{\rho}^q\phi^2\dv\\
		&\qquad+C(1+\left\|H_3\right\|_{L^{\infty}}+\left\|H_6\right\|_{L^{\infty}})\cdot q\int_M |\nabla \phi|^2F_{\rho}^q\dv+\int_M |H_4|^{2q}\phi^2\dv.
	\end{align*}
Again, the constants appearing in this estimate only depend on $\Lambda,n$ and $\epsilon_1$.		
		Because there are terms containing $H_5$ and $H_6$ which are quadratic in $q$, we are not able to prove an estimate uniform in $q$ for all large $q$. However, assuming additionally a bound $2q=p\leq p_1<\infty$, we may proceed as in part (i) to finish the proof of part (ii). 
\end{proof}

\subsection{Short-time estimates for the heat flow of the modified Lichnerowicz Laplacian}
In this section, we establish short-time estimates for solutions of the linear heat equation
\begin{align}\label{eq : a linear heat equation}
	\partial_tk+\Delta_{L,g,h}k
		= 0.
\end{align}

The involved scalar products, covariant derivatives and curvatures appearing here are induced by $h$ and hence also depend on time.
\begin{lem}\label{lem : short-time estimates}
	Let  $\hat{h}$, $\U$ and $\F$ be as in Proposition \ref{prop: near metrics}.
	Let $g_t \in \U$ and $h_t \in \U \cap \F$ be 1-parameter families of metrics on $M$, such that
	\begin{align*}
		|\nabla^{h_t} g_t|_{h_t} 		\leq \Lambda, 
		\qquad \left| \d_t h_t \right|_{h_t} 
%		+ \left| \nabla^{h_t} \d_t h_t \right|_{h_t}
		  \leq \Lambda,
	\end{align*}
	for all $t\geq0$ and a time-independent constant $\Lambda>0$. \\
	Let $k_t$, $t\in [0,T]$ be a solution of equation \eqref{eq : a linear heat equation} with initial data $k_0$ and $1<p_0<p_1<\infty$.
	\begin{itemize}
		\item[(i)] If $p\in [p_0,\infty)$, $l\in\left\{0,1\right\}$ and $k_0\in W^{l,p}$, then $k(t)\in W^{l,p}$ for all $t\geq0$ and we have
		\begin{align*}
%		\label{short-timew1p}
		\left\|k_t\right\|_{W^{l,p}}\leq e^{C\cdot t}\left\|k_0\right\|_{W^{l,p}}\end{align*}
		for some constant $C = C(\Lambda, n,\U,p_0,T,l)$.
		\item[(ii)] If $p\in [p_0,p_1]$ and $k_0\in W^{2,p}$, then $k_t\in W^{2,p}$ for all $t\geq 0$ and we have
		\begin{align*}
%		\label{short-timew2p}
		\left\|k_t\right\|_{W^{2,p}}\leq e^{C\cdot t}\left\|k_0\right\|_{W^{2,p}}\end{align*}
		for some constant $C=C(\Lambda, n,\U,p_0,p_1,T,l)$.		
		\item[(iii)] If $p\in [p_0,p_1]$ and $k_0\in W^{1,p}$, then $k_t\in W^{2,p}$ for all $t>0$ and we have
		\begin{align*}
%		\label{extshort-timew2p}
		\left\|k_t\right\|_{W^{1,p}}+C_1\cdot t^{1/2}\cdot \left\|\nabla^2k_t\right\|_{L^p}\leq e^{C_2\cdot t}\left\|k_0\right\|_{W^{1,p}}\end{align*}
		for some constants $C_i=C_i(\Lambda, n,\U,p_0, p_1,T)$, $i=1,2$.
		\item[(iv)] If $p\in [p_0,p_1]$ and $k_0\in L^{p}$, then $k_t\in W^{2,p}$ for all $t>0$ and we have
		\begin{align*}
		\left\|k_t\right\|_{L^{p}}+C_1\cdot t^{1/2}\cdot \left\|\nabla k_t\right\|_{L^p}+C_2\cdot t\cdot \left\|\nabla^2k_t\right\|_{L^p} 
		\leq e^{C_3\cdot t}\left\|k_0\right\|_{L^p}
%		 \label{nextshort-timew2p}
		\end{align*}
		for some constants $C_i=C_i(\Lambda, n,\U,T, p_0,p_1)$, $i=1,2,3$.
		\item[(v)] If $p\in [p_0,\infty)$, $q\in [p,\infty)$, $k_0\in L^{q}$ and $\nabla k_0\in L^{p}$, then $\nabla k_t\in L^{p}$ for all $t\in [0,T]$ and we have
		\begin{align*}
			\left\|\nabla k_t\right\|_{L^{p}}
				\leq e^{C_1\cdot t}(\left\|\nabla k_0\right\|_{L^{p}}+ C_2t^{1/p}\left\| k_0\right\|_{L^{q}})
		\end{align*}
		for some constants $C_i=C_i(\Lambda, n,\U,T,p_0)$, $i=1,2$.
		\item[(vi)] If $p\in [p_0,p_1]$, $q\in [p,\infty)$, $k_0\in W^{1,q}$ and $\nabla^2 k_0\in L^{p}$, then $\nabla^2 k_t\in L^{p}$ for all $t\in [0,T]$ and we have
		\begin{align*}
			\left\|\nabla^2 k_t\right\|_{L^{p}}
				\leq e^{C_1\cdot t}(\left\|\nabla^2 k_0\right\|_{L^{p}}+C_2(t^{1/p}+t^{1/2})\left\|k_0\right\|_{W^{1,q}})
		\end{align*}
		for some constants $C_i=C_i(\Lambda, n,\U,T,p_0,p_1)$, $i=1,2$.
	\end{itemize}
\end{lem}
\begin{proof}
Let us first remark that because $g_t,h_t\in\mathcal{U}$, we have
\begin{align*}
&\frac{1}{\Lambda}h_t
			\leq g_t\leq \Lambda\cdot h_t,
			\qquad 
		\frac1{\Lambda} h_t 
			\leq \hat h 
			\leq \Lambda h_t,
\end{align*}
provided that $\Lambda>0$ is chosen sufficiently large. Thus, the assumptions of Theorem \ref{Lpmaxprinc} are satisfied. We can write equation \eqref{eq : a linear heat equation} as
\begin{align*}
\partial_tk=g^{ab}\nabla^2_{ab}k+	\hat{R}[k],\qquad 
\end{align*}
	where
	\[
		\hat{R}[k]
			=g^{-1}*g*R*k.
	\]
	To prove (i), we compute
	\begin{align}
		\partial_t|k|^2&=g^{ab}\nabla^2_{ab}|k|^2-2g^{ab}\langle\nabla_ak,\nabla_bk\rangle+2\langle \hat{R}[k]+\partial_t h*k,k\rangle \label{eq: evolution k squared} \\
		\begin{split}
		\partial_t|\nabla k|^2&=g^{ab}\nabla^2_{ab}|\nabla k|^2-2g^{ab}\langle\nabla_a\nabla k,\nabla_b\nabla k\rangle+2\langle [\nabla,g^{ab}\nabla^2_{ab}]k+ \nabla\hat{R}[k],\nabla k\rangle\\
		&\qquad +\langle   \partial_t h*\nabla k+\nabla \partial_t h*k,\nabla k\rangle.
		\end{split} \label{eq: evolution nabla k squared}
	\end{align}
	Because we have
	\begin{align*}
		[\nabla,g^{ab}\nabla^2_{ab}]k&=\nabla g^{-1}*\nabla^2 k+g^{-1}*R*\nabla k+g^{-1}*\nabla R*k,
	\end{align*}
	and $R$ and $\nabla R$ are bounded for $h\in\mathcal{F}$,
	part (i) follows from Theorem \ref{Lpmaxprinc} (i) applied to
	\begin{align*}
		u=k\in C^{\infty}(S^2M) \quad \text{ and } \quad u=(k,\nabla k)\in C^{\infty}(S^2M\oplus T^*M\otimes S^2M).
	\end{align*}
	It is convenient to prove (v) now. In this case, we apply Theorem \ref{Lpmaxprinc} (i) to $u=\nabla k\in C^{\infty}(S^2M)$ and regard the terms in \eqref{eq: evolution nabla k squared} containing $k$ as part of the inhomogeneity. 
	All these terms are of the form 
	\begin{equation} \label{eq: k terms, first kind}
		\nabla  \partial_t h*k, \quad \n g^{-1} * g * R * k, \quad g^{-1} * \n g * R * k, \quad g^{-1} * g * \nabla R * k.
	\end{equation}
	We have
	\begin{align} \label{eq: h terms, first kind}
		\left\|\nabla  \partial_t h\right\|_{L^r}\leq 
		C\left\| \partial_t h\right\|_{L^{\infty}}<C,\qquad	\left\| R\right\|_{W^{1,s}}
			\leq C<\infty
	\end{align}
	for all $r\in [p_0,\infty]$ and $s\in [1,\infty]$, with constants depending only on $\mathcal{U},\Lambda$ and $p_0$. Here, the first estimate follows from Proposition \ref{prop : kernel properties} and the fact that $T_h\mathcal{F}$ is finite-dimensional so that all the norms on it are equivalent. The second estimate follows again from Proposition \ref{prop : kernel properties} and the fact that every Ricci-flat ALE manifold is of order $-n+1$ \cite{BKN89}.
	
Thus, applying the short-time estimates (i) for the $L^{q}$-norm of $k$, shows that all terms of the form \eqref{eq: k terms, first kind} are actually in $L^{p}$. 
	Therefore, an application of Theorem \ref{Lpmaxprinc} (i) now proves (v). Note that the factor $t^{1/p}$ in the statement of (v) comes from the estimate 
\begin{align}\label{eq : inhomogeneity estimate}
\left\|e^{\int_s^t\psi(r)dr} H_4(s)\right\|_{L^p([0,t]\times M)}
\leq e^{Ct}\left\|H_4(s)\right\|_{L^p([0,t]\times M)}\leq e^{Ct}t^{1/p}\sup_{s\in [0,T]}\left\|H_4(s)\right\|_{L^p}
\end{align} 
for the expression involving $H_4$ in Theorem \ref{Lpmaxprinc}.

	For (ii), we additionally compute
	\begin{equation} \label{eq: evolution nabla quared k}
	\begin{split}
		\partial_t|\nabla^2k|^2&=g^{ab}\nabla^2_{ab}|\nabla^2k|^2-2g^{ab}\langle\nabla_a\nabla^2k,\nabla_b\nabla^2k\rangle+2\langle[\nabla^2,g^{ab}\nabla^2_{ab}]k+ \nabla^2\hat{R}[k],\nabla^2k\rangle\\
		&\qquad +\langle \partial_t h*\nabla^2k+\nabla^2\partial_th*k+\nabla\partial_th*\nabla k,\nabla^2 k\rangle.
	\end{split}
	\end{equation}
	Using 
	\begin{align*}
	[\nabla^2,g^{ab}\nabla^2_{ab}]k
		&=[\nabla^2,g^{ab}]\nabla^2_{ab}k+g^{ab}[\nabla^2,\nabla^2_{ab}]k\\
		&= \nabla(\nabla g^{-1}*\nabla^2k )+\nabla g^{-1}*\nabla^3k+ g^{-1}* \nabla^2 R*k \\
		&\qquad +g^{-1}* \nabla R*\nabla k+ g^{-1} * R*\nabla^2 k,
	\end{align*}
	we can rewrite equation \eqref{eq: evolution nabla quared k} as
	\begin{align*}
		\partial_t|\nabla^2k|^2
			&=g^{ab}\nabla^2_{ab}|\nabla^2k|^2-2g^{ab}\langle\nabla_a\nabla^2k,\nabla_b\nabla^2k\rangle + 2\langle\nabla(\nabla g^{-1}*\nabla^2k+ \nabla\hat{R}[k]),\nabla^2k\rangle \\
			&\qquad + 2\langle\nabla g^{-1}*\nabla^3k+g^{-1}*\nabla^2 R*k+g^{-1}*\nabla R*\nabla k+g^{-1}*R*\nabla^2 k,\nabla^2k\rangle\\
			&\qquad +\langle \partial_t h*\nabla^2k+\nabla(\nabla\partial_th*k)+\nabla\partial_th*\nabla k,\nabla^2 k\rangle.
	\end{align*}
	Part (ii) follows from Theorem \ref{Lpmaxprinc} (ii), applied to 
	\begin{align*}
		u=(k,\nabla k,\nabla^2k)\in C^{\infty}(S^2M\oplus T^*M\otimes S^2M \oplus T^*M\otimes T^*M\otimes S^2M).
	\end{align*}
	Here, we have to apply the extended version of the maximum principle because we need to deal with the terms $\nabla(\nabla g^{-1}*\nabla^2k+ \nabla\hat{R}[k])$ and $\nabla(\nabla\partial_th*k)$ without using second derivatives of $g^{-1}$ and $\partial_th$, respectively.
		For part (iii), we first compute
	\begin{align*}
		\partial_t|\nabla^2(At^{\frac{1}{2}}k)|^2&=\partial_t (A^2t|\nabla^2k|^2)\\
		&=A^2|\nabla^2k|^2+A^2t\left( g^{ab}\nabla^2_{ab}|\nabla^2k|^2-2g^{ab}\langle\nabla_a\nabla^2k,\nabla_b\nabla^2k\rangle\right)\\
		&\qquad +2A^2t\left(\langle\nabla(\nabla g^{-1}*\nabla^2k+ \nabla\hat{R}[k])+\nabla g^{-1}*\nabla^3k+\nabla R*k+R*\nabla k,\nabla^2k\rangle\right)\\
		&\qquad +A^2t\left(\langle \partial_t h*\nabla^2k+\nabla^2\partial_th*k+\nabla\partial_th*\nabla k,\nabla^2 k\rangle\right)\\
		&=A^2|\nabla^2k|^2+\left( g^{ab}\nabla^2_{ab}|\nabla^2(At^{\frac{1}{2}}k)|^2-2g^{ab}\langle\nabla_a\nabla^2(At^{\frac{1}{2}}k),\nabla_b\nabla^2(At^{\frac{1}{2}}k)\rangle\right)\\
		&\qquad+2\langle\nabla(\nabla g^{-1}*\nabla^2(At^{\frac{1}{2}}k)+ At^{\frac{1}{2}}\nabla\hat{R}[k]),\nabla^2(At^{\frac{1}{2}}k)\rangle \\
		&\qquad +\langle\nabla g^{-1}*\nabla^3(At^{\frac{1}{2}}k)+At^{\frac{1}{2}}\nabla R*k+At^{\frac{1}{2}}R*\nabla k,\nabla^2(At^{\frac{1}{2}}k)\rangle\\
		&\qquad +\langle \partial_t h*\nabla^2(At^{\frac{1}{2}}k)+At^{\frac{1}{2}}\nabla^2\partial_th*k+At^{\frac{1}{2}}\nabla\partial_th*\nabla k,\nabla^2(At^{\frac{1}{2}}k)\rangle.
	\end{align*}
	Part (iii) then also follows from Theorem \ref{Lpmaxprinc} (ii), applied to 
	\begin{align*}
		u=(k,\nabla k,\nabla^2(At^{\frac{1}{2}}k))\in C^{\infty}(S^2M\oplus T^*M\otimes S^2M \oplus T^*M\otimes T^*M\otimes S^2M),
	\end{align*}
	where $A$ has to be chosen small in dependence of $p_0$.
	Similarly as (iii), part (iv)
	 follows from Theorem \ref{Lpmaxprinc} (ii), applied to 
	\begin{align*}
		u=(k,\nabla (At^{\frac{1}{2}}k),\nabla^2(Bt k))\in C^{\infty}(S^2M\oplus T^*M\otimes S^2M \oplus T^*M\otimes T^*M\otimes S^2M),
	\end{align*}
	where $A$ has to be chosen small in dependence of $p_0$ and $B$ has to be chosen small in dependence of $p_0$ and $A$. The details are left to the reader.
	
It remains to prove (vi). It seems natural to argue as in the proof of (v) and to regard the terms in the evolution of $\nabla^2k$ which contain $k$ and $\nabla k$ as a part of the inhomogeneity. However, in this case, one of the inhomogeneous terms coming from $\n^2\hat{R}[k]$ would be of the form $\n^2(g^{-1}* g) * R * k$ and this can not be covered by Theorem \ref{Lpmaxprinc} without assumptions on $\n^2g$.
To overcome this problem, let us first work with the simpler equation
\begin{align}
\label{eq : a simpler linear heat equation}
\partial_t\varphi=g^{ab}\nabla^2_{ab}\varphi.
\end{align}
By removing some of the curvature terms in the above proofs, we see that the estimates in (i)-(v) of this lemma do also hold for solutions of \eqref{eq : a simpler linear heat equation}. Now we are going to show that (vi) holds for these solutions as well. We see that 
	\begin{align*}
		\partial_t|\nabla^2\varphi|^2
			&=g^{ab}\nabla^2_{ab}|\nabla^2\varphi|^2-2g^{ab}\langle\nabla_a\nabla^2\varphi,\nabla_b\nabla^2\varphi\rangle + 2\langle\nabla(\nabla g^{-1}*\nabla^2\varphi),\nabla^2\varphi\rangle \\
			&\qquad + 2\langle\nabla g^{-1}*\nabla^3\varphi+g^{-1}*\nabla^2 R*\varphi+g^{-1}*\nabla R*\nabla \varphi+g^{-1}*R*\nabla^2 \varphi,\nabla^2\varphi\rangle\\
			&\qquad +\langle \partial_t h*\nabla^2\varphi+\nabla(\nabla\partial_th*\varphi)+\nabla\partial_th*\nabla \varphi,\nabla^2 \varphi\rangle,
	\end{align*}
	where in comparison to the respective expression for $		\partial_t|\nabla^2{k}|^2$ further above, the critical term dropped.
Now we regard all terms which contain $\varphi$ and $\nabla \varphi$ as a part of the inhomogeneity. 
	All these terms are of the form
	\begin{align} \label{eq: higher order mixed terms}
		\nabla(\nabla\partial_th* \varphi),\qquad \nabla\partial_th*\nabla \varphi,\qquad g^{-1} * \nabla^2 R*\varphi,\qquad g^{-1}*\nabla R*\nabla \varphi.
	\end{align}
Similarly as in \eqref{eq: h terms, first kind},
	\begin{align*}
		\left\|\nabla  \partial_t h\right\|_{W^{1,r}}\leq C		\left\|\partial_t h\right\|_{L^{\infty}}<C,\qquad		
		\left\| R\right\|_{W^{2,s}}\leq C<\infty
	\end{align*}
	for all $r\in [p_0,\infty]$ and $s\in [1,\infty]$, with constants depending only on $\mathcal{U},\Lambda$ and $p_0$.
	Applying the short-time estimates (i) for the $W^{1,q}$-norm of $k$ thus shows that all terms in \eqref{eq: higher order mixed terms}  are in $L^p$. 	Therefore, an application of Theorem \ref{Lpmaxprinc} (ii) and \eqref{eq : inhomogeneity estimate} now proves (vi) for solutions of \eqref{eq : a simpler linear heat equation} (even without the $t^{1/2}$-factor).
	
For $t\geq s$, define the propagation operator $P(g,h)_{s\to t}(\varphi_s):=\varphi|_{r=t}$, where $\varphi$ is the solution of the initial value problem
\begin{align*}
\partial_r\varphi=g^{ab}\nabla^2_{ab}\varphi,\qquad \varphi|_{r=s}=\varphi_s.
\end{align*}
By the Duhamel principle, the solution of the initial value problem
\begin{align*}
\partial_tk=g^{ab}\nabla^2_{ab}k+\hat{R}[k],\qquad k|_{t=0}=k_0
\end{align*}
is given by
\begin{align*}
k_t=P(g,h)_{0\to t}k_0+\int_0^tP(g,h)_{s\to t}\hat{R}[k]_sds.
\end{align*}
Because (vi) (without the $t^{1/2}$-factor) holds for solutions of \eqref{eq : a simpler linear heat equation}, we have
\begin{align*}
\left\|\n^2P(g,h)_{0\to t}k_0\right\|_{L^p}\leq e^{C\cdot t}(\left\|\nabla^2 k_0\right\|_{L^{p}}+Ct^{1/p}\left\|k_0\right\|_{W^{1,q}})
\end{align*}
 and because (iii) holds as well, 
\begin{align*}
\left\| \nabla^2 \int_0^tP(g,h)_{s\to t}\hat{R}[k]_sds\right\|_{L^p}&\leq C\int_0^t(t-s)^{-1/2}e^{C(t-s)}\left\| \hat{R}[k]_s\right\|_{W^{1,p}}ds\\
&\leq  Ct^{1/2}e^{Ct}\sup_{s\in[0,t]}\left\| \hat{R}[k]_s\right\|_{W^{1,p}}\\
&\leq  Ct^{1/2}e^{Ct}\sup_{s\in[0,t]}\left\| R\right\|_{W^{1,r}}\left\|k_s\right\|_{W^{1,q}}\\
&\leq  Ct^{1/2}e^{Ct}\left\|k_0\right\|_{W^{1,q}},
\end{align*}
where $\frac{1}{r}=\frac{1}{p}-\frac{1}{q}$.
Here, we used the H\"{o}lder inequality in the third estimate and part (i) of the lemma in the last one. Together, we get the claimed estimate for $\nabla^2k_t$.
\end{proof}
\begin{lem}\label{lem : difference of heat flows}
Let $\U$ and $\F$ be as in Proposition \ref{prop: near metrics}.
Let $g_t, \tilde{g}_t \in \U$ and $h_t, \tilde{h}_t \in \U \cap \F$ be 1-parameter families of metrics on $M$, such that
	\begin{align*}
					|\nabla^{h_t} g_t|_{h_t} 
			\leq \Lambda, \qquad
		\left| \d_t h_t \right|_{h_t} 
%		+ \left| \nabla^{h_t} \d_t h_t \right|_{h_t}
		\leq \Lambda,  \qquad
							|\nabla^{\tilde{h}_t} \tilde{g}_t|_{\tilde{h}_t} 
			\leq \Lambda, \qquad
		\left| \d_t \tilde{h}_t \right|_{\tilde{h}_t} 
%		 + \left| \nabla^{\tilde{h}_t} \d_t \tilde{h}_t \right|_{\tilde{h}_t} 
		 \leq \Lambda
	\end{align*}
	for all $t \geq 0$ and a time-independent constant $\Lambda>0$.  
	Let $1<p_0\leq p\leq p_1<\infty$ and suppose that $\left\|g-\tilde{g}\right\|_{W^{1,\infty}}<\epsilon$ where $\epsilon=\epsilon(p_0,p_1)>0$ is a small constant.
	Let $k_t$ $\tilde{k}_t$, $t\in [0,T]$ be solutions of the evolution equations
	\begin{align*}
		\partial_tk+\Delta_{L,g,h}k=0,\qquad \partial_t\tilde{k}+\Delta_{L,\tilde{g},\tilde{h}}\tilde{k}=0
	\end{align*}
	with initial data $k_0$, $\tilde{k}_0$, respectively.
	In the following, all covariant derivatives and norms are taken with respect to $h_t$.
	\begin{itemize}
		\item[(i)] If $l\in\left\{0,1,2\right\}$ and $k_0,\tilde{k}_0\in W^{l,p}$, then $k_t-\tilde{k}_t\in W^{l,p}$ for all $t\geq0$ and we have
		\begin{align*}
		\left\|k_t-\tilde{k}_t\right\|_{W^{l,p}}
			&\leq e^{C_1\cdot t}\left\|k_0-\tilde{k}_0\right\|_{W^{l,p}}\\
			&\qquad+\sup_{s\in [0,t]}\left(\left\|g_s-\tilde{g}_s\right\|_{W^{1,\infty}}+\left\|h_s-\tilde{h}_s\right\|_{W^{2+l,\infty}}\right)e^{C_2\cdot t}\left\|\tilde{k}_0\right\|_{W^{l,p}}
		\end{align*}
		for some constants $C_i=C_i(\Lambda,n,\mathcal{U},\epsilon,p_0,p_1, T, l)$, $i=1,2$.
		\item[(ii)] If $k_0,\tilde{k}_0\in W^{1,p}$, then $k_t-\tilde{k}_t\in W^{2,p}$ for all $t>0$ and we have
		\begin{align*}\left\|k_t-\tilde{k}_t\right\|_{W^{1,p}}&+C_1\cdot t^{1/2}\cdot \left\|\nabla^2(k_t-\tilde{k}_t)\right\|_{L^p}\leq e^{C_2\cdot t}\left\|k_0-\tilde{k}_0\right\|_{W^{1,p}}\\
			&\qquad+\sup_{s\in [0,t]}\left(\left\|g_s-\tilde{g}_s\right\|_{W^{1,\infty}}+\left\|h_s-\tilde{h}_s\right\|_{W^{4,\infty}}\right)e^{C_3\cdot t}\left\|\tilde{k}_0\right\|_{W^{1,p}}
		\end{align*}		
		for some constants $C_i=C_i(\Lambda,n,\mathcal{U}, \epsilon,p_0, p_1,T)$, $i=1,2,3$.
		\item[(iii)] If $k_0,\tilde{k}_0\in L^{p}$, then $k_t-\tilde{k}_t\in W^{2,p}$ for all $t>0$ and we have
		\begin{align*}&\left\|k_t-\tilde{k}_t\right\|_{W^{1,p}}+C_1\cdot t^{1/2}\cdot \left\|\nabla (k_t-\tilde{k}_t)\right\|_{L^p}+C_2\cdot t\cdot \left\|\nabla^2(k_t-\tilde{k}_t)\right\|_{L^p}\\
			&	\leq e^{C_3\cdot t}\left\|k_0-\tilde{k}_0\right\|_{L^p}
			+\sup_{s\in [0,t]}\left(\left\|g_s-\tilde{g}_s\right\|_{W^{1,\infty}}+\left\|h_s-\tilde{h}_s\right\|_{W^{4,\infty}}\right)e^{C_4\cdot t}\left\|\tilde{k}_0\right\|_{L^p}
		\end{align*}
		for some constants $C_i=C_i(\Lambda,n,\mathcal{U},\epsilon,p_0, p_1,T)$, $i=1,2,3,4$.		
		
		\item[(iv)] If $q\in [p,p_1]$, $k(0),\tilde{k}_0\in L^{q}$ and $\nabla k_0,\nabla \tilde{k}_0\in L^{p}$, then $\nabla( k_t-\tilde{k}_t)\in L^{p}$ for all $t\in [0,T]$ and we have
		\begin{align*}
		&\left\|\nabla (k_t-\tilde{k}_t)\right\|_{L^{p}}
			\leq e^{C_1\cdot t}\left(\left\|\nabla (k_0-\tilde{k}_0)\right\|_{L^{p}}+C_2t^{1/p}\left\| k_0-\tilde{k}_0\right\|_{L^{q}}\right)\\
			&\qquad\qquad+\sup_{s\in [0,t]}\left(\left\|g_s-\tilde{g}_s\right\|_{W^{1,\infty}}+\left\|h_s-\tilde{h}_s\right\|_{W^{3,\infty}}\right)e^{C_1\cdot t}\left(\left\|\nabla\tilde{k}_0\right\|_{L^{p}}+C_2t^{1/p}\left\|\tilde{k}_0\right\|_{L^{q}}\right)
			%			\left\|\nabla\tilde{k}_0\right\|_{L^{p}}\\
%			&\qquad+\sup_{s\in [0,t]}\left(\left\|g_s-\tilde{g}_s\right\|_{W^{1,\infty}}+\left\|h_s-\tilde{h}_s\right\|_{W^{3,\infty}}\right) e^{C_1\cdot t}t^{1/p}\left\|\tilde{k}_0\right\|_{L^{q}}
		\end{align*}
		for some constants $C_i=C_i(\Lambda,n,\mathcal{U},\epsilon, p_0,p_1,T)$, $i=1,2$.
		\item[(v)] If $q\in [p,p_1]$, $k_0,\tilde{k}_0\in W^{1,q}$ and $\nabla^2 k_0,\nabla^2 \tilde{k}_0\in L^{p}$, then $\nabla^2( k_t-\tilde{k}_t)\in L^{p}$ for all $t\in [0,T]$ and we have
		\begin{align*}
		&\left\|\nabla^2 (k_t-\tilde{k}_t)\right\|_{L^{p}}
			\leq e^{C_1\cdot t}\left(\left\|\nabla^2 (k_0-\tilde{k}_0)\right\|_{L^{p}}+C_2(t^{1/p}+t^{1/2}) \left\| k_0-\tilde{k}_0\right\|_{W^{1,q}}\right)\\
			&\quad+\sup_{s\in [0,t]}\left(\left\|g_s-\tilde{g}_s\right\|_{W^{1,\infty}}+\left\|h_s-\tilde{h}_s\right\|_{W^{4,\infty}}\right) e^{C_1\cdot t}
			\left(\left\|\nabla^2\tilde{k}_0\right\|_{L^{p}}+C_2(t^{1/p}+t^{1/2})\left\|\tilde{k}_0\right\|_{W^{1,q}}\right)
%			&\qquad+\sup_{s\in [0,t]}\left(\left\|g_s-\tilde{g}_s\right\|_{W^{1,\infty}}+\left\|h_s-\tilde{h}_s\right\|_{W^{4,\infty}}\right) e^{C_1\cdot t}C_2(t^{1/p}+t^{1/2})\left\|\tilde{k}_0\right\|_{W^{1,q}}
		\end{align*}
		for some constants $C_i=C_i(\Lambda,n,\mathcal{U},\epsilon, p_0,p_1,T)$, $i=1,2$.
	\end{itemize}
\end{lem}
\begin{proof}
Let us first remark that because $g_t,\tilde{g}_t,h_t,\tilde{h}_t\in\mathcal{U}$, we have
\begin{align*}
&\frac{1}{\Lambda}h_t
			\leq g_t\leq \Lambda\cdot h_t,
			\qquad 
		\frac1{\Lambda} h_t 
			\leq \hat h 
			\leq \Lambda h_t,\\
			&\frac{1}{\Lambda}\tilde{h}_t
			\leq \tilde{g}_t\leq \Lambda\cdot \tilde{h}_t,
			\qquad 
		\frac1{\Lambda} \tilde{h}_t 
			\leq \hat h 
			\leq \Lambda \tilde{h}_t,
\end{align*}
provided that $\Lambda>0$ is chosen sufficiently large. Thus, the assumptions of Theorem \ref{Lpmaxprinc} are satisfied. 
Because $T_h\mathcal{F}$ is finite-dimensional, all the norms on it are equivalent. By Proposition \ref{prop : kernel properties} and the assumption $|\partial_th_t|_{h_t}\leq\Lambda$, we thus get for every $k\in\N_0$ and $r\in(1,\infty]$ a bound on the $W^{k,r}$-norm of $\partial_th_t$. We will use this fact frequently in the proof.

	Since there are several metrics involved in this proof, let us use the notation
	\[
		R_{g,h}[k]_{ij}
			:= -k_{ab}g^{ka}h^{lb}g_{ip}h^{pq}R_{jklq}-k_{ab}g^{ka}h^{lb}g_{jp}h^{pq}R_{iklq}
			= g^{-1} * g * R * k.
	\]
	We have the evolution equations
	\begin{align*}
		\partial_tk
			&=g^{ab}\nabla^2_{ab}k+R_{g,h}[k],\\
		\partial_t\tilde{k}
			&=\tilde{g}^{ab}\tilde{\nabla}^2_{ab}\tilde{k}+R_{\tilde{g},\tilde{h}}[\tilde{k}]=g^{ab}\nabla^2_{ab}\tilde{k}+(\tilde{g}^{ab}-g^{ab}){\nabla}^2_{ab}\tilde{k}+\tilde{g}^{ab}(\tilde{\nabla}^2_{ab}-\nabla^2_{ab})(\tilde{k})+R_{\tilde{g},\tilde{h}}[\tilde{k}] \\
			&=g^{ab}\nabla^2_{ab}\tilde{k}+\nabla_a[(\tilde{g}^{ab}-g^{ab}){\nabla}_{b}\tilde{k}]-[\nabla_a(\tilde{g}^{ab}-g^{ab})]\nabla_b\tilde{k}+\tilde{g}^{ab}(\tilde{\nabla}^2_{ab}-\nabla^2_{ab})(\tilde{k})+R_{\tilde{g},\tilde{h}}[\tilde{k}], \\
		\partial_t(k-\tilde{k})
			&= g^{ab}\nabla^2_{ab}(k-\tilde{k})+(g^{ab}-\tilde{g}^{ab})\nabla^2_{ab}\tilde{k}
		+\tilde{g}^{ab}(\nabla^2_{ab}-\tilde{\nabla}^2_{ab})(\tilde{k})+R_{g,h}[k]-R_{\tilde{g},\tilde{h}}[\tilde{k}]\\
			&= g^{ab}\nabla^2_{ab}(k-\tilde{k})+\nabla_a[(g^{ab}-\tilde{g}^{ab})\nabla_{b}\tilde{k}]-[\nabla_a(g^{ab}-\tilde{g}^{ab})]\nabla_{b}\tilde{k}\\
			&\qquad+\tilde{g}^{ab}(\nabla^2_{ab}-\tilde{\nabla}^2_{ab})(\tilde{k})+R_{g,h}[k]-R_{\tilde{g},\tilde{h}}[\tilde{k}].
	\end{align*}
	Note that
	\begin{align*}
		R_{g,h}[k]-R_{\tilde{g},\tilde{h}}[\tilde{k}]&=(g^{-1}-\tilde{g}^{-1})*\tilde{g}*\tilde{R}*\tilde{k}
		+g^{-1}*(g-\tilde{g})*\tilde{R}*\tilde{k}\\
		&\qquad +g^{-1}*g*(R-\tilde{R})*\tilde{k}+g^{-1}*g*R*(k-\tilde{k})
	\end{align*}
	and
	\begin{align*}
		\nabla^2\tilde{k}-\tilde{\nabla}^2\tilde{k}
			&=\nabla^2(h-\tilde{h})*\tilde{h}^{-1}*\tilde{k}+\tilde{h}^{-2}*\nabla(h-\tilde{h})*\nabla(h-\tilde{h})*\tilde{k}+\tilde{h}^{-1}*\nabla(h-\tilde{h})*\nabla\tilde{k},\\
		R-\tilde{R}
			&=\nabla^2(h-\tilde{h})*\tilde{h}^{-1}+\tilde{h}^{-2}*\nabla(h-\tilde{h})*\nabla(h-\tilde{h}).
	\end{align*}
	Thus all these terms are easy to handle because they are at most first order in $\tilde{k}$ and $k-\tilde{k}$.
	We have the evolution equations
	\begin{align*}
	\partial_t|\tilde{k}|^2
		&\leq g^{ab}{\nabla}^2_{ab}|\tilde{k}|^2-2g^{ab}\langle\nabla_a\tilde{k},\nabla_b\tilde{k}\rangle 
		+2\langle\nabla_a[(\tilde{g}^{ab}-g^{ab}){\nabla}_{b}\tilde{k}],\tilde{k}\rangle\\&\qquad
		+2\langle\nabla_a[g^{ab}-\tilde{g}^{ab}]\nabla_b\tilde{k}+\tilde{g}^{ab}(\tilde{\nabla}^2_{ab}-\nabla^2_{ab})(\tilde{k})+R_{\tilde{g},\tilde{h}}[\tilde{k}] + \d_t h* \tilde k ,\tilde{k}\rangle,\\
	\partial_t|k-\tilde{k}|^2
		&\leq g^{ab}{\nabla}^2_{ab}|k-\tilde{k}|^2-2g^{ab}\langle\nabla_a(k-\tilde{k}),\nabla_b(k-\tilde{k})\rangle + 2\langle \nabla_a[(g^{ab}-\tilde{g}^{ab})\nabla_{b}\tilde{k}], k-\tilde{k}\rangle\\
		&\qquad + 2\langle -
		[\nabla_a(g^{ab}-\tilde g^{ab})]\nabla_{b}\tilde{k}+\tilde{g}^{ab}(\nabla^2_{ab}-\tilde{\nabla}^2_{ab})(\tilde{k})+R_{g,h}[k]-R_{\tilde{g},\tilde{h}}[\tilde{k}], k-\tilde{k}\rangle \\
		&\qquad + \langle \d_t h * (k - \tilde k), k-\tilde{k}\rangle.
	\end{align*}
	The crucial point in applying Theorem \ref{Lpmaxprinc} (i) and (ii) is to handle the off diagonal terms appropriately. For this purpose, we write
	\begin{align*}
		u
			=(u_1,u_2)
			=(k-\tilde{k},AB\tilde{k})\in C^{\infty}(S^2M\oplus S^2M),
	\end{align*}
	where
	\begin{align*}
			A:=\sup_{t\in[0,T]}\left\{\left\|g-\tilde{g}\right\|_{W^{1,\infty}}+\left\|h-\tilde{h}\right\|_{W^{2,\infty}}\right\},
	\end{align*}
	and $B>0$ is a constant which is yet to be chosen.
	Writing
	\begin{align*}
		|u|^2=|u_1|^2+|u_2|^2=|k-\tilde{k}|^2+A^2B^2|\tilde{k}|^2,
	\end{align*}
	we get the inequality
	\begin{align*}
		\partial_t|u|^2&\leq g^{ab}{\nabla}^2_{ab}|u|^2-2g^{ab}\langle\nabla_au,\nabla_bu\rangle +2\langle\nabla_a[(\tilde{g}^{ab}-g^{ab}){\nabla}_{b}u_2],u_2\rangle\\&\qquad
		+2\langle\nabla_a[g^{ab}-\tilde{g}^{ab}]\nabla_bu_2+\tilde{g}^{ab}(\tilde{\nabla}^2_{ab}-\nabla^2_{ab})(u_2)+R_{\tilde{g},\tilde{h}}[u_2],u_2\rangle\\
		&\qquad +\frac{2}{AB}\langle - \nabla_a[(g^{ab}-\tilde{g}^{ab})\nabla_{b}u_2]+ [\nabla_a(\tilde{g}^{ab}-g^{ab})]\nabla_{b}u_2+\tilde{g}^{ab}(\nabla^2_{ab}-\tilde{\nabla}^2_{ab})(u_2), u_1\rangle \\
		&\qquad + 2\langle R_{g,h}[k]-R_{\tilde{g},\tilde{h}}[\tilde{k}], u_1\rangle + C\abs{u}^2.
	\end{align*}
	Observe that for every $\epsilon>0$, we can choose $B>0$ so large that
	\begin{align*}
		\frac{1}{AB}|(g^{ab}-\tilde{g}^{ab})\nabla_{b}u_2|\leq \frac{1}{B}|\nabla u|\leq \epsilon |\nabla u|,
	\end{align*}
	where we used the definition of $A$.
	It is also straightforward to see that
	\begin{align*}
		\frac{1}{AB}\left(|[\nabla_a(\tilde{g}^{ab}-g^{ab})]\nabla_{b}u_2|+|\tilde{g}^{ab}(\nabla^2_{ab}-\tilde{\nabla}^2_{ab})(u_2)|\right)\leq C (|\nabla u_2|+|u_2|)
	\end{align*}
	where $C$ is independent of $A$ by the definition of $A$. Furthermore,
	\begin{align*}
		2 |R_{g,h}[k]-R_{\tilde{g},\tilde{h}}[\tilde{k}]|
			\leq C(|u_1|+|u_2|)\leq C|u|
	\end{align*}
	with $C$ independent of $A$.
	For this reason, we can write schematically
	\begin{align*}
		2\langle\nabla_a[(\tilde{g}^{ab}-g^{ab}){\nabla}_{b}u_2],u_2+\frac{1}{AB}u_1\rangle&=  \langle \nabla_a ( (H_3)^{ab}\nabla_b u),u\rangle,\\
		\langle \tilde{g}^{ab}(\tilde{\nabla}^2_{ab}-\nabla^2_{ab})(u_2),u_2+\frac{1}{AB}u_1\rangle&=\langle H_1(u)+H_2(\nabla u),u\rangle, \\
				\langle\nabla_a[g^{ab}-\tilde{g}^{ab}]\nabla_bu_2,u_2+\frac{1}{AB}u_1\rangle&=\langle \overline{H}_2(\nabla u),u\rangle,\\
		\langle R_{\tilde{g},\tilde{h}}[u_2],u_2\rangle +
		\langle R_{g,h}[k]-R_{\tilde{g},\tilde{h}}[\tilde{k}], u_1\rangle
		&=\langle \overline{H}_1(u),u\rangle,
	\end{align*}
	for  some endomorphisms $H_i$, $\overline{H}_i$ satisfying the conditions of Theorem \ref{Lpmaxprinc} (i), with $E$ being $E=S^2M\oplus S^2M$. Note that the scalar products are taken on $S^2M$ on the left hand side and on $S^2M\oplus S^2M$ on the right hand side. Applying Theorem \ref{Lpmaxprinc} (i) yields (i) for $l=0$.
	
	Before we continue estimating derivatives, we remark that for any $(2,0)$-tensor $v$, we have
	\begin{align*}
		[\nabla,v^{ab}\nabla^2_{ab}]k
		&=\nabla v*\nabla^2k+v*R*\nabla k+v*\nabla R*k,\\
		[\nabla^2,v^{ab}\nabla^2_{ab}]k&=
		\nabla(\nabla v*\nabla^2k)+\nabla v*\nabla^3k+v*\nabla^2R*k+v*\nabla R*\nabla k+v*R*\nabla^2 k,
	\end{align*}
	from which we conclude
	\begin{align*}
		\partial_t\nabla\tilde{k}&=g^{ab}\nabla^2_{ab}\nabla\tilde{k}+
		\nabla\left\{\tilde{g}^{ab}(\tilde{\nabla}^2_{ab}-\nabla^2_{ab})(\tilde{k})+R_{\tilde{g},\tilde{h}}[\tilde{k}]\right\}+\nabla_a[(\tilde{g}^{ab}-g^{ab})*\nabla_{b}\nabla\tilde{k}]\\
		&\qquad +\nabla\partial_th*\tilde{k}+\partial_th*\nabla\tilde{k}\\
		&\qquad +\nabla (g^{-1}-\tilde{g}^{-1})*\nabla^2\tilde{k}+(g^{-1}-\tilde{g}^{-1})*R*\nabla \tilde{k}+(g^{-1}-\tilde{g}^{-1})*\nabla R*\tilde{k},\\
		\partial_t\nabla(k-\tilde{k})&=
		g^{ab}\nabla^2_{ab}\nabla(k-\tilde{k})
		+\nabla\left\{\tilde{g}^{ab}(\nabla^2_{ab}-\tilde{\nabla}^2_{ab})(\tilde{k})+R_{g,h}[k]-R_{\tilde{g},\tilde{h}}[\tilde{k}]\right\}
		\\
		&\qquad +\nabla_a[(g^{ab}-\tilde{g}^{ab})*\nabla_{b}\nabla\tilde{k}]+\nabla\partial_th*(k-\tilde{k})+\partial_th*\nabla(k-\tilde{k})\\
		&\qquad +\nabla g^{-1}*\nabla^2(k-\tilde{k})+g^{-1}*R*\nabla (k-\tilde{k})+g^{-1}*\nabla R*(k-\tilde{k})
		\\
		&\qquad +\nabla (g^{-1}-\tilde{g}^{-1})*\nabla^2\tilde{k}+(g^{-1}-\tilde{g}^{-1})*R*\nabla \tilde{k}+(g^{-1}-\tilde{g}^{-1})*\nabla R*\tilde{k},\\
		\partial_t\nabla^2\tilde{k}&=g^{ab}\nabla^2_{ab}\nabla^2\tilde{k}+
		\nabla^2\left\{\tilde{g}^{ab}(\tilde{\nabla}^2_{ab}-\nabla^2_{ab})(\tilde{k})+R_{\tilde{g},\tilde{h}}[\tilde{k}]\right\}
		+\nabla_a[(\tilde{g}^{ab}-g^{ab})*\nabla_{b}\nabla^2\tilde{k}]
		\\&\qquad+\nabla^2\partial_th*\tilde{k}+\nabla\partial_th*\nabla\tilde{k}+\partial_th*\nabla^2\tilde{k}\\
		&\qquad+\nabla(\nabla (g^{-1}-\tilde{g}^{-1})*\nabla^2\tilde{k})+\nabla (g^{-1}-\tilde{g}^{-1})*\nabla^3\tilde{k}\\&\qquad+(g^{-1}-\tilde{g}^{-1})*\nabla^2R*\tilde{k}+(g^{-1}-\tilde{g}^{-1})*\nabla R*\nabla \tilde{k}+(g^{-1}-\tilde{g}^{-1})*R*\nabla^2 \tilde{k},\\
		\partial_t\nabla^2(k-\tilde{k})&=
		g^{ab}\nabla^2_{ab}\nabla^2(k-\tilde{k})
		+\nabla^2\left\{\tilde{g}^{ab}(\nabla^2_{ab}-\tilde{\nabla}^2_{ab})(\tilde{k})+R_{g,h}[k]-R_{\tilde{g},\tilde{h}}[\tilde{k}]\right\}
		\\
		&\qquad +\nabla_a[(g^{ab}-\tilde{g}^{ab})*\nabla_{b}\nabla^2\tilde{k}]\\
		&\qquad+\nabla^2\partial_th*(k-\tilde{k})+\nabla\partial_th*\nabla(k-\tilde{k})+\partial_th*\nabla^2(k-\tilde{k})\\
		&\qquad +\nabla g^{-1}*\nabla^2(k-\tilde{k})+g^{-1}*R*\nabla (k-\tilde{k})+g^{-1}*\nabla R*(k-\tilde{k})
		\\
		&\qquad +\nabla (g^{-1}-\tilde{g}^{-1})*\nabla^2\tilde{k}+(g^{-1}-\tilde{g}^{-1})*R*\nabla \tilde{k}+(g^{-1}-\tilde{g}^{-1})*\nabla R*\tilde{k}\\
		&\qquad+\nabla(\nabla (g^{-1}-\tilde{g}^{-1})*\nabla^2\tilde{k})+\nabla (g^{-1}-\tilde{g}^{-1})*\nabla^3\tilde{k}\\&\qquad+(g^{-1}-\tilde{g}^{-1})*\nabla^2R*\tilde{k}+(g^{-1}-\tilde{g}^{-1})*\nabla R*\nabla \tilde{k}+(g^{-1}-\tilde{g}^{-1})*R*\nabla^2 \tilde{k}\\
		&\qquad+\nabla(\nabla g^{-1}*\nabla^2(k-\tilde{k}))+\nabla g^{-1}*\nabla^3(k-\tilde{k})\\&\qquad+g^{-1}*\nabla^2R*(k-\tilde{k})+g^{-1}*\nabla R*\nabla(k-\tilde{k})+g^{-1}*R*\nabla^2 (k-\tilde{k}).
	\end{align*}
	In the following, we sketch to which expressions we have to apply Theorem \ref{Lpmaxprinc} (ii) in order to get all the other cases of the Lemma. 
	The details are left to the reader.
	
	The cases $l=1,2$ in (i) follow from applying Theorem \ref{Lpmaxprinc} (ii) to
	\begin{align*}
		u&=(k-\tilde{k},AB\tilde{k},\nabla(k-\tilde{k}),AB\nabla\tilde{k})\in C^{\infty}(S^2M^{\oplus2}, (T^*M\otimes S^2M)^{\oplus2}),\\
		A&:=\sup_{t\in[0,T]}\left\{\left\|g-\tilde{g}\right\|_{W^{1,\infty}}+\left\|h-\tilde{h}\right\|_{W^{3,\infty}}\right\}
	\end{align*}
	and to 
	\begin{align*}
		u&=(k-\tilde{k},AB\tilde{k},\nabla(k-\tilde{k}),AB\nabla\tilde{k},\nabla^2(k-\tilde{k}),AB\nabla^2\tilde{k}) \\
		&\qquad \in C^{\infty}\left((\R\oplus T^*M\oplus T^*M^{\otimes2})\otimes S^2M^{\oplus2}\right),
	\end{align*}
	where
	\begin{align*}
		A :=\sup_{t\in[0,T]}\left\{\left\|g-\tilde{g}\right\|_{W^{1,\infty}}+\left\|h-\tilde{h}\right\|_{W^{4,\infty}}\right\},
	\end{align*}
	and in both cases, $B>0$ is a constant which is suitably chosen.
	To prove (ii), we  apply Theorem \ref{Lpmaxprinc} (ii) to
	\begin{align*}
		u&=(k-\tilde{k},AB_1\tilde{k},\nabla(k-\tilde{k}),AB_1\nabla\tilde{k},AB_2t^{\frac{1}{2}}\nabla^2(k-\tilde{k}),AB_1B_2t^{\frac{1}{2}}\nabla^2\tilde{k}),\\
		A&:=\sup_{t\in[0,T]}\left\{\left\|g-\tilde{g}\right\|_{W^{1,\infty}}+\left\|h-\tilde{h}\right\|_{W^{4,\infty}}\right\},
	\end{align*}
	where $B_1>0$ is a large constant and $B_2>0$ is a small one.
	To prove (iii), we apply it to
	\begin{align*}
		u&=(k-\tilde{k},AB_1\tilde{k},B_2 t^{\frac{1}{2}}\nabla(k-\tilde{k}),AB_1 B_2t^{\frac{1}{2}} \nabla\tilde{k},B_3t\nabla^2(k-\tilde{k}),AB_1B_3t\nabla^2\tilde{k}),\\
		A&:=\sup_{t\in[0,T]}\left\{\left\|g-\tilde{g}\right\|_{W^{1,\infty}}+\left\|h-\tilde{h}\right\|_{W^{4,\infty}}\right\},
	\end{align*}
	where $B_1>0$ is a large constant and $B_2,B_3>0$ are small ones.
	Case (iv) follows from applying Theorem \ref{Lpmaxprinc} (ii) to
	\begin{align*}
		u&=(\nabla(k-\tilde{k}),AB\nabla\tilde{k}),\\
		A&:=\sup_{t\in[0,T]}\left\{\left\|g-\tilde{g}\right\|_{W^{1,\infty}}+\left\|h-\tilde{h}\right\|_{W^{3,\infty}}\right\},
	\end{align*}
	where $B>0$ is a large constant. Here, the terms $k-\tilde{k},AB\tilde{k}$ are treated as inhomogeneities which can be bounded using (i), cf. also the proof of the previous lemma.

It remains to prove (v). Analogous to part (iv), we would want to apply Theorem \ref{Lpmaxprinc} (ii) to
	\begin{align*}
		u&=(\nabla^2(k-\tilde{k}),AB\nabla^2\tilde{k}),\\
		A&:=\sup_{t\in[0,T]}\left\{\left\|g-\tilde{g}\right\|_{W^{1,\infty}}+\left\|h-\tilde{h}\right\|_{W^{4,\infty}}\right\},
	\end{align*}
	with a large constant $B>0$ and treat the terms $k-\tilde{k},AB\tilde{k},\nabla(k-\tilde{k}),AB\nabla\tilde{k}$ as inhomogeneities. However, as in the proof of Lemma \ref{lem : short-time estimates} (vi), we have problematic inhomogeneous terms. In the evolution of  $\nabla^2\tilde{k}$ and $\nabla^2(k-\tilde{k})$, we have the terms $\nabla^2I$, $\nabla^2J$, respectively, where
\begin{equation}\label{eq:define_I,J}
\begin{split}
	I&:=\tilde{g}^{ab}(\tilde{\nabla}^2_{ab}-\nabla^2_{ab})(\tilde{k})+R_{\tilde{g},\tilde{h}}[\tilde{k}],\\
	J&:=\tilde{g}^{ab}(\nabla^2_{ab}-\tilde{\nabla}^2_{ab})(\tilde{k})+R_{g,h}[k]-R_{\tilde{g},\tilde{h}}[\tilde{k}].
\end{split}
\end{equation}
The second derivatives of $I$ and $J$ contain tensor products of second derivatives of $g$ and $\tilde{g}$ with zero and first derivatives of $\tilde{k}$ and $k-\tilde{k}$.
Thus we can not apply Theorem \ref{Lpmaxprinc} directly without making assumptions on second derivatives of $g$ and $\tilde{g}$. To overcome this problem, we proceed similarly as in the proof of Lemma \ref{lem : short-time estimates} (vi) and first consider the reduced system 
\begin{equation}\label{eq:simpler_system}
\begin{split}
	\partial_t\tilde{\varphi}
					&=g^{ab}\nabla^2_{ab}\tilde{\varphi}+\nabla_a[(\tilde{g}^{ab}-g^{ab}){\nabla}_{b}\tilde{\varphi}]-[\nabla_a(\tilde{g}^{ab}-g^{ab})]\nabla_b\tilde{\varphi}, \\	
		\partial_t(\varphi-\tilde{\varphi})
			&= g^{ab}\nabla^2_{ab}(\varphi-\tilde{\varphi})+\nabla_a[(g^{ab}-\tilde{g}^{ab})\nabla_{b}\tilde{\varphi}]-[\nabla_a(g^{ab}-\tilde{g}^{ab})]\nabla_{b}\tilde{\varphi},
\end{split}
\end{equation}
where in comparison to the system on $\tilde{k}$ and $k-\tilde{k}$, the critical terms $I$ and $J$ are dropped.
It is clear from the above arguments that $\varphi_t-\tilde{\varphi}_t$
satisfies the estimates (i)-(iv) of this lemma (with $k$ and $\tilde{k}$ obviously being replaced by $\varphi$ and $\tilde{\varphi}$). To check (v), observe that if we compare the equations for $\partial_t\nabla^2\tilde{\varphi}$ and $\partial_t\nabla^2(\varphi-\tilde{\varphi})$ to those for $\partial_t\nabla^2\tilde{k}$ and $\partial_t\nabla^2(k-\tilde{k})$, the expressions $\nabla^2I$ and $\nabla^2J$ just drop out. Under the assumptions of the lemma, we can therefore apply Theorem \ref{Lpmaxprinc} (ii) to
	\begin{align*}
		u&=(\nabla^2(\varphi-\tilde{\varphi}),AB\nabla^2\tilde{\varphi}),\\
		A&:=\sup_{t\in[0,T]}\left\{\left\|g-\tilde{g}\right\|_{W^{1,\infty}}+\left\|h-\tilde{h}\right\|_{W^{4,\infty}}\right\},
	\end{align*}
	with a large constant $B>0$, where we treat the terms $\varphi-\tilde{\varphi},AB\tilde{\varphi},\nabla(\varphi-\tilde{\varphi}),AB\nabla\tilde{\varphi}$ as inhomogeneities. Using \eqref{eq : inhomogeneity estimate} in addition, we obtain estimate (v) (even without the $t^{1/2}$-factors) for $\varphi-\tilde{\varphi}$.

To handle the full system, define for $s\geq r$ the propagation operator \begin{align*}Q(g,h)_{r\to s}(\tilde{\varphi}_r,\varphi_r-\tilde{\varphi}_r):=(\tilde{\varphi}|_{t=s},(\varphi-\tilde{\varphi})|_{t=s}),
\end{align*}
 where the tuple $(\tilde{\varphi},\varphi-\tilde{\varphi})$
 solves \eqref{eq:simpler_system}
 with initial data $(\tilde{\varphi}_r,\varphi_r-\tilde{\varphi}_r)$ given at time $t=r$. By the Duhamel principle, the solution for the full system on $(\tilde{k},k-\tilde{k})$ is given by
 \begin{align*}
(\tilde{k}_t,k_t-\tilde{k}_t)= Q(g,h)_{0\to t}(\tilde{k}_0,k_0-\tilde{k}_0)
+\int_0^tQ(g,h)_{s\to t}(I_s,J_s)ds,
 \end{align*}
 where $I$ and $J$ are defined in \eqref{eq:define_I,J}.
For obtaining the desired estimate, it suffices to consider the second component of this expression, which we can write as
\begin{align*}
k_t-\tilde{k}_t= Q_2(g,h)_{0\to t}(\tilde{k}_0,k_0-\tilde{k}_0)
+\int_0^tQ_2(g,h)_{s\to t}(I_s,J_s)ds.
\end{align*}
Here, $Q_2(g,h)_{0\to t}$ denotes the second component of the propagation operator $Q(g,h)_{0\to t}$. By the triangle inequality, 
\begin{align*}
\left\|\nabla^2(k_t-\tilde{k}_t)\right\|_{L^p}\leq
\left\|\nabla^2(Q_2(g,h)_{0\to t}(\tilde{k}_0,k_0-\tilde{k}_0))\right\|_{L^p}
+ \int_0^t\left\|\nabla^2(Q_2(g,h)_{s\to t}(I_s,J_s))\right\|_{L^p}ds.
\end{align*}
Because the estimate (v) holds for the reduced system \eqref{eq:simpler_system} without the $t^{1/2}$-factors, we have
\begin{align*}
&\left\|\nabla^2(Q_2(g,h)_{0\to t}(\tilde{k}_0,k_0-\tilde{k}_0))\right\|_{L^p} \leq e^{C t}\left(\left\|\nabla^2 (k_0-\tilde{k}_0)\right\|_{L^{p}}+Ct^{1/p}\left\| k_0-\tilde{k}_0\right\|_{W^{1,q}}\right)\\
			&\qquad +\sup_{s\in [0,t]}\left(\left\|g_s-\tilde{g}_s\right\|_{W^{1,\infty}}+\left\|h_s-\tilde{h}_s\right\|_{W^{4,\infty}}\right) e^{C t}
			\left(\left\|\nabla^2\tilde{k}_0\right\|_{L^{p}}
			+Ct^{1/p}\left\|\tilde{k}_0\right\|_{W^{1,q}}\right).
%			&\qquad\qquad \qquad +\sup_{s\in [0,t]}\left(\left\|g_s-\tilde{g}_s\right\|_{W^{1,\infty}}+\left\|h_s-\tilde{h}_s\right\|_{W^{3,\infty}}\right) e^{C t}Ct^{1/p}\left\|\tilde{k}_0\right\|_{W^{1,q}}.
\end{align*}
Because (ii) holds for the reduced system as well, we have
\begin{align*}
&\int_0^t\left\|\nabla^2(Q_2(g,h)_{s\to t}(I_s,J_s))\right\|_{L^p}ds\leq
C\int_0^t(t-s)^{-1/2}e^{C(t-s)}\left\|J_s\right\|_{W^{1,p}}ds\\
&\qquad\qquad+C\int_0^t(t-s)^{-1/2}e^{C (t-s)}\sup_{r\in [s,t]}\left(\left\|g_r-\tilde{g}_r\right\|_{W^{1,\infty}}+\left\|h_r-\tilde{h}_r\right\|_{W^{4,\infty}}\right)\left\|I_s\right\|_{W^{1,p}}ds\\
&\qquad\leq Ct^{1/2}e^{Ct}\left(\sup_{s\in [0,t]}\left\|J_s\right\|_{W^{1,p}}
+
\sup_{s\in [0,t]}\left(\left\|g_s-\tilde{g}_s\right\|_{W^{1,\infty}}+\left\|h_s-\tilde{h}_s\right\|_{W^{4,\infty}}\right)
\sup_{s\in [0,t]}\left\|I_s\right\|_{W^{1,p}}\right).
\end{align*}
The desired inequality in (v) follows from estimating the $W^{1,p}$ norms of $I$ and $J$ using the H\"{o}lder inequality, Lemma \ref{lem : short-time estimates} (i) and (vi) applied to $\tilde{k}$ and Lemma \ref{lem : difference of heat flows} (i) applied to $k-\tilde{k}$.
\end{proof}
\subsection{Short-time estimates for the Ricci-de Turck flow}\label{subsec: shorttime estimates}
We consider a fixed Ricci-flat background metric $h$ and the $h$-gauged Ricci-de Turck flow $g(t)$.
Defining $k(t)=g(t)-h$, we recall equation \eqref{eq : expression RdT3}, which takes the form
\begin{equation} \label{eq : expression RdT3 local}
	\partial_tk+\Delta k	
		=F_4(g^{-1},g^{-1},\nabla k,\nabla k)+F_5(g^{-1},g,R,k)+\nabla_a((g^{ab}-h^{ab})\nabla_bk_{ij}),
\end{equation}
where the covariant derivatives, Laplacians and curvature are all with respect to $h$.
For convenience, let us abbreviate
\begin{align*}
	R[k]
		&=F_5(g^{-1},g,R,k), \\
	Q_0[k]
		&=F_4(g^{-1},g^{-1},\nabla k,\nabla k), \\
	Q_1[k]
		&= \n_\a ((h+k)^{\a\b}-h^{\a\b})\nabla_\b k_{ij}.
\end{align*}
We use the convention here that a latin letter power of $\n$,  like $\n^l$, denotes the $l$'th covariant derivative, whereas a greek index, like $\n^\a$ or $\n_\a$ denotes a component of the covariant derivative in abstract index notation.
Thus, if $k$ evolves according to \eqref{eq : expression RdT3 local}, then we have
\begin{align*}
	\partial_t |k|^{2}+\Delta |k|^{2}
		&=-2|\nabla k|^2+2\langle R[k]+Q_0[k]+Q_1[k],k\rangle
\end{align*}
and for covariant derivatives,
\begin{equation} \label{eq: evolution nabla l k}
\begin{split}
	\partial_t |\nabla^lk|^{2}+\Delta |\nabla^lk|^{2}
		=&2\langle [\nabla^l,\Delta]k+\nabla^l(R[k]+Q_0[k])+[\nabla^l, Q_1][k]+ Q_1[\nabla^lk],\nabla^lk\rangle\\ 
		&\quad-2|\nabla^{l+1} k|^2.
\end{split}
\end{equation}
\begin{lem}\label{lem : Ricci flow short-time estimates in Ck}
	Let $h$ be a fixed Ricci-flat ALE metric and let $T > 0$.
	For every $m_0\in \N_0$, there exists an $\e_0 > 0$ such that if $g(t)$ is a solution to the $h$-gauged Ricci-de Turck flow for $t \in [0, T]$ and $k(t):=g(t)-h$ satisfies
	\begin{align*}
		\left\|k(0)\right\|_{L^{\infty}}<\e_0,
	\end{align*}
	then for every $l \in \N_0$ such that $l \leq m_0$ there exists a constant $C = C(h, T, m_0, \e_0)$ such that
	\[
		\left\|\nabla^l k(t)\right\|_{L^{\infty}}
			\leq C t^{-l/2}\cdot\left\| k(0)\right\|_{L^{\infty}}.
	\]
\end{lem}
\begin{proof}
This is a standard short-time existence result, see e.g.~\cite{Bam14}*{Proposition 2.8}, which easily carries over to the present situation.
\end{proof}
\begin{lem}\label{F_k_estimate}
	Let $h$ be a fixed Ricci-flat ALE metric and let $T > 0$.
	For every $\delta>0$ and $m\in \N_0$, there exists an $\e > 0$ and constants 
	\[
		c_0, \hdots, c_m > 0,
	\]
	independent of $h$, $T$, and $\e$ and a constant
	\[
		d_m > 0
	\]
	depending on $h$, $T$, $\de$, $m$ and $\e$, such that if $g(t)$ is a solution to the $h$-gauged Ricci-de Turck flow for $t \in [0, T]$ and $k(t):=g(t)-h$ satisfies
	\[
		\left\|k(0)\right\|_{L^\infty} 
			< \e,
	\]
	then the function
	\begin{align*}
		P_m(t)
			:=\sum_{l=0}^m c_l\cdot t^l\cdot |\nabla^lk(t)|^2,
	\end{align*}
	satisfies the evolution inequality
	\begin{align*}
		\partial_tP_m + \Delta P_m
			\leq -(2-\delta)G_m + d_m P_m + 2\sum_{l=0}^m c_l\cdot t^l\cdot \langle Q_1[\nabla^{l}k],\nabla^{l}k\rangle
	\end{align*}
	for all $t \in [0, T]$, where
	\[
		G_m(t)
			=\sum_{l=0}^m c_l\cdot t^l\cdot |\nabla^{l+1}k(t)|^2.
	\]
\end{lem}
\begin{proof}
	Fix an $\e_0 > 0$ small enough in order to apply Lemma \ref{lem : Ricci flow short-time estimates in Ck} with $m_0 := m + 1$.
	All constants denoted $C$ in this proof may depend on $h, T, \de, m, \e_0$, without further mentioning, and might change from line to line.
	By Lemma \ref{lem : Ricci flow short-time estimates in Ck}, we have
	\begin{equation} \label{eq: nabla l k derivatives}
		\left\|\nabla^lk(t)\right\|_{L^{\infty}}
			\leq C t^{-l/2}\left\|k(0)\right\|_{L^{\infty}}
	\end{equation}
	for all $0 \leq l\leq m + 1$ and $t\in (0,T]$.
	We assume from now on that
	\[
		\norm{k(0)}_{L^\infty}
			< \e,
	\]
	where $\e \in (0, \e_0]$ is still to be chosen (the constants denoted $C$ will be independent of $\e$).
	Since $\norm{k(0)}_{L^\infty} < \e_0$, the estimate \eqref{eq: nabla l k derivatives} implies that if $\e$ is small enough, then it follows that $\norm{g(t)^{-1}}_{L^\infty} \leq C$, and hence
	\[
		\left\|\nabla^l g\right\|_{L^{\infty}} + \left\|\nabla^l g^{-1}\right\|_{L^{\infty}}
			\leq C t^{-l/2}\left\|k(0)\right\|_{L^{\infty}}.
	\]
	We estimate each term on the right hand side of \eqref{eq: evolution nabla l k} one-by-one.
	Standard computations and \eqref{eq: nabla l k derivatives} imply that
	\begin{equation} \label{eq: short time first commutator}
		\langle[\nabla^l,\Delta]k,\nabla^lk\rangle
			=\sum_{j=0}^l\nabla^jR*\nabla^{l-j}k*\nabla^l k
			\leq C \sum_{j=0}^l|\nabla^j k|^2
	\end{equation}
	for all $t \in [0, T]$ and all $l \leq m$.
	Recalling that $R[k] = F_5(g^{-1},g,R,k)$, we estimate
	\begin{equation} \label{eq: short time R term}
	\begin{split}
	\langle \n^l R[k],\nabla^lk\rangle
		&\leq C \sum_{l_1+l_2+l_3+l_4=l}|\nabla^{l_1}k||\nabla^{l_2}g^{-1}||\nabla^{l_3}g||\nabla^{l_4}R||\nabla^lk|\\
		&\leq C \sum_{i + j \leq l} |\nabla^j k| t^{- \frac i2} \norm{k(0)}_{L^\infty}^i |\nabla^lk|\\
		&= C \left( \sum_{j = 0}^l |\nabla^j k| |\nabla^l k | + \sum_{\stackrel{i + j \leq l}{i \geq 1}} |\nabla^j k| t^{- \frac i2} \norm{k(0)}_{L^\infty}^i |\nabla^lk| \right) \\	
		&\leq C \sum_{j = 0}^l |\nabla^j k|^2 + C\cdot \e \cdot  \sum_{\stackrel{i + j \leq l}{i \geq 1}} t^{- i}|\nabla^j k|^2 \\
		&\leq C \sum_{j = 0}^{l} | \nabla^j k|^2 + C\cdot \e \cdot \sum_{j = 0}^{l-1} t^{j-l}|\nabla^j k|^2\\
		&\leq C \sum_{j = 0}^{l} | \nabla^j k|^2 + C\cdot \e \cdot \sum_{j = 1}^{l-1} t^{j-(l+1)}|\nabla^j k|^2+Ct^{-l}|k|^2		
	\end{split}
	\end{equation}
	Recalling that $Q_0[k] = F_4(g^{-1},g^{-1},\nabla k,\nabla k)$, we estimate
	\begin{equation} \label{eq: short time Q 0 term}
	\begin{split}
	\langle \nabla^lQ_0[k],\nabla^l k\rangle
		&\leq C \sum_{j = 0}^l |\nabla^{l-j}\left( g^{-1}*g^{-1} * \n k \right)| \cdot |\nabla^{j + 1}k| \cdot |\nabla^lk|\\
		&\leq C \left( t \e^{-1}\sum_{j = 0}^l |\nabla^{l-j}\left( g^{-1}*g^{-1} * \n k \right)|^2 \cdot |\nabla^{j + 1}k|^2 + \frac 14 t^{-1}\e |\nabla^l k|^2 \right)\\
		&\leq C \left(  \sum_{j = 0}^l t^{j-l} \norm{k(0)}_{L^\infty}^2 \e^{-1} \cdot |\nabla^{j + 1}k|^2 + \frac 14 t^{-1}\e |\nabla^l k|^2 \right)\\
		&\leq C\cdot \e \left(  \sum_{j = 0}^l t^{j-l} |\nabla^{j + 1}k|^2 + \frac 14 t^{-1} |\nabla^l k|^2 \right)\\
		&\leq C\cdot \e \sum_{j = 1}^{l+1} t^{j-(l+1)} |\nabla^jk|^2
	\end{split}
	\end{equation}
	for all $l \leq m$.
	Continuing with the next term, note that
	\begin{align*}
		[\n^l, Q_1][k]
			= \nabla^l\nabla_\a(\Psi^{\a\b}\nabla_\b k)-\nabla_\a \Psi^{\a\b}\nabla_\b\nabla^l k
	\end{align*}
	where $\Psi^{\a\b}=(h+k)^{\a\b}-h^{\a\b}$. 
	We rewrite this as
	\begin{align*}
		\nabla^l\nabla_\a(\Psi^{\a\b}\nabla_\b k)-\nabla_\a\Psi^{\a\b}\nabla_\b\nabla^l k
			&=[\nabla^l,\nabla_\a]\Psi^{\a\b}\nabla_\b k+
		\nabla_\a([\nabla^l,\Psi^{\a\b}]\nabla_\b k)+\nabla_\a(\Psi^{\a\b}[\nabla^l,\nabla_\b]k)\\
		&=\sum_{j=0}^l\nabla^{j + 1} \Psi*\nabla^{l+1-j}k+\sum_{l_1+l_2+l_3=l} \nabla^{l_1} R *\nabla^{l_2}\Psi*\nabla^{l_3}k
	\end{align*}
	and therefore
	\begin{align*}
		\langle[\nabla^l, Q_1][k],\nabla^lk\rangle
			=\sum_{j=0}^l\nabla^{j + 1}\Psi*\nabla^{l+1-j}k*\nabla^lk+\sum_{l_1+l_2+l_3=l} \nabla^{l_1}R*\nabla^{l_2}\Psi*\nabla^{l_3}k*\nabla^lk.
	\end{align*}
	By \eqref{eq: nabla l k derivatives}, we have that
	$\left\|\nabla^j  \Psi \right\|_{L^{\infty}}
		\leq C \cdot t^{-j/2}\left\|k(0)\right\|_{L^{\infty}}$ for all $j \leq l+1$.
	Hence the first of these two expressions is estimated by
	\begin{equation} \label{eq: first part PSI estimate}
	\begin{split}
		\sum_{j=0}^l\nabla^{j+1}\Psi*\nabla^{l+1-j}k*\nabla^lk
			&\leq C \left( \epsilon^{-1}\cdot t \sum_{j=0}^l |\nabla^{j+1}\Psi|^2|\nabla^{l+1-j}k|^2+\epsilon\cdot t^{-1}\cdot|\nabla^lk|^2\right)\\
		&\leq C \cdot\epsilon\left(\sum_{j=0}^lt^{-j}|\nabla^{l+1-j}k|^2+ t^{-1}\cdot|\nabla^lk|^2\right) \\
		&\leq  C \cdot\epsilon\cdot\sum_{j=1}^{l+1}t^{-(l+1)+j}|\nabla^jk|^2,
	\end{split}
	\end{equation}
	for all $l \leq m$, while the other one is estimated by
	\begin{equation} \label{eq: second part PSI estimate}
	\begin{split}
		\sum_{l_1+l_2+l_3=l} &\nabla^{l_1}R*\nabla^{l_2}\Psi*\nabla^{l_3}k*\nabla^lk \\
			&\leq C\left(\epsilon^{-1}\sum_{l_1+l_2+l_3=l} |\nabla^{l_1}R|^2|\nabla^{l_2}\Psi|^2|\nabla^{l_3}k|^2+ \epsilon|\nabla^lk|^2\right)\\
			&\leq C\cdot\epsilon\left(\sum_{l_2+l_3\leq l}t^{-l_2}|\nabla^{l_3}k|^2+ |\nabla^lk|^2\right)\\
			&\leq C\cdot\epsilon\left(\sum_{l_2+l_3= l}t^{-l_2}|\nabla^{l_3}k|^2+ |\nabla^lk|^2\right)\\
			&= C\cdot\epsilon\sum_{j=0}^lt^{-l+j}|\nabla^{j}k|^2\\
			&\leq C\cdot\epsilon\sum_{j=1}^lt^{-(l+1)+j}|\nabla^{j}k|^2+Ct^{-l}|k|^2
	\end{split}
	\end{equation}
	for all $l \leq m$.
	By equation \eqref{eq: evolution nabla l k} and the estimates \eqref{eq: short time first commutator}, \eqref{eq: short time R term}, \eqref{eq: short time Q 0 term}, \eqref{eq: first part PSI estimate} and \eqref{eq: second part PSI estimate}, we get
\begin{align*}
	\partial_tP_m+\Delta P_m
		&\leq \sum_{l=0}^ml c_l t^{l-1}|\nabla^lk|^2 
			+ C \sum_{l=0}^m c_l t^l \sum_{j=0}^l|\nabla^jk|^2 \\*
		&\qquad +   C \cdot \e \cdot \sum_{l=0}^m c_l\sum_{j=1}^{l+1}t^{j-1}|\nabla^{j}k|^2  +C\left(\sum_{l=0}^mc_l\right)|k|^2\\*
		&\qquad + 2\sum_{l=0}^m c_l t^l\langle Q_1[\nabla^lk],\nabla^lk\rangle - 2\sum_{l=0}^mc_l\cdot t^l\cdot|\nabla^{l+1}k|^2\\
		&= (- 2 + C\cdot \e) c_m t^m |\nabla^{m+1}k|^2+C\left(\sum_{l=0}^mc_l\right)|k|^2 \\*
		&\qquad + \sum_{l=0}^{m-1}\left( (l+1)c_{l+1} - 2c_l + C \cdot\e  \cdot \sum_{j = l}^{m} c_j \right) \cdot t^l\cdot |\nabla^{l+1}k|^2 \\*
		&\qquad + C \sum_{l=0}^m c_l t^l \sum_{j=0}^l|\nabla^jk|^2
			+2\sum_{l=0}^m c_l\cdot t^l\langle Q_1[\nabla^lk],\nabla^lk\rangle.
\end{align*}
In the equality above, we used that by index shifting and interchanging $j$ and $l$,
\begin{align*}
\sum_{l=0}^m c_l\sum_{j=1}^{l+1}t^{j-1}|\nabla^{j}k|^2&=\sum_{l=0}^m c_l\sum_{j=0}^{l}t^{j}|\nabla^{j+1}k|^2\\
&=\sum_{0\leq l\leq j\leq m}c_j t^{l}|\nabla^{l+1}k|^2
\\ &=
c_mt^m|\nabla^{m+1}k|^2 +\sum_{l=0}^{m-1}\left(\sum_{j=l}^mc_j\right)t^{l}|\nabla^{l+1}k|^2.
\end{align*}
For a sufficiently large $d_m > 0$, we obtain
\begin{align*}
	\partial_tP_m+\Delta P_m
				&\leq  (- 2 + C\cdot \e) c_m t^m |\nabla^{m+1}k|^2 \\*
		&\qquad + \sum_{l=0}^{m-1}\left( (l+1)c_{l+1} - 2c_l+C\cdot \e \cdot c_l + C \cdot\e  \cdot \sum_{j = l+1}^{m} c_j \right) \cdot t^l\cdot |\nabla^{l+1}k|^2 \\*
		&\qquad+ d_m P_m 
			+2\sum_{l=0}^m c_l\cdot t^l\langle Q_1[\nabla^lk],\nabla^lk\rangle.
%
%
%	\partial_tP_m+\Delta P_m
%		&\leq \sum_{l=0}^m\left( -2 + C\cdot \e \cdot \sum_{j = l}^m c_j \right) \cdot t^l\cdot |\nabla^{l+1}k|^2 \\*
%		&\qquad + d_m P_m + 2\sum_{l=0}^m c_l\cdot t^l\langle Q_1[\nabla^lk],\nabla^lk\rangle
\end{align*}
Choosing $\e>0$ so small that $C\cdot\e<\frac{\delta}{2}$, $c_m = 1$ and then iteratively $c_{m-1}, \hdots, c_0$ such that
\[
	(l + 1)c_{l+1} +\frac{1}{2}\delta \sum_{j=l+1}^mc_j
		\leq \frac{1}{2}\delta c_l
\]
ensures
\[
	(l + 1)c_{l+1} +C\cdot \e\sum_{j=l+1}^mc_j
		\leq (\delta-C\cdot\e)c_l
\]
%we get the bound
%
%Choosing $\e > 0$ such that
%\[
%	C \cdot \e \cdot \sum_{j = l}^m c_j 
%		< \de\cdot c_l
%\]
%for all $l\in\left\{0,\ldots,m\right\}$
and completes the proof.
\end{proof}
\begin{rem}
	By multiplying all $c_i$ in Lemma \ref{F_k_estimate} with $c_0^{-1}$, we can always choose $c_0 = 1$, so that $P_0(0)=|k|^2$. 
	We assume that this convention holds from now on.
\end{rem}
\begin{lem}\label{lem : Ricci flow short-time estimates}
	Let $h$ be a fixed Ricci-flat ALE metric and let $T > 0$ and let $p_0 > 1$.
	Then for every $m \in \N$, there exists an $\e > 0$ such that if $g(t)$ is a solution to the $h$-gauged Ricci-de Turck flow for $t \in [0, T]$ and $k(t):=g(t)-h$ satisfies
	\begin{align*}
		\norm{k(0)}_{L^\infty}
			&< \e, \\
		\norm{k(0)}_{L^p}
			&< \infty,
	\end{align*}
	for some $p \in [p_0, \infty)$, then for every $l \in \N_0$, $l \leq m$, there exists a constant 
	\[
		C = C(h,m,\e,T,p_0),
	\]
	but independent of $p$, such that
	\begin{align*}
		\left\|\nabla^l k(t)\right\|_{L^p}
			\leq C t^{-l/2} \left\| k(0)\right\|_{L^p}.
	\end{align*}
\end{lem}
\begin{proof}
Choose $\de > 0$ such that $1 + \de < p_0$.
Applying Lemma \ref{F_k_estimate} with this choice of $\de$ implies that there is an $\e > 0$ and constants $c_0, \hdots, c_m > 0$, such that
\begin{align*}
	\partial_t|u_m|^2
		\leq h^{ab}\nabla^2_{ab}|u_m|^2 +D_m |u_m|^2+\langle \nabla_a((h+k)^{ab}-h^{ab})\nabla_b u_m,u_m\rangle-2(1-\delta)|\nabla u_m|^2,
\end{align*}
where
\begin{align*}
	u_m(t)
		:= \left(\sqrt{c_l}^{1/2} t^{l/2} \nabla^l k(t) \right)_{l = 0}^m \in C^{\infty}\left(\bigoplus_{l=0}^m (T^*M)^{\otimes l}\otimes S^2M\right),
\end{align*}
and where we have used that
\begin{align*}
	P_m
		&= |u_m|^2 \\
	G_m
		&= |\nabla u_m|^2.
\end{align*}
Now, if $\epsilon>0$ is chosen small enough,
\begin{align*}
\left\|k(t)\right\|_{L^{\infty}}<C\epsilon
\end{align*}
for all $t\in (0,T]$. Thus,
 $H_3^{ab}:=(h+k)^{ab}-h^{ab}$ is so small that we can apply Theorem \ref{Lpmaxprinc} with $h$ playing the role of $g$ there (which means that $\n h$ is trivially bounded, since it vanishes). Then Theorem \ref{Lpmaxprinc} implies
\begin{align*}
	\sqrt{c_l} t^{l/2}\left\|\nabla^lk(t)\right\|_{L^p}\leq\left\| u_m(t)\right\|_{L^p}\leq C \left\|u_m(0)\right\|_{L^p}=C\left\|k(0)\right\|_{L^p}
\end{align*}
and the assertion follows.
\end{proof}

\section{The Ricci-de Turck flow and a mixed evolution problem}\label{sec : RdT mixed evolution}
Throughout this section, let $\widehat{h}$ be an integrable Ricci-flat ALE metric with a parallel spinor, $p\in (1,\infty)$ and $\mathcal{U}$ be an $L^{[p,\infty]}$-neighbourhood of $\widehat{h}$ in the space of metrics, on which the projection map 
	\begin{align*}
		\Phi:\mathcal{U}\to\mathcal{U}\cap\mathcal{F}
	\end{align*}
of Subsection \ref{subsec: proj map} is well-defined. 

\subsection{A Ricci-de Turck flow with moving gauge}

\begin{definition}
A family of metrics $g_t$ in $\U$ is called a Ricci-de Turck flow with moving gauge, if it satisfies the evolution equation
	\[
		\partial_tg
			=-2\Ric_{g}+\mathcal{L}_{V(g,\Phi(g))}g.
	\]
\end{definition}
By Lemma \ref{lem : expression RdT1} and since $\Phi(g)$ is Ricci-flat, the Ricci-de Turck flow with moving gauge, written in terms of $h=\Phi(g)$ and $k=g-h$, expands as
\begin{align}
	\label{eq : RdT1}\partial_tg+\Delta_{L,h}k&=H_1:=F_1(g^{-1},g^{-1},\nabla k,\nabla k)+F_2(g^{-1},R,k,k)+F_3(g^{-1},k,\nabla^2 k),\\
	\label{eq : RdT2}\partial_tg+\Delta_{L,g,h}k&=H_2:=F_1(g^{-1},g^{-1},\nabla k,\nabla k),\\
	\label{eq : RdT3}	\partial_tg+\Delta_{h}k&=F_4(g^{-1},g^{-1},\nabla k,\nabla k)+F_5(g^{-1},g,R,k)+\nabla_a((g^{ab}-h^{ab})\nabla_bk_{ij}),
\end{align}
where the $F_i$ are tensor fields, viewed as $C^{\infty}$-multilinear maps.  We will work with these expressions in the following.
\begin{lem}\label{lem : Ricci flow moving gauge short-time estimates}
	Let $g_t\subset \U$, $t\in[0,T]$ be a solution of the Ricci-de Turck flow with moving gauge $h_t=\Phi(g_t)$ and $k_t:=g_t-h_t$.
	\begin{itemize}
\item[(i)]	For all $r\in (n,\infty)$, there is an $\epsilon>0$ such that if 
	\begin{align*}
		\left\| k_t\right\|_{W^{2,r}}<\epsilon\qquad\forall t\in [0,T],\qquad \left\|k_0\right\|_{L^{p}}<\infty \text{ for some }p\in [p_0,\infty),
	\end{align*}
 there exists for every $m\in \N_0$ a constant $C_m=C(m,\epsilon,h,T,r,p_0,\U)$ (but independent from $p$) such that
	\begin{align*}
		\left\|\nabla^m k_t\right\|_{L^p}\leq C_m\cdot t^{-m/2}\cdot(\left\| k_0\right\|_{L^p}+\sup_{s\in [0,t]}\left\| k_s\right\|_{W^{2,r}}).
	\end{align*}
	\item[(ii)] For all $r\in (n,\infty)$ and $m\in \N_0$, there is an $\epsilon>0$ such that if 
		\begin{align*}
\left\| k_t\right\|_{W^{2,r}}<\epsilon\qquad\forall t\in [0,T],\qquad \left\|k_0\right\|_{W^{m,p}}<\infty \text{ for some }p\in [p_0,\infty),
	\end{align*} 
 there exists a constant $C_m=C(m,\epsilon,h,T,r,p_0,\U)$ (but independent from $p$) such that
\begin{align*}
		\left\| k_t\right\|_{W^{m,p}}\leq C_m\cdot (\left\| k_0\right\|_{W^{m,p}}+\sup_{s\in [0,t]}\left\| k_s\right\|_{W^{2,r}}).
	\end{align*}
	\end{itemize}
\end{lem}
\begin{proof}
%[Sketch of proof]
%The proof of part (i) is very similar to the proof of Lemma \ref{lem : Ricci flow short-time estimates}.
With the notation from Subsection \ref{subsec: shorttime estimates} we can write \eqref{eq : RdT3} as
\begin{align*}
\partial_tg+\Delta_hk&=R[k]+Q_0[k]+\nabla Q_1[k].
\end{align*}
From that, we get
\begin{align*}
	\partial_t k+\Delta_hk=\partial_tg+\Delta_hk-\partial_th
&=	R[k]+Q_0[k]+\nabla Q_1[k]-D_g\Phi(\partial_tg)\\
&=R[k]+Q_0[k]+\nabla Q_1[k]-D_g\Phi(-\Delta_{L,h}k+H_1)\\
&=R[k]+Q_0[k]+\nabla Q_1[k]-D_g\Phi(H_1).
\end{align*}
Here, we used the chain rule applied to $h=\Phi(g)$, Formula \eqref{eq : RdT1} and Lemma \ref{lem: DPhi vanishes}. In comparison to the Ricci de Turck flow with fixed gauge, we thus get an additional term which is
\begin{align*}
D_g\Phi(H_1(t))= D_g\Phi(g^{-1}*g^{-1}*\nabla k*\nabla k+g^{-1}*R*k*k+g^{-1}*k*\nabla^2k).
\end{align*}
A straightforward calculation now yields
\begin{align*} 
%\begin{split}
	\partial_t |\nabla^lk|^{2}+\Delta |\nabla^lk|^{2}
		=&2\langle [\nabla^l,\Delta]k+\nabla^l(R[k]+Q_0[k])+[\nabla^l, Q_1][k]+ Q_1[\nabla^lk],\nabla^lk\rangle\\ 
		&\quad-2|\nabla^{l+1} k|^2
+(\partial_t h)(\nabla^l k,\nabla^lk)\\
&\qquad+2\langle [\partial_t,\nabla^l]k,\nabla^lk\rangle-2\langle\nabla^l(	D_g\Phi(H_1)),\nabla^lk\rangle		,
%\end{split}
\end{align*}
where the norms and the covariant derivatives are both taken with respect to $h_t$.
Observe that in comparison to the corresponding formula for the Ricci-de Turck flow with fixed gauge \eqref{eq: evolution nabla l k}, we have obtained three additional terms.
For the first two of them, we have
\begin{equation}
\begin{split}\label{eq : additional terms 1}
(\partial_t h)(\nabla^l k,\nabla^lk)&\leq C|\partial_th||\nabla^lk|^2,\\
\langle [\partial_t,\nabla^l]k,\nabla^lk\rangle &\leq
C \sum_{m=1}^l |\nabla^m\partial_t h| |\nabla^{l-m}k||\nabla^lk|.
\end{split}
\end{equation}
Now recall that
\begin{align*}
\partial_th=D_g\Phi(H_1),\qquad H_1=g^{-1}*g^{-1}*\nabla k*\nabla k+g^{-1}*R*k*k+g^{-1}*k*\nabla^2k.
\end{align*}
Thus, a combination of Lemma \ref{lem : projection lemma} (iii) and the H\"{o}lder inequality shows that
\begin{align}\label{eq : additional terms 2}
	\left\| \nabla^m \partial_th\right\|_{L^\infty}
\leq C \left\|H_1\right\|_{L^{r/2}}	
	\leq C\cdot\left\| k\right\|_{W^{2,r}}^2\leq C \epsilon
\end{align}
for all $m\in\N_0$.

Now let us prove (i). Define  
	\begin{align*}
		u_m(t):=(\sqrt{C_l\cdot t^l}\nabla^lk(t))_{l=0}^m\in C^{\infty}\left(\bigoplus_{l=0}^m (T^*M)^{\otimes l}\otimes S^2M\right).
	\end{align*}
We get by the same computation as in the proof of Lemma \ref{lem : Ricci flow short-time estimates} that
\begin{equation}\label{eq : diff inequality 1}
\begin{split}
		\partial_t|u_m|^2&\leq h^{ab}\nabla^2_{ab}|u_m|^2+D_m |u_m|^2+\langle 
		\nabla_a((h+k)^{ab}-h^{ab})\nabla_b u_m,u_m\rangle\\
		&\qquad -2(1-\delta)|\nabla u_m|^2-2\langle w_m,u_m\rangle.
		\end{split}
\end{equation}
Here, two of the additional terms in the evolution of $\partial_tk$ were captured by $D_m |u_m|^2$ due to \eqref{eq : additional terms 1} and \eqref{eq : additional terms 2}
and the third of them is captured by
\begin{align*}
w_m(t):=(\sqrt{C_l\cdot t^l}\nabla^lD_g\Phi(H_1(t)))_{l=0}^m\in C^{\infty}\left(\bigoplus_{l=0}^m (T^*M)^{\otimes l}\otimes S^2M\right).
\end{align*}
Again, a combination of Lemma \ref{lem : projection lemma} (iii) and the Hölder inequality shows that
\begin{align*}
	\left\| \nabla^m\circ D_g\Phi(H_1)\right\|_{L^p}
\leq C \left\|H_1\right\|_{L^{r/2}}	
	\leq C\cdot\left\| k\right\|_{W^{2,r}}^2	
	\leq C\cdot\left\| k\right\|_{W^{2,r}}
\end{align*}
for all $m\in\N_0$.
Therefore
\begin{align*}
\left\| w_m(s)\right\|_{L^p([0,t]\times M)}\leq C\sup_{s\in [0,T]}
\left\| w_m(s)\right\|_{L^p(M)}
\leq  C\sup_{s\in [0,T]}\left\| k_s\right\|_{W^{2,r}}.
\end{align*}
Note that the first term in \eqref{eq : diff inequality 1} is $ h^{ab}\nabla^2_{ab}|u_m|^2$, as opposed to the more general form $g^{ab}\nabla^2_{ab}|u_m|^2$ in Theorem \ref{Lpmaxprinc}.
This means that we now can apply Theorem  \ref{Lpmaxprinc} with the $g_t$ \emph{there} (not the $g_t=h_t+k_t$ in the \emph{current} lemma) given by $g_t=h_t$.
We conclude  
 	\begin{align*}
		\sqrt{C_l\cdot t^l}\left\|\nabla^lk(t)\right\|_{L^p}\leq\left\| u_m(t)\right\|_{L^p}&\leq C \left\|u_m(0)\right\|_{L^p}+  C\sup_{s\in [0,T]}\left\| k_s\right\|_{W^{2,r}}
		\\
		&	
			=C\left\|k(0)\right\|_{L^p}+ C\sup_{s\in [0,T]}\left\| k_s\right\|_{W^{2,r}},
	\end{align*} 
which proves (i).

For proving (ii),
 we define  
	\begin{align*}
		u_m(t):=(\nabla^lk(t))_{l=0}^m\in C^{\infty}\left(\bigoplus_{l=0}^m (T^*M)^{\otimes l}\otimes S^2M\right).
	\end{align*}
	Doing the same computations as in the proof of Lemma \ref{lem : Ricci flow short-time estimates} (up to ignoring the derivatives of the $t$-factors) together with \eqref{eq : additional terms 1} and \eqref{eq : additional terms 2}, we get
\begin{align*}
		\partial_t|u_m|^2&\leq h^{ab}\nabla^2_{ab}|u_m|^2+D_m |u_m|^2+\langle 
		\nabla_a((h+k)^{ab}-h^{ab})\nabla_b u_m,u_m\rangle\\
		&\qquad -2(1-\delta)|\nabla u_m|^2-2\langle w_m,u_m\rangle
\end{align*}
	with
	\begin{align*}
w_m(t)=(\nabla^lD_g\Phi(H_1(t)))_{l=0}^m\in C^{\infty}\left(\bigoplus_{l=0}^m (T^*M)^{\otimes l}\otimes S^2M\right).
\end{align*}
As in part (i),
\begin{align*}
\left\| w_m(s)\right\|_{L^p([0,t]\times M)}\leq   C\sup_{s\in [0,T]}\left\| k_s\right\|_{W^{2,r}}.
\end{align*}
In this case, Theorem \ref{Lpmaxprinc} (again with $g_t=h_t$ for the $g_t$ in this theorem) yields
 	\begin{align*}
	\left\|k(t)\right\|_{W^{m,p}}=\left\| u_m(t)\right\|_{L^p}&\leq C \left\|u_m(0)\right\|_{L^p}+  C\sup_{s\in [0,T]}\left\| k_s\right\|_{W^{2,r}}\\
		&		=C\left\|k(0)\right\|_{W^{m,p}}+ C\sup_{s\in [0,T]}\left\| k_s\right\|_{W^{2,r}},
	\end{align*} 
which proves (ii).
\end{proof}
\begin{prop}\label{prop : alternative mod RdTf}
	Let $g_t$ be a smooth family of Riemannian metrics in $\U$ and let $h_{\infty}$ be some fixed metric in $\U \cap \F$.
	Suppose that $\U$ is so small that $\Pi_{h,\bar{h}}^\perp$ is invertible for all $h, \bar{h} \in \U \cap \F$ (cf. Lemma \ref{lem : invert projection}).
	Then $g_t$ is a Ricci-de Turck flow with moving gauge if and only if the components 
	\begin{align*}
		h_t&=\Phi(g_t),\qquad k_t=g_t-h_t,\qquad \overline{k}_t=\Pi^{\perp}_{h_t,h_{\infty}}(k_t)
	\end{align*}
	satisfy the coupled system
	\begin{align*}
		\partial_th&=D_g\Phi(H_1),\\
		\partial_t\overline{k}+\Delta_{L,\infty}\overline{k}&=\Pi^{\perp}_{\infty}[(\Delta_{L,\infty}-\Delta_{L,h})(k)+
		(1-D_g\Phi)(H_1)],
	\end{align*}
	where $H_1$ is defined in \eqref{eq : RdT1} and $\Pi^{\perp}_{\infty} := \Pi^{\perp}_{h_{\infty}}$. 
	On the other hand, it is also equivalent to the system
	\begin{align*}
		\partial_th&=D\Phi_g(H_1),\\
		\partial_t\overline{k}+\Delta_{L,g,h}\overline{k}&=[\Delta_{L,g,h},\Pi^{\perp}_{\infty}](k)+\Pi^{\perp}_{\infty}[D_g\Phi((\Delta_{L,g,h}-\Delta_{L,h})(k))+(1-D_g\Phi)(H_2)].
	\end{align*}
\end{prop}
\begin{proof}
At first, note that $\Delta_{L,h}k$ is orthogonal to $\ker_{L^2}(\Delta_{L,h})$, because $\Delta_{L,h}$ is self-adjoint. Thus, Lemma \ref{lem: DPhi vanishes} implies $D_g\Phi(\Delta_{L,h}k)=0$.
	Splitting \eqref{eq : RdT1} up into $g=h+k=\Phi(g)+(1-\Phi)(g)$ therefore yields the equivalent system
	\begin{align*}
		\partial_th&=D_g\Phi(\partial_tg)=D_g\Phi(-\Delta_{L,h}k+H_1)=
		D_g\Phi(H_1),\\
		\partial_tk&=(1-D_g\Phi)(\partial_tg)=(1-D_g\Phi)(-\Delta_{L,h}k+H_1)=-\Delta_{L,h}k+(1-D_g\Phi)(H_1).
	\end{align*}
	Applying $\Pi^{\perp}_{\infty}$ to the second equation yields
	\begin{align*}
		\partial_t\overline{k}=\Pi^{\perp}_{\infty}(\partial_tk)&=\Pi^{\perp}_{\infty}(-\Delta_{L,h}k+(1-D_g\Phi)(H_1))\\
		&=-\Pi^{\perp}_{\infty}(\Delta_{L,\infty}k)+ \Pi^{\perp}_{\infty}((\Delta_{L,\infty}-\Delta_{L,h})(k)+(1-D_g\Phi)(H_1))\\
		&=-\Delta_{L,\infty}\overline{k}+ \Pi^{\perp}_{\infty}((\Delta_{L,\infty}-\Delta_{L,h})(k)+(1-D_g\Phi)(H_1)).
	\end{align*}
	By construction, $\overline{k}=\Pi_{h,h_{\infty}}^\perp(k)$. 
	By assumption, $\Pi_{h,h_{\infty}}^\perp$ is invertible, so that $k$ is determined by $\overline{k}$. 
	Therefore, the evolution equations on  $k$ and $\overline{k}$ are actually equivalent.
	
	It remains to show equivalence of the Ricci flow to the second system. 
	This is done similarly.
	The first equation of the system is the same as the first equation from the first system. 
	For the second equation, we get from \eqref{eq : RdT2} and the chain rule that
	\begin{align*}
		\partial_tk&=(1-D_g\Phi)(\partial_tg)=(1-D_g\Phi)(-\Delta_{L,g,h}k+H_2).
	\end{align*}
	Applying $\Pi^{\perp}_{\infty}$ yields
	\begin{align*}
		\partial_t\overline{k}&=\Pi^{\perp}_{\infty}(\partial_tk)=\Pi^{\perp}_{\infty}[(1-D_g\Phi)(-\Delta_{L,g,h}k+H_2)]\\
		&=-\Pi^{\perp}_{\infty}(\Delta_{L,g,h}k)+\Pi^{\perp}_{\infty}[D_g\Phi(\Delta_{L,g,h}k)+(1-D_g\Phi)(H_2)]\\
		&=-\Delta_{L,g,h}\overline{k}+[\Delta_{L,g,h},\Pi^{\perp}_{\infty}](k)+\Pi^{\perp}_{\infty}[D_g\Phi((\Delta_{L,g,h}-\Delta_{L,h})(k))+(1-D_g\Phi)(H_2)].
	\end{align*}
	Note again that by assumption, $k$ and $\overline{k}$ contain the same information. Therefore, the evolution equations on  $k$ and $\overline{k}$ are equivalent.
\end{proof}
We will use both forms of the Ricci-de Turck flow equation at once to obtain the Ricci flow as a fixed point argument.
\subsection{A mixed evolution operator}
Let $h_{\infty}\in\U \cap\F$ be a fixed metric and $g_t\in\U$ and 
	${h_t\in \U \cap \F}$, 
$t\geq1$ be smooth families of evolving metrics. 
We then have the operators $\Delta_{L,\infty}:=\Delta_{L,h_{\infty}}$ and $\Delta_{L,g,h}$ where we suppressed the dependence of $g,h$ on $t$. 
We build now a  mixed evolution operator, depending on both operators, and hence on the metrics $g_t,h_t$ and $h_{\infty}$. 
For $t\geq 1$ and $s\in [1,t]$, we consider the evolution problem
\begin{align*}
	\partial_rk+\Delta_{L,\infty}k
		&=0,\qquad \text{ for }r\in [s,\max\left\{t-1,s\right\}],\\
	\partial_rk+\Delta_{L,g,h}k
		&=0,\qquad \text{ for }r\in [\max\left\{t-1,s\right\},t],\\
	k|_{r = s}
		&=k'.
\end{align*}
We now define for $1\leq s\leq t$ the operator $P(g,h,h_{\infty})_{s\to t}$ as the map which associates to given initial data $k'$ the solution of one of the above initial value problems.
\begin{rem}
	\begin{itemize}
		\item
		Note that by construction, the mixed solution operator $P(g,h,h_{\infty})_{s\to t}$ depends continuously on all involved parameters. 
		\item Note that if $t>2$ and $s\in [1,t-1)$, we have 
		\begin{align*}P(g,h,h_{\infty})_{s\to t}=P(g,h,h_{\infty})_{t-1\to t}\circ e^{-(t-1-s)\Delta_{L,h_{\infty}}}.
		\end{align*}
	\end{itemize}
\end{rem}
Let us now extend this construction to the inhomogeneous problem. 
For $t\geq 1$ and $s\in [1,t]$, we consider the inhomogeneous evolution problem
\begin{align*}
	\partial_rk+\Delta_{L,\infty}k
		&=F_r,\qquad \text{ for }r\in [s,\max\left\{t-1,s\right\}],\\
	\partial_rk+\Delta_{L,g,h}k
		&=F_r,\qquad \text{ for }r\in [\max\left\{t-1,s\right\},t],\\
	k|_{r = s}
		&=k'.
\end{align*}
We now define for $1\leq s\leq t$ the operator $Q(g,h,h_{\infty})_{s\to t}(k',F)$ as the map which associates to given initial data $k=k_s$ the solution of one of the above initial value problems. Observe that the Duhamel principle also holds in a very general setting so that
\begin{align}\label{eq : inhom mixed system}
	Q(g,h,h_{\infty})_{s\to t}(k',F)=P(g,h,h_{\infty})_{s\to t}(k')+\int_{s}^tP(g,h,h_{\infty})_{r\to t}(F_r)dr.
\end{align}
\subsection{The Ricci flow as a mixed evolution problem}\label{subsec : mixed_ev_problem}
Now let us turn back to the Ricci-de Turck flow with moving gauge. Let $g_t$ be such a flow and $h_{\infty}\in\U\cap \mathcal{F}$ be arbitrary. Introduce the quantities
\begin{align*}
	h_t&=\Phi(g_t),\qquad k_t=g_t-h_t,\qquad \overline{k}_t=\Pi^{\perp}_{h_t,h_{\infty}}(k_t).
\end{align*}
Due to Proposition \ref{prop : alternative mod RdTf}, the flow equation is equivalent to	
\begin{equation}\begin{split}\label{eq : RdT moving 1}
		\partial_th&=D_g\Phi(H_1),\\
		\partial_t\overline{k}+\Delta_{L,\infty}\overline{k}&=\Pi^{\perp}_{\infty}[(\Delta_{L,\infty}-\Delta_{L,h})(k)+
		(1-D_g\Phi)(H_1)]
	\end{split}
\end{equation}
and
\begin{equation}\begin{split}\label{eq : RdT moving 2}
		\partial_th&=D_g\Phi(H_1),\\
		\partial_t\overline{k}+\Delta_{L,g,h}\overline{k}&=[\Delta_{L,g,h},\Pi^{\perp}_{\infty}](k)+\Pi^{\perp}_{\infty}[D_g\Phi((\Delta_{L,g,h}-\Delta_{L,h})(k))+(1-D_g\Phi)(H_2)].
	\end{split}
\end{equation}
Let $g_0$ be an initial metric and $g_t$, $t\in[0,1]$ be a $\widehat{h}$-gauged Ricci-de Turck flow, starting at $g_0$. Split $g_1=\Phi(g_1)+(1-\Phi)(g_1)$ and let $h_1=\Phi(g_1)$ and $k_1=(1-\Phi)(g_1)$. 
Continue now with the Ricci-de Turck flow with moving gauge and let $h_t,\overline{k}_t$ as above. 
For fixed $t\in[1,2]$, we regard $(h_t,\overline{k}_t)$ as a solution to \eqref{eq : RdT moving 2}. 
For $t>2$, we regard $(h_t,\overline{k}_t)$ as the tuple obtained from solving \eqref{eq : RdT moving 1} for time $s\in [1,t-1]$ and \eqref{eq : RdT moving 2} for time $s\in [t-1,t]$. 
Due to \eqref{eq : inhom mixed system}, this implies that
\begin{equation}
	\begin{split}\label{eq : identity1}
		h_t&=h_1+\int_1^tD_g\Phi(H_1(s))ds,\\
		\overline{k}_t&=P(g,h,h_{\infty})_{1\to t}(\Pi^{\perp}_{h_\infty}(k_1))+\int_1^tP(g,h,h_{\infty})_{s\to t}I(t,s,k,h,h_{\infty})ds,
	\end{split}
\end{equation}
with 
\begin{align*}
	&I(t,s,k,h,h_{\infty})=\chi_{[1,\max\left\{t-1,1\right\}]}(s)\cdot\left\{\Pi^{\perp}_{\infty}[(\Delta_{L,\infty}-\Delta_{L,h})(k)+
	(1-D_g\Phi)(H_1)]\right\}\\&\qquad+\chi_{(\max\left\{t-1,1\right\},t]}(s)\cdot\left\{[\Delta_{L,g,h},\Pi^{\perp}_{\infty}](k)+\Pi^{\perp}_{\infty}[D_g\Phi((\Delta_{L,g,h}-\Delta_{L,h})(k))+(1-D_g\Phi)(H_2)]\right\}
\end{align*}
for $t>1$ and $s\in [1,t]$.
\begin{remark}
	This complicated construction resolves the regularity problem that was addressed in Subsection \ref{subsubseq : boundary terms}.
\end{remark}
Now we use \eqref{eq : identity1} to identify the map of which the Ricci-de Turck flow with moving gauge is a fixed point. For this purpose, let  $h_t$, $t\in [1,\infty]$ a smooth curve in $\U\cap\mathcal{F}$ which converges to a limit metric $h_{\infty}$. Furthermore, let $k_t$, $t\in [1,\infty)$ be family of symmetric $2$-tensors such that $k_t\perp \ker_{L^2}(\Delta_{L,h_t})$.
Finally, assume that $g_t:=h_t+k_t\in \U$. Define
\begin{align*}
	\overline{\psi}_1(h,k)_t&:=h_1+\int_1^tD_g\Phi(H_1(s))ds,\qquad t\in [1,\infty],\\
	\overline{\psi}_2(h,k)_t&:=P(g,h,h_{\infty})_{1\to t}(\Pi^{\perp}_{h_\infty}(k_1))+\int_1^tP(g,h,h_{\infty})_{s\to t}I(t,s,k,h,h_{\infty})ds, \qquad t\in [1,\infty].
\end{align*}
In order to get again a curve in $\U\cap\F$ and family of symmetric $2$-tensors which are orthogonal to $ \ker_{L^2}(\Delta_{L,h_t})$, we define as correction terms
\begin{align*}
	\psi_1(h,k)&=\Phi(\overline{\psi}_1(h,k)),\\
	\psi_2(h,k)&=(\Pi^{\perp}_{h,h_{\infty}})^{-1}(\Pi^{\perp}_{h_{\infty}}\overline{\psi}_2(h,k)).
\end{align*}
These two maps unify to the map $\psi(h,k)=(\psi_1(h,k),\psi_2(h,k))$ and due to \eqref{eq : identity1}, the Ricci-de Turck flow with moving gauge (viewed as the tuple $(h_t,k_t)$) is a fixed point of this map. Our goal in the next section is to identify this map as a contraction map in a suitable Banach space so that a Ricci-de Turck flow with moving gauge can be found via an iteration procedure.
\section{The iteration map}\label{sec : iteration map}
In this section we are going to study the map $\psi$, which we just defined at the end of Subsection \ref{subsec : mixed_ev_problem}, in detail.
\subsection{Estimates for the linear problem}
Let us first summarize some results for the linearized version of the problem.
For this purpose, let $\widehat{h}$ be an integrable Ricci-flat ALE metric with a parallel spinor, $\mathcal{U}$ a small neighbourhood of $\widehat{h}$ in the space of metrics such that Proposition \ref{prop: near metrics} and Lemma \ref{lem : elliptic regularity} hold.
For $h\in \mathcal{U}\cap\mathcal{F}$ let $\Delta_{L,h}$ be its Lichnerowicz Laplacian. The norms and covariant derivatives in this subsection are taken with respect to $h$. 
 The following result is Theorem 6.9 in our companion paper \cite{KP2020}.
\begin{thm}[Heat kernel and derivative estimates]\label{thm : linear estimates}Let $h\in \mathcal{U}\cap\mathcal{F}$.
	\begin{itemize}
		\item[(i)] 	For each $1 < p \leq q < \infty$, there exists a constant $C = C(p, q,\mathcal{U})$ such that for all $t>0$, we have
		\[
		\norm{e^{-t\Delta_{L,h} }\circ \Pi^{\perp}_h}_{L^p,L^q} \leq Ct^{-\frac n2 \left(\frac1p - \frac1q\right)}
		\]
		\item[(ii)] For each $i\in\left\{1,\ldots,n-1\right\}$ and $p\in (1,\frac ni)$, there exists a constant $C = C(p,i,\mathcal{U})$ such that for all $t>0$, we have
		\[
		\norm{\nabla^i \circ e^{-t\Delta_{L,h} }\circ \Pi^{\perp}_h}_{L^p,L^p} \leq Ct^{-\frac i2}.
		\]
		\item[(iii)]
		For each $i\in\N$, $p\in [\frac ni,\infty)\cap (1,\infty)$, $t_0>0$ and $\epsilon>0$, there exists $C=C(p,i,t_0,\epsilon,\mathcal{U})$ such that for all $t>t_0$, we have
		\[
		\norm{\nabla^i\circ e^{-t\Delta_{L,h} }\circ \Pi^{\perp}_h}_{L^{p},L^p} \leq Ct^{-\frac{n}{2p}+\epsilon}.
		\]
	\end{itemize}
\end{thm}
In \cite{KP2020}, the constants are given for a fixed metric $h\in \U\cap\F$. It follows however from the proof of this theorem and the estimates for two different $h,\overline{h}\in \U\cap\F$ in Lemma \ref{lem : elliptic regularity} that the constants can be chosen uniformly for all $h\in\U\cap\F$.
We can also state the result more generally as follows.
\begin{cor}
[Heat kernel and derivative estimates]\label{cor : linear estimates}Let $h\in \mathcal{U}\cap\mathcal{F}$, $1 < p \leq q < \infty$ and $i\in\N$.
\begin{itemize}
	\item[(i)] If $\frac{n}{2}\left( \frac{1}{p}-\frac{1}{q} \right)+\frac{i}{2} < \frac{n}{2p}$ there exists a constant $C = C(p, q,i,\mathcal{U})$ such that for all $t>0$, we have
	\[
		\norm{\nabla^i\circ e^{-t\Delta_{L,h} }\circ \Pi^{\perp}_h}_{L^p,L^q} \leq Ct^{-\frac n2 \left(\frac1p - \frac1q\right)-\frac{i}{2}}.
	\]
		\item[(ii)] If $\frac{n}{2}\left(\frac{1}{p}-\frac{1}{q}\right)+\frac{i}{2}\geq\frac{n}{2p}$ there exists for each  $t_0>0$ and $\epsilon>0$ a constant $C = C(p, q,i,t_0,\epsilon,\mathcal{U})$ such that for all $t>t_0$, we have
	\[
	\norm{\nabla^i\circ e^{-t\Delta_{L,h} }\circ \Pi^{\perp}_h}_{L^p,L^q} \leq Ct^{-\frac{n}{2p}+\epsilon}.
	\]
\end{itemize}
\end{cor}
\begin{proof}
	Writing
	\begin{align*}
	\nabla^i\circ e^{-t\Delta_{L,h} }\circ \Pi^{\perp}_h=\nabla^i\circ e^{-\frac{t}{2}\Delta_{L,h} }\circ e^{-\frac{t}{2}\Delta_{L,h} }\circ \Pi^{\perp}_h\circ \Pi^{\perp}_h=\nabla^i\circ e^{-\frac{t}{2}\Delta_{L,h} }\circ \Pi^{\perp}_h\circ e^{-\frac{t}{2}\Delta_{L,h} }\circ  \Pi^{\perp}_h,
	\end{align*}
	we immediately get
	\begin{align*}
		\norm{\nabla^i\circ e^{-t\Delta_{L,h} }\circ \Pi^{\perp}_h}_{L^p,L^q}\leq 	\norm{\nabla^i\circ e^{-\frac{t}{2}\Delta_{L,h} }\circ \Pi^{\perp}_h}_{L^q,L^q}	\norm{ e^{-\frac{t}{2}\Delta_{L,h} }\circ \Pi^{\perp}_h}_{L^p,L^q},
	\end{align*}
	and the estimate follows from Theorem \ref{thm : linear estimates} and a case by case analysis.
\end{proof}
These estimates are in sharp contrast to the scalar Laplacian on Euclidean space. For its heat flow, the estimates in Theorem \ref{thm : linear estimates} (ii) and Corollary \ref{cor : linear estimates} (i) hold for any choice of $1\leq p\leq q\leq\infty$ and $i\in\N_0$. These estimates can be derived from the explicit form of the Euclidean heat kernel. In particular, there is no dichotomy as in the parts (ii) and (iii) of Theorem \ref{thm : linear estimates} and as in Corollary \ref{cor : linear estimates} for the heat flow on Euclidean space.

 However, we can get rid of this dichotomy if we use particular differential operators instead of full covariant derivatives. Let
\begin{align*}
	DV_h(k)=\frac{d}{dt}|_{t=0}V(h+tk,h),\qquad
	D\Ric_h(k)=\frac{d}{dt}|_{t=0}\Ric_{h+tk},
\end{align*}
be the Fr\'{e}chet derivatives of the de Turck vector field in the first component and the Ricci tensor, respectively. The following result is \cite{KP2020}*{Theorem 6.11}.
\begin{thm}[Special derivative estimates]\label{thm : special derivative estimates}
	Let $h\in \mathcal{U}\cap\mathcal{F}$ and $\Delta_{L,h}$ its Lichnerowicz Laplacian. Then for each $p\in (1,\infty)$ there exists a constant $C=C(p,\mathcal{U})$ such that
	\begin{align*}
	\norm{ DV_h\circ e^{-t\Delta_{L,h}}}_{L^p,L^p}\leq C t^{-\frac12},\qquad
	\norm{D\Ric_h\circ e^{-t\Delta_{L,h}}}_{L^p,L^p}\leq C t^{-1}.
	\end{align*}
\end{thm}
Thus for large $p$, the decay rates for the first order operator $DV_h$ and the second order operator $D\Ric_h$ are stronger than for $\nabla$ and $\nabla^2$, respectively.

\subsection{The Banach space}\label{subsec : Banach space}
Let $\widehat{h}$ be an integrable Ricci-flat ALE metric with a parallel spinor.
Let $q\in (1,n)$ and $\mathcal{U}$ be an $L^{[q,\infty]}$-neighbourhood of $\widehat{h}$ in the space of metrics, which is so small that the projection map 
	\begin{align*}
		\Phi:\mathcal{U}\to\mathcal{U}\cap\mathcal{F}
	\end{align*}
of Subsection \ref{subsec: proj map} is well-defined. In addition let $r \in (n,\infty)$ be so large that
\begin{align*}
\frac{n}{2}\left(\frac{1}{q}-\frac{1}{r}\right)&>\frac{1}{2}.
\end{align*}
The H\"{o}lder exponent $q$ is the one in which our initial data will lie. The exponent $r$ is an auxiliary exponent which is on the one hand finite (which is useful for the projection map $\Phi$) but on the other hand so large that $W^{i,r}\subset W^{i-1,\infty}$ by Sobolev embedding. 

We equip the space of maps $k:[1,\infty)\to C^{\infty}(S^2M)$ with the norms
\begin{align*}
\left\|k\right\|_{X_{q,r}(\hat{h})}
	&:=\sup_{t\geq1}\Big\{\left\|k\right\|_{L^q(\hat{h})}+t^{\frac{1}{2}}\left\|\nabla k\right\|_{L^q(\hat{h})} \\
	&\qquad \qquad +t^{\min\left\{1,\frac{n}{2} \left( \frac{1}{q}-\frac{1}{r} \right) \right\}}\left\|\nabla^2 k\right\|_{L^q(\hat{h})}+ t^{\frac{n}{2} \left( \frac{1}{q}-\frac{1}{r} \right) }\left\|k\right\|_{W^{2,r}(\hat{h})}\Big\},\\
\left\|k\right\|_{Z_{q,r}(\hat{h})}
	&:=\sup_{t\geq1}\left\{\left\|k\right\|_{L^q(\hat{h})}+t^{n\left(\frac{1}{q}-\frac{1}{r}\right)}\left\|\partial_tk\right\|_{L^q(\hat{h})}\right\},
\end{align*}
and the set of maps $(h,k):[1,\infty)\to C^{\infty}(S^2M\oplus S^2M)$ with the norms
\begin{align*}
\left\|(h,k)\right\|_{Y_{q,r}(\hat{h})}&:=\left\|h\right\|_{Z_{q,r}(\hat{h})}+\left\|k\right\|_{X_{q,r}(\hat{h})}
\end{align*}
In the first slot, we will typically insert $h-\hat{h}$, where $h=h_t$ is a family of Ricci-flat metrics in $\U\cap\F$. Here, we defined the norms with respect to the fixed metric $\hat{h}$, but the norms can also be taken with respect to a family $\overline{h}_t$, $t\in [1,\infty)$ of Riemannian metrics in a natural way.
\begin{lem}
 Let $\overline{h}_t$, $t\in [1,\infty)$ be a family of metrics in $\U\cap\F$.
%  and denote by $X_{q,r}(\overline{h}),Z_{q,r}(\overline{h})$ the above norms, but measured with respect to the family $\overline{h}_t$. 
  Then there exists a constant $C=C(q,r,\U)$ such that
\begin{align*}
\frac{1}{C}\left\|k\right\|_{X_{q,r}(\hat{h})}\leq \left\|k\right\|_{X_{q,r}(\overline{h})}\leq C\left\|k\right\|_{X_{q,r}(\hat{h})},\qquad
\frac{1}{C}\left\|h\right\|_{Z_{q,r}(\hat{h})}\leq \left\|h\right\|_{Z_{q,r}(\overline{h})}\leq C\left\|h\right\|_{Z_{q,r}(\hat{h})}.	
	\end{align*}
\end{lem}
\begin{proof}
By Lemma \ref{lem : elliptic regularity}, we have uniform bounds on $\overline{h}_t-\hat{h}$ in $W^{2,q'}$ for all $q'\in [q,\infty]$. Therefore,  the $W^{2,q'}$-norms coming from $\overline{h}_t$ and $\hat{h}$ are uniformly equivalent, and so the $Z_{q,r}$-norms are by definition equivalent as well. For the $X_{q,r}$-norm, we first consider $L^q$ norms of the first two derivatives of $k$.
We have
\begin{align*}
\overline{\nabla} k-\hat{\nabla} k
=T*k,\qquad
\overline{\nabla}^2 k-\hat{\nabla}^2 k= \hat{\nabla}T*k+T*\hat{\nabla}k+T*T*k,
\end{align*}
where $T_{ij}^l:=\Gamma(\overline{h})_{ij}^l-\Gamma(\hat{h})_{ij}^l$.
Because $1<q<n<r<\infty$, we find $r'\in (1,\infty)$ such that 
\begin{align*}
\frac{1}{q}=\frac{1}{r'}+\frac{1}{r}.
\end{align*}
Therefore, with $t\geq 1$, we get by the H\"{o}lder inequality, the equivalence of the $W^{2,q'}$-norms for all $q'\in [q,\infty]$ and Lemma \ref{lem : elliptic regularity} that
\begin{align*}
t^{\frac{1}{2}}\left\|\overline{\nabla} k\right\|_{L^q(\overline{h})}
&\leq 
C t^{\frac{1}{2}}\left\|\hat{\nabla} k\right\|_{L^q(\overline{h})}+ 
C t^{\frac{1}{2}}\left\| T\right\|_{L^{r'}(\overline{h})}
\left\|\overline{\nabla} k\right\|_{L^r(\overline{h})}\\
&\leq C t^{\frac{1}{2}}\left\|\hat{\nabla} k\right\|_{L^q(\hat{h})}+ 
C \left\| T\right\|_{L^{r'}(\hat{h})}
 t^{\frac{1}{2}}\left\|\overline{\nabla} k\right\|_{L^r(\hat{h})}\\
&\leq C t^{\frac{1}{2}}\left\|\hat{\nabla} k\right\|_{L^q(\hat{h})}+ 
C \left\| \overline{h}-\hat{h}\right\|_{W^{1,r'}(\hat{h})}
t^{\frac{n}{2} \left( \frac{1}{q}-\frac{1}{r} \right) }\left\|\overline{\nabla} k\right\|_{L^r(\hat{h})}\\
&\leq C t^{\frac{1}{2}}\left\|\hat{\nabla} k\right\|_{L^q(\hat{h})}+ 
C \left\| \overline{h}-\hat{h}\right\|_{L^q(\hat{h})}
t^{\frac{n}{2} \left( \frac{1}{q}-\frac{1}{r} \right) }\left\|\overline{\nabla} k\right\|_{L^r(\hat{h})}\\
&\leq C \left\|k\right\|_{X_{q,r}(\hat{h})}.
\end{align*}
Similarly, with $t\geq1$ and $\beta=\min\left\{1,\frac{n}{2} \left( \frac{1}{q}-\frac{1}{r}\right) \right\}$,
\begin{align*}
t^{\beta}\left\|\overline{\nabla}^2 k\right\|_{L^q(\overline{h})}
&\leq Ct^{\beta}\left\|\hat{\nabla}^2 k\right\|_{L^q(\overline{h})}
+Ct^{\beta}\left\| T\right\|_{W^{1,r'}(\overline{h})}
\left\| k\right\|_{W^{1,r}(\overline{h})}
+Ct^{\beta}\left\|T* T\right\|_{W^{1,r'}(\overline{h})}
\left\| k\right\|_{L^{r}(\overline{h})}\\
&\leq Ct^{\beta}\left\|\hat{\nabla}^2 k\right\|_{L^q(\overline{h})}+
C\left( \left\| T\right\|_{W^{1,r'}(\overline{h})}+
\left\| T\right\|_{W^{1,2r'}(\overline{h})}^2\right)t^{\beta}\left\| k\right\|_{W^{1,r}(\overline{h})}\\
&\leq Ct^{\beta}\left\|\hat{\nabla}^2 k\right\|_{L^q(\hat{h})}+
C\left( \left\| T\right\|_{W^{1,r'}(\hat{h})}+
\left\| T\right\|_{W^{1,2r'}(\hat{h})}^2\right)t^{\beta}\left\| k\right\|_{W^{1,r}(\hat{h})}\\
&\leq Ct^{\beta}\left\|\hat{\nabla}^2 k\right\|_{L^q(\hat{h})}+
C\left( \left\| \overline{h}-\hat{h}\right\|_{W^{2,r'}(\hat{h})}+
\left\| \overline{h}-\hat{h}\right\|_{W^{2,2r'}(\hat{h})}^2\right)t^{\beta}\left\| k\right\|_{W^{1,r}(\hat{h})}\\
&\leq Ct^{\beta}\left\|\hat{\nabla}^2 k\right\|_{L^q(\hat{h})}+
C \left\| \overline{h}-\hat{h}\right\|_{L^{q}(\hat{h})}t^{\frac{n}{2} \left( \frac{1}{q}-\frac{1}{r}\right)}\left\| k\right\|_{W^{1,r}(\hat{h})}\\
&\leq C \left\|k\right\|_{X_{q,r}(\hat{h})}
\end{align*}
Because the $W^{2,q'}$-norms are equivalent for all $q'\in [q,\infty]$, this already implies the inequality $\left\|k\right\|_{X_{q,r}(\overline{h})}\leq C \left\|k\right\|_{X_{q,r}(\hat{h})}$. The converse estimate is shown analogously.
\end{proof}
Due to this lemma, we will from now on suppress the dependence of the norm on the metric for notational convenience and allow any curve $\overline{h}_t$ in $\U\cap\F$. 
The next lemma justifies the $X$-part of the norm
\begin{lem}\label{lem: Linear evolution and Xnorm}
Let $h\in\U\cap\F$ and $\Delta_{L,h}$ its Lichnerowicz Laplacian
Then there exists a constant $C=C(q,r,\U)$ such that
	\begin{align*}
	\left\|e^{-t\Delta_{L,h}}\circ \Pi_{h}^\perp(k_0) |_{[1,\infty)}\right\|_{X_{q,r}}\leq C \left\|k_0\right\|_{L^q}.
	\end{align*}
	for all  $k_0\in L^q(S^2M)$.
	\end{lem}
\begin{proof}
	This is immediate from Theorem \ref{thm : linear estimates}.\end{proof}

%From now on, let $\U$ be an $L^{[q,\infty]}$-neighbourhood of $\hat{h}$, which is so small that Proposition \ref{prop: near metrics} holds and that the projection map $\Phi$ of Subsection \ref{subsec: proj map} is defined. 

 \begin{lem}\label{lem : inftycontrol}
 Let $h:[1,\infty)\to \U$ be smooth and suppose that $\left\|h-\widehat{h}\right\|_{Z_{q,r}}<\infty$. Then, 
  $h_t$ converges in $L^q$ to a metric $h_{\infty}\to \overline{\U}$ as $t\to\infty$ and 
  \begin{align*}
  \left\|h_{\infty}-\widehat{h}\right\|_{L^q}\leq \left\|h-\widehat{h}\right\|_{Z_{q,r}}.
  \end{align*} 
  Moreover, if $\left\|h-\widehat{h}\right\|_{Z_{q,r}}$ is sufficiently small, $h_{\infty}\in \U$.
 \end{lem}
 \begin{proof}
 Let $\left\{t_i\right\}_{i\in\N}$ be a sequence of real numbers converging to $\infty$. Then,
 \begin{align*}
 \left\|h_{t_{i+1}}-h_{t_{i}}\right\|_{L^q}\leq \int_{t_i}^{t_{i+1}}\left\|\partial_th\right\|_{L^q}dt&\leq \int_{t_{i}}^{t_{i+1}}\left\|\partial_th\right\|_{L^q}dt \\
 &\leq \int_{t_{i}}^{t_{i+1}}t^{-n(\frac{1}{q}-\frac{1}{r})}dt\left\|h\right\|_{Z_{q,r}}\leq [t^{1-n(\frac{1}{q}-\frac{1}{r})}]_{t_{i}}^{t_{i+1}}\left\|h\right\|_{Z_{q,r}}.
 \end{align*}
 By the assumptions on $q$ and $r$, $1-n(\frac{1}{q}-\frac{1}{r})<0$. Thus, $[t^{1-n(\frac{1}{q}-\frac{1}{r})}]_{t_{i}}^{t_{i+1}}$ is a Cauchy sequence and $h_{t_i}$ is as well. Thus $h_{t_i}$ converges in $L^q$ as $i\to\infty$. Because sequence was arbitrary the limit $h_{\infty}$ is independent of the choice of sequence. Hence $h_t\to h_{\infty}$ as $t\to\infty$ and we get
 \begin{align*}
 \left\|h_{\infty}-\widehat{h}\right\|_{L^q}=\lim_{t\to\infty} \left\|h_{t}-\widehat{h}\right\|_{L^q}\leq \sup_{t\geq 1}\left\|h_{t}-\widehat{h}\right\|_{L^q}\leq \left\|h\right\|_{Z_{q,r}},
 \end{align*}
 which establishes the estimate. The final statement of the lemma is obvious.
 \end{proof}
\begin{lem}\label{lem : normcontrol}
Let   $(h_t)_{t\geq 1}$ be a family of metrics in $\U\cap\F$, $(k_t)_{t\geq1}$ be a family of symmetric $2$-tensors, and $g_t=h_t+k_t$, $t\in [1,\infty)$. Recall that
\begin{align*}
	\left\|k\right\|_{L^{[q,\infty]}}=\left\|k\right\|_{L^q}+\left\|k\right\|_{L^{\infty}}.
	\end{align*}
	Then, there exists a constant $C=C(q,r,\U)$ such that
	\begin{align*}
	\left\|g-\hat{h}\right\|_{L^{[q,\infty]}}+	\left\|g-\hat{h}\right\|_{W^{1,\infty}}\leq C\left\|(h-\hat{h},k)\right\|_{Y_{q,r}}.
	\end{align*}
	\end{lem}
\begin{proof}We have $g-\hat{h}=k+h-\hat{h}$. Thus
by Lemma \ref{lem : elliptic regularity}, Sobolev embedding and the definition of the time-dependant norms, we have
\begin{align*}
	\left\|g-\hat{h}\right\|_{{L^{[q,\infty]}}}+	\left\|g-\hat{h}\right\|_{W^{1,\infty}}&\leq 
C\left(\left\|k\right\|_{L^{q}}+\left\|k\right\|_{W^{1,\infty}}+\left\|h-\widehat{h}\right\|_{L^{q}}+\left\|h-\widehat{h}\right\|_{W^{1,\infty}}\right)\\
&\leq C\left(\left\|k\right\|_{L^{q}}+\left\|k\right\|_{W^{2,r}}+\left\|h-\widehat{h}\right\|_{L^{q}}\right)\\
&\leq C\left(\left\|k\right\|_{X_{q,r}}+\left\|h-\widehat{h}\right\|_{Z_{q,r}}\right)\\
&= C\left\|(h-\hat{h},k)\right\|_{Y_{q,r}},
\end{align*}	
which proves the lemma.
\end{proof}

\begin{lem}\label{lem : Phiwelldefined}Let   $(h_t)_{t\geq 1}$ be a family of metrics in $\U\cap\F$ and $(k_t)_{t\geq1}$ be a family of symmetric $2$-tensors.
	Then there exists an $\epsilon>0$ such that $\psi$ is well-defined
	if $\left\|(h-\hat{h},k)\right\|_{Y_{q,r}}<\epsilon$.
\end{lem}
\begin{proof}Let $g_t=h_t+k_t$.
	In order to show that $\psi$ is well-defined, we have to ensure that the terms involving the projection map $\Phi$ make sense. These terms are given by $D_g\Phi$ and 
	\begin{align*}
	\psi_1(h,k)
		=\Phi(\overline{\psi}_1(h,k))
		=\Phi\left(h_1+\int_1^tD_g\Phi(H_1(s))ds\right).
	\end{align*}
	In order to do so, we have to show that
	\begin{align*}
	g\in \mathcal{U} ,\qquad  h_1+\int_1^tD_g\Phi(H_1(s))ds\in \mathcal{U},
	\end{align*}
	where $\U$ is the $L^{[q,\infty]}$-neighbourhood which is the domain of definition of $\Phi$.
	At first, Lemma \ref{lem : normcontrol} shows that $g\in \mathcal{U}$, if $\left\|(h-\hat{h},k)\right\|_{Y_{q,r}}$ is chosen small enough.
	Let now $q'\in (1,q]$, $r'\in[r,\infty)$ and $q''\in (1,\infty]$.
	Using Lemma \ref{lem : projection lemma} (iii), 
	\begin{align*}
	-n\left(\frac{1}{q'}-\frac{1}{r'}\right)
		\leq - n\left(\frac{1}{q}-\frac{1}{r}\right)
		< - 1,
	\end{align*}
	and \eqref{eq : RdT1}, we establish the estimate	
	\begin{equation}
	\begin{split}\label{eq : H_1 estimate}
	\left\| \int_1^tD_g\Phi(H_1(s))ds\right\|_{L^{q''}} &\leq C\int_1^t\left\| H_1(s)\right\|_{L^{\frac{r'}{2}}}ds\\
		&\leq C\int_1^t [\left\|\nabla k\right\|_{L^{r'}}^2+\left\|\nabla^2k\right\|_{L^{r'}}\left\|k\right\|_{L^{r'}}+\left\|R\right\|_{L^{\infty}}\left\| k\right\|_{L^{r'}}^2]ds\\
		&\leq C\int_1^t s^{-n(\frac{1}{q'}-\frac{1}{r'})}ds\cdot \left\|k\right\|_{X_{q',r'}}^2\\
		&\leq C\left\|k\right\|_{X_{q',r'}}^2.
	\end{split}
	\end{equation}
    Here, we use \eqref{eq : H_1 estimate} only for $q''\in \left\{q,\infty\right\}$, $q'=q$ and $r'=r$, but later, we will make use of this inequality in full generality.
	We get
	\begin{align*}
	\left\| h_1+\int_1^tD_g\Phi(H_1(s))ds-\hat{h} \right\|_{L^{[q,\infty]}}&\leq \left\| \int_1^tD_g\Phi(H_1(s))ds \right\|_{L^{[q,\infty]}}+ \left\| h_1-\hat{h} \right\|_{L^{[q,\infty]}}\\
	&\leq C\left\|k\right\|_{X_{q,r}}^2+ \left\| h_1-\hat{h} \right\|_{Z_{q,r}}
	\end{align*}
	so that
	\begin{align*}
    h_1+\int_1^tD_g\Phi(H_1(s))ds\in \mathcal{U},
	\end{align*}
	if $\left\|(h-\hat{h},k)\right\|_{Y_{q,r}}$ is chosen small enough.
	\end{proof}

\subsection{Mapping properties of the iteration map}\label{subsection : mapping properties}
Let  $(M^n,\hat{h})$ be an integrable ALE manifold with a parallel spinor and $q\in (1,n)$, $r\in(n,\infty)$ so that
\begin{align*}
	\frac{n}{2}\left(\frac{1}{q}-\frac{1}{r}\right)
		&>\frac{1}{2}.
\end{align*}
Moreover, let $\U$ be an $L^{[q,\infty]}$-neighbourhood of $\hat{h}$ in the space of metrics, which is so small that the projection map $\Phi:\mathcal{U}\to\mathcal{U}\cap\mathcal{F}$ of Subsection \ref{subsec: proj map} is defined. Let furthermore
$h:=(h_t)_{t\geq 1}$ be a family of metrics in $\U\cap\F$ and $k:=(k_t)_{t\geq1}$ be a family of symmetric $2$-tensors. We finally set $g:=(g_t)_{t\geq 1}:=(h_t+k_t)_{t\geq 1}$. 

%The constants $\epsilon,C$ will appear in various assertions in this subsection. The small constant $\epsilon>0$ always depends on $\mathcal{U},q,r$. The constant $C>0$ will appear in every assertion in this subsection and depends in each assertion on $\mathcal{U}$ and all the appearing H\"{o}lder exponents. 

The goal of this subsection is to derive the following mapping property:
\begin{thm}\label{thm : psimapping}
	There exists an $\epsilon>0$ such that if  $\left\|(h-\hat{h},k)\right\|_{Y_{q,r}}<\epsilon$, there exists a constant $C=C(q,r,\U,\epsilon)$ such that
	 the map $\psi$ satisfies the estimate
	\begin{align*}
	\left\|\psi(h,k)-(\hat{h},0)\right\|_{Y_{q,r}}\leq C\left(\left\|k_1\right\|_{W^{2,q}}+\left\|k_1\right\|_{W^{2,\infty}}+\left\|h_1-\hat{h}\right\|_{L^q}+\left\|k\right\|_{X_{q,r}}\left\|(h-\hat{h},k)\right\|_{Y_{q,r}}\right).
	\end{align*}
	\end{thm}
The proof of this theorem is split up in Propositions \ref{psi1mapping} and \ref{psi2mapping} below, in which the estimates for the components of $\psi$ are established. In fact we will also prove estimates  for certain other H\"{o}lder exponents $q'\in (1,q]$ and $r'\in [r,\infty)$. These more general estimates will be important later in the paper for detecting the optimal convergence behaviour for the Ricci-de Turck flow with moving gauge.
\begin{prop}\label{psi1mapping}
	There exists an $\epsilon>0$ such that if  $\left\|(h-\hat{h},k)\right\|_{Y_{q,r}}<\epsilon$, there exists for all  $q'\in (1,q]$ and $r'\in [r,\infty)$ a constant $C=C(q,r,q',r',\U,\epsilon)$ such that	
\begin{align*}
\left\|\psi_1(h,k)-\hat{h}\right\|_{Z_{q',r'}}\leq C\left(\left\|h_1-\hat{h}\right\|_{L^q}+\left\|k\right\|_{X_{q',r'}}^2\right).
\end{align*}
	\end{prop}
\begin{proof}
	Due to Lemma \ref{lem : Phiwelldefined}, $\psi$ is well defined under the assumption of the proposition.
		Due to Lemma \ref{lem : projection lemma} (iii), $D_g\Phi$ is bounded on $L^{q'}$ for each $q'\in(1,\infty)$ and $g\in\mathcal{U}$. Therefore, $\Phi$ is also Lipschitz with respect to the $L^{q'}$-norm.
	Using $\hat{h}=\Phi(\hat{h})$, $h_1=\Phi(h_1)$ and the estimate in \eqref{eq : H_1 estimate}, we then get
	\begin{align*}
	\left\|\psi_1(h,k)-\hat{h}\right\|_{L^{q'}}
	&\leq \left\| \Phi\left(h_1+\int_1^tD_g\Phi(H_1(s))ds \right)-\Phi(h_1)\right\|_{L^{q'}}
	+\left\|\Phi(h_1)-\Phi(\hat{h})\right\|_{L^{q'}}\\
	&\leq C\left\| \int_1^tD_g\Phi(H_1(s))ds\right\|_{L^{q'}}+\left\|h_1-\hat{h}\right\|_{L^{q'}}\\
	&\leq C\left\|k\right\|_{X_{q',r'}}^2+C\left\|h_1-\hat{h}\right\|_{L^{q'}}.
	\end{align*}
	Secondly, we estimate, similarly as in \eqref{eq : H_1 estimate}
	\begin{align*}
	t^{n\left(\frac{1}{q'}-\frac{1}{r'}\right)}\left\|\partial_t\psi_1(h,k)\right\|_{L^{q'}}
		&\leq Ct^{n\left(\frac{1}{q'}-\frac{1}{r'}\right)} \left\| D_{\overline{\psi}_1(h,k)}\Phi[D_g\Phi(H_1)]\right\|_{L^{q'}}\\
		&\leq C t^{n\left(\frac{1}{q'}-\frac{1}{r'}\right)}\left[\left\|\nabla k\right\|_{L^{r'}}^2+\left\|\nabla^2k\right\|_{L^{r'}}\left\|k\right\|_{L^{r'}}+\left\|R\right\|_{L^{\infty}}\left\| k\right\|_{L^{r'}}^2\right]\\
		&\leq C\left\|k\right\|_{X_{q',r'}}^2,
	\end{align*}
	which finishes the proof of the proposition.
\end{proof}
\begin{prop}\label{psi2mapping}
There exists an $\epsilon>0$ such that if  $\left\|(h-\hat{h},k)\right\|_{Y_{q,r}}<\epsilon$, there exists a constant $C=C(q,r,\U,\epsilon)$ such that
	\begin{align*}
		\left\|\psi_2(h,k)\right\|_{X_{q,r}}&\leq C
		\left[\left\|k_1 \right\|_{W^{2,q}}+\left\|k_1 \right\|_{W^{2,r}}
	+\left(\left\|k\right\|_{X_{q,r}}+\left\|h-\hat{h}\right\|_{Z_{q,r}}\right)\left\|k\right\|_{X_{q,r}}\right].
	\end{align*}
Moreover, for any $q'\in (1,q]$ satisfying
	\begin{align*}
	\frac{n}{2} \left( \frac{1}{q'}-\frac{1}{q} \right)
		\leq \min\left\{\frac{n}{2r},n\left(\frac{1}{q}-\frac{1}{r}\right)-1,\frac{1}{2}\right\},
	\end{align*}
 there exists a constant $C=C(q,q',r,\epsilon)$ such that
		\begin{align*}
	\left\|\psi_2(h,k)\right\|_{X_{q',r}}
		&\leq C \left[\left\|k_1 \right\|_{W^{2,q'}}+\left\|k_1 \right\|_{W^{2,r}}
	+\left( \left\|k\right\|_{X_{q,r}}+\left\|h-\hat{h}\right\|_{Z_{q,r}}\right) \left\|k\right\|_{X_{q,r}}\right].
	\end{align*} 
	\end{prop}
	We prove the second inequality as the first one is a special case of the second. It will follow from a series of lemmas which cover the remainder of this subsection. For all these lemmas, we will assume that 
	\begin{align*}\left\|(h-\hat{h},k)\right\|_{Y_{q,r}}<\epsilon
\end{align*}	
	 for some sufficiently small $\epsilon>0$.
	\begin{lem}\label{lem : projection and Banach space}
	For any $q'\in (1,q]$ and $r'\in [r,\infty)$, there exists a constant $C=C(q',r',\U,\epsilon)$ such that
		\begin{align*}
	\left\|\psi_2(h,k)\right\|_{X_{q',r'}}=\left\|(\Pi^{\perp}_{h,h_{\infty}})^{-1}\Pi^{\perp}_{h_{\infty}}(\overline{\psi}_2(h,k))\right\|_{X_{q',r'}}\leq C	\left\|\overline{\psi}_2(h,k)\right\|_{X_{q',r'}}.
	\end{align*}
	\end{lem}
	\begin{proof}
		Let us abbreviate $\overline{\psi}_2:=\overline{\psi}_2(h,k)_t$ for fixed $h,k,t$ in the proof. 
Due to the assumptions on $q$ and $r$, we have
\begin{align*}
	\beta
		:=\min\left\{\frac{n}{2}\left( \frac{1}{q'}-\frac{1}{r'} \right) ,1\right\}
		\geq \min\left\{\frac{n}{2}\left( \frac{1}{q}-\frac{1}{r} \right),1\right\}
		\geq \frac{1}{2}.
\end{align*}		
		Therefore, using Lemma \ref{lem : projection lemma} (ii) and (iv), we get
	\begin{align*}
	\left\|(\Pi^{\perp}_{h,h_{\infty}})^{-1}\Pi^{\perp}_{h_{\infty}}(\overline{\psi}_2)\right\|_{L^{q'}}
		&\leq C	\left\|\overline{\psi}_2\right\|_{L^{q'}},\\
	t^{\frac{1}{2}}\left\|\nabla(\Pi^{\perp}_{h,h_{\infty}})^{-1}\Pi^{\perp}_{h_{\infty}}(\overline{\psi}_2)\right\|_{L^{q'}}
		&\leq C	t^{\frac{1}{2}}\left\|\nabla\overline{\psi}_2\right\|_{L^{q'}}+Ct^{\frac{n}{2} \left( \frac{1}{q'}-\frac{1}{r'}\right)} \left\|\overline{\psi}_2\right\|_{L^{r'}},\\
	t^{\beta}\left\|\nabla^2(\Pi^{\perp}_{h,h_{\infty}})^{-1}\Pi^{\perp}_{h_{\infty}}(\overline{\psi}_2)\right\|_{L^{q'}}
		&\leq C	t^{\beta}\left\|\nabla^2\overline{\psi}_2\right\|_{L^{q'}}+	t^{\frac{n}{2} \left( \frac{1}{q'}-\frac{1}{r'} \right) }\left\|\overline{\psi}_2\right\|_{L^{r'}},\\
	t^{\frac{n}{2}\left(\frac{1}{q'}-\frac{1}{r'}\right)}\left\|(\Pi^{\perp}_{h,h_{\infty}})^{-1}\Pi^{\perp}_{h_{\infty}}(\overline{\psi}_2)\right\|_{W^{2,r'}}
		&\leq C	t^{\frac{n}{2}\left(\frac{1}{q'}-\frac{1}{r'}\right)}\left\|\overline{\psi}_2\right\|_{W^{2,r'}},
	\end{align*}
and the lemma follows from the definition of the $X_{q',r'}$-norm.
\end{proof}
\begin{proof}[Proof of Proposition \ref{psi2mapping}]
By Lemma \ref{lem : projection and Banach space} and the triangle inequality,  we have
	\begin{align*}
	\left\|{\psi}_2(h,k)\right\|_{X_{q',r}}\leq C	\left\|\overline{\psi}_2(h,k)\right\|_{X_{q',r}}
		&\leq C\left\| P(g,h,h_{\infty})_{1\to t}(\Pi^{\perp}_{h_\infty}(k_1))\right\|_{X_{q',r'}} \\
		&\quad + C\left\|\int_1^{t}  P(g,h,h_{\infty})_{s\to t}I(t,s,k,h,h_{\infty})ds\right\|_{X_{q',r'}}.
	\end{align*}
The first term on the right hand side is treated in Lemma \ref{lem : linear part} below. The integrand will be split up in three terms which are estimated in Lemma \ref{lem : nonlinear part 1}, Lemma \ref{lem : nonlinear part 1.5} and Lemma \ref{lem : nonlinear part 2} below. 
\end{proof}
\begin{lem}\label{lem : linear part}
For any $q'\in (1,q]$ and $r'\in [r,\infty)$, there exists a constant $C=C(q,r,q',r',\U,\epsilon)$ such that
	\begin{align}\label{eq : estimates mixed linear evolution 1}
	\left\|P(g,h,h_{\infty})_{1\to t}(\Pi^{\perp}_{h_\infty}(k_1))\right\|_{X_{q',r'}}&\leq C
		(\left\|k_1 \right\|_{W^{2,q'}}+\left\|k_1 \right\|_{W^{2,r'}}	).
	\end{align}
	 Moreover, there exists for all $p'\in (1,q']$, $q''\in \left\{q',r' \right\}$, $i\in\left\{0,1,2 \right\}$ a constant $C=C(q,r,p',q'',i,\U)$ such that for all times $t\geq 3$ and $s\in [1,t-2]$  we get the estimate
	 		\begin{align}\label{eq : estimates mixed linear evolution2}
			\left\|\nabla^i\circ P(g,h,h_{\infty})_{s\to t} (\Pi^{\perp}_{h_\infty}(\Theta))\right\|_{L^{q''}} 
				&\leq C (t-s)^{-\frac{n}{2}\left( \frac{1}{p'}-\frac{1}{q'} \right)-\alpha(q'',i)} \left\|\Theta\right\|_{L^{p'}},
		\end{align}
		for any $\Theta\in L^{p'}(S^2M)$. Here, the function    $\alpha:\left\{q',r' \right\}\times\left\{0,1,2 \right\}\to\R$ is defined by
		$	\alpha(q',0)=0$, $\alpha(q',1)=\frac{1}{2}, \alpha(q',2)=\min\left\{1,\frac{n}{2}(\frac{1}{q'}-\frac{1}{r'})\right\}$ and $\alpha(r',i)=\frac{n}{2}(\frac{1}{q'}-\frac{1}{r'})$.
	\end{lem}
\begin{proof}
By Lemma \ref{lem : normcontrol}, the assumption $\left\|(h-\hat{h},k)\right\|_{Y_{q,r}}<\epsilon$ implies smallness of $\left\| k\right\|_{W^{1,\infty}}$. In particular, $\left\|g^{-1}\right\|_{W^{1,\infty}}<\infty$ (recall that $g=h+k$). Therefore, for $t\in[1,3]$, the short-time estimates in Lemma \ref{lem : short-time estimates} (i) and (ii) ensure that
\begin{align*}
		\left\|P(g,h,h_{\infty})_{1\to t}(\Pi^{\perp}_{h_\infty}(k_1))\right\|_{W^{2,q'}}&\leq C\left\|k_1 \right\|_{W^{2,q'}},\\ \left\|P(g,h,h_{\infty})_{1\to t}(\Pi^{\perp}_{h_\infty}(k_1))\right\|_{W^{2,r'}}&\leq C\left\|k_1 \right\|_{W^{2,r'}}.
\end{align*}
From now on, let $t\geq 3$ and $s\in [1,t-2]$. We combine different parts of Lemma \ref{lem : short-time estimates} (to pass from $t-1$ to $t$) and Corollary \ref{cor : linear estimates} (to pass from $1$ to $t-1$). We first prove \eqref{eq : estimates mixed linear evolution2}. The more special estimate \eqref{eq : estimates mixed linear evolution 1} then follows from setting $\Theta=k_1$, $s=1$, $p'=q'$ and using the definition of the $X_{q',r'}$-norm.
At first, by Lemma \ref{lem : short-time estimates} (i), Corollary \ref{cor : linear estimates} and Lemma \ref{lem : projection lemma} (ii), we clearly have
\begin{align*}
	\left\| P(g,h,h_{\infty})_{1\to t}(\Pi^{\perp}_{h_\infty}(\Theta))\right\|_{L^{q'}}
		&= \left\|P(g,h,h_{\infty})_{t-1\to t}\circ e^{-(t-1-s)\Delta_{L,h_{\infty}}}(\Pi^{\perp}_{h_\infty}(\Theta))\right\|_{L^{q'}}\\
		&\leq C\left\| e^{-(t-1-s)\Delta_{L,h_{\infty}}}(\Pi^{\perp}_{h_\infty}(\Theta))\right\|_{L^{q'}}\\
		&\leq C(t-1-s)^{-\frac{n}{2}\left(\frac{1}{p'}-\frac{1}{q'}\right)}\left\|\Theta\right\|_{L^{p'}}
\end{align*}
We again have
\begin{align*}
	\beta:=\min\left\{\frac{n}{2}\left(\frac{1}{q'}-\frac{1}{r'}\right),1\right\}\geq 
\min\left\{\frac{n}{2}\left(\frac{1}{q}-\frac{1}{r}\right),1\right\}\geq \frac{1}{2}.
\end{align*}
Thus by using Lemma \ref{lem : short-time estimates} (v), Corollary \ref{cor : linear estimates} and Lemma \ref{lem : projection lemma} (ii), we get
\begin{align*}
	&\left\|\nabla\circ P(g,h,h_{\infty})_{s\to t}(\Pi^{\perp}_{h_\infty}(\Theta))\right\|_{L^{q'}}
		=\left\|\nabla\circ P(g,h,h_{\infty})_{t-1\to t}\circ e^{-(t-1-s)\Delta_{L,h_{\infty}}}(\Pi^{\perp}_{h_\infty}(\Theta))\right\|_{L^{q'}}\\ 
	&\qquad\qquad \leq C\left( \left\|\nabla\circ e^{-(t-1-s)\Delta_{L,h_{\infty}}}(\Pi^{\perp}_{h_\infty}(\Theta))\right\|_{L^{q'}}+\left\| e^{-(t-1-s)\Delta_{L,h_{\infty}}}(\Pi^{\perp}_{h_\infty}(\Theta))\right\|_{L^{r'}}
\right)\\ 
	&\qquad\qquad \leq  C (t-1-s)^{-\frac{n}{2}\left( \frac{1}{p'}-\frac{1}{q'} \right)-\frac{1}{2}}\left\| \Theta\right\|_{L^{p'}}+C(t-1-s)^{-\frac{n}{2}\left( \frac{1}{p'}-\frac{1}{q'} \right) -\frac{n}{2}\left( \frac{1}{q'}-\frac{1}{r'} \right) }\left\| \Theta\right\|_{L^{p'}} \\
	&\qquad\qquad \leq C (t-1-s)^{-\frac{n}{2}\left( \frac{1}{p'}-\frac{1}{q'} \right) -\frac{1}{2}}\left\| \Theta\right\|_{L^{p'}}
\end{align*}
and by using Lemma \ref{lem : short-time estimates} (vi), Corollary \ref{cor : linear estimates} and Lemma \ref{lem : projection lemma} (ii),
\begin{align*}
	&\left\|\nabla^2\circ P(g,h,h_{\infty})_{s\to t}(\Pi^{\perp}_{h_\infty}(\Theta))\right\|_{L^{q'}}
		= \left\|\nabla^2\circ P(g,h,h_{\infty})_{t-1\to t}\circ e^{-(t-s-1)\Delta_{L,h_{\infty}}}(\Pi^{\perp}_{h_\infty}(\Theta))\right\|_{L^{q'}}\\
 		&\qquad\qquad\leq C\left( \left\|\nabla^2\circ e^{-(t-1-s)\Delta_{L,h_{\infty}}}(\Pi^{\perp}_{h_\infty}(\Theta))\right\|_{L^{q'}}+\left\| e^{-(t-1-s)\Delta_{L,h_{\infty}}}(\Pi^{\perp}_{h_\infty}(\Theta))\right\|_{W^{1,r'}}
\right) \\ 
		&\qquad\qquad\leq  C (t-1-s)^{-\frac{n}{2} \left( \frac{1}{p'}-\frac{1}{q'} \right) -\beta}\left\| \Theta\right\|_{L^{p'}}+C(t-1-s)^{-\frac{n}{2}\left( \frac{1}{p'}-\frac{1}{q'} \right) -\frac{n}{2}\left( \frac{1}{q'}-\frac{1}{r'} \right) }\left\| \Theta\right\|_{L^{p'}}\\
		&\qquad\qquad\leq  C (t-1-s)^{-\frac{n}{2}\left( \frac{1}{p'}-\frac{1}{q'} \right) -\beta}\left\| \Theta\right\|_{L^{p'}}.
\end{align*}
Finally, by Lemma \ref{lem : short-time estimates} (iv), Corollary \ref{cor : linear estimates} and Lemma \ref{lem : projection lemma} (ii), we have
\begin{align*}
	\left\|P(g,h,h_{\infty})_{s\to t}(\Pi^{\perp}_{h_\infty}(\Theta))\right\|_{W^{2,r'}}
		&= \left\| P(g,h,h_{\infty})_{t-1\to t}\circ e^{-(t-1-s)\Delta_{L,h_{\infty}}}(\Pi^{\perp}_{h_\infty}(\Theta))\right\|_{W^{2,r'}}\\ 
		&\leq C \left\| e^{-(t-1-s)\Delta_{L,h_{\infty}}}(\Pi^{\perp}_{h_\infty}(\Theta))\right\|_{L^{r'}} \\ 
		&\leq  C (t-1-s)^{-\frac{n}{2} \left( \frac{1}{p'}-\frac{1}{q'} \right)-\frac{n}{2}\left( \frac{1}{q'}-\frac{1}{r'}\right) } \left\| \Theta\right\|_{L^{p'}},
\end{align*}
which finishes the proof of the lemma.
\end{proof}	
\begin{lem}\label{lem : nonlinear part 0.1}
	For all $q'\in (1,q]$, $r'\in [r,\infty)$ and $p'\in (1,r']$, 
there exists a constant $C=C(q',r',p',\U,\epsilon)$ such that for all times	
	 $t\in [1,\infty)$ and $t'\in [t,\infty]$, we have
	\begin{align*}
	\left\|(\Delta_{L,h_{t'}}-\Delta_{L,h_t})(k_t)\right\|_{L^{p'}} 
		\leq Ct^{1-\frac{3n}{2}\left(\frac{1}{q'}-\frac{1}{r'}\right)} \left\|h-\hat{h}\right\|_{Z_{q',r'}}\left\|k\right\|_{X_{q',r'}}.
	\end{align*}
\end{lem}
\begin{proof}
	Choose $p''\in (p',\infty]$ so that $\frac{1}{p'}=\frac{1}{p''}+\frac{1}{r'}$. Then by the H\"{o}lder inequality and Lemma \ref{lem : elliptic regularity},
\begin{equation}
\begin{split}
\label{eq : h minus limit h}
	\left\|(\Delta_{L,h_{t'}}-\Delta_{L,h_t})(k_t)\right\|_{L^{p'}}&\leq C\left\| h_t-h_{t'}\right\|_{W^{2,p''}}\left\|k_t\right\|_{W^{2,r'}}\\
		&\leq C\left\| h_t-h_{t'}\right\|_{L^{q'}}\left\|k_t\right\|_{W^{2,r'}}\\
		&\leq C \int_t^{t'}\left\|\partial_sh_s\right\|_{L^{q'}}ds\cdot\left\|k_t\right\|_{W^{2,r'}}\\
		&\leq C\int_t^{\infty} s^{-n\left(\frac{1}{q'}-\frac{1}{r'}\right)}ds\cdot \left\|h-\hat{h}\right\|_{Z_{q',r'}} t^{-\frac{n}{2}\left(\frac{1}{q'}-\frac{1}{r'}\right)}\left\|k\right\|_{X_{q',r'}}\\
		&\leq Ct^{1-\frac{3n}{2}\left(\frac{1}{q}-\frac{1}{r}\right)} \left\|h-\hat{h}\right\|_{Z_{q',r'}}\left\|k\right\|_{X_{q',r'}},
\end{split}
\end{equation}
which finishes the proof.	
\end{proof}
\begin{lem}\label{lem : nonlinear part 0.2}
	For $q'\in (1,q]$ satisfying
	\begin{align*}
	\frac{n}{2}\left(\frac{1}{q'}-\frac{1}{q}\right)
		<\frac{n}{2r},	
	\end{align*}
	there exists a constant $C=C(q,q',r,\U,\epsilon)$ such that for all $t\geq 1$, we have
	\begin{align*}
	\left\|(1-D_g\Phi)(H_1)(t)\right\|_{L^{q'}}&\leq Ct^{-\beta-\frac{n}{2}\left(\frac{1}{q}-\frac{1}{r}\right)}\left\|k\right\|_{X_{q,r}}^2,
	\end{align*}
	where $\beta=\min\left\{1,\frac{n}{2}\left(\frac{1}{q}-\frac{1}{r}\right)\right\}>\frac{1}{2}$.
\end{lem}
\begin{proof}
	By the condition on $q'$, we may choose $r'\in (r,\infty]$, such that
	\begin{align*}
	\frac{1}{q'}=\frac{1}{q}+\frac{1}{r'}.
	\end{align*}	
	By Lemma \ref{lem : projection lemma} (iii), the H\"{o}lder inequality and Sobolev embedding, we get
	\begin{align*}
	\left\|(1-D_g\Phi)(H_1)(t)\right\|_{L^{q'}}\leq C\left\|H_1(t)\right\|_{L^{q'}}
	&\leq C(\left\|\nabla^2k_t\right\|_{L^{q}}\left\|k_t\right\|_{L^{r'}}+\left\|\nabla k_t\right\|_{L^{2q'}}^2+\left\| R\right\|_{L^{q'}}\left\|\nabla k_t\right\|_{L^{\infty}}^2)\\
%	&\leq C(\left\|\nabla^2k_t\right\|_{L^{q}}\left\|k_t\right\|_{L^{r'}}+\left\| R\right\|_{L^{q'}}\left\|\nabla k_t\right\|_{L^{r'}}^2)\\
	&\leq C(\left\|\nabla^2k_t\right\|_{L^{q}}\left\|k_t\right\|_{W^{1,r}}+(1+\left\| R\right\|_{L^{q'}})\left\|\nabla k_t\right\|_{W^{1,r}}^2)\\	
	&\leq C\left(t^{-\beta-\frac{n}{2}\left(\frac{1}{q}-\frac{1}{r}\right)}+t^{-n\left(\frac{1}{q}-\frac{1}{r}\right)}\right)\left\|k\right\|_{X_{q,r}}^2\\
	&\leq Ct^{-\beta-\frac{n}{2}\left(\frac{1}{q}-\frac{1}{r}\right)}\left\|k\right\|_{X_{q,r}}^2
	\end{align*}
	with $\beta=\min\left\{1,\frac{n}{2}\left(\frac{1}{q}-\frac{1}{r}\right)\right\}>\frac{1}{2}$.
	\end{proof}
\begin{lem}\label{lem : nonlinear part 1}
For all $q'\in (1,q],r'\in [r,\infty)$ satisfying
	\begin{align}\label{eq : condition for q' r'}
		\frac{n}{2}\left(\frac{1}{q'}-\frac{1}{q}\right)+\frac{n}{2}\left(\frac{1}{r}-\frac{1}{r'}\right)
			< \min\left\{n\left(\frac{1}{q}-\frac{1}{r}\right)-1,\frac{1}{2}\right\},
			\qquad 	\frac{n}{2}\left(\frac{1}{q'}-\frac{1}{q}\right)
			<\frac{n}{2r},	
	\end{align}
there exists a constant $C=C(q,r,q',r',\U,\epsilon)$ such that for all $t\geq3$, we get the estimate
	\begin{align*}
	\left\|\int_1^{t-2}  P(g,h,h_{\infty})_{s\to t}I(t,s,k,h,h_{\infty})ds\right\|_{X_{q',r'}}\leq C	(\left\|k\right\|_{X_{q,r}}+\left\|h-\hat{h}\right\|_{Z_{q,r}})\left\|k\right\|_{X_{q,r}}.	\end{align*}
\end{lem}
\begin{proof}	
By the triangle inequality, we first split the term further up as
	\begin{align*}
	& \left\|\int_1^{t-2}  P(g,h,h_{\infty})_{s\to t}I(t,s,k,h,h_{\infty})ds\right\|_{X_{q',r'}} \\
		&\qquad \leq \left\|\int_1^{t-2} P(g,h,h_{\infty})_{s\to t}\Pi^{\perp}_{\infty}[(\Delta_{L,\infty}-\Delta_{L,h})(k)ds\right\|_{X_{q',r'}}\\
	 	&\qquad \quad + \left\|\int_1^{t-2} P(g,h,h_{\infty})_{s\to t}\Pi^{\perp}_{\infty}(1-D_g\Phi)(H_1)ds\right\|_{X_{q',r'}}.
	\end{align*}
	Now let us consider the first of these terms. We estimate each of the terms of the $X_{q',r'}$-norm separately.
Let $q''\in\left\{q',r'\right\}$, $i\in\left\{0,1,2\right\}$ and $\alpha:\left\{q',r' \right\}\times\left\{0,1,2 \right\}\to\R$ be the function from Lemma \ref{lem : linear part}. Let $p'\in (1,q')$ be small.
By Lemma \ref{lem : linear part} and Lemma \ref{lem : nonlinear part 0.1}, we get
	\begin{align*}
	&\left\|\nabla^i\int_1^{t-2} P(g,h,h_{\infty})_{s\to t}\Pi^{\perp}_{\infty}[(\Delta_{L,\infty}-\Delta_{L,h})(k)ds\right\|_{L^{q''}}\\
	&\qquad\leq\int_1^{t-2}	\left\|\nabla^i\circ P(g,h,h_{\infty})_{s\to t}\Pi^{\perp}_{\infty}\right\|_{L^{p'},L^{q''}} \left\|(\Delta_{L,{\infty}}-\Delta_{L,h})(k)\right\|_{L^{p'}}ds\\
	&\qquad\leq C\int_1^{t-1}(t-s)^{-\frac{n}{2}(\frac{1}{p'}-\frac{1}{q'})-\alpha(q'',i)}s^{1-\frac{3n}{2}(\frac{1}{q}-\frac{1}{r})}ds\cdot\left\|h-\hat{h}\right\|_{Z_{q,r}}\left\|k\right\|_{X_{q,r}}\\
	&\qquad \leq C t^{-\alpha(q'',i)}\left\|h-\hat{h}\right\|_{Z_{q,r}}\left\|k\right\|_{X_{q,r}}.
	\end{align*}
	The last inequality is justified by Lemma \ref{important_technical_lemma}:
We have
\begin{align*}
	\frac{n}{2}\left(\frac{1}{p'}-\frac{1}{q'}\right)+\alpha(q'',i)>	\alpha(q'',i),\qquad\qquad
		\frac{3n}{2}\left(\frac{1}{q}-\frac{1}{r}\right)-1
			>\frac{n}{2}\left(\frac{1}{q}-\frac{1}{r}\right)
			\geq \alpha(q'',i)
\end{align*}	
and \eqref{eq : condition for q' r'} implies
	\begin{align*}
	\frac{n}{2}\left(\frac{1}{p'}-\frac{1}{q'}\right)+\alpha(q'',i)+\frac{3n}{2}\left(\frac{1}{q}-\frac{1}{r}\right)-2
>\frac{n}{2}\left(\frac{1}{p'}-\frac{1}{r'}\right)-1+\alpha(q'',i)	
	>	\alpha(q'',i),
	\end{align*}
	for all choices of $q''$ and $i$, provided that $p'$ is chosen small enough.
We have thus shown that
	\begin{align*}
	 &\left\|\int_1^{t-2} P(g,h,h_{\infty})_{s\to t}\Pi^{\perp}_{\infty}[(\Delta_{L,\infty}-\Delta_{L,h})(k)ds\right\|_{X_{q',r'}}\leq C\left\|h-\hat{h}\right\|_{Z_{q,r}}\left\|k\right\|_{X_{q,r}}.
		\end{align*}
For the other part of the integral, the estimate is slightly different.			
				 By Lemma \ref{lem : linear part} and Lemma \ref{lem : nonlinear part 0.2}
			\begin{align*}
	&	\left\|\nabla^i\int_1^{t-2} P(g,h,h_{\infty})_{s\to t}\Pi^{\perp}_{\infty}
		(1-D_g\Phi)(H_1)ds\right\|_{L^{q''}}\\ 
		&\qquad\leq 	\int_1^{t-2}
		\left\| \nabla^iP(g,h,h_{\infty})_{s\to t}\Pi^{\perp}_{\infty}\right\|_{L^{q'},L^{q''}}\left\|
		(1-D_g\Phi)(H_1)\right\|_{L^{q'}}ds\\
		&\qquad\leq 	\int_1^{t-1}(t-s)^{-\alpha(q'',i)}s^{-\frac{1}{2}-\frac{n}{2}(\frac{1}{q}-\frac{1}{r})}ds\cdot \left\|k\right\|_{X_{q,r}}^2\\
		&\qquad\leq C t^{-\alpha(q'',i)}\cdot\left\|k\right\|_{X_{q,r}}^2.	
				\end{align*}
			The last inequality is again justified by Lemma \ref{important_technical_lemma} and \eqref{eq : condition for q' r'}, since
		\begin{align*}
			\frac{1}{2}+\frac{n}{2}\left(\frac{1}{q}-\frac{1}{r}\right)
				>1 ,
				\qquad \frac{1}{2}+\frac{n}{2}\left(\frac{1}{q}-\frac{1}{r}\right)
				>\frac{n}{2}\left(\frac{1}{q'}-\frac{1}{r'}\right)
				\geq \alpha(q'',i).
		\end{align*}
This proves
		\begin{align*}
			\left\|\int_1^{t-2} P(g,h,h_{\infty})_{s\to t}\Pi^{\perp}_{\infty}
	(1-D_g\Phi)(H_1)ds\right\|_{X_{q',r'}}
				\leq C\left\|k\right\|_{X_{q,r}}^2
		\end{align*}	
and we conclude
		\begin{align*}
			\left\|\int_1^{t-2} P(g,h,h_{\infty})_{s\to t}I(t,s,k,h,h_{\infty})ds\right\|_{X_{q',r'}}
				\leq C\left(\left\|k\right\|_{X_{q,r}}+\left\|h-\hat{h}\right\|_{Z_{q,r}}\right) \left\|k\right\|_{X_{q,r}},
		\end{align*}
		as desired
\end{proof}

\begin{lem}\label{lem : nonlinear part 1.5}
For all $q'\in (1,q]$ with
	\begin{align*}
	\frac{n}{2}\left(\frac{1}{q'}-\frac{1}{q}\right)
		<\frac{n}{2r},
\end{align*}
there exists a constant $C=C(q,r,q',\U,\epsilon)$ such that for all $t\geq2$, we have
	\begin{align*}
		&\left\|\int_{\max\left\{t-2,1\right\}}^{t-1}  P(g,h,h_{\infty})_{s\to t}I(t,s,k,h,h_{\infty})ds\right\|_{X_{q',r}}
%		 \\		&\qquad \qquad\qquad \qquad \qquad \qquad \qquad \qquad \qquad \qquad 
			 \leq C \left(\left\|k\right\|_{X_{q,r}}+\left\|h-\hat{h}\right\|_{Z_{q,r}}\right) \left\|k\right\|_{X_{q,r}}.
	\end{align*}
\end{lem}
\begin{proof}
	By Lemma \ref{lem : short-time estimates} (iv), we first have
	\begin{align*}
	\left\|\int_{\max\left\{t-2,1\right\}}^{t-1}  P(g,h,h_{\infty})_{s\to t}I(t,s,k,h,h_{\infty})ds\right\|_{W^{2,q''}}
	\leq C\sup_{s\in [\max\left\{t-2,1\right\},t-1]}\left\|I(t,s,k,h,h_{\infty})\right\|_{L^{q''}}
		\end{align*}
	for $q''\in \left\{q',r\right\}$. Recall that 
	\begin{align*}
	I(t,s,k,h,h_{\infty})=\Pi^{\perp}_{h_\infty}[(\Delta_{L,\infty}-\Delta_{L,h})(k)
 + 	(1-D_g\Phi)(H_1)]
	\end{align*}
for $s\in [\max\left\{t-2,1\right\},t-1]$. With the help of Lemma \ref{lem : projection lemma} (ii) and Lemma \ref{lem : nonlinear part 0.1}, we get
\begin{align*}
	\sup_{s\in [\max\left\{t-2,1\right\},t-1]} \left\|\Pi^{\perp}_{h_\infty}(\Delta_{L,\infty}-\Delta_{L,h_s})(k_s)\right\|_{L^{q''}} 
		&\leq \sup_{s\in [\max\left\{t-2,1\right\},t-1]} \left\|(\Delta_{L,\infty}-\Delta_{L,h_s})(k_s)\right\|_{L^{q''}} 
		\\
		&\leq Ct^{1-\frac{3n}{2} \left( \frac{1}{q}-\frac{1}{r} \right) } \left\|h-\hat{h}\right\|_{Z_{q,r}}\left\|k\right\|_{X_{q,r}}.
\end{align*}
 Lemma \ref{lem : projection lemma} (ii) and Lemma \ref{lem : nonlinear part 0.2} yield
\begin{align*}
	\sup_{s\in [\max\left\{t-2,1\right\},t-1]} \left\|\Pi^{\perp}_{h_\infty}(1-D_g\Phi)(H_1)(s)\right\|_{L^{q'}}
	&\leq 	\sup_{s\in [\max\left\{t-2,1\right\},t-1]} \left\|(1-D_g\Phi)(H_1)(s)\right\|_{L^{q'}}
	\\
		&\leq Ct^{-\beta-\frac{n}{2}\left( \frac{1}{q}-\frac{1}{r} \right) }\left\|k\right\|_{X_{q,r}}^2,
\end{align*}
and estimates as in the proof of Lemma \ref{lem : nonlinear part 0.2} yield
	\begin{align*}
	\sup_{s\in [\max\left\{t-2,1\right\},t-1]}	\left\|(1-D_g\Phi)(H_1)(s)\right\|_{L^{r}}\leq
    	\sup_{s\in [\min\left\{t-2,1\right\},t-1]}	\left\|k_s\right\|_{W^{2,r}}^2	
    	\leq Ct^{-n \left( \frac{1}{q}-\frac{1}{r} \right) }\left\|k\right\|_{X_{q,r}}^2.
\end{align*}
The function $\alpha$ defined in Lemma \ref{lem : linear part} satisfies
\begin{align*}
	\alpha(q'',i)\leq \beta+\frac{n}{2} \left( \frac{1}{q}-\frac{1}{r} \right)
		\leq n \left( \frac{1}{q}-\frac{1}{r} \right),
		\qquad\qquad \alpha(q'',i)
		\leq \frac{3n}{2}\left( \frac{1}{q}-\frac{1}{r} \right) -1
\end{align*}
and the statement follows from putting these estimates together.
\end{proof}
In the remainder of this subsection, we estimate the integral from $\max\left\{t-1,1\right\}$ to $t$. We start with a technical lemma.
\begin{lem}\label{lem : integral term}
For  $q'\in (1,q]$ and $r'\in [r,\infty)$ satisfying
\begin{align*}
\frac{1}{q'}<\frac{1}{q}+\frac{1}{r'},
\end{align*}
there exists a constant $C=C(q,r,q',r',\U,\epsilon)$ such that for all times $t\geq 1$ and
$s\in [\max\left\{t-1,1\right\},t]$,
  we have
 \begin{align*}
 \left\|I(t,s,k,h,h_{\infty})\right\|_{W^{1,r'}}&\leq C(\left\|k\right\|_{W^{2,r'}}+\left\|h-h_{\infty}\right\|_{L^q})\left\|k\right\|_{W^{2,r'}},\\
  \left\|I(t,s,k,h,h_{\infty})\right\|_{W^{1,q'}}&\leq C(\left\|k\right\|_{W^{2,r'}}+\left\|\nabla k\right\|_{W^{1,q}}+\left\|h-h_{\infty}\right\|_{L^q})\left\|k\right\|_{W^{2,r'}}.
 \end{align*}
	\end{lem}
\begin{proof}
%Let $p'\in \left\{q',r'\right\}$.
Recall that for $s\in[\max\left\{t-1,1\right\},t]$, we have
\begin{align*}
I(t,s,k,h,h_{\infty})=[\Delta_{L,g,h},\Pi^{\perp}_{\infty}](k)+\Pi^{\perp}_{\infty}[D_g\Phi((\Delta_{L,g,h}-\Delta_{L,h})(k))+(1-D_g\Phi)(H_2)]
\end{align*}
Because $[\Delta_{L,g,h},\Pi^{\perp}_{\infty}]=[\Delta_{L,g,h},1-\Pi^{\parallel}_{\infty}]=-[\Delta_{L,g,h},\Pi^{\parallel}_{\infty}]$, we can rewrite the above expression as
\begin{align*}
I(t,s,k,h,h_{\infty})&=-[\Delta_{L,g,h}-\Delta_{L,h},\Pi^{\parallel}_{\infty}](k)\\
&\qquad-[\Delta_{L,h}-\Delta_{L,\infty},\Pi^{\parallel}_{\infty}](k)\\
&\qquad+\Pi^{\perp}_{\infty}[D_g\Phi((\Delta_{L,g,h}-\Delta_{L,h})(k))]\\
&\qquad +\Pi^{\perp}_{\infty}[(1-D_g\Phi)(H_2)].
\end{align*}
To obtain the estimates of the lemma, we will discuss each of these four terms separately.
Using $g=h+k$, the first of these terms can be written as
\begin{align*}
[\Delta_{L,g,h}-\Delta_{L,h},\Pi^{\parallel}_{\infty}](k)&=(\Delta_{L,g,h}-\Delta_{L,h})\Pi^{\parallel}_{\infty}(k)-\Pi^{\parallel}_{\infty}((\Delta_{L,g,h}-\Delta_{L,h})(k))
\\&=k*\nabla^2\Pi^{\parallel}_{\infty}(k)+R*k*\Pi^{\perp}_{\infty}(k)-\Pi^{\parallel}_{\infty}(k*\nabla^2k+R*k*k).
\end{align*}
Using the H\"{o}lder inequality and Lemma \ref{lem : projection lemma} (i) yields
\begin{align*}
\left\| [\Delta_{L,g,h}-\Delta_{L,h},\Pi^{\parallel}_{\infty}](k)\right\|_{W^{1,r'}}&\leq \left\| k*\nabla^2\Pi^{\parallel}_{\infty}(k) \right\|_{W^{1,r'}}+\left\|R*k*\Pi^{\perp}_{\infty}(k) \right\|_{W^{1,r'}}\\
&\qquad+\left\|\Pi^{\parallel}_{\infty}(k*\nabla^2k+R*k*k) \right\|_{W^{1,r'}}\\
&\leq C\left(\left\| k \right\|_{W^{1,r'}}
\left\|\Pi^{\parallel}_{\infty}(k) \right\|_{W^{3,r'}}
+\left\| k*\nabla^2k+R*k*k\right\|_{L^{\frac{r'}{2}}}\right)\\
&\leq C\left\| k \right\|_{W^{2,r'}}^2
\end{align*}
To prove the same bound for the $W^{1,q'}$-norm, we proceed analogously, but with different H\"{o}lder exponents. At first, because $1< q'<r'<\infty$ and
\begin{align*}
\frac{1}{q'}<\frac{1}{q}+\frac{1}{r'},
\end{align*}
we find $r''\in(q,\infty)$ such that
\begin{align*}
\frac{1}{q'}=\frac{1}{r'}+\frac{1}{r''}.
\end{align*}
Now we get
\begin{align*}
\left\| [\Delta_{L,g,h}-\Delta_{L,h},\Pi^{\parallel}_{\infty}](k)\right\|_{W^{1,q'}}&\leq \left\| k*\nabla^2\Pi^{\parallel}_{\infty}(k) \right\|_{W^{1,q'}}+\left\|R*k*\Pi^{\perp}_{\infty}(k) \right\|_{W^{1,q'}}\\
&\qquad+\left\|\Pi^{\parallel}_{\infty}(k*\nabla^2k+R*k*k) \right\|_{W^{1,q'}}\\
&\leq C\left(\left\| k \right\|_{W^{1,r'}}
\left\|\Pi^{\parallel}_{\infty}(k) \right\|_{W^{3,r''}}
+\left\| k*\nabla^2k+R*k*k\right\|_{L^{\frac{r'}{2}}}\right)\\
&\leq C\left\| k \right\|_{W^{2,r'}}^2,
\end{align*}
which finishes the discussion for the first term. For the second and the third term, we can develop the estimates of the $W^{1,r'}$ and the $W^{1,q'}$-norms simultaneously.
The second term which can be schematically written as
\begin{align*}
[\Delta_{L,h}-\Delta_{L,\infty},\Pi^{\parallel}_{\infty}](k)&=(\Delta_{L,h}-\Delta_{L,\infty})\Pi^{\parallel}_{\infty}(k)-\Pi^{\parallel}_{\infty}((\Delta_{L,h}-\Delta_{L,\infty})(k))\\
&=(h-h_{\infty})*\nabla^2\Pi^{\parallel}_{\infty}(k)+\nabla(h-h_{\infty})*\nabla\Pi^{\parallel}_{\infty}(k)\\
&\qquad +\nabla^2(h-h_{\infty})*\Pi^{\parallel}_{\infty}(k)+R*(h-h_{\infty})*\Pi^{\parallel}_{\infty}(k)\\
&\qquad- \Pi^{\parallel}_{\infty}( (h-h_{\infty})*\nabla^2k+\nabla(h-h_{\infty})*\nabla k)\\
&\qquad- \Pi^{\parallel}_{\infty}(\nabla^2(h-h_{\infty})*k+R*(h-h_{\infty})*k).
\end{align*}
Now, let $p'\in\left\{q',r'\right\}$.
By the H\"{o}lder inequality, Lemma \ref{lem : projection lemma} (i) and Lemma \ref{lem : elliptic regularity}, we obtain
\begin{align*}
&\left\| [\Delta_{L,h}-\Delta_{L,\infty},\Pi^{\parallel}_{\infty}](k)\right\|_{W^{1,p'}} \\
	&\quad \leq
\left\|(h-h_{\infty})*\nabla^2\Pi^{\parallel}_{\infty}(k) \right\|_{W^{1,p'}}+
\left\| \nabla(h-h_{\infty})*\nabla\Pi^{\parallel}_{\infty}(k)\right\|_{W^{1,p'}}\\
	&\quad \qquad +\left\| \nabla^2(h-h_{\infty})*\Pi^{\parallel}_{\infty}(k)\right\|_{W^{1,p'}}+
\left\|R*(h-h_{\infty})*\Pi^{\parallel}_{\infty}(k)\right\|_{W^{1,p'}}\\
	&\quad \qquad +\left\|\Pi^{\parallel}_{\infty}(\nabla^2(h-h_{\infty})*k+R*(h-h_{\infty})*k)\right\|_{W^{1,p'}}\\
	&\quad \leq C\left\|h-h_{\infty}\right\|_{W^{3,\infty}}
\left\|\Pi^{\parallel}_{\infty}(k)\right\|_{W^{3,p'}}
\\
	&\quad \qquad
 +C\left\| \nabla^2(h-h_{\infty})*k+R*(h-h_{\infty})*k\right\|_{L^{r'}}\\
 	&\quad \leq C\left\|h-h_{\infty}\right\|_{W^{3,\infty}}\left\|k\right\|_{W^{2,r'}}\\
	&\quad \leq C\left\|h-h_{\infty}\right\|_{L^{q}}\left\|k\right\|_{W^{2,r'}}
\end{align*}
Using Lemma \ref{lem : projection lemma} (ii) and (iii) and the H\'{o}lder inequality, the third term is estimated as
\begin{align*}
\left\|\Pi^{\perp}_{\infty}[D_g\Phi((\Delta_{L,g,h}-\Delta_{L,h})(k))\right\|_{W^{1,p'}}&\leq
C\left\|(\Delta_{L,g,h}-\Delta_{L,h})(k)\right\|_{L^{\frac{r'}{2}}}\\
&= C\left\|k*\nabla^2k+\nabla k*\nabla k+R*k*k\right\|_{L^{\frac{r'}{2}}}\\
&\leq
 C\left\|k\right\|_{W^{2,r'}}^2.
\end{align*}
In order to estimate the final term,
%$\Pi^{\parallel}_{\infty}(1-D_g\Phi)(H_2)$,
  we
recall that $H_2=g^{-1}*g^{-1}*\nabla k*\nabla k$ which yields pointwise bounds
\begin{align*}
|H_2|\leq C|\nabla k|^2,\qquad |\nabla H_2|\leq C( |\nabla k||\nabla^2 k|+ |\nabla k|^3).
\end{align*}
By using Lemma \ref{lem : projection lemma} (i) and (iii), applying the H\"{o}lder inequality to the above pointwise bounds and finally using Sobolev embedding, we get
\begin{align*}
\left\|\Pi^{\parallel}_{\infty}(1-D_g\Phi)(H_2)\right\|_{W^{1,r'}}
\leq C \left\|H_2\right\|_{W^{1,r'}}
\leq C  \left\|\nabla k\right\|_{W^{1,r'}}\left\|\nabla k\right\|_{L^{\infty}}
\leq C\left\|k\right\|_{W^{2,r'}}^2.
\end{align*}
This finishes the discussion of the $W^{1,r'}$-norm. Let us proceed with the $W^{1,q'}$-norm.
Because $\frac{1}{q}\leq \frac{1}{q'}<\frac{1}{q}+\frac{1}{r}$, we may find $r''\in (r',\infty]$ such that 
\begin{align*}
\frac{1}{q'}=\frac{1}{q}+\frac{1}{r''}.
\end{align*}
Again by using Lemma \ref{lem : projection lemma} (i) and (iii), applying the H\"{o}lder inequality to the above pointwise bounds and finally using Sobolev embedding, we get
\begin{align*}
\left\|\Pi^{\parallel}_{\infty}(1-D_g\Phi)(H_2)\right\|_{W^{1,q'}}\leq
C \left\|H_2\right\|_{W^{1,q'}}\leq C \left\|\nabla k\right\|_{W^{1,q}}\left\|\nabla k\right\|_{L^{r''}}
\leq C\left\|\nabla k\right\|_{W^{1,q}}\left\|k\right\|_{W^{2,r'}}.
\end{align*}
This establishes the estimates for the fourth and final term and finishes the proof of the lemma.
	\end{proof}
\begin{lem}\label{lem : nonlinear part 2} For $q'\in (1,q]$ satisfying
\begin{align*}
\frac{1}{q'}<\frac{1}{q}+\frac{1}{r},
\end{align*}
there exists a constant $C=C(q,r,q',\U,\epsilon)$ such that for all $t\geq1$, we have
	\begin{align*}
	&\left\|\int_{\max\left\{t-1,1\right\}}^t  P(g,h,h_{\infty})_{s\to t}I(t,s,k,h,h_{\infty})ds\right\|_{X_{q',r}}
%	\\	&\qquad\qquad\qquad\qquad\qquad\qquad\qquad\qquad
	\leq C \left(\left\|k\right\|_{X_{q,r}}+\left\|h-\hat{h}\right\|_{Z_{q,r}}\right)\left\|k\right\|_{X_{q,r}}.
	\end{align*}
\end{lem}
\begin{proof}
 Let $p'\in\left\{q',r\right\}$.
By short-time estimates (Lemma \ref{lem : short-time estimates} (iii)), we have
					\begin{align*}
\left\|\int_{\max\left\{t-1,1\right\}}^t  P(g,h,h_{\infty})_{s\to t}I(t,s,k,h,h_{\infty})ds\right\|_{W^{2,p'}}&\leq C\sup_ {s\in[\max\left\{t-1,1\right\},t]}	\left\|I(t,s,k,h,h_{\infty})\right\|_{W^{1,p'}}.
\end{align*}
From Lemma \ref{lem : integral term}, we have
 \begin{align*}
\left\|I(t,s,k,h,h_{\infty})\right\|_{W^{1,p'}}&\leq C(\left\|k\right\|_{W^{2,r}}+\left\|\nabla k\right\|_{W^{1,q}}+\left\|h-h_{\infty}\right\|_{L^q})\left\|k\right\|_{W^{2,r}}.
\end{align*}
By definition of the norms, we have
\begin{align*}
	\sup_ {s\in[\max\left\{t-1,1\right\},t]}\left\|h-h_{\infty}\right\|_{L^{q}}\left\|k\right\|_{W^{2,r}}
		&\leq Ct^{1-n\left(\frac{1}{q}-\frac{1}{r}\right)-\frac{n}{2}\left(\frac{1}{q}-\frac{1}{r}\right) }\left\|h-\hat{h}\right\|_{Z_{q,r}}\left\|k\right\|_{X_{q,r}},\\
	\sup_ {s\in[\max\left\{t-1,1\right\},t]}\left\|k\right\|_{W^{2,r}}^2
		&\leq Ct^{-n \left( \frac{1}{q}-\frac{1}{r} \right)} \left\|k\right\|_{X_{q,r}}^2,\\
	\sup_ {s\in[\max\left\{t-1,1\right\},t]}\left\|\nabla k\right\|_{W^{1,q}}\left\|k\right\|_{W^{2,r}}
		&\leq C t^{-\frac{1}{2}-\frac{n}{2}\left( \frac{1}{q}-\frac{1}{r}\right) }\left\|k\right\|_{X_{q,r}}^2.
\end{align*}	
Combining all these estimates and using the conditions on $q'$, we get
\begin{align*}
	&\left\|\int_{\max\left\{t-1,1\right\}}^t P(g,h,h_{\infty})_{s\to t}I(t,s,k,h,h_{\infty})ds\right\|_{W^{2,p'}} \\ 
	&\qquad\qquad\qquad\qquad \leq C t^{-\frac{n}{2}\left(\frac{1}{q'}-\frac{1}{r}\right) }\left( \left\|k\right\|_{X_{q,r}}+\left\|h-\hat{h}\right\|_{Z_{q,r}} \right) \left\|k\right\|_{X_{q,r}},
\end{align*}
which yields the desired estimate.
\end{proof}
\subsection{Contraction properties of the iteration map}
Let  $(M^n,\hat{h})$ be an integrable ALE manifold with a parallel spinor and $q\in (1,n)$, $r\in(n,\infty)$ be H\"{o}lder exponents so that
\begin{align*}
	\frac{n}{2}\left(\frac{1}{q}-\frac{1}{r}\right)
		&>\frac{1}{2},\qquad 	\frac{n}{2}\left(\frac{1}{q}-\frac{1}{r}\right)
		\neq 1
\end{align*}
Moreover, let $\U$ be an $L^{[q,\infty]}$-neighbourhood of $\hat{h}$ in the space of metrics, which is so small that the projection map $\Phi:\mathcal{U}\to\mathcal{U}\cap\mathcal{F}$ of Subsection \ref{subsec: proj map} is defined.

 Let furthermore
$h:=(h_t)_{t\geq 1}$, $\tilde{h}=(\tilde{h}_t)_{t\geq1}$ be families of metrics in $\U\cap\F$ with  $\tilde{h}_1=h_1$  and $k:=(k_t)_{t\geq1}$, $(\tilde{k}_t)_{t\geq1}$
 be families of symmetric $2$-tensors with $k_1=\tilde{k}_1$. Define the families $g:=(g_t)_{t\geq 1}:=(h_t+k_t)_{t\geq 1}$ and $\tilde{g}:=(\tilde{g}_t)_{t\geq 1}:=(\tilde{h}_t+\tilde{k}_t)_{t\geq 1}$. 
%\noindent
 The goal of this subsection is to establish the following contraction estimate:
\begin{thm}\label{thm : psicontracting}
	There exists an $\epsilon>0$ such that if $\left\|(h-\hat{h},k)\right\|_{Y_{q,r}}+\left\|(\tilde{h}-\hat{h},\tilde{k})\right\|_{Y_{q,r}}<\epsilon$, there exists a constant $C=C(q,r,\U,\epsilon)$ such that
	the operator $\psi$ satisfies the estimate
	\begin{align*}
	\left\|\psi(h,k)-\psi(\tilde{h},\tilde{k})\right\|_{Y_{q,r}}&\leq C\left(\left\|k_1\right\|_{W^{2,q}}+\left\|k_1\right\|_{W^{2,r}}\right)\left\|(h-\tilde{h},k-\tilde{k})\right\|_{Y_{q,r}}\\
	&\qquad +C\left(\left\|(h-\hat{h},k)\right\|_{Y_{q,r}}+\left\|(\tilde{h}-\hat{h},\tilde{k})\right\|_{Y_{q,r}}\right)\left\|(h-\tilde{h},k-\tilde{k})\right\|_{Y_{q,r}}.
	\end{align*}
\end{thm}
The proof of this theorem is split up in Propositions \ref{psi1contraction} and \ref{psi2contraction} below, in which the estimates for the components of $\psi$ are established.
At first, recall that 
	\begin{align*}
H_1=F_1(g^{-1},g^{-1},\nabla k,\nabla k)+F_2(g^{-1},R,k,k)+F_3(g^{-1},k,\nabla^2k)
\end{align*}
and abbreviate
	\begin{align*}
&F_1:=F_1(g^{-1},g^{-1},\nabla k,\nabla k),\qquad &\tilde{F}_1:=F_1(\tilde{g}^{-1},\tilde{g}^{-1},\tilde{\nabla}\tilde{k},\tilde{\nabla}\tilde{k}),
\\
&F_2:=F_2(g^{-1},R,k,k),\qquad &\tilde{F}_2:=F_2(\tilde{g}^{-1},\tilde{R},\tilde{k},\tilde{k}),\\
&F_3:=F_3(g^{-1},k,\nabla^2k),\qquad &\tilde{F}_3:=F_3(\tilde{g}^{-1},\tilde{k},\tilde{\nabla}^2\tilde{k}).
\end{align*}
For the remainder of this Subsection, we will furthermore always assume that 
\begin{align*}\left\|(h-\hat{h},k)\right\|_{Y_{q,r}}+\left\|(\tilde{h}-\hat{h},\tilde{k})\right\|_{Y_{q,r}}<\epsilon,
\end{align*}
where $\epsilon>0$ is chosen sufficiently small.
\begin{lem}\label{lem : F_1-tildeF_1}
There exists a constant $C=C(\U,\epsilon)$ such that we have the pointwise estimates
\begin{align*}
	|F_1-\tilde{F}_1|&\leq C|k-\tilde{k}||\nabla k|^2+C|\nabla (k-\tilde{k})|(|\nabla k|+|\nabla\tilde{k}|)\\
		&\qquad +C(|h-\tilde{h}|+|\nabla (h-\tilde{h})|)(|\nabla k|+|\nabla\tilde{k}|+|k|+|\tilde{k}|)^2, \\
	|F_2-\tilde{F}_2|
		&\leq C |k-\tilde{k}|(|k|+|\tilde{k}|)(|R|+|\tilde{R}|)+C(|h-\tilde{h}|+|\nabla(h-\tilde{h})|+|\nabla^2(h-\tilde{h})|)		|k|^2,\\
	|F_3-\tilde{F}_3|&\leq C(|k-\tilde{k}||\nabla^2 k|+|\tilde{k}||\nabla^2(k-\tilde{k})|) \\
	&\qquad + C(|h-\tilde{h}||k||\nabla^2k|+|\nabla(h-\tilde{h})||\nabla \tilde{k}||\tilde{k}|+|\nabla(h-\tilde{h})|^2|\tilde{k}|^2).	
\end{align*}
\end{lem}
\begin{proof}Note that Lemma \ref{lem : normcontrol} ensures that $\left\|k\right\|_{W^{1,\infty}}+\left\|\tilde{k}\right\|_{W^{1,\infty}}$ is small.
	We first look at the difference
	\begin{align*}
	F_1(g^{-1},g^{-1},\nabla k,\nabla k)-F_1(\tilde{g}^{-1},\tilde{g}^{-1},\tilde{\nabla}\tilde{k},\tilde{\nabla}\tilde{k})&=
	F_1	(g^{-1}-\tilde{g}^{-1},g^{-1},\nabla k,\nabla k)\\
	&\qquad +F_1(\tilde{g}^{-1},g^{-1}-\tilde{g}^{-1},\nabla k,\nabla k)\\
	&\qquad +F_1(\tilde{g}^{-1},\tilde{g}^{-1},\nabla k-\tilde{\nabla}\tilde{k},\nabla k)\\
	&\qquad +F_1(\tilde{g}^{-1},\tilde{g}^{-1},\tilde{\nabla}\tilde{k},\nabla k-\tilde{\nabla}\tilde{k}).
	\end{align*}
	Note that 
	\begin{align*}	|g^{-1}-\tilde{g}^{-1}|\leq C(|k-\tilde{k}|+|h-\tilde{h}|).
	\end{align*}
	 Note further, that by using the tensor $T_{ij}^k:=\Gamma_{ij}^k-\tilde{\Gamma}_{ij}^k$, we can write $\nabla \tilde{k}-\tilde{\nabla}\tilde{k}=T*\tilde{k}$.
The tensor $T$ is schematically of the form $T=\tilde{h}^{-1}*\nabla(h-\tilde{h})$. Therefore we get
	\begin{align*}
 |\nabla k-\tilde{\nabla}\tilde{k}|\leq C(|\nabla (k-\tilde{k})|+|T||\tilde{k}|),
 \qquad |T|\leq C|\nabla (h-\tilde{h})|,
	\end{align*}
 and the first inequality is obtained from combining all these estimates.
The other estimates are performed similarly, using in addition
\begin{align*}
R-\tilde{R}=\nabla T+T*T, \qquad \nabla^2 \tilde{k}-\tilde{\nabla}^2\tilde{k}=\nabla T*\tilde{k}+T*\nabla\tilde{k}+T* T*\tilde{k}.
\end{align*}
The details are left to the reader.
\end{proof}
\begin{lem}\label{lem : H_1-tildeH_1} 
There exists a constant $C=C(q,r,\U,\epsilon)$ such that we have
	\begin{align*}
	\left\|H_1-\tilde{H}_1\right\|_{L^q}
		&\leq Ct^{-\frac{1}{2}-\frac{n}{2}\left(\frac{1}{q}-\frac{1}{r}\right) }
	\left(\left\|h-\tilde{h}\right\|_{Z_{q,r}}+\left\|k-\tilde{k}\right\|_{X_{q,r}}\right)
	\left(\left\|k\right\|_{X_{q,r}}+\left\|\tilde{k}\right\|_{X_{q,r}}\right).
	\end{align*}
	\end{lem}
\begin{proof}
	By Lemma \ref{lem : F_1-tildeF_1} and the H\"{o}lder inequality, we conclude
	\begin{align*}
	\left\|F_1-\tilde{F}_1\right\|_{L^{q}}&\leq C\left\|k-\tilde{k}\right\|_{L^{q}}\left\|\nabla k\right\|_{L^{\infty}}^2+C \left\|\nabla(k-\tilde{k})\right\|_{L^{q}} \left( \left\|\nabla k\right\|_{L^{\infty}}+\left\|\nabla\tilde{k}\right\|_{L^{\infty}}\right) \\
	&\qquad +C\left\|h-\tilde{h}\right\|_{W^{1,{q}}} \left( \left\|k\right\|_{W^{1,\infty}}+\left\|\tilde{k}\right\|_{W^{1,\infty}} \right)^2 \\
	&\leq C\left\|k-\tilde{k}\right\|_{L^{q}}\left\| k\right\|_{W^{2,r}}^2+C \left\|\nabla(k-\tilde{k})\right\|_{L^{q}} \left(\left\| k\right\|_{W^{2,r}}+\left\|\tilde{k}\right\|_{W^{2,r}}\right)\\
	&\qquad +C\left\|h-\tilde{h}\right\|_{W^{1,{q}}} \left( \left\|k\right\|_{W^{2,r}}+\left\|\tilde{k}\right\|_{W^{2,r}} \right)^2\\	
	&\leq Ct^{-n\left(\frac{1}{q}-\frac{1}{r}\right)}\left\|k-\tilde{k}\right\|_{X_{q,r}}\left\|k\right\|_{X_{q,r}}^2
	+t^{-\frac{1}{2}-\frac{n}{2}\left(\frac{1}{q}-\frac{1}{r}\right)}\left\|k-\tilde{k}\right\|_{X_{q,r}} \left( \left\|k\right\|_{X_{q,r}}+\left\|\tilde{k}\right\|_{X_{q,r}}\right) \\
	&\qquad +Ct^{-n\left(\frac{1}{q}-\frac{1}{r}\right)}\left\|h-\tilde{h}\right\|_{X_{q,r}} \left( \left\|k\right\|_{X_{q,r}}+\left\|\tilde{k}\right\|_{X_{q,r}} \right)^2\\
	&\leq Ct^{-\frac{1}{2}-\frac{n}{2}\left(\frac{1}{q}-\frac{1}{r}\right)} \left( \left\|h-\tilde{h}\right\|_{X_{q,r}}+\left\|k-\tilde{k}\right\|_{X_{q,r}} \right) \left( \left\|k\right\|_{X_{q,r}}+\left\|\tilde{k}\right\|_{X_{q,r}} \right).
	\end{align*}
Furthermore, we have
	\begin{align*}
	\left\|F_2-\tilde{F}_2\right\|_{L^{q}}
		&\leq C\left\|k-\tilde{k}\right\|_{L^{\infty}}\left( \left\|k\right\|_{L^{\infty}}+\left\|\tilde{k}\right\|_{L^{\infty}} \right) \left( \left\|{R}\right\|_{L^{q}}+\left\|\tilde{R}\right\|_{L^{q}} \right)
	+ C\left\|h-\tilde{h}\right\|_{W^{2,q}}\left\|k\right\|_{L^{\infty}}^2\\
	&\leq C\left\|k-\tilde{k}\right\|_{W^{2,r}} \left( \left\|k\right\|_{W^{2,r}}+\left\|\tilde{k}\right\|_{W^{2,r}} \right)+ C\left\|h-\tilde{h}\right\|_{L^q}\left\|k\right\|_{W^{2,r}}^2\\
	&\leq Ct^{-n\left(\frac{1}{q}-\frac{1}{r}\right)}
	\left( \left\|h-\tilde{h}\right\|_{Z_{q,r}}+\left\|k-\tilde{k}\right\|_{X_{q,r}} \right)\left(\left\|k\right\|_{X_{q,r}}+\left\|\tilde{k}\right\|_{X_{q,r}} \right)\\
	&\leq Ct^{-n\left(\frac{1}{q}-\frac{1}{r}\right)}
	\left( \left\|h-\tilde{h}\right\|_{Z_{q,r}}+\left\|k-\tilde{k}\right\|_{X_{q,r}} \right)\left(\left\|k\right\|_{X_{q,r}}+\left\|\tilde{k}\right\|_{X_{q,r}} \right).
	\end{align*}
Choose $r'\in (q,\infty)$ such that $\frac{1}{q}=\frac{1}{r'}+\frac{1}{r}$.
 Then the H\"{o}lder inequality yields
\begin{align*}
	\left\|F_3-\tilde{F}_3\right\|_{L^{q}}
		&\leq C\left( \left\|\nabla^2 k\right\|_{L^q}\left\|k-\tilde{k}\right\|_{L^{\infty}}+\left\|\nabla^2(k-\tilde{k})\right\|_{L^{q}}\left\|\tilde{k}\right\|_{L^{\infty}} \right) \\
		&\qquad + C\left\|h-\tilde{h}\right\|_{W^{2,r'}} \left( \left\|k\right\|_{L^{\infty}}\left\|\nabla^2 k\right\|_{L^r}+\left\|\nabla\tilde{k}\right\|_{L^{r}}\left\|\tilde{k}\right\|_{L^{\infty}}+\left\|\tilde{k}\right\|_{L^{r}}\left\|\tilde{k}\right\|_{L^{\infty}} \right) \\
		&\leq C \left( \left\|\nabla^2 k\right\|_{L^q}\left\|k-\tilde{k}\right\|_{W^{2,r}}+\left\|\nabla^2(k-\tilde{k})\right\|_{L^{q}}\left\|\tilde{k}\right\|_{W^{2,r}} \right) \\
		&\qquad + C\left\|h-\tilde{h}\right\|_{L^q} \left( \left\|k\right\|_{W^{2,r}}^2+\left\|\tilde{k}\right\|_{W^{2,r}}^2 \right) \\
		&\leq Ct^{-\beta-\frac{n}{2}\left(\frac{1}{q}-\frac{1}{r}\right) }
\left(\left\|h-\tilde{h}\right\|_{Z_{q,r}}+\left\|k-\tilde{k}\right\|_{X_{q,r}}\right)
\left(\left\|k\right\|_{X_{q,r}}+\left\|\tilde{k}\right\|_{X_{q,r}}\right)\\
&\leq Ct^{-\beta-\frac{n}{2}\left(\frac{1}{q}-\frac{1}{r}\right) }
\left(\left\|h-\tilde{h}\right\|_{Z_{q,r}}+\left\|k-\tilde{k}\right\|_{X_{q,r}}\right)
\left(\left\|k\right\|_{X_{q,r}}+\left\|\tilde{k}\right\|_{X_{q,r}}\right),
\end{align*}
where
\begin{align*}
\beta=\min\left\{1,\frac{n}{2}\left(\frac{1}{q}-\frac{1}{r}\right)\right\}.
\end{align*}
Using that
\begin{align*}
\frac{n}{2}\left(\frac{1}{q}-\frac{1}{r}\right)\geq\beta >\frac{1}{2}
\end{align*}
finishes the proof of the lemma.
\end{proof}
\begin{prop}\label{psi1contraction}
There exists a constant $C=C(q,r,\U,\epsilon)$ such that we have
	\begin{align*}
	\left\|\psi_1(h,k)-\psi_1(\tilde{h},\tilde{k})\right\|_{Z_{q,r}}\leq C\left(\left\|k\right\|_{X_{q,r}}+\left\|\tilde{k}\right\|_{X_{q,r}}\right)\left(\left\|h-\tilde{h}\right\|_{Z_{q,r}}+\left\|k-\tilde{k}\right\|_{X_{q,r}}\right).
	\end{align*}
\end{prop}
\begin{proof}
	We estimate, using Lemma \ref{lem : projection lemma} (iii) and (vi) and Lemma \ref{lem : H_1-tildeH_1},
	\begin{align*}
	&\left\|\psi_1(h,k)-\psi_1(\tilde{h},\tilde{k})\right\|_{L^{[q,\infty]}} \\
	&\quad \leq C \left\|\overline{\psi}_1(h,k)-\overline{\psi}_1(\tilde{h},\tilde{k})\right\|_{L^{[q,\infty]}}\\
	&\quad \leq C\int_1^t\left\| D_g\Phi(H_1(h,k))-D_{\tilde{g}}\Phi(H_1(\tilde{h},\tilde{k}))\right\|_{L^{[q,\infty]}}ds\\
	&\quad \leq C\int_1^t\left\| (D_g\Phi-D_{\tilde{g}}\Phi)(H_1(\tilde{h},\tilde{k}))+D_g\Phi((H_1(h,k)-H_1(\tilde{h},\tilde{k}))\right\|_{L^{[q,\infty]}}ds\\
	&\quad \leq C\int_1^t\left[\left\|g-\tilde{g}\right\|_{L^{[q,\infty]}}\left\|H_1(\tilde{h},\tilde{k})\right\|_{L^{r/2}}+\left\|H_1(h,k)-H_1(\tilde{h},\tilde{k})\right\|_{L^{q}}\right]ds\\
	&\quad \leq C\int_1^t\left[\left(\left\|h-\tilde{h}\right\|_{L^{[q,\infty]}}+\left\|k-\tilde{k}\right\|_{L^{[q,\infty]}}\right)\left\|\tilde{k}\right\|_{W^{2,r}}^2+\left\|H_1(h,k)-H_1(\tilde{h},\tilde{k})\right\|_{L^{q}}\right]ds\\
	&\quad \leq C\int_1^t s^{-\frac{1}{2}-\frac{n}{2}\left(\frac{1}{q}-\frac{1}{r}\right)}ds\left(\left\|k\right\|_{X_{q,r}}+\left\|\tilde{k}\right\|_{X_{q,r}}\right)\left(\left\|h-\tilde{h}\right\|_{Z_{q,r}}+\left\|k-\tilde{k}\right\|_{X_{q,r}}\right)\\
	&\quad \leq C\left(\left\|k\right\|_{X_{q,r}}+\left\|\tilde{k}\right\|_{X_{q,r}}\right)\left(\left\|h-\tilde{h}\right\|_{Z_{q,r}}+\left\|k-\tilde{k}\right\|_{X_{q,r}}\right).
	\end{align*}
	Note that we have used $-\beta>-\frac{1}{2}$ in the last inequality. Note also that we have estimated the $L^q$-norm and the $L^{\infty}$-norm at once. The $L^q$-estimate gives us the first part of the $Z_{q,r}$-norm. Both estimates will be needed at the end of the proof.
	
	Now for the other term of the norm, we estimate using Lemma \ref{lem : projection lemma} (iii) and (vi) again,
	\begin{align*}
	t^{n\left( \frac{1}{q}-\frac{1}{r} \right)}
		&\left\|\partial_t(\psi_1(h,k)-\psi_1(\tilde{h},\tilde{k}))\right\|_{L^q} \\
		&= t^{n\left(\frac{1}{q}-\frac{1}{r}\right)} \left\| D_{\overline{\psi}_1(h,k)}\Phi(D_g\Phi(H_1(h,k)))-D_{\overline{\psi}_1(\tilde{h},\tilde{k})}\Phi(D_{\tilde{g}}\Phi(H_1(\tilde{h},\tilde{k}))) \right\|_{L^q}\\
		&\leq t^{n\left( \frac{1}{q}-\frac{1}{r} \right)} \left\|
	(D_{\overline{\psi}_1(h,k)}\Phi-D_{\overline{\psi}_1(\tilde{h},\tilde{k})}\Phi)(D_{\tilde{g}}\Phi(H_1(\tilde{h},\tilde{k})))\right\|_{L^q}\\
		&\qquad+ t^{n\left( \frac{1}{q}-\frac{1}{r} \right)}\left\|D_{\overline{\psi}_1(h,k)}\Phi(D_g\Phi-D_{\tilde{g}}\Phi)(H_1(\tilde{h},\tilde{k}))\right\|_{L^q}\\
		&\qquad+ t^{n\left(\frac{1}{q}-\frac{1}{r}\right)}\left\|D_{\overline{\psi}_1(h,k)}\Phi(D_g\Phi(H_1(h,k)-H_1(\tilde{h},\tilde{k})))\right\|_{L^q}\\
		&\leq C t^{n\left( \frac{1}{q}-\frac{1}{r} \right)}\left\|H_1(\tilde{h},\tilde{k})\right\|_{L^{r/2}}\left( \left\|\overline{\psi}_1(h,k)-\overline{\psi}_1(\tilde{h},\tilde{k})\right\|_{L^{[q,\infty]}}+\left\|g-\tilde{g}\right\|_{L^{[q,\infty]}}\right) \\
		&\qquad+  C t^{n\left(\frac{1}{q}-\frac{1}{r}\right)}\left\|H_1(h,k)-H_1(\tilde{h},\tilde{k}))\right\|_{L^{q}}\\
		&\leq C t^{n\left(\frac{1}{q}-\frac{1}{r}\right)}\left\|\tilde{k}\right\|_{W^{2,r}}^2\left(\left\|\overline{\psi}_1(h,k)-\overline{\psi}_1(\tilde{h},\tilde{k})\right\|_{L^{[q,\infty]}}+\left\|k-\tilde{k}\right\|_{L^{[q,\infty]}}+\left\|h-\tilde{h}\right\|_{L^{[q,\infty]}}\right)\\
		&\qquad +  C t^{n\left(\frac{1}{q}-\frac{1}{r}\right)}\left\|H_1(h,k)-H_1(\tilde{h},\tilde{k}))\right\|_{L^{q}}\\
		&\leq C\left(\left\|k\right\|_{X_{q,r}}+\left\|\tilde{k}\right\|_{X_{q,r}}\right)\left(\left\|h-\tilde{h}\right\|_{Z_{q,r}}+\left\|k-\tilde{k}\right\|_{X_{q,r}}\right).
	\end{align*}
	In the last inequality, we used  Lemma \ref{lem : H_1-tildeH_1} again and the estimate from the first part of this proof.
\end{proof}

\begin{prop}\label{psi2contraction}
	There exists a constant $C=C(q,r,\U,\epsilon)$ such that we have
	\begin{align*}
		\left\|\psi_2(h,k)-\psi_2(\tilde{h},\tilde{k})\right\|_{X_{q,r}}&\leq C	\left(\left\|k_1 \right\|_{W^{2,r}}+\left\|k_1 \right\|_{W^{2,q}}+\left\|k\right\|_{X_{q,r}}+\left\|\tilde{k}\right\|_{X_{q,r}}
		+\left\|h-\hat{h}\right\|_{Z_{q,r}}\right)\\
		&\qquad\cdot 
		\left(\left\|h-\tilde{h}\right\|_{Z_{q,r}}+\left\|k-\tilde{k}\right\|_{X_{q,r}}\right)
		\end{align*}
	\end{prop}
\begin{proof}
Let us first write
	\begin{equation}\label{eq : Prop 5.32}
	\begin{split}
\psi_2(h,k)-\psi_2(\tilde{h},\tilde{k})
)	&=[(\Pi^{\perp}_{h,h_{\infty}})^{-1}-(\Pi^{\perp}_{\tilde{h},\tilde{h}_{\infty}})^{-1}]\Pi^{\perp}_{h_{\infty}}(\overline{\psi}_2(h,k))\\
&\qquad +(\Pi^{\perp}_{\tilde{h},\tilde{h}_{\infty}})^{-1}
[\Pi^{\perp}_{h_{\infty}}-\Pi^{\perp}_{\tilde{h}_{\infty}}]\overline{\psi}_2(h,k)\\
&\qquad+(\Pi^{\perp}_{\tilde{h},\tilde{h}_{\infty}})^{-1}\Pi^{\perp}_{\tilde{h}_{\infty}}(\overline{\psi}_2(h,k)-\overline{\psi}_2(\tilde{h},\tilde{k}))
\end{split}
	\end{equation}
and estimate term by term. Let us first define a function
 $\alpha:\left\{q,r \right\}\times\left\{0,1,2 \right\}\to\R$ by
		$	\alpha(q,0)=0$, $\alpha(q,1)=\frac{1}{2}, \alpha(q,2)=\min\left\{1,\frac{n}{2}(\frac{1}{q}-\frac{1}{r})\right\}$ and $\alpha(r,i)=\frac{n}{2}(\frac{1}{q}-\frac{1}{r})$ for $i=0,1,2$.	 
Let us also abbreviate for notational simplicity $\overline{\psi}_2:=\overline{\psi}_2(h,k)$  for the moment. Let $p'\in \left\{q,r \right\}$ and $i\in \left\{0,1,2 \right\}$.
Then by Lemma \ref{lem : projection lemma} (vii) and (ii) and Lemma \ref{lem : inftycontrol},
	\begin{align*}
	\left\|\nabla^i[(\Pi^{\perp}_{h,h_{\infty}})^{-1}-(\Pi^{\perp}_{\tilde{h},\tilde{h}_{\infty}})^{-1}]\Pi^{\perp}_{h_{\infty}}(\overline{\psi}_2)\right\|_{L^{p'}}
		&\leq C\left(\left\|h-\tilde{h}\right\|_{L^q}+\left\|h_{\infty}-\tilde{h}_{\infty}\right\|_{L^q}\right)\left\|\Pi^{\perp}_{h_{\infty}}(\overline{\psi}_2)\right\|_{L^{r}},
		\\
&\leq 		C\left\|h-\tilde{h}\right\|_{Z_{q,r}}\left\|\overline{\psi}_2\right\|_{L^{r}}
\\
&\leq 	Ct^{-\alpha(p',i)} \left\|h-\tilde{h}\right\|_{Z_{q,r}} \left\|\overline{\psi}_2\right\|_{X_{q,r}},
	\end{align*}
so that	
	\begin{align*}
	\left\|[(\Pi^{\perp}_{h,h_{\infty}})^{-1}-(\Pi^{\perp}_{\tilde{h},\tilde{h}_{\infty}})^{-1}]\Pi^{\perp}_{h_{\infty}}(\overline{\psi}_2)\right\|_{X_{q,r}}\leq C\left\|h-\tilde{h}\right\|_{Z_{q,r}} \left\|\overline{\psi}_2\right\|_{X_{q,r}}
	\end{align*}
	Similarly, by Lemma \ref{lem : projection lemma} (iv) and (v) and Lemma \ref{lem : inftycontrol},
			\begin{align*}
	\left\|\nabla^i[(\Pi^{\perp}_{\tilde{h},\tilde{h}_{\infty}})^{-1}
[\Pi^{\perp}_{h_{\infty}}-\Pi^{\perp}_{\tilde{h}_{\infty}}]\overline{\psi}_2]\right\|_{L^{p'}}
		&\leq C
	\left\|[\Pi^{\perp}_{h_{\infty}}-\Pi^{\perp}_{\tilde{h}_{\infty}}]\overline{\psi}_2]\right\|_{W^{2,p'}}	\\
	&\leq  C\left\|h_{\infty}-\tilde{h}_{\infty}\right\|_{L^q}
	 \left\|\overline{\psi}_2\right\|_{L^{r}}
	\\
&\leq 	Ct^{-\alpha(p',i)}\left\|h-\tilde{h}\right\|_{Z_{q,r}} \left\|\overline{\psi}_2\right\|_{X_{q,r}}
	\end{align*}
	so that
	\begin{align*}
\left\|[(\Pi^{\perp}_{\tilde{h},\tilde{h}_{\infty}})^{-1}
[\Pi^{\perp}_{h_{\infty}}-\Pi^{\perp}_{\tilde{h}_{\infty}}]\overline{\psi}_2]\right\|_{X_{q,r}}\leq 	C\left\|h-\tilde{h}\right\|_{Z_{q,r}} \left\|\overline{\psi}_2\right\|_{X_{q,r}}.	
	\end{align*}	
	Furthermore, we know from the proof of Proposition \ref{psi2mapping} that
\begin{align*}
\left\|\overline{\psi}_2\right\|_{X_{q,r}}=\left\|\overline{\psi}_2(h,k)\right\|_{X_{q,r}}&\leq C( \left\|k_1\right\|_{W^{2,q}}+\left\|k_1\right\|_{W^{2,r}}+(\left\|k\right\|_{X_{q,r}}+\left\|h-\hat{h}\right\|_{Z_{q,r}})\left\|k\right\|_{X_{q,r}})\\
&\leq 
 C( \left\|k_1\right\|_{W^{2,q}}+\left\|k_1\right\|_{W^{2,r}}+\left\|k\right\|_{X_{q,r}})
\end{align*}
	from which we conclude
	\begin{align*}
&	\left\|[(\Pi^{\perp}_{h,h_{\infty}})^{-1}-(\Pi^{\perp}_{\tilde{h},\tilde{h}_{\infty}})^{-1}]\Pi^{\perp}_{h_{\infty}}(\overline{\psi}_2)\right\|_{X_{q,r}}+
	\left\|[(\Pi^{\perp}_{\tilde{h},\tilde{h}_{\infty}})^{-1}
[\Pi^{\perp}_{h_{\infty}}-\Pi^{\perp}_{\tilde{h}_{\infty}}]\overline{\psi}_2]\right\|_{X_{q,r}}\\
%&\qquad\qquad\qquad\qquad \leq
%C( \left\|k_1\right\|_{W^{2,q}}+\left\|k_1\right\|_{W^{2,r}}+(\left\|k\right\|_{X_{q,r}}+\left\|h-\hat{h}\right\|_{Z_{q,r}})\left\|k\right\|_{X_{q,r}})\left\|h-\tilde{h}\right\|_{Z_{q,r}}\\
&\qquad\qquad\qquad\qquad \leq C
( \left\|k_1\right\|_{W^{2,q}}+\left\|k_1\right\|_{W^{2,r}}+\left\|k\right\|_{X_{q,r}})\left\|h-\tilde{h}\right\|_{Z_{q,r}}
	\end{align*}
We have so far estimated the first two of the three terms in \eqref{eq : Prop 5.32}.	
It remains to estimate the term
\begin{align*}
(\Pi^{\perp}_{\tilde{h},\tilde{h}_{\infty}})^{-1}\Pi^{\perp}_{\tilde{h}_{\infty}}(\overline{\psi}_2(h,k)-\overline{\psi}_2(\tilde{h},\tilde{k}))
\end{align*}
To do this, note that the proof of Lemma \ref{lem : projection and Banach space} shows that the time-dependant operator $(\Pi^{\perp}_{\tilde{h},\tilde{h}_{\infty}})^{-1}\Pi^{\perp}_{\tilde{h}_{\infty}}$ acting on a family of time-dependent tensors is a bounded map on $X_{q,r}$.
		 Therefore,
	\begin{align*}
	\left\|(\Pi^{\perp}_{\tilde{h},\tilde{h}_{\infty}})^{-1}\Pi^{\perp}_{\tilde{h}_{\infty}}(\overline{\psi}_2(h,k)-\overline{\psi}_2(\tilde{h},\tilde{k}))\right\|_{X_{q,r}}\leq C\left\|\overline{\psi}_2(h,k)-\overline{\psi}_2(\tilde{h},\tilde{k})\right\|_{X_{q,r}}.
	\end{align*}
Thus, to finish the proof, it suffices to establish the estimate
		\begin{align*}
		\left\|\overline{\psi}_2(h,k)-\overline{\psi}_2(\tilde{h},\tilde{k})\right\|_{X_{q,r}}
		&\leq C\left(\left\|h-\tilde{h}\right\|_{Z_{q,r}}+\left\|k-\tilde{k}\right\|_{X_{q,r}}\right)\\
		&\qquad \cdot
		\left(\left\|k_1 \right\|_{W^{2,r}}+\left\|k_1 \right\|_{W^{2,q}}+\left\|k\right\|_{X_{q,r}}+\left\|\tilde{k}\right\|_{X_{q,r}}\right).
	\end{align*}
We rewrite this difference as
\begin{align*}
\overline{\psi}_2(h,k)-&\overline{\psi}_2(\tilde{h},\tilde{k}) \\
	&= P(g,h,h_{\infty})_{1\to t}\circ \Pi^{\perp}_{h_{\infty}}(k_1)-P(\tilde{g},\tilde{h},\tilde{h}_{\infty})_{1\to t}\circ \Pi^{\perp}_{\tilde{h}_{\infty}}(k_1)\\
		&\qquad+\int_1^{t}[  P(g,h,h_{\infty})_{s\to t}I(t,s,k,h,h_{\infty})- P(\tilde{g},\tilde{h},\tilde{h}_{\infty})_{s\to t}I(t,s,\tilde{k},\tilde{h},\tilde{h}_{\infty}) ]ds\\
	&=P(g,h,h_{\infty})_{1\to t}\circ \Pi^{\perp}_{h_{\infty}}(k_1)-P(\tilde{g},\tilde{h},\tilde{h}_{\infty})_{1\to t}\circ \Pi^{\perp}_{\tilde{h}_{\infty}}(k_1)\\
		&\qquad +\int_1^{\max\left\{t-1,1\right\}}  P(g,h,h_{\infty})_{s\to t}[I(t,s,k,h,h_{\infty})-I(t,s,\tilde{k},\tilde{h},\tilde{h}_{\infty})] ds\\
		&\qquad+\int_1^{\max\left\{t-4,1\right\}}  [P(g,h,h_{\infty})_{s\to t}- P(\tilde{g},\tilde{h},\tilde{h}_{\infty})_{s\to t}]I(t,s,\tilde{k},\tilde{h},\tilde{h}_{\infty}) ds\\
		&\qquad+\int_{\max\left\{t-4,1\right\}}^{\max\left\{t-1,1\right\}}  [P(g,h,h_{\infty})_{s\to t}- P(\tilde{g},\tilde{h},\tilde{h}_{\infty})_{s\to t}]I(t,s,\tilde{k},\tilde{h},\tilde{h}_{\infty}) ds\\
		&\qquad+\int_{\max\left\{t-1,1\right\}}^{t}[  P(g,h,h_{\infty})_{s\to t}I(t,s,k,h,h_{\infty})- P(\tilde{g},\tilde{h},\tilde{h}_{\infty})_{s\to t}I(t,s,\tilde{k},\tilde{h},\tilde{h}_{\infty}) ]ds.
\end{align*}	
The terms on the right hand side are estimated 
in a sequence of lemmas which we prove in the remainder of this subsection. More precisely, the first term is estimated by Lemma \ref{lem : difference term 1}, the second term by Lemma \ref{lem : difference term 2.2} and Lemma \ref{lem : difference term 2.4}, the third one by Lemma \ref{lem : difference term 3}, the fourth one by Lemma \ref{lem : difference term 4} and the last one by Lemma \ref{lem : difference term 5}.
\end{proof}	
\begin{lem}\label{lem : difference term 1}
There exists a constant $C=C(q,r,\U,\epsilon)$ such that we have
	\begin{align*}
&	\left\|P(g,h,h_{\infty})_{1\to t}\circ \Pi^{\perp}_{h_{\infty}}(k_1)-P(\tilde{g},\tilde{h},\tilde{h}_{\infty})_{1\to t}\circ \Pi^{\perp}_{\tilde{h}_{\infty}}(k_1) \right\|_{X_{q,r}} \\
	&\qquad\qquad \leq C(\left\|k_1\right\|_{W^{2,q}}+\left\|k_1\right\|_{W^{2,r}})\left( \left\|\tilde{h}-h\right\|_{Z_{q,r}}+ \left\|\tilde{k}-k\right\|_{X_{q,r}}\right).
	\end{align*}
	For any $t\geq5$ and $p'\in (1,q]$ satisfying $\frac{n}{2}\left( \frac{1}{p'}-\frac{1}{r} \right) \neq 1$, we even have a constant  $C=C(q,r,p',\U,\epsilon)$ such that 
	\begin{align*}
	&\left\|t^{\frac{n}{2}\left( \frac{1}{p'}-\frac{1}{q} \right) }(P(g,h,h_{\infty})_{1\to t}\circ \Pi^{\perp}_{h_{\infty}}(k_1)-P(\tilde{g},\tilde{h},\tilde{h}_{\infty})_{1\to t}\circ \Pi^{\perp}_{\tilde{h}_{\infty}}(k_1)) \right\|_{X_{q,r}}\\  
	&\qquad\qquad\leq C \left\|k_1\right\|_{L^{p'}}\left( \left\|\tilde{h}-h\right\|_{Z_{q,r}}+ \left\|\tilde{k}-k\right\|_{X_{q,r}}\right).
\end{align*}	
\end{lem}
\begin{proof}
	We abbreviate, for each $s\in[1,t]$,
	\begin{align*}
	\varphi_s=P(g,h,h_{\infty})_{1\to s}\circ \Pi^{\perp}_{h_{\infty}}(k_1),\qquad \tilde{\varphi}_s=P(\tilde{g},\tilde{h},\tilde{h}_{\infty})_{1\to s}\circ \Pi^{\perp}_{\tilde{h}_{\infty}}(k_1)
	\end{align*}
For $t\leq5$, the desired estimate of the lemma follows from
showing for $q'\in\left\{q,r\right\}$ that
	\begin{align*}
	\left\|\varphi_t-\tilde{\varphi}_t\right\|_{W^{2,q'}}&\leq C
\left\|(	\Pi^{\perp}_{h_{\infty}}-\Pi^{\perp}_{\tilde{h}_{\infty}})(k_1)\right\|_{W^{2,q'}}\\
&\qquad+C\sup_{s\in [1,t]}
\left(\left\|g_s-\tilde{g}_s\right\|_{W^{1,\infty}}+\left\|h_s-\tilde{h}_s\right\|_{W^{4,\infty}}\right)\left\|\Pi^{\perp}_{\tilde{h}_{\infty}}(k_1) \right\|_{W^{2,q'}}\\
&\leq C\left\|h_{\infty}-\tilde{h}_{\infty}\right\|_{L^{q}}\left\|k_1\right\|_{L^{q'}}\\
&\qquad  
+C\sup_{s\in [1,t]}
\left(\left\|k_s-\tilde{k}_s\right\|_{W^{1,\infty}}+\left\|h_s-\tilde{h}_s\right\|_{W^{4,\infty}}\right)\left\|k_1 \right\|_{W^{2,q'}}\\
&\leq C\left\|h_{\infty}-\tilde{h}_{\infty}\right\|_{L^{q}}\left\|k_1\right\|_{L^{q'}}\\
&\qquad  
+C\sup_{s\in [1,t]}
\left(\left\|k_s-\tilde{k}_s\right\|_{W^{2,r}}+\left\|h_s-\tilde{h}_s\right\|_{L^{q}}\right)\left\|k_1 \right\|_{W^{2,q'}}\\
&\leq C\left\|k_1\right\|_{W^{2,q'}}\left(\left\|h-\tilde{h}\right\|_{Z_{q,r}}+\left\|k-\tilde{k}\right\|_{X_{q,r}}\right).
	\end{align*}
Here, we used Lemma \ref{lem : difference of heat flows} (i) in the first inequality,  Lemma \ref{lem : projection lemma} (ii) and (v) as well as the triangle inequality in the second inequality, Sobolev embedding and Lemma \ref{lem : elliptic regularity} in the third inequality and Lemma \ref{lem : inftycontrol} and the definition of the norms in the fourth inequality.
Note that Lemma 3.4 applies to the heat flow of the modified Lichnerowicz Laplacian as well as to the heat flow of the (unmodified) Lichnerowicz Laplacian (as it is just a special case of the former). This is why we can use it here.

	For $t>5$ we proceed as follows:
At first, we estimate $\varphi_t-\tilde{\varphi}_{t}$ by $\varphi_{t-1}-\tilde{\varphi}_{t-1}$ which we do by short-time estimates. 
The term $\varphi_{t-1}-\tilde{\varphi}_{t-1}$ is then easier to estimate in terms of the initial data because we have by construction
\begin{align*}
\varphi_{t-1}=e^{-(t-2)\Delta_{L,h_{\infty}}}\circ\Pi^{\perp}_{h_\infty}(k_1),\qquad \tilde{\varphi}_{t-1}=e^{-(t-2)\Delta_{L,\tilde{h}_{\infty}}}\circ\Pi^{\perp}_{\tilde{h}_\infty}(k_1).
\end{align*}
	Again by Lemma \ref{lem : difference of heat flows} (i), we estimate at first
\begin{align*}
	\left\|\varphi_t-\tilde{\varphi}_t\right\|_{W^{2,r}}
		&\leq C\left(\left\|\varphi_{t-1}-\tilde{\varphi}_{t-1}\right\|_{W^{2,r}}+ \left\|\tilde{\varphi}_{t-1}\right\|_{W^{2,r}}\sup_{s\in [t-1,t]}\left(\left\|g-\tilde{g}\right\|_{W^{1,\infty}}+\left\|h-\tilde{h}\right\|_{W^{4,r}}\right) \right).
\end{align*}	
	For the $W^{2,q}$-norm, we proceed a little bit differently. In this case, we obtain from Lemma \ref{lem : difference of heat flows} (i), (iv) and (v) that
\begin{align*}
	\left\|\varphi_t-\tilde{\varphi}_t\right\|_{L^q}
		&\leq C\left(\left\|\varphi_{t-1}-\tilde{\varphi}_{t-1}\right\|_{L^{q}}+ \left\|\tilde{\varphi}_{t-1}\right\|_{L^{q}}\sup_{s\in [t-1,t]}\left(\left\|g-\tilde{g}\right\|_{W^{1,\infty}}+\left\|h-\tilde{h}\right\|_{W^{2,\infty}}\right)\right),\\
	\left\|\nabla(\varphi_t-\tilde{\varphi}_t)\right\|_{L^q}
		&\leq C\left( \left\|\nabla(\varphi_{t-1}-\tilde{\varphi}_{t-1})\right\|_{L^{q}}+ \left\|\varphi_{t-1}-\tilde{\varphi}_{t-1}\right\|_{L^{r}} \right) \\
		&\qquad + C\left( \left\|\nabla\tilde{\varphi}_{t-1}\right\|_{L^{q}} + \left\|\tilde{\varphi}_{t-1}\right\|_{L^{r}} \right)
	\sup_{s\in [t-1,t]}\left(\left\|g-\tilde{g}\right\|_{W^{1,\infty}} +\left\|h-\tilde{h}\right\|_{W^{3,\infty}} \right),\\
	\left\|\nabla^2(\varphi_t-\tilde{\varphi}_t)\right\|_{L^q}
		&\leq C\left( \left\|\nabla^2(\varphi_{t-1}-\tilde{\varphi}_{t-1})\right\|_{L^{q}}+ \left\|\nabla(\varphi_{t-1}-\tilde{\varphi}_{t-1})\right\|_{L^{r}}+
\left\|\varphi_{t-1}-\tilde{\varphi}_{t-1}\right\|_{L^{r}} \right) \\
		&\qquad +C\left( \left\|\nabla^2\tilde{\varphi}_{t-1}\right\|_{L^{q}}+ 
\left\|\nabla\tilde{\varphi}_{t-1}\right\|_{L^{r}}+
\left\|\tilde{\varphi}_{t-1}\right\|_{L^{r}} \right) \\
		&\qquad\qquad \cdot\sup_{s\in [t-1,t]} \left( \left\|g-\tilde{g}\right\|_{W^{1,\infty}} + \left\|h-\tilde{h}\right\|_{W^{4,\infty}} \right).
\end{align*}
By using $g=h+k$, $\tilde{g}=\tilde{h}+\tilde{k}$ and Lemma \ref{lem : elliptic regularity}, we find
\begin{align*}
	\sup_{s\in [t-1,t]} \left( \left\|g-\tilde{g}\right\|_{W^{1,\infty}}
+\left\|h-\tilde{h}\right\|_{W^{4,\infty}} \right)
		\leq C \left(\left\|k-\tilde{k}\right\|_{X_{q,r}}+\left\|h-\tilde{h}\right\|_{Z_{q,r}}\right).
\end{align*}
Putting these estimates together and multiplying by $t^{\frac{n}{2}\left( \frac{1}{p'}-\frac{1}{q} \right)}$, we get
\begin{align*}
	&\left\| t^{\frac{n}{2}\left(\frac{1}{p'}-\frac{1}{q}\right)}(\varphi_t-\tilde{\varphi}_t)\right\|_{X_{q,r}}
		\leq  C\left\|t^{\frac{n}{2}\left(\frac{1}{p'}-\frac{1}{q}\right)}(\varphi_{t-1}-\tilde{\varphi}_{t-1})\right\|_{X_{q,r}}\\
		&\qquad\qquad\qquad\qquad\qquad\qquad\qquad+ C\left\|t^{\frac{n}{2}\left( \frac{1}{p'}-\frac{1}{q} \right) }\tilde{\varphi}_{t-1}\right\|_{X_{q,r}}\left( \left\|\tilde{h}-h\right\|_{Z_{q,r}}+ \left\|\tilde{k}-k\right\|_{X_{q,r}}\right)\\
		&\leq  C\left\{\left\| t^{\frac{n}{2}\left(\frac{1}{p'}-\frac{1}{q}\right)}(\varphi_{t-1}-\tilde{\varphi}_{t-1})\right\|_{X_{q,r}}+ \left\|k_1\right\|_{L^{p'}}\left( \left\|\tilde{h}-h\right\|_{Z_{q,r}}+ \left\|\tilde{k}-k\right\|_{X_{q,r}}\right)\right\}.
\end{align*}
Note that we used Lemma \ref{lem: Linear evolution and Xnorm} in the last inequality.
From now on we abbreviate for notational convenience
\begin{align*}
\Delta_L:=\Delta_{L,h_{\infty}},\qquad \tilde{\Delta}_L:=\Delta_{L,\tilde{h}_\infty},\qquad \Pi^{\perp}:=\Pi^{\perp}_{h_{\infty}},\qquad \tilde{\Pi}^{\perp}:=\Pi^{\perp}_{\tilde{h}_\infty}.
\end{align*}
We can rewrite the difference  $\varphi_{t-1}-\tilde{\varphi}_{t-1}$ as
\begin{align*}
\varphi_{t-1}-\tilde{\varphi}_{t-1}&=\varphi_{t-1}-\tilde{\Pi}^{\perp}(\varphi_{t-1})+\tilde{\Pi}^{\perp}(\varphi_{t-1}-\tilde{\varphi}_{t-1})\\&=(\Pi^{\perp}-\tilde{\Pi}^{\perp})(\varphi_{t-1})+(\Pi^{\perp}_{\tilde{h}_{\infty},h_{\infty}})^{-1}\circ\Pi^{\perp}(\varphi_{t-1}-\tilde{\varphi}_{t-1}),
\end{align*}
where we used that $\varphi_{t-1}=\Pi^{\perp}(\varphi_{t-1})$ and $\tilde{\varphi}_{t-1}=\tilde{\Pi}^{\perp}(\tilde{\varphi}_{t-1})$. Let us use the notation
\begin{align*}
 \psi_{t-1}=\Pi^{\perp}(\varphi_{t-1}-\tilde{\varphi}_{t-1}).
 \end{align*}
From Lemma \ref{lem : projection lemma} (iv) and (v), we get the estimates
\begin{align*}
\left\|\varphi_{t-1}-\tilde{\varphi}_{t-1}\right\|_{W^{2,r}}
&\leq \left\|(\Pi^{\perp}-\tilde{\Pi}^{\perp})(\varphi_{t-1})\right\|_{W^{2,r}}+\left\|(\Pi^{\perp}_{\tilde{h}_{\infty},h_{\infty}})^{-1}(  \psi_{t-1})\right\|_{W^{2,r}}\\
&\leq C \left(\left\|\tilde{h}-h\right\|_{Z_{q,r}}\left\|\varphi_{t-1}\right\|_{L^r}+\left\|\psi_{t-1}\right\|_{W^{2,r}}\right),
\end{align*}
and
\begin{align*}
\left\|\varphi_{t-1}-\tilde{\varphi}_{t-1}\right\|_{L^q}&\leq C \left(\left\|\tilde{h}-h\right\|_{Z_{q,r}}\left\|\varphi_{t-1}\right\|_{L^r}+\left\|\psi_{t-1}\right\|_{L^q}\right),\\
\left\|\nabla(\varphi_{t-1}-\tilde{\varphi}_{t-1})\right\|_{L^q}&\leq C \left(\left\|\tilde{h}-h\right\|_{Z_{q,r}}\left\|\varphi_{t-1}\right\|_{L^r}+\left\|\nabla\psi_{t-1}\right\|_{L^q}+\left\|\psi_{t-1}\right\|_{L^r}\right),\\
\left\|\nabla^2(\varphi_{t-1}-\tilde{\varphi}_{t-1})\right\|_{L^q}&\leq C \left(\left\|\tilde{h}-h\right\|_{Z_{q,r}}\left\|\varphi_{t-1}\right\|_{L^r}+\left\|\nabla^2\psi_{t-1}\right\|_{L^q}+\left\|\psi_{t-1}\right\|_{W^{1,r}}\right).
\end{align*}
By Corollary \ref{cor : linear estimates}, $\left\|\varphi_{t-1}\right\|_{L^r}\leq Ct^{-\frac{n}{2}\left(\frac{1}{p'}-\frac{1}{r}\right)} \left\|k_1\right\|_{L^{p'}}$ and we conclude
\begin{align*}
	\left\|t^{\frac{n}{2}\left(\frac{1}{p'}-\frac{1}{q}\right)}(\varphi_{t-1}-\tilde{\varphi}_{t-1})\right\|_{X_{q,r}}
		\leq C \left(\left\|k_1\right\|_{L^{p'}}\left\|\tilde{h}-h\right\|_{Z_{q,r}}+\left\|t^{\frac{n}{2}\left(\frac{1}{p'}-\frac{1}{q}\right)}\psi_{t-1}\right\|_{X_{q,r}}\right).
\end{align*}
Thus to finish the proof, it suffices to show
\begin{align*}
	\left\|t^{\frac{n}{2}\left(\frac{1}{p'}-\frac{1}{q}\right)}\psi_{t-1}\right\|_{X_{q,r}}\leq C \left\|k_1\right\|_{L^{p'}}\left\|\tilde{h}-h\right\|_{Z_{q,r}}.
\end{align*}
We are going to establish this estimate for the remainder of this proof.
Because for $s\in [1,t-1]$, $\varphi_s$ and $\tilde{\varphi}_s$ are solutions of the evolution problems
\begin{align*}
\partial_s\varphi_s+\Delta_L\varphi_s&=0,\qquad \varphi_1=\Pi^{\perp}(k_1),\\
\partial_s\tilde{\varphi}_s+\tilde{\Delta}_L\tilde{\varphi}_s&=0,\qquad \tilde{\varphi}_1=\tilde{\Pi}^{\perp}(k_1),
\end{align*}
the quantity $\psi_s:=\Pi^{\perp}(\varphi_{s}-\tilde{\varphi}_{s})$ is a solution of the problem
\begin{align*}
\partial_s\psi_s+\Delta_L\psi_s=\Pi^{\perp}\circ(\tilde{\Delta}_L-\Delta_L)(\tilde{\varphi}_s),\qquad
\psi_1=\Pi^{\perp}\circ(\Pi^{\perp}-\tilde{\Pi}^{\perp})(k_1),
\end{align*}
which is then written by the Duhamel principle as
\begin{align*}
\psi_{t-1}&=e^{-(t-2)\Delta_L}\psi_1+\int_1^{t-1}e^{-(t-1-s)\Delta_L}\circ\Pi^{\perp}(\tilde{\Delta}_L-\Delta_L)(\tilde{\varphi})ds\\
&=e^{-(t-2)\Delta_L}\psi_1+\int_1^{2}e^{-(t-1-s)\Delta_L}\circ\Pi^{\perp}(\tilde{\Delta}_L-\Delta_L)(\tilde{\varphi})ds\\&\qquad
+
\int_{t-2}^{t-1}e^{-(t-1-s)\Delta_L}\circ\Pi^{\perp}(\tilde{\Delta}_L-\Delta_L)(\tilde{\varphi})ds\\
&\qquad+\int_2^{t-2}e^{-(t-1-s)\Delta_L}\circ\Pi^{\perp}(\tilde{\Delta}_L-\Delta_L)(\tilde{\varphi})ds.
\end{align*}
We are now going to estimate these terms separately.
At first, we obtain
\begin{align*}
	t^{\frac{n}{2}\left(\frac{1}{p'}-\frac{1}{q}\right)}
		&\left(\left\| e^{-(t-2)\Delta_L}\psi_1\right\|_{L^q}+
t^{\frac{1}{2}}\left\|\nabla e^{-(t-2)\Delta_L}\psi_1\right\|_{L^q}+t^{\frac{n}{2}\left(\frac{1}{q}-\frac{1}{r}\right)}
\left\|\nabla^2 e^{-(t-2)\Delta_L}\psi_1\right\|_{L^q}\right) \\
		&\leq C(t-2)^{\frac{n}{2}\left(\frac{1}{p'}-\frac{1}{q}\right)} \Big( \left\| e^{-(t-2)\Delta_L}\psi_1\right\|_{L^q}+
(t-2)^{\frac{1}{2}}\left\|\nabla e^{-(t-2)\Delta_L}\psi_1\right\|_{L^q}\\
&\qquad+C(t-2)^{\frac{n}{2}\left(\frac{1}{p'}-\frac{1}{q}\right)}(t-2)^{\frac{n}{2}\left(\frac{1}{q}-\frac{1}{r}\right)}
\left\|\nabla^2 e^{-(t-2)\Delta_L}\psi_1\right\|_{L^q} \Big) \\
&\leq C\left\|\psi_1\right\|_{L^{p'}}\\
&\leq C\left\|\tilde{h}-h\right\|_{Z_{q,r}}\left\|k_1\right\|_{L^{p'}}
\end{align*}
and similarly,
\begin{align*}
	t^{\frac{n}{2}\left(\frac{1}{p'}-\frac{1}{q}\right)} \cdot t^{\frac{n}{2}\left(\frac{1}{p'}-\frac{1}{r}\right)}\left\| e^{-(t-2)\Delta_L}\psi_1\right\|_{W^{2,r}}
		&\leq C(t-2)^{\frac{n}{2}\left(\frac{1}{p'}-\frac{1}{r}\right)}\left\| e^{-(t-2)\Delta_L}\psi_1\right\|_{W^{2,r}}\\
		&\leq C\left\|\psi_1\right\|_{L^{p'}}\\
		&		\leq C\left\|\tilde{h}-h\right\|_{Z_{q,r}}\left\|k_1\right\|_{L^{p'}}
\end{align*}
by Corollary \ref{cor : linear estimates}.
For the next term
\begin{align*}
\int_1^{2}e^{-(t-1-s)\Delta_L}\circ\Pi^{\perp}(\tilde{\Delta}_L-\Delta_L)(\tilde{\varphi})ds=e^{-(t-3)\Delta_L}\left(\int_1^{2}e^{-(2-s)\Delta_L}\circ\Pi^{\perp}(\tilde{\Delta}_L-\Delta_L)(\tilde{\varphi})ds\right),
\end{align*}
we need to use Sobolev spaces of negative order to get rid of the second derivatives of $\tilde{\varphi}$. Let us first abbreviate
\begin{align*}
\eta_2&=\int_1^{2}e^{-(2-s)\Delta_L}\circ\Pi^{\perp}(\tilde{\Delta}_L-\Delta_L)(\tilde{\varphi})ds,\\ \eta_3&=e^{-\Delta_L}\eta_2,\qquad \eta_{t-1}=e^{-(t-4)\Delta_L}\eta_3=e^{-(t-3)\Delta_L}\eta_2.
\end{align*}
Then $\eta_{t-1}$ is the term we wish to estimate. Similarly as above, we have
\begin{align*}
	&t^{\frac{n}{2}\left(\frac{1}{p'}-\frac{1}{q}\right)}\left(\left\|\eta_{t-1}\right\|_{L^q} + t^{\frac{1}{2}}\left\|\nabla\eta_{t-1}\right\|_{L^q}+t^{\frac{n}{2}\left(\frac{1}{q}-\frac{1}{r} \right)} \left\|\nabla^2\eta_{t-1}\right\|_{L^q} \right) \\
	&\qquad\qquad\leq C(t-4)^{\frac{n}{2}\left(\frac{1}{p'}-\frac{1}{q}\right)}\left(\left\|\eta_{t-1}\right\|_{L^q} + (t-4)^{\frac{1}{2}}\left\|\nabla\eta_{t-1}\right\|_{L^q}+(t-4)^{\frac{n}{2}\left(\frac{1}{q}-\frac{1}{r}\right)} \left\|\nabla^2\eta_{t-1}\right\|_{L^q}\right) \\
  	&\qquad\qquad\leq C\left\|\eta_3\right\|_{L^{p'}}
\end{align*}
and similarly,
\begin{align*}
	&t^{\frac{n}{2}\left(\frac{1}{p'}-\frac{1}{q}\right)}\cdot t^{\frac{n}{2}\left(\frac{1}{q}-\frac{1}{r}\right)}\left\| \eta_{t-1}\right\|_{W^{2,r}}
		\leq C\left\|\eta_3\right\|_{L^{p'}}
\end{align*}
by Corollary \ref{cor : linear estimates}. Now because $e^{-\Delta_L}$ extends to bounded maps
\begin{align*}
e^{-\Delta_L}:L^{q'}&\to W^{2,q'},\qquad \forall q'\in (1,\infty),\\
e^{-s\Delta_L}:W^{2,q'}&\to W^{2,q'},\qquad  \forall q'\in (1,\infty), s\in[0,1],
\end{align*}
duality implies that it is also a bounded map
\begin{align*}
e^{-\Delta_L}: W^{-2,q'}&\to L^{q'},\qquad \forall q'\in (1,\infty),\\
e^{-s\Delta_L}:W^{-2,q'}&\to W^{-2,q'},\qquad  \forall q'\in (1,\infty), s\in[0,1].
\end{align*}
Because $\Pi^{\perp}:W^{2,q'}\to W^{2,q'}$ is bounded and self-adjoint on $L^2$, it also admits a bounded extension $\Pi^{\perp}:W^{-2,q'}\to W^{-2,q'}$.
These observations imply
\begin{align*}
\left\|\eta_3\right\|_{L^q}\leq C\left\|\eta_2\right\|_{W^{-2,q}}&=C\left\|\int_1^{2}e^{-(2-s)\Delta_L}\circ\Pi^{\perp}(\tilde{\Delta}_L-\Delta_L)(\tilde{\varphi})ds\right\|_{W^{-2,q}}
\\&
\leq C\sup_{s\in [1,2]}\left\| (\tilde{\Delta}_L-\Delta_L)(\tilde{\varphi})\right\|_{W^{-2,q}}.
\end{align*}
Using the tensor $T_{ij}^k=\Gamma_{ij}^k-\tilde{\Gamma}_{ij}^k$, the difference of two Lichnerowicz Laplacians can be written as
\begin{align}\label{eq : difference LL}
(\tilde{\Delta}_L-\Delta_L)(\tilde{\varphi})&=(h_{\infty}^{-1}-\tilde{h}_{\infty}^{-1})*(\nabla^2\tilde{\varphi}+R*\tilde{\varphi})+\tilde{h}_{\infty}^{-1}*(\nabla T*\tilde{\varphi}+T*\nabla\tilde{\varphi}+T*T*\tilde{\varphi}).
\end{align}
Let now $\chi\in C^{\infty}_{cs}(S^2M)$ be a compactly supported test tensor.
Since we have the schematic form $T=\tilde{h}_{\infty}^{-1}*\nabla(h_{\infty}-\tilde{h}_{\infty})$, suitable integration by parts yields
\begin{align*}
((\tilde{\Delta}_L-\Delta_L)\tilde{\varphi},\chi)_{L^2}&
=((h^{-1}-\tilde{h}_{\infty}^{-1})*(\nabla^2\tilde{\varphi}+R*\tilde{\varphi}),\chi)_{L^2}\\
&\qquad+ (\tilde{h}_{\infty}^{-1}*(\nabla T*\tilde{\varphi}+T*\nabla\tilde{\varphi}+T*T*\tilde{\varphi}),\chi)_{L^2}\\
&\leq C\left\|\tilde{\varphi}\right\|_{L^{q}}\left\|\chi\right\|_{W^{2,q^*}}\left\|\tilde{h}_{\infty}-h_{\infty}\right\|_{W^{2,\infty}}.
\end{align*}
Here, $q^*$ is the conjugate H\"{o}lder exponent of $q$. Using the definition of negative Sobolev spaces and Lemma \ref{lem : elliptic regularity}, we obtain
\begin{align*}
\left\| (\tilde{\Delta}_L-\Delta_L)(\tilde{\varphi})\right\|_{W^{-2,q}}\leq C\left\|\tilde{\varphi}\right\|_{L^{q}}\left\|\tilde{h}_{\infty}-h_{\infty}\right\|_{W^{2,\infty}}\leq C\left\|\tilde{\varphi}\right\|_{L^{q}}\left\|\tilde{h}-h\right\|_{Z_{q,r}},
\end{align*} 
so that
\begin{align*}
\left\|\eta_3\right\|_{L^q}\leq C\sup_{s\in [1,2]}\left\| (\tilde{\Delta}_L-\Delta_L)(\tilde{\varphi})\right\|_{W^{-2,q}}\leq C\left\|\tilde{h}-h\right\|_{Z_{q,r}}\sup_{s\in [1,2]}\left\|\tilde{\varphi}\right\|_{L^{q}}\leq  C\left\|\tilde{h}-h\right\|_{Z_{q,r}}\left\| k_1\right\|_{L^q}.
\end{align*}
Consequently,
\begin{align*}
\left\|\int_1^{2}e^{-(t-1-s)\Delta_L}\circ\Pi^{\perp}(\tilde{\Delta}_L-\Delta_L)(\tilde{\varphi})ds\right\|_{X_{q,r}}\leq  C\left\|\tilde{h}-h\right\|_{Z_{q,r}}\left\| k_1\right\|_{L^q}.
\end{align*}
Let $q'\in\left\{p,r\right\}$. Then,
\begin{align*}
	\left\|\int_{t-2}^{t-1}e^{-(t-1-s)\Delta_L}\circ\Pi^{\perp}(\tilde{\Delta}_L-\Delta_L)(\tilde{\varphi})ds\right\|_{W^{2,q'}}
		&\leq C\sup_{s\in [t-2,t-1]}\left\|(\tilde{\Delta}_L-\Delta_L)(\tilde{\varphi}_s)\right\|_{W^{2,q'}}\\
		&\leq C \left\|h_{\infty}-\tilde{h}_{\infty}\right\|_{W^{4,q'}}\sup_{s\in [t-2,t-1]}\left\|\tilde{\varphi}_s\right\|_{W^{4,\infty}}\\
		&\leq C\left\|h-\tilde{h}\right\|_{Z_{q,r}}\left\|\tilde{\varphi}_{t-3}\right\|_{L^{r}}\\
		&\leq C(t-3)^{-\frac{n}{2}\left(\frac{1}{p'}-\frac{1}{r}\right)}\left\|h-\tilde{h}\right\|_{Z_{q,r}}\left\|k_1\right\|_{L^{p'}}\\
		&\leq Ct^{-\frac{n}{2}\left(\frac{1}{p'}-\frac{1}{r}\right)}\left\|h-\tilde{h}\right\|_{Z_{q,r}}\left\|k_1\right\|_{L^{p'}}.
\end{align*}
Now we estimate the final term. 
Let $\alpha:\left\{q,r\right\}\times \left\{0,1,2\right\}\to\R$ be the function from Lemma \ref{lem : linear part}.
Choose $p''\in  (1,p')$ small and let $q''\in (p'',\infty)$ such that $\frac{1}{p''}=\frac{1}{q''}+\frac{1}{r}$. 
Under these assumptions,
\begin{align*}
	1
		\neq \frac{n}{2}\left(\frac{1}{p''}-\frac{1}{q}\right)+\alpha(q',i)
		&> \frac{n}{2}\left(\frac{1}{p'}-\frac{1}{q}\right)+\alpha(q',i), \\
	1
		\neq \frac{n}{2}\left(\frac{1}{p'}-\frac{1}{r}\right)
		= \frac{n}{2}\left(\frac{1}{p'}-\frac{1}{q}\right)+ \frac{n}{2}\left(\frac{1}{q}-\frac{1}{r}\right) 
		&\geq\frac{n}{2}\left(\frac{1}{p'}-\frac{1}{q}\right)+\alpha(q',i), \\
	\frac{n}{2}\left(\frac{1}{p''}-\frac{1}{q}\right)+\alpha(q',i)+\frac{n}{2}\left(\frac{1}{p'}-\frac{1}{r}\right)-1
		&>\frac{n}{2}(\frac{1}{p'}-\frac{1}{q})+\alpha(q',i).
\end{align*}
By the H\"{o}lder inequality, Corollary \ref{cor : linear estimates} and Lemma \ref{important_technical_lemma}, we therefore get
\begin{align*}
	&\left\|\nabla^i\int_2^{t-2}e^{-(t-1-s)\Delta_L}\circ\Pi^{\perp}(\tilde{\Delta}_L-\Delta_L)(\tilde{\varphi})ds\right\|_{L^{q'}} \\
 		&\qquad\qquad\leq \int_2^{t-2} \left\|\nabla^i e^{-(t-1-s)\Delta_L}\circ\Pi^{\perp}\right\|_{L^{p''}\to L^{q'}}\left\|\tilde{h}_{\infty}-h_{\infty}\right\|_{W^{2,q''}}\left\|\tilde{\varphi}\right\|_{W^{2,r}}ds\\
		&\qquad\qquad\leq C\int_2^{t-2}(t-1-s)^{-\frac{n}{2}\left(\frac{1}{p''}-\frac{1}{q}\right) -\alpha(q',i)}(s-1)^{-\frac{n}{2}\left(\frac{1}{p'}-\frac{1}{r}\right)}ds\cdot \left\|\tilde{h}-h\right\|_{Z_{q,r}}\left\|k_1\right\|_{L^{p'}}\\
		&\qquad\qquad\leq C t^{-\frac{n}{2}\left(\frac{1}{p'}-\frac{1}{q}\right)-\alpha(q',i)}\cdot\left\|\tilde{h}-h\right\|_{Z_{q,r}}\left\|k_1\right\|_{L^{p'}},
\end{align*}
which finishes the proof.
\end{proof}
\begin{rem}\label{rem : difference heat flows long-time estimate}
By shifting the time parameter, we also get under the conditions  of Lemma \ref{lem : difference term 1}, that there exists for $q'\in\left\{q,r\right\}$ a constant  $C=C(q',p',\U,\epsilon)$ such that for $t>s\geq1$  and $t-s\geq 4$, we have
	\begin{align*}
	&\left\|\nabla^i\circ P(g,h,h_{\infty})_{s\to t}\circ \Pi^{\perp}_{h_{\infty}}(k_1)-\nabla^i\circ P(\tilde{g},\tilde{h},\tilde{h}_{\infty})_{s\to t}\circ \Pi^{\perp}_{\tilde{h}_{\infty}}(k_1) \right\|_{L^{q'}}  \\ 
	&\qquad\qquad\leq C (t-s)^{-\frac{n}{2}\left(\frac{1}{p'}-\frac{1}{q}\right)-\alpha(q',i)} \left(\left\|\tilde{k}-k\right\|_{X_{q,r}}+\left\|\tilde{h}-h\right\|_{Z_{q,r}}\right)\left\|k_1\right\|_{L^{p'}},
\end{align*}
where $\alpha$ is the function of Lemma \ref{lem : linear part}.
	\end{rem}
\noindent	
Recall that for $t\geq 2$,
	\begin{align*}
	&\chi_{[1,\max\left\{t-1,1\right\}]}(s)[I(t,s,k,h,h_{\infty})-I(t,s,\tilde{k},\tilde{h},\tilde{h}_{\infty})] \\
		&\qquad =\chi_{[1,\max\left\{t-1,1\right\}]}(s)\Pi^{\perp}_{\infty}[(\Delta_{L,h_{\infty}}-\Delta_{L,h})(k)-(\Delta_{L,\tilde{h}_{\infty}}-\Delta_{L,\tilde{h}})(\tilde{k})]\\
		&\qquad \qquad +\chi_{[1,\max\left\{t-1,1\right\}]}(s)\Pi^{\perp}_{\infty}[(1-D_g\Phi)(H_1)-(1-D_{\tilde{g}}\Phi)(\tilde{H}_1)],
	\end{align*}
so that 
\begin{align*}
&\int_1^{\max\left\{t-1,1\right\}}  P(g,h,h_{\infty})_{s\to t}[I(t,s,k,h,h_{\infty})-I(t,s,\tilde{k},\tilde{h},\tilde{h}_{\infty})] ds\\
&=\int_1^{\max\left\{t-1,1\right\}}  P(g,h,h_{\infty})_{s\to t}\circ \Pi^{\perp}_{\infty}[(\Delta_{L,h_{\infty}}-\Delta_{L,h})(k)-(\Delta_{L,\tilde{h}_{\infty}}-\Delta_{L,\tilde{h}})(\tilde{k})] ds\\
&\qquad+\int_1^{\max\left\{t-1,1\right\}}  P(g,h,h_{\infty})_{s\to t}\circ\Pi^{\perp}_{\infty}[(1-D_g\Phi)(H_1)-(1-D_{\tilde{g}}\Phi)(\tilde{H}_1)] ds.
\end{align*}
We deal with these two terms in the next four lemmas.
\begin{lem}\label{lem : difference term 2.1}
We have for every $p'\in (1,r]$ a constant $C=C(q,r,p',\U,\epsilon)$ such that
\begin{align*}
	&\left\|(\Delta_{L,h_{\infty}}-\Delta_{L,h})(k)-(\Delta_{L,\tilde{h}_{\infty}}-\Delta_{L,\tilde{h}})(\tilde{k})\right\|_{L^{p'}}\\
	&\qquad \qquad \qquad \leq Cs^{-\frac{n}{2}\left(\frac{1}{q}-\frac{1}{r}\right)}\left( \left\|h-\hat{h}\right\|_{Z_{q,r}}\left\|k-\tilde{k}\right\|_{X_{q,r}}+\left\|h-\tilde{h}\right\|_{Z_{q,r}}\left\|\tilde{k}\right\|_{X_{q,r}}\right),
\end{align*}
\end{lem}
\begin{proof}
We first rewrite
\begin{align*}
(\Delta_{L,h_{\infty}}-\Delta_{L,h})(k)-(\Delta_{L,\tilde{h}_{\infty}}-\Delta_{L,\tilde{h}})(\tilde{k})&=
(\Delta_{L,h_{\infty}}-\Delta_{L,\tilde{h}_{\infty}})(\tilde{k})+(\Delta_{L,h}-\Delta_{L,\tilde{h}})(\tilde{k})\\
&\qquad+
(\Delta_{L,h_{\infty}}-\Delta_{L,h})(k-\tilde{k}).
\end{align*}
Pick $r'\in (1,\infty)$ so that $\frac{1}{p'}=\frac{1}{r'}+\frac{1}{r}$.
Then the H\"{o}lder inequality applied to \eqref{eq : difference LL} implies together with Lemma \ref{lem : elliptic regularity} and Lemma \ref{lem : inftycontrol} that
\begin{align*}
\left\|(\Delta_{L,h_{\infty}}-\Delta_{L,\tilde{h}_{\infty}})(\tilde{k})\right\|_{L^{p'}} + &\left\|(\Delta_{L,h}-\Delta_{L,\tilde{h}})(\tilde{k})\right\|_{L^{p'}}\\
&\leq C\left(\left\|h_{\infty}-\tilde{h}_{\infty}\right\|_{W^{2,r'}}+\left\|h-\tilde{h}\right\|_{W^{2,r'}}\right)
\left\|\tilde{k}\right\|_{W^{2,r}}\\
	&\leq C\left\|h-\tilde{h}\right\|_{Z_{q,r}}\left\|\tilde{k}\right\|_{W^{2,r}}\\
	&\leq C\cdot s^{-\frac{n}{2}\left(\frac{1}{q}-\frac{1}{r}\right)}\left\|h-\tilde{h}\right\|_{Z_{q,r}}\left\|\tilde{k}\right\|_{X_{q,r}},
\end{align*}
and it is shown as in \eqref{eq : h minus limit h} that
\begin{align*}
	\left\|(\Delta_{L,h_{\infty}}-\Delta_{L,h})(k-\tilde{k})\right\|_{L^{p'}}
		&\leq C\cdot s^{1-n\left(\frac{1}{q}-\frac{1}{r}\right)-\frac{n}{2}\left(\frac{1}{q}-\frac{1}{r}\right)}\left\|h-\hat{h}\right\|_{Z_{q,r}}\left\|k-\tilde{k}\right\|_{X_{q,r}}\\ 
		&\leq C\cdot s^{-\frac{n}{2}\left(\frac{1}{q}-\frac{1}{r}\right)}\left\|h-\hat{h}\right\|_{Z_{q,r}}\left\|k-\tilde{k}\right\|_{X_{q,r}},
\end{align*}
which yields the desired result.
\end{proof}
\begin{lem}\label{lem : difference term 2.2}
For $t\geq 2$, there exists a constant $C=C(q,r,\U,\epsilon)$ such that we have
\begin{align*}
&\left\|\int_1^{\max\left\{t-1,1\right\}}  P(g,h,h_{\infty})_{s\to t}\circ \Pi^{\perp}_{\infty}[(\Delta_{L,h_{\infty}}-\Delta_{L,h})(k)-(\Delta_{L,\tilde{h}_{\infty}}-\Delta_{L,\tilde{h}})(\tilde{k})] ds\right\|_{X_{q,r}}\\
&\qquad\qquad\leq C \left( \left\|k-\tilde{k}\right\|_{X_{q,r}}+\left\|h-\tilde{h}\right\|_{Z_{q,r}}\right)\left( \left\|h-\hat{h}\right\|_{Z_{q,r}}+\left\|\tilde{k}\right\|_{X_{q,r}}\right)
\end{align*}
\end{lem}
\begin{proof}
Let $\alpha:\left\{q,r\right\}\times \left\{0,1,2\right\}\to\R$ be the function from Lemma \ref{lem : linear part}, $q'\in\left\{q,r\right\}$ and choose $p'\in  (1,q)$ small. 
Then we have
\begin{align*}
	1
		\neq \frac{n}{2}\left(\frac{1}{p'}-\frac{1}{q}\right)+\alpha(q',i)&>\alpha(q',i),\\
1\neq\frac{n}{2}\left(\frac{1}{q}-\frac{1}{r}\right)
		&\geq \alpha(q',i),\\
\frac{n}{2}\left(\frac{1}{p'}-\frac{1}{q}\right)+\alpha(q',i)+\frac{n}{2}\left(\frac{1}{q}-\frac{1}{r}\right)
		&> \alpha(q',i),
\end{align*}
and Lemma \ref{lem : linear part}, Lemma \ref{lem : difference term 2.1} and Lemma \ref{important_technical_lemma} yield
	\begin{align*}
	&\left\|\nabla^{i}	\int_1^{\max\left\{t-1,1\right\}}  P(g,h,h_{\infty})_{s\to t}\circ \Pi^{\perp}_{\infty}[(\Delta_{L,h_{\infty}}-\Delta_{L,h})(k)-(\Delta_{L,\tilde{h}_{\infty}}-\Delta_{L,\tilde{h}})(\tilde{k})] ds\right\|_{L^{q'}}\\
		&\leq \int_1^{\max\left\{t-1,1\right\}}\left\|\nabla^{i}\circ P(g,h,h_{\infty})_{s\to t}\circ \Pi^{\perp}_{\infty}\right\|_{L^{p'},L^{q'}}\left\|
(\Delta_{L,h_{\infty}}-\Delta_{L,h})(k)-(\Delta_{L,\tilde{h}_{\infty}}-\Delta_{L,\tilde{h}})(\tilde{k})\right\|_{L^{p'}}ds\\
		&\leq\int_1^{\max\left\{t-1,1\right\}}(t-s)^{-\frac{n}{2}\left(\frac{1}{p'}-\frac{1}{q}\right)-\alpha(q',i)}s^{-\frac{n}{2}\left(\frac{1}{q}-\frac{1}{r}\right)}ds\left( \left\|h-\hat{h}\right\|_{Z_{q,r}}\left\|k-\tilde{k}\right\|_{X_{q,r}}+\left\|h-\tilde{h}\right\|_{Z_{q,r}}\left\|\tilde{k}\right\|_{X_{q,r}}\right) \\
		&\leq C t^{-\alpha(q',i)}\left( \left\|k-\tilde{k}\right\|_{X_{q,r}}+\left\|h-\tilde{h}\right\|_{Z_{q,r}}\right)\left( \left\|h-\hat{h}\right\|_{Z_{q,r}}+\left\|\tilde{k}\right\|_{X_{q,r}}\right),
	\end{align*}
as desired.
\end{proof}
	\begin{lem}\label{lem : difference term 2.3}
There exists a constant $C=C(q,r,\U,\epsilon)$ such that we have
	\begin{align*}
	\left\|(1-D_g\Phi)(H_1)-(1-D_{\tilde{g}}\Phi)(\tilde{H}_1)\right\|_{L^q}
		&\leq Cs^{-n\left(\frac{1}{q}-\frac{1}{r}\right)} \left(\left\|h-\tilde{h}\right\|_{Z_{q,r}}+\left\|k-\tilde{k}\right\|_{X_{q,r}}\right)\\
		&\qquad\cdot \left(\left\|k\right\|_{X_{q,r}}+\left\|\tilde{k}\right\|_{X_{q,r}}\right)
	\end{align*}
	\end{lem}
	\begin{proof}
We first write 
\begin{align*}
(1-D_g\Phi)(H_1)-(1-D_{\tilde{g}}\Phi)(\tilde{H}_1)=H_1-\tilde{H}_1+(D_g\Phi-D_{\tilde{g}}\Phi)(H_1)+D_{\tilde{g}}\Phi(\tilde{H}_1-H_1)
\end{align*}	
By Lemma \ref{lem : H_1-tildeH_1}, we already know
\begin{align*}
\left\| H_1-\tilde{H}_1\right\|_{L^q}
	\leq Cs^{-n\left(\frac{1}{q}-\frac{1}{r}\right)} \left(\left\|h-\tilde{h}\right\|_{Z_{q,r}}+\left\|k-\tilde{k}\right\|_{X_{q,r}}\right) \left(\left\|k\right\|_{X_{q,r}}+\left\|\tilde{k}\right\|_{X_{q,r}}\right).
\end{align*}
Moreover by Lemma \ref{lem : projection lemma} (iii),
\begin{align*}
\left\|D_{\tilde{g}}\Phi(\tilde{H}_1-H_1)\right\|_{L^q}
	&\leq \left\|\tilde{H}_1-H_1\right\|_{L^q} \\
	&\leq  Cs^{-n\left(\frac{1}{q}-\frac{1}{r}\right)}
\left(\left\|h-\tilde{h}\right\|_{Z_{q,r}}+\left\|k-\tilde{k}\right\|_{X_{q,r}}\right)
\left(\left\|k\right\|_{X_{q,r}}+\left\|\tilde{k}\right\|_{X_{q,r}}\right).
\end{align*}
Finally by Lemma \ref{lem : projection lemma} (vi) and the definition of $H_1$ (see \eqref{eq : RdT1}),
\begin{align*}
	\left\|(D_g\Phi-D_{\tilde{g}}\Phi)(H_1)\right\|_{L^q}
		&\leq C\left(\left\|h-\tilde{h}\right\|_{L^q}+\left\|k-\tilde{k}\right\|_{L^q} \right) \left\|H_1\right\|_{L^{\frac{r}{2}}}\\
		&\leq C\left( \left\|h-\tilde{h}\right\|_{L^q}+\left\|k-\tilde{k}\right\|_{L^q} \right) \left\|k\right\|_{W^{2,r}}^2\\
		&\leq Cs^{-n \left( \frac{1}{q}-\frac{1}{r} \right) }\left( \left\|h-\tilde{h}\right\|_{Z_{q,r}}+\left\|k-\tilde{k}\right\|_{X_{q,r}} \right) \left\|k\right\|_{X_{q,r}}^2.
\end{align*}
This finishes the proof of the lemma.
	\end{proof}
\begin{lem}\label{lem : difference term 2.4}
For $t\geq 2$, there exists a constant $C=C(q,r,\U,\epsilon)$ such that we have
\begin{align*}
&\left\|\int_{1}^{\max\left\{t-1,1\right\}}P(g,h,h_{\infty})_{s\to t}\circ\Pi^{\perp}_{\infty}	[(1-D_g\Phi)(H_1)-(1-D_{\tilde{g}}\Phi)(\tilde{H}_1)]ds \right\|_{X_{q,r}}\\
&\qquad\qquad\qquad\leq C \left(\left\|h-\tilde{h}\right\|_{Z_{q,r}}+\left\|k-\tilde{k}\right\|_{X_{q,r}}\right)
	\left(\left\|k\right\|_{X_{q,r}}+\left\|\tilde{k}\right\|_{X_{q,r}}\right)
\end{align*}
\end{lem}	
\begin{proof}
Let $\alpha:\left\{q,r\right\}\times \left\{0,1,2\right\}\to\R$ be the function from Lemma \ref{lem : linear part} and $q'\in\left\{q,r\right\}$. Then we get by Lemma \ref{lem : linear part} and Lemma \ref{lem : difference term 2.3}   that
\begin{align*}
	&\left\|\nabla^i\int_{1}^{\max\left\{t-1,1\right\}}P(g,h,h_{\infty})_{s\to t}\circ\Pi^{\perp}_{\infty}	[(1-D_g\Phi)(H_1)-(1-D_{\tilde{g}}\Phi)(\tilde{H}_1)]ds \right\|_{L^{q'}}\\
	&\qquad\leq \int_1^{\max\left\{t-1,1\right\}}\left\|\nabla^i\circ P(g,h,h_{\infty})_{s\to t}\circ\Pi^{\perp}_{\infty}\right\|_{L^q,L^{q'}}\left\|(1-D_g\Phi)(H_1)-(1-D_{\tilde{g}}\Phi)(\tilde{H}_1)\right\|_{L^q}ds\\
	&\qquad\leq C\int_1^{\max\left\{t-1,1\right\}}(t-s)^{-\alpha(q',i)}s^{-n\left(\frac{1}{q}-\frac{1}{r}\right)}ds\left[\left\|h-\tilde{h}\right\|_{Z_{q,r}}+\left\|k-\tilde{k}\right\|_{X_{q,r}}\right]
	\left[\left\|k\right\|_{X_{q,r}}+\left\|\tilde{k}\right\|_{X_{q,r}}\right]\\
	&\qquad\leq  C\cdot t^{-\alpha(q',i)}\left(\left\|h-\tilde{h}\right\|_{Z_{q,r}}+\left\|k-\tilde{k}\right\|_{X_{q,r}}\right)
	\left(\left\|k\right\|_{X_{q,r}}+\left\|\tilde{k}\right\|_{X_{q,r}}\right).
\end{align*}
The last inequality here follows from Lemma \ref{important_technical_lemma} and
\begin{align*}
	n\left(\frac{1}{q}-\frac{1}{r}\right)>1,
		\qquad n \left(\frac{1}{q}-\frac{1}{r}\right)
		>\frac{n}{2}\left(\frac{1}{q}-\frac{1}{r}\right)
		\geq \alpha(q',i).
\end{align*}
The result is immediate from the definition of the norm.
\end{proof}
\begin{lem}\label{lem : difference term 3}
For $t\geq 5$, there exists a constant $C=C(q,r,\U,\epsilon)$ such that we have
\begin{align*}
&\left\| \int_{1}^{\max\left\{t-4,1\right\}}  [P(g,h,h_{\infty})_{s\to t}- P(\tilde{g},\tilde{h},\tilde{h}_{\infty})_{s\to t}]I(t,s,\tilde{k},\tilde{h},\tilde{h}_{\infty}) ds\right\|_{X_{q,r}}\\
&\qquad\qquad\leq
C \left(\left\|k-\tilde{k}\right\|_{X_{q,r}}+\left\|h-\tilde{h}\right\|_{Z_{q,r}}\right)\left\|\tilde{k}\right\|_{X_{q,r}}
\end{align*}
\end{lem}
\begin{proof}
We split up
		\begin{align*}
 \int_{1}^{\max\left\{t-4,1\right\}}  &[P(g,h,h_{\infty})_{s\to t}- P(\tilde{g},\tilde{h},\tilde{h}_{\infty})_{s\to t}]I(t,s,\tilde{k},\tilde{h},\tilde{h}_{\infty}) ds\\
	&= \int_{1}^{\max\left\{t-4,1\right\}}  [P(g,h,h_{\infty})_{s\to t}\circ \Pi^{\perp}_{\infty}- P(\tilde{g},\tilde{h},\tilde{h}_{\infty})_{s\to t}\circ \tilde{\Pi}^{\perp}_{\infty}][(\Delta_{L,\tilde{h}_{\infty}}-\Delta_{L,\tilde{h}})(\tilde{k})] ds\\
	&\qquad +\int_{1}^{\max\left\{t-4,1\right\}}  [P(g,h,h_{\infty})_{s\to t}\circ \Pi^{\perp}_{\infty}- P(\tilde{g},\tilde{h},\tilde{h}_{\infty})_{s\to t}\circ \tilde{\Pi}^{\perp}_{\infty}][(1-D_{\tilde{g}}\Phi)(\tilde{H}_1)] ds
	\end{align*}
 Let us now estimate the term
\begin{align*}
\int_{1}^{\max\left\{t-4,1\right\}}  [P(g,h,h_{\infty})_{s\to t}\circ \Pi^{\perp}_{\infty}- P(\tilde{g},\tilde{h},\tilde{h}_{\infty})_{s\to t}\circ \tilde{\Pi}^{\perp}_{\infty}][(\Delta_{L,\tilde{h}_{\infty}}-\Delta_{L,\tilde{h}})(\tilde{k})] ds
\end{align*}
Let $\alpha:\left\{q,r\right\}\times \left\{0,1,2\right\}\to\R$ be the function from Lemma \ref{lem : linear part}, $q'\in\left\{q,r\right\}$ and $p'\in (1,q)$ small. 
Then we can use Lemma \ref{lem : nonlinear part 0.1}, Lemma \ref{lem : difference term 1} and Remark \ref{rem : difference heat flows long-time estimate} to obtain
\begin{align*}
	&\left\|\nabla^i\int_{1}^{\max\left\{t-4,1\right\}}  [P(g,h,h_{\infty})_{s\to t}\circ \Pi^{\perp}_{\infty}- P(\tilde{g},\tilde{h},\tilde{h}_{\infty})_{s\to t}\circ \tilde{\Pi}^{\perp}_{\infty}][(\Delta_{L,\tilde{h}_{\infty}}-\Delta_{L,\tilde{h}})(\tilde{k})] ds\right\|_{L^{q'}}\\
		&\leq\int_{1}^{\max\left\{t-4,1\right\}} \left\|\nabla^i[P(g,h,h_{\infty})_{s\to t}\circ \Pi^{\perp}_{\infty}- P(\tilde{g},\tilde{h},\tilde{h}_{\infty})_{s\to t}\circ \tilde{\Pi}^{\perp}_{\infty}]\right\|_{L^{p'},L^{q'}}
\left\|(\Delta_{L,\tilde{h}_{\infty}}-\Delta_{L,\tilde{h}})(\tilde{k})\right\|_{L^{p'}}ds\\
		&\leq C\int_1^{\max\left\{t-4,1\right\}}(t-s)^{-\frac{n}{2}\left( \frac{1}{p'}-\frac{1}{q}\right)-\alpha(q',i)}s^{1-\frac{3n}{2}\left(\frac{1}{q}-\frac{1}{r}\right)}ds \\
		&\qquad \cdot\left(\left\|\tilde{h}-h\right\|_{Z_{q,r}}+\left\|\tilde{k}-k\right\|_{X_{q,r}}\right)\left\|\tilde{h}-\hat{h}\right\|_{Z_{q,r}}\left\|\tilde{k}\right\|_{X_{q,r}}\\
		&\leq C t^{-\alpha(q',i)}\left(\left\|\tilde{h}-h\right\|_{Z_{q,r}}+\left\|\tilde{k}-k\right\|_{X_{q,r}}\right)
		\left\|\tilde{h}-\hat{h}\right\|_{Z_{q,r}}\left\|\tilde{k}\right\|_{X_{q,r}}
\end{align*}
The last inequality is justified by Lemma \ref{important_technical_lemma}, since
\begin{align*}
	\frac{n}{2}\left(\frac{1}{p'}-\frac{1}{q}\right)+\alpha(q',i)
		&>\alpha(q',i),\\
 	\frac{3n}{2}\left(\frac{1}{q}-\frac{1}{r}\right)-1
 		&> \frac{n}{2}\left(\frac{1}{q}-\frac{1}{r}\right)
 		\geq\alpha(q',i),\\
	\frac{n}{2}\left(\frac{1}{p'}-\frac{1}{q}\right)+\alpha(q',i)+\frac{3n}{2}\left(\frac{1}{q}-\frac{1}{r}\right)-2
		&=\frac{n}{2}\left(\frac{1}{p'}-\frac{1}{r}\right)+\alpha(q',i)-1
>\alpha(q',i).
\end{align*}
For the term
\begin{align*}
\int_{1}^{\max\left\{t-4,1\right\}}  [P(g,h,h_{\infty})_{s\to t}\circ \Pi^{\perp}_{\infty}- P(\tilde{g},\tilde{h},\tilde{h}_{\infty})_{s\to t}\circ \tilde{\Pi}^{\perp}_{\infty}][(1-D_{\tilde{g}}\Phi)(\tilde{H}_1)] ds,
	\end{align*}	
			we have, using Lemma \ref{lem : nonlinear part 0.2}, Lemma \ref{lem : difference term 1} and Remark \ref{rem : difference heat flows long-time estimate},
\begin{align*}
	&\left\|\nabla^i \int_{1}^{\max\left\{t-4,1\right\}}  [P(g,h,h_{\infty})_{s\to t}\circ \Pi^{\perp}_{\infty}- P(\tilde{g},\tilde{h},\tilde{h}_{\infty})_{s\to t}\circ \tilde{\Pi}^{\perp}_{\infty}][(1-D_{\tilde{g}}\Phi)(\tilde{H}_1)] ds\right\|_{L^{q'}}\\
	&\qquad\leq \int_{1}^{\max\left\{t-4,1\right\}}\left\|\nabla^i( P(g,h,h_{\infty})_{s\to t}\circ \Pi^{\perp}_{\infty}- P(\tilde{g},\tilde{h},\tilde{h}_{\infty})_{s\to t})\circ \tilde{\Pi}^{\perp}_{\infty}\right\|_{L^{q},L^{q'}}\left\| (1-D_{\tilde{g}}\Phi)(\tilde{H}_1)\right\|_{L^q}ds\\
	&\qquad\leq C\int_{1}^{\max\left\{t-4,1\right\}}(t-s)^{-\alpha(q',i)}s^{-\beta-\frac{n}{2}\left(\frac{1}{q}-\frac{1}{r}\right)}ds
\left(\left\|h-\tilde{h}\right\|_{Z_{q,r}}
+\left\|k-\tilde{k}\right\|_{X_{q,r}}
\right)\left\|\tilde{k}\right\|_{X_{q,r}}^2\\
	&\qquad\leq Ct^{-\alpha(q',i)}\left(\left\|h-\tilde{h}\right\|_{Z_{q,r}}
+\left\|k-\tilde{k}\right\|_{X_{q,r}}\right)	
	\left\|\tilde{k}\right\|_{X_{q,r}}^2,
\end{align*}	
and the last inequality follows from Lemma \ref{important_technical_lemma} since
\begin{align*}
	\beta +\frac{n}{2} \left( \frac{1}{q}-\frac{1}{r} \right)
		>1,
	\qquad  \beta+\frac{n}{2}\left(\frac{1}{q}-\frac{1}{r}\right)
		> \frac{n}{2}\left(\frac{1}{q}-\frac{1}{r}\right)
		\geq \alpha(q',i).
\end{align*}
This finishes the proof.
\end{proof}
\begin{lem}\label{lem : difference term 4}
For $t\geq 2$, there exists a constant $C=C(q,r,\U,\epsilon)$ such that we have
\begin{align*}
&\left\|\int_{\max\left\{t-4,1\right\}}^{\max\left\{t-1,1\right\}}  [P(g,h,h_{\infty})_{s\to t}- P(\tilde{g},\tilde{h},\tilde{h}_{\infty})_{s\to t}]I(t,s,\tilde{k},\tilde{h},\tilde{h}_{\infty}) ds\right\|_{X_{q,r}}\\
	&\qquad\qquad\leq C \left( \left\|k-\tilde{k}\right\|_{X_{q,r}}+\left\|h-\tilde{h}\right\|_{Z_{q,r}} \right) \left\| \tilde{k}\right\|_{X_{q,r}}
\end{align*}
\end{lem}
\begin{proof}
Let $q'=\left\{q,r\right\}$.
 Lemma \ref{lem : difference of heat flows} (iii), Lemma \ref{lem : elliptic regularity} and Sobolev embedding yield
\begin{align*}
	&\left\|\int_{\max\left\{t-4,1\right\}}^{\max\left\{t-1,1\right\}}  [P(g,h,h_{\infty})_{s\to t}- P(\tilde{g},\tilde{h},\tilde{h}_{\infty})_{s\to t}]I(t,s,\tilde{k},\tilde{h},\tilde{h}_{\infty}) ds\right\|_{W^{2,q'}}\\
		&\qquad\qquad \leq C \sup_{s\in [t-5,t]}\left( \left\|g-\tilde{g}\right\|_{W^{1,\infty}}+\left\|h-\tilde{h}\right\|_{W^{4,\infty}}\right) \left\|
I(t,s,\tilde{k},\tilde{h},\tilde{h}_{\infty})\right\|_{L^{q'}}\\
		&\qquad\qquad\leq C \sup_{s\in [t-5,t]}\left(\left\|k-\tilde{k}\right\|_{W^{2,r}}+\left\|h-\tilde{h}\right\|_{L^q} \right) \left\|
I(t,s,\tilde{k},\tilde{h},\tilde{h}_{\infty})\right\|_{L^{q'}}\\
		&\qquad\qquad\leq C \left( \left\|k-\tilde{k}\right\|_{X_{q,r}}+\left\|h-\tilde{h}\right\|_{Z_{q,r}} \right) \sup_{s\in [t-5,t]}\left\|I(t,s,\tilde{k},\tilde{h},\tilde{h}_{\infty})\right\|_{L^{q'}}
\end{align*}
Exacly as in the proof of Lemma \ref{lem : nonlinear part 2}, we get the estimate
\begin{align*}
\sup_{s\in [t-5,t]} \left\|I(t,s,\tilde{k},\tilde{h},\tilde{h}_{\infty})\right\|_{L^{q'}}
	\leq C t^{-\frac{n}{2}\left(\frac{1}{q}-\frac{1}{r}\right)}
\left\|\tilde{k}\right\|_{X_{q,r}},
\end{align*}
which yields the result by definition of the norm.
\end{proof}

\begin{lem}\label{lem : preparation for difference term 5}
There exist a constant $C=C(q,r,\U,\epsilon)$ such that
for $s\in [\max\left\{t-1,1\right\},t]$ and $q'\in\left\{q,r\right\}$, we have
\begin{align*}
&\left\| I(t,s,k,h,h_{\infty})-I(t,s,\tilde{k},\tilde{h},\tilde{h}_{\infty})\right\|_{W^{1,q'}}\\
&\leq C
\left(
\left\|k\right\|_{W^{2,r}}+\left\|\tilde{k}\right\|_{W^{2,r}}
\right)
 \cdot
\left(\left\|k-\tilde{k}\right\|_{W^{2,r}}
+\left\|k-\tilde{k}\right\|_{W^{2,q}}
+\left\|h-\tilde{h}\right\|_{L^q}+\left\|h_{\infty}-\tilde{h}_{\infty}\right\|_{L^q}\right)\\
&\qquad +
C
\left\|k-\tilde{k}\right\|_{W^{2,r}}\left(\left\|\tilde{k}\right\|_{W^{2,r}}+\left\|\tilde{h}-\tilde{h}_{\infty}\right\|_{L^q}\right).
\end{align*}
\end{lem}
\begin{proof}
According to the proof of Lemma \ref{lem : integral term},
we can write, for $s\in [\max\left\{t-1,1\right\},t]$,
\begin{align*}
I(t,s,k,h,h_{\infty})=[\Delta_{L,g,h},\Pi^{\perp}_{\infty}](k)+\Pi^{\perp}_{\infty}[D_g\Phi((\Delta_{L,g,h}-\Delta_{L,h})(k))+(1-D_g\Phi)(H_2)]
\end{align*}
and an analogous expression exists for $I(t,s,\tilde{k},\tilde{h},\tilde{h}_{\infty})$. We are going to estimate all the terms separately. We first analyze the terms containing the different versions of $\Delta_{L}$.

For an arbitrary symmetric $2$-tensor $\Phi$, we can by definition of $\Delta_{L,g,h}$ write schematically
\begin{align*}
\Delta_{L,g,h}\Phi=g^{-1}*\nabla^2\Phi+\Phi*g^{-1}*h^{-1}*h^{-1}*g*R,
\end{align*}
so that
\begin{align*}
\Delta_{L,g,h}\Phi-\Delta_{L,\tilde{g},\tilde{h}}\Phi&=(g^{-1}-\tilde{g}^{-1})*\nabla^2\Phi+\tilde{g}^{-1}*(
\nabla^2\Phi-\tilde{\nabla}^2\Phi)\\
&\qquad +\tilde{g}^{-1}*\nabla(h-\tilde{h})*\nabla\Phi
+\Phi*(g^{-1}-\tilde{g}^{-1})*h^{-1}*h^{-1}*g*R\\
&\qquad+\Phi*\tilde{g}^{-1}*(h^{-1}-\tilde{h}^{-1})*(h^{-1}+\tilde{h}^{-1})*g*R\\
&\qquad +\Phi*\tilde{g}^{-1}*\tilde{h}^{-1}*\tilde{h}^{-1}*(g-\tilde{g})*R\\
&\qquad +\Phi*\tilde{g}^{-1}*\tilde{h}^{-1}*\tilde{h}^{-1}*\tilde{g}*(R-\tilde{R}).
\end{align*}
By the H\"{o}lder inequality, Lemma \ref{lem : elliptic regularity} and Sobolev embedding, we thus get for every $r'\in [1,\infty]$ that
\begin{equation}
\begin{split}\label{eq : estimate difference LL in L^r}
\left\|\Delta_{L,g,h}\Phi-\Delta_{L,\tilde{g},\tilde{h}}\Phi \right\|_{L^{r'}}&\leq C\left(\left\|g-\tilde{g}\right\|_{L^{\infty}}+\left\|h-\tilde{h}\right\|_{W^{2,\infty}}\right)\left\|\Phi\right\|_{W^{2,r'}}\\
&\leq C\left(\left\|k-\tilde{k}\right\|_{L^{\infty}}+\left\|h-\tilde{h}\right\|_{W^{2,\infty}}\right)\left\|\Phi\right\|_{W^{2,r'}}\\
&\leq C\left(\left\|k-\tilde{k}\right\|_{W^{2,r}}+\left\|h-\tilde{h}\right\|_{L^q}\right)\left\|\Phi\right\|_{W^{2,r'}}
\end{split}
\end{equation}
and
\begin{equation}
\begin{split}\label{eq : estimate difference LL in W^1,r}
\left\|\Delta_{L,g,h}\Phi-\Delta_{L,\tilde{g},\tilde{h}}\Phi \right\|_{W^{1,r'}}&\leq C\left(\left\|g-\tilde{g}\right\|_{W^{1,\infty}}+\left\|h-\tilde{h}\right\|_{W^{3,\infty}}\right)\left\|\Phi\right\|_{W^{3,r'}}\\
&\leq C\left(\left\|k-\tilde{k}\right\|_{W^{1,\infty}}+\left\|h-\tilde{h}\right\|_{W^{3,\infty}}\right)\left\|\Phi\right\|_{W^{3,r'}}\\
&\leq C\left(\left\|k-\tilde{k}\right\|_{W^{2,r}}+\left\|h-\tilde{h}\right\|_{L^q}\right)\left\|\Phi\right\|_{W^{3,r'}}.
\end{split}
\end{equation}
Since $\Delta_{L,h}=\Delta_{L,h,h}$, $g=h+k$ and the corresponding identities with tildes hold, we get as special cases that
\begin{align}
\label{eq : estimate difference LL 2 in L^r}
\left\|\Delta_{L,h}\Phi-\Delta_{L,\tilde{h}}\Phi \right\|_{L^{r'}}&\leq C\left\|h-\tilde{h}\right\|_{L^q}\left\|\Phi\right\|_{W^{2,r'}},\\
\label{eq : estimate difference LL 2 in W^1,r}
\left\|\Delta_{L,h}\Phi-\Delta_{L,\tilde{h}}\Phi \right\|_{W^{1,r'}}&\leq C\left\|h-\tilde{h}\right\|_{L^q}\left\|\Phi\right\|_{W^{3,r'}}.\\
\label{eq : estimate difference LL 3 in L^r}
\left\|\Delta_{L,g,h}\Phi-\Delta_{L,h}\Phi \right\|_{L^{r'}}&\leq C\left\|k\right\|_{W^{2,r}}\left\|\Phi\right\|_{W^{2,r'}},\\
\label{eq : estimate difference LL 3 in W^1,r}
\left\|\Delta_{L,g,h}\Phi-\Delta_{L,h}\Phi \right\|_{W^{1,r'}}&\leq C\left\|k\right\|_{W^{2,r}}\left\|\Phi\right\|_{W^{3,r'}}.
\end{align}
A formal computation shows that
\begin{align*}
[\Delta_{L,g,h},\Pi^{\perp}_{\infty}](k)=
-[\Delta_{L,g,h},\Pi^{\parallel}_{\infty}](k)
&=-\Delta_{L,g,h}(\Pi^{\parallel}_{\infty}(k))+\Pi^{\parallel}_{\infty}
(\Delta_{L,g,h}(k))
\end{align*}
and an analogous expression holds for $[\Delta_{L,\tilde{g},\tilde{h}},\tilde{\Pi}^{\perp}_{\infty}](\tilde{k})$. Therefore,
\begin{align*}
[\Delta_{L,g,h},\Pi^{\perp}_{\infty}](k)&-[\Delta_{L,\tilde{g},\tilde{h}},\tilde{\Pi}^{\perp}_{\infty}](\tilde{k})\\
&=
-(\Delta_{L,g,h}-\Delta_{L,\tilde{g},\tilde{h}})\Pi^{\parallel}_{\infty}(k)
-\Delta_{L,\tilde{g},\tilde{h}}((\Pi^{\parallel}_{\infty}-\tilde{\Pi}^{\parallel}_{\infty})(k)
+\tilde{\Pi}_{\infty}^{\parallel}(k-\tilde{k}))\\
&\qquad+({\Pi}_{\infty}^{\parallel}-\tilde{\Pi}_{\infty}^{\parallel})(\Delta_{L,g,h}k)
+{\tilde{\Pi}}_{\infty}^{\parallel}((\Delta_{L,g,h}-\Delta_{L,\tilde{g},\tilde{h}})(k))
+{\tilde{\Pi}}_{\infty}^{\parallel}(\Delta_{L,\tilde{g},\tilde{h}}
(k-\tilde{k}))\\
&=(\Delta_{L,\tilde{g},\tilde{h}}-\Delta_{L,g,h})\Pi^{\parallel}_{\infty}(k)
+\Delta_{L,\tilde{g},\tilde{h}}((\Pi^{\perp}_{\infty}-\tilde{\Pi}^{\perp}_{\infty})(k))\\
&\qquad+[(\Delta_{L,\tilde{h}}-\Delta_{L,\tilde{g},\tilde{h}})
+(\Delta_{L,\tilde{h}_{\infty}}-\Delta_{L,\tilde{h}})
](\tilde{\Pi}_{\infty}^{\parallel}(k-\tilde{k}))\\
&\qquad
+(\tilde{\Pi}_{\infty}^{\perp}-{\Pi}_{\infty}^{\perp})(\Delta_{L,g,h}k)
+{\tilde{\Pi}}_{\infty}^{\parallel}((\Delta_{L,g,h}-\Delta_{L,\tilde{g},\tilde{h}})(k))\\
&\qquad+{\tilde{\Pi}}_{\infty}^{\parallel}(
[(\Delta_{L,\tilde{g},\tilde{h}}-\Delta_{L,\tilde{h}})+(\Delta_{L,\tilde{h}}-\Delta_{L,\tilde{h}_{\infty}})]
(k-\tilde{k})),
\end{align*}
where in the last equation, we used the fact that 
\begin{align*}\Delta_{L,\tilde{h}_{\infty}}\circ {\tilde{\Pi}}_{\infty}^{\parallel}={\tilde{\Pi}}_{\infty}^{\parallel}\circ\Delta_{L,\tilde{h}_{\infty}}=0.
\end{align*}
Now let us estimate this expression term by term. 
By \eqref{eq : estimate difference LL in W^1,r} and Lemma \ref{lem : projection lemma} (i),
we have
\begin{align*}
\left\|
(\Delta_{L,\tilde{g},\tilde{h}}-\Delta_{L,g,h})\Pi^{\parallel}_{\infty}(k)
\right\|_{W^{1,q'}}&\leq
C\left(\left\|k-\tilde{k}\right\|_{W^{2,r}}+\left\|h-\tilde{h}\right\|_{L^{q}}\right)\left\|\Pi^{\parallel}_{\infty}(k)\right\|_{W^{3,q'}}\\
&\leq C\left(\left\|k-\tilde{k}\right\|_{W^{2,r}}+\left\|h-\tilde{h}\right\|_{L^{q}}\right)
\left\|k\right\|_{L^r}.
\end{align*}
By Lemma \ref{lem : projection lemma} (v),
\begin{align*}
\left\|
\Delta_{L,\tilde{g},\tilde{h}}((\Pi^{\perp}_{\infty}-\tilde{\Pi}^{\perp}_{\infty})(k))
\right\|_{W^{1,q'}}\leq C\left\|(\Pi^{\perp}_{\infty}-\tilde{\Pi}^{\perp}_{\infty})(k)\right\|_{W^{3,q'}}
\leq C\left\|h-\tilde{h}\right\|_{L^q}\left\|k\right\|_{L^r}.
\end{align*}
By \eqref{eq : estimate difference LL 2 in W^1,r} and \eqref{eq : estimate difference LL 3 in W^1,r} and Lemma \ref{lem : projection lemma} (i),
\begin{align*}
&\left\|
[(\Delta_{L,\tilde{h}}-\Delta_{L,\tilde{g},\tilde{h}})
+(\Delta_{L,\tilde{h}_{\infty}}-\Delta_{L,\tilde{h}})
](\tilde{\Pi}_{\infty}^{\parallel}(k-\tilde{k}))
\right\|_{W^{1,q'}}\\
&\qquad\qquad\qquad\leq C
\left(\left\|\tilde{k}\right\|_{W^{2,r}}
+\left\|\tilde{h}-\tilde{h}_{\infty}\right\|_{L^{q}}\right)
\left\|\tilde{\Pi}_{\infty}^{\parallel}(k-\tilde{k})
\right\|_{W^{3,q'}}\\
&\qquad\qquad\qquad\leq C
\left(\left\|\tilde{k}\right\|_{W^{2,r}}
+\left\|\tilde{h}-\tilde{h}_{\infty}\right\|_{L^{q}}\right)
\left\|k-\tilde{k}
\right\|_{L^{r}}.
\end{align*}
By Lemma \ref{lem : projection lemma} (v),
\begin{align*}
\left\|(\tilde{\Pi}_{\infty}^{\perp}-{\Pi}_{\infty}^{\perp})(\Delta_{L,g,h}k)\right\|_{W^{1,q'}}\leq
C\left\|h_{\infty}-\tilde{h}_{\infty}\right\|_{L^q}
\left\|\Delta_{L,g,h}k
\right\|_{L^r}
\leq
C\left\|h_{\infty}-\tilde{h}_{\infty}\right\|_{L^q}
\left\|k
\right\|_{W^{2,r}}.
\end{align*}
By Lemma \ref{lem : projection lemma} (i) and \eqref{eq : estimate difference LL in L^r},
\begin{align*}
\left\|
{\tilde{\Pi}}_{\infty}^{\parallel}((\Delta_{L,g,h}-\Delta_{L,\tilde{g},\tilde{h}})(k))
\right\|_{W^{1,q'}}&\leq
C\left\|
(\Delta_{L,g,h}-\Delta_{L,\tilde{g},\tilde{h}})(k)
\right\|_{L^r}\\
&\leq
C\left(\left\|k-\tilde{k}\right\|_{W^{2,r}}+\left\|
h-\tilde{h}
\right\|_{L^{q}}\right)\left\|k\right\|_{W^{2,r}}.
\end{align*}
By Lemma \ref{lem : projection lemma} (i), \eqref{eq : estimate difference LL 2 in L^r} and \eqref{eq : estimate difference LL 3 in L^r},
\begin{align*}
&\left\|
{\tilde{\Pi}}_{\infty}^{\parallel}(
[(\Delta_{L,\tilde{g},\tilde{h}}-\Delta_{L,\tilde{h}})+(\Delta_{L,\tilde{h}}-\Delta_{L,\tilde{h}_{\infty}})]
(k-\tilde{k}))
\right\|_{W^{1,q'}}\\
&\qquad\qquad\qquad\leq
C
\left\|
[(\Delta_{L,\tilde{g},\tilde{h}}-\Delta_{L,\tilde{h}})+(\Delta_{L,\tilde{h}}-\Delta_{L,\tilde{h}_{\infty}})]
(k-\tilde{k})
\right\|_{L^{r}}\\
%&\qquad\qquad\qquad\leq C \left(
%\left\|\tilde{g}-\tilde{h}\right\|_{L^{\infty}}+\left\|\tilde{h}-\tilde{h}_{\infty}\right\|_{W^{2,\infty}}
%\right)\left\|k-\tilde{k}\right\|_{W^{2,r}}\\
&\qquad\qquad\qquad\leq C \left(
\left\|\tilde{k}\right\|_{W^{2,r}}+\left\|\tilde{h}-\tilde{h}_{\infty}\right\|_{L^{q}}
\right)\left\|k-\tilde{k}\right\|_{W^{2,r}}.
\end{align*}
Now we write
\begin{align*}
\Pi^{\perp}_{\infty}[D_g\Phi((\Delta_{L,g,h}-\Delta_{L,h})(k))]
&-\tilde{\Pi}^{\perp}_{\infty}[D_{\tilde{g}}\Phi((\Delta_{L,\tilde{g},\tilde{h}}-\Delta_{L,\tilde{h}})(\tilde{k}))]\\
&=(\Pi^{\perp}_{\infty}-\tilde{\Pi}^{\perp}_{\infty})[D_g\Phi((\Delta_{L,g,h}-\Delta_{L,h})(k))]\\
&\qquad+\tilde{\Pi}^{\perp}_{\infty}[(D_g\Phi-D_{\tilde{g}}\Phi)((\Delta_{L,g,h}-\Delta_{L,h})(k))]\\
&\qquad+
\tilde{\Pi}^{\perp}_{\infty}[D_{\tilde{g}}\Phi
((\Delta_{L,g,h}-\Delta_{L,\tilde{g},\tilde{h}})(k)
-(\Delta_{L,h}-\Delta_{L,\tilde{h}})(k))]\\
&\qquad+\tilde{\Pi}^{\perp}_{\infty}[D_{\tilde{g}}\Phi((\Delta_{L,\tilde{g},\tilde{h}}-\Delta_{L,\tilde{h}})(k-\tilde{k}))]
\end{align*}
Again, we estimate term by term. Using Lemma \ref{lem : projection lemma} (v), \eqref{eq : estimate difference LL in L^r} and Sobolev embedding, we get
\begin{align*}
\left\|
(\Pi^{\perp}_{\infty}-\tilde{\Pi}^{\perp}_{\infty})[D_g\Phi((\Delta_{L,g,h}-\Delta_{L,h})(k))]
\right\|_{W^{1,q'}}&\leq C\left\|h_{\infty}-\tilde{h}_{\infty}\right\|_{L^p}\left\| 
(\Delta_{L,g,h}-\Delta_{L,h})(k) \right\|_{L^r}\\
&\leq C\left\|h_{\infty}-\tilde{h}_{\infty}\right\|_{L^q}
\left\|k\right\|_{L^{\infty}}\left\|k\right\|_{W^{2,r}}\\
&\leq C\left\|h_{\infty}-\tilde{h}_{\infty}\right\|_{L^q}
\left\|k\right\|_{W^{2,r}}^2.
\end{align*}
By Lemma \ref{lem : projection lemma} (i) (applied to $
\tilde{\Pi}^{\parallel}_{\infty}=1-
\tilde{\Pi}^{\perp}_{\infty}$) and (vi) (applied to $p=q$), \eqref{eq : estimate difference LL 3 in L^r} and Sobolev embedding,
\begin{align*}
&\left\|\tilde{\Pi}^{\perp}_{\infty}[(D_g\Phi-D_{\tilde{g}}\Phi)((\Delta_{L,g,h}-\Delta_{L,h})(k))]
\right\|_{W^{1,q'}}
\leq C\left\|g-\tilde{g}\right\|_{L^{[q,\infty]}}\left\| 
(\Delta_{L,g,h}-\Delta_{L,h})(k) \right\|_{L^r}\\
&\qquad\qquad\qquad\qquad\leq C\left(\left\|k-\tilde{k}\right\|_{L^{q}}+
\left\|k-\tilde{k}\right\|_{L^{\infty}}+\left\|h-\tilde{h}\right\|_{L^{q}}\right)\left\|k\right\|_{L^{\infty}}\left\|k\right\|_{W^{2,r}}\\
&\qquad\qquad\qquad\qquad\leq C\left(\left\|k-\tilde{k}\right\|_{L^{q}}+
\left\|k-\tilde{k}\right\|_{W^{2,r}}+\left\|h-\tilde{h}\right\|_{L^{q}}\right)\left\|k\right\|_{W^{2,r}}^2
\end{align*}
By Lemma \ref{lem : projection lemma} (i) (applied to $
\tilde{\Pi}^{\parallel}_{\infty}=1-
\tilde{\Pi}^{\perp}_{\infty}$) and (iii), as well as \eqref{eq : estimate difference LL in L^r} and \eqref{eq : estimate difference LL 2 in L^r},
\begin{align*}
&\left\|
\tilde{\Pi}^{\perp}_{\infty}[D_{\tilde{g}}\Phi
((\Delta_{L,g,h}-\Delta_{L,\tilde{g},\tilde{h}})(k)
-(\Delta_{L,h}-\Delta_{L,\tilde{h}})(k))]
\right\|_{W^{1,q'}}
\\
& \qquad\qquad\qquad\qquad\leq
 C\left(\left\|(\Delta_{L,g,h}-\Delta_{L,\tilde{g},\tilde{h}})(k)
\right\|_{L^{r}}
+
\left\|(\Delta_{L,h}-\Delta_{L,\tilde{h}})(k)
\right\|_{L^{r}}\right)
\\
& \qquad\qquad\qquad\qquad\leq
C\left( \left\|k-\tilde{k}\right\|_{W^{2,r}}+\left\|h-\tilde{h}\right\|_{L^q}\right)\left\|k\right\|_{W^{2,r}}
\end{align*}
Finally,
by Lemma \ref{lem : projection lemma} (i) (applied to $
\tilde{\Pi}^{\parallel}_{\infty}=1-
\tilde{\Pi}^{\perp}_{\infty}$) and (iii), as well as \eqref{eq : estimate difference LL 3 in L^r},
\begin{align*}
\left\|\tilde{\Pi}^{\perp}_{\infty}[D_{\tilde{g}}\Phi((\Delta_{L,\tilde{g},\tilde{h}}-\Delta_{L,\tilde{h}})(k-\tilde{k}))]\right\|_{W^{1,q'}} &\leq C\left\|(\Delta_{L,\tilde{g},\tilde{h}}-\Delta_{L,\tilde{h}})(k-\tilde{k})\right\|_{L^r}\\
%&\leq C\left\|\tilde{g}-\tilde{h}\right\|_{L^{\infty}}
%\left\|k-\tilde{k}\right\|_{W^{2,r}}\\
&\leq C\left\|\tilde{k}\right\|_{W^{2,r}}\left\|k-\tilde{k}\right\|_{W^{2,r}}.
\end{align*}
Now let us look at the remaining expression
\begin{align*}
\Pi^{\perp}_{\infty}[(1-D_g\Phi)(H_2)]
-\tilde{\Pi}^{\perp}_{\infty}[(1-D_{\tilde{g}}\Phi)(\tilde{H}_2)]&=
(\Pi^{\perp}_{\infty}-\tilde{\Pi}^{\perp}_{\infty})[(1-D_g\Phi)(H_2)]\\
&\qquad+
\tilde{\Pi}^{\perp}_{\infty}[(D_{\tilde{g}}\Phi-D_g\Phi)({H}_2)]\\
&\qquad+\tilde{\Pi}^{\perp}_{\infty}[(1-D_{\tilde{g}}\Phi)(H_2-\tilde{H}_2)].
\end{align*}
We estimate the three terms on the right hand side separately.
Recall that $H_2$ has the schematic expression 
\begin{align*}
H_2=g^{-1}*g^{-1}*\nabla k*\nabla k.
\end{align*}
Thus by Lemma \ref{lem : projection lemma} (iii) and (v), and by Sobolev embedding
\begin{align*}
\left\|
(\Pi^{\perp}_{\infty}-\tilde{\Pi}^{\perp}_{\infty})[(1-D_g\Phi)(H_2)]
\right\|_{W^{1,q'}}&\leq
C \left\|h_{\infty}-\tilde{h}_{\infty}\right\|_{L^{q}}\left\|H_2\right\|_{L^r}\\
&\leq C \left\|h_{\infty}-\tilde{h}_{\infty}\right\|_{L^{q}}
\left\|\nabla k\right\|_{L^r}\left\|\nabla k\right\|_{L^{\infty}}\\
&\leq C \left\|h_{\infty}-\tilde{h}_{\infty}\right\|_{L^{q}}
\left\| k\right\|_{W^{2,r}}^2.
\end{align*}
By Lemma \ref{lem : projection lemma} (i) (applied to $
\tilde{\Pi}^{\parallel}_{\infty}=1-
\tilde{\Pi}^{\perp}_{\infty}$) and (vi) (applied to $p=q$), as well as Lemma \ref{lem : elliptic regularity} and Sobolev embedding,
\begin{align*}
\left\|
\tilde{\Pi}^{\perp}_{\infty}[(D_{\tilde{g}}\Phi-D_g\Phi)({H}_2)]
\right\|_{W^{1,q'}}&\leq
C(\left\|g-\tilde{g}\right\|_{L^q}+\left\|g-\tilde{g}\right\|_{L^{\infty}})\left\|H_2\right\|_{L^r}\\
&\leq C\left(\left\|k-\tilde{k}\right\|_{L^q}+
\left\|k-\tilde{k}\right\|_{W^{2,r}}
+\left\|h-\tilde{h}\right\|_{L^{q}}\right)\left\| k\right\|_{W^{2,r}}^2.
\end{align*}
To estimate the final term, we first write schematically
\begin{align*}
H_2-\tilde{H}_2&=g^{-1}*g^{-1}*\nabla k*\nabla k-\tilde{g}^{-1}*\tilde{g}^{-1}*\tilde{\nabla} \tilde{k}*\tilde{\nabla} \tilde{k}\\
&=(g^{-1}-\tilde{g}^{-1})*g^{-1}*\nabla k*\nabla k
+\tilde{g}^{-1}*(g^{-1}-\tilde{g}^{-1})*\nabla k*\nabla k\\
&\qquad
+\tilde{g}^{-1}*\tilde{g}^{-1}*(\nabla-\tilde{\nabla}) k*\nabla k+
\tilde{g}^{-1}*\tilde{g}^{-1}*\tilde{\nabla} (k-\tilde{k})*\nabla k\\
&\qquad
+\tilde{g}^{-1}*\tilde{g}^{-1}*\tilde{\nabla} \tilde{k}*(\nabla-\tilde{\nabla}) k+
\tilde{g}^{-1}*\tilde{g}^{-1}*\tilde{\nabla} \tilde{k}*\tilde{\nabla} (k-\tilde{k})\\
&=(g^{-1}-\tilde{g}^{-1})*g^{-1}*\nabla k*\nabla k
+\tilde{g}^{-1}*(g^{-1}-\tilde{g}^{-1})*\nabla k*\nabla k\\
&\qquad
+\tilde{g}^{-1}*\tilde{g}^{-1}*\tilde{h}^{-1}*\nabla(h-\tilde{h})* k*\nabla k+
\tilde{g}^{-1}*\tilde{g}^{-1}*\tilde{\nabla} (k-\tilde{k})*\nabla k\\
&\qquad
+\tilde{g}^{-1}*\tilde{g}^{-1}*\tilde{\nabla} \tilde{k}*\tilde{h}^{-1}*\nabla(h-\tilde{h})*k+
\tilde{g}^{-1}*\tilde{g}^{-1}*\tilde{\nabla} \tilde{k}*\tilde{\nabla} (k-\tilde{k})
.
\end{align*}
By Lemma \ref{lem : projection lemma} (i) (applied to $
\tilde{\Pi}^{\parallel}_{\infty}=1-
\tilde{\Pi}^{\perp}_{\infty}$) and (iii), as well as Lemma \ref{lem : elliptic regularity} and Sobolev embedding,
\begin{align*}
\left\|
\tilde{\Pi}^{\perp}_{\infty}[(1-D_{\tilde{g}}\Phi)(H_2-\tilde{H}_2)]
\right\|_{W^{1,q'}}&\leq
C\left\|
H_2-\tilde{H}_2
\right\|_{W^{1,q'}}\\
&\leq
C\left\|g-\tilde{g}\right\|_{W^{1,\infty}}\left\|\nabla k\right\|_{W^{1,\infty}}\left\|\nabla k\right\|_{W^{1,q'}}\\
&\qquad +C\left(\left\|\nabla k\right\|_{W^{1,\infty}}+\left\|\tilde{\nabla}\tilde{k}\right\|_{W^{1,\infty}}\right)
\left\|h-\tilde{h}\right\|_{W^{2,\infty}}\left\|k\right\|_{W^{1,q'}}\\
&\qquad  +C \left(\left\|\nabla k\right\|_{W^{1,\infty}}+\left\|\tilde{\nabla}\tilde{k}\right\|_{W^{1,\infty}}\right)
\left\|\tilde{\nabla}(k-\tilde{k})\right\|_{W^{1,q'}}\\
&\leq C
\left(\left\|h-\tilde{h}\right\|_{L^q}
+\left\|k-\tilde{k}\right\|_{W^{2,q'}}
\right)\left(\left\|k\right\|_{W^{2,r}}+\left\|\tilde{k}\right\|_{W^{2,r}}\right),
\end{align*}
which finishes the proof of the lemma.
\end{proof}

\begin{lem}\label{lem : difference term 5}
There exist a constant $C=C(q,r,\U,\epsilon)$ such that we have
\begin{align*}
	&\left\| \int_{\max\left\{t-1,1\right\}}^{t}[  P(g,h,h_{\infty})_{s\to t}I(t,s,k,h,h_{\infty})- P(\tilde{g},\tilde{h},\tilde{h}_{\infty})_{s\to t}I(t,s,\tilde{k},\tilde{h},\tilde{h}_{\infty}) ]ds\right\|_{X_{q,r}}\\
		&\leq C \left( \left\| k-\tilde{k}\right\|_{X_{q,r}}+\left\|h-\tilde{h}\right\|_{Z_{q,r}} \right) \left(
  \left\|\tilde{k}\right\|_{X_{q,r}}  +  \left\|\tilde{h}-\hat{h}\right\|_{Z_{q,r}} 
    \right).
\end{align*}
\end{lem}
\begin{proof}
By Lemma \ref{lem : difference of heat flows} (ii) we get for $q'\in\left\{q,r\right\}$ that
\begin{align*}
	&\left\| \int_{\max\left\{t-1,1\right\}}^{t}[  P(g,h,h_{\infty})_{s\to t}I(t,s,k,h,h_{\infty})- P(\tilde{g},\tilde{h},\tilde{h}_{\infty})_{s\to t}I(t,s,\tilde{k},\tilde{h},\tilde{h}_{\infty}) ]ds\right\|_{W^{2,q'}}\\
		&\qquad\leq C \int_{\max\left\{t-1,1\right\}}^{t}\left[  (t-s)^{-1/2}\sup_{s\in [t-1,t]} \left( \left\|g-\tilde{g}\right\|_{W^{1,\infty}}+\left\|h-\tilde{h}\right\|_{W^{4,\infty}} \right)	\left\| I(t,s,\tilde{k},\tilde{h},\tilde{h}_{\infty})\right\|_{W^{1,q'}} \right]ds\\	
&\qquad\qquad +		  \int_{\max\left\{t-1,1\right\}}^{t}\left[  (t-s)^{-1/2}\sup_{s\in [t-1,t]}
\left\|
I(t,s,k,h,h_{\infty})-
I(t,s,\tilde{k},\tilde{h},\tilde{h}_{\infty})
\right\|_{W^{1,q'}}
	\right]ds	\\
			&\qquad\leq C \sup_{s\in [\max\left\{t-1,1\right\},t]} \left( \left\|g-\tilde{g}\right\|_{W^{1,\infty}}+\left\|h-\tilde{h}\right\|_{W^{4,\infty}} \right) \left\| I(t,s,\tilde{k},\tilde{h},\tilde{h}_{\infty})\right\|_{W^{1,q'}}\\
			&\qquad\qquad +
\sup_{s\in [\max\left\{t-1,1\right\},t]}
\left\|
I(t,s,k,h,h_{\infty})-
I(t,s,\tilde{k},\tilde{h},\tilde{h}_{\infty})
\right\|_{W^{1,q'}}			
			.
\end{align*}		
By the triangle inequality, Sobolev embedding and Lemma \ref{lem : elliptic regularity},
\begin{align*}
\sup_{s\in [\max\left\{t-1,1\right\},t]} \left( \left\|g-\tilde{g}\right\|_{W^{1,\infty}}+\left\|h-\tilde{h}\right\|_{W^{4,\infty}} \right)
&\leq\sup_{s\in [\max\left\{t-1,1\right\},t]} \left( \left\|k-\tilde{k}\right\|_{W^{1,\infty}}+\left\|h-\tilde{h}\right\|_{W^{4,\infty}} \right)\\
&\leq C\sup_{s\in [\max\left\{t-1,1\right\},t]} \left( \left\|k-\tilde{k}\right\|_{W^{2,r}}+\left\|h-\tilde{h}\right\|_{L^q} \right)\\
&\leq C\left( \left\|k-\tilde{k}\right\|_{X_{q,r}}+\left\|h-\tilde{h}\right\|_{Z_{q,r}} \right) 
\end{align*}
By Lemma \ref{lem : integral term} and the definition of the $X_{q,r}$-norm,		\begin{align*}
\sup_{s\in [\max\left\{t-1,1\right\},t]} \left\| I(t,s,\tilde{k},\tilde{h},\tilde{h}_{\infty})\right\|_{W^{1,q'}}
&\leq
C\sup_{s\in [\max\left\{t-1,1\right\},t]}\left(\left\|\tilde{k}\right\|_{W^{2,r}}+\left\|\nabla \tilde{k}\right\|_{W^{1,q}}+\left\|\tilde{h}-\tilde{h}_{\infty}\right\|_{L^q}\right)\left\| \tilde{k}\right\|_{W^{2,r}}
	\\
&\leq C\sup_{s\in [\max\left\{t-1,1\right\},t]}\left(\left\|\tilde{k}\right\|_{W^{2,r}}+\left\|\nabla \tilde{k}\right\|_{W^{1,q}}+\left\|\tilde{h}-\tilde{h}_{\infty}\right\|_{L^q}\right)					\left\| \tilde{k}\right\|_{W^{2,r}}
				\\
				&\leq Ct^{-\frac{n}{2} \left( \frac{1}{q}-\frac{1}{r} \right) }  \left\| \tilde{k}\right\|_{X_{q,r}}.
\end{align*}
Finally, by Lemma \ref{lem : preparation for difference term 5}, Lemma \ref{lem : inftycontrol} and the definition of the $X_{q,r}$- and $Z_{q,r}$-norms,
\begin{align*}
&\sup_{s\in [\max\left\{t-1,1\right\},t]}\left\| I(t,s,k,h,h_{\infty})-I(t,s,\tilde{k},\tilde{h},\tilde{h}_{\infty})\right\|_{W^{1,q'}}\\
&\leq C\sup_{s\in [\max\left\{t-1,1\right\},t]}
\left(
\left\|k\right\|_{W^{2,r}}+\left\|\tilde{k}\right\|_{W^{2,r}}
\right)
 \cdot
\left(\left\|k-\tilde{k}\right\|_{W^{2,r}}
+\left\|k-\tilde{k}\right\|_{W^{2,q}}
+\left\|h-\tilde{h}\right\|_{L^q}+\left\|h_{\infty}-\tilde{h}_{\infty}\right\|_{L^q}\right)\\
&\qquad +
C\sup_{s\in [\max\left\{t-1,1\right\},t]}
\left\|k-\tilde{k}\right\|_{W^{2,r}}\left(\left\|\tilde{k}\right\|_{W^{2,r}}+\left\|\tilde{h}-\tilde{h}_{\infty}\right\|_{L^q}\right)\\
&\leq  
 Ct^{-\frac{n}{2} \left( \frac{1}{q}-\frac{1}{r} \right) } 
\left(  \left\| k\right\|_{X_{q,r}}+\left\| \tilde{k}\right\|_{X_{q,r}}\right)
 \left( \left\|k-\tilde{k}\right\|_{X_{q,r}}+\left\|h-\tilde{h}\right\|_{Z_{q,r}} \right)\\
 &\qquad +
  Ct^{-\frac{n}{2} \left( \frac{1}{q}-\frac{1}{r} \right) } 
  \left\| k-\tilde{k}\right\|_{X_{q,r}}
    \left(
  \left\|\tilde{k}\right\|_{X_{q,r}}  +  \left\|\tilde{h}-\hat{h}\right\|_{Z_{q,r}} 
    \right).
\end{align*}
Putting these estimates together proves the statement.
\end{proof}
\section{Long-time existence and convergence}\label{sec : long-time existence convergence}
In this section, we are going to prove the main results of the paper.
% which is Theorem \ref{thm : mainthm 1 introduction}.
 Throughout the section, let $(M,\hat{h})$ be a Ricci-flat ALE manifold which is integrable and carries a parallel spinor.
\subsection{Establishing a fixed point of the iteration map}
\begin{definition}Let $\U\subset\M$ be a neighbourhood of $\hat{h}$ on which the map $\Phi$ of Subsection \ref{subsec: proj map} is defined.
	We call a family of metrics $g_t$, $t\in [0,\infty)$ in $\U$ a modified Ricci-de Turck flow starting at $g_0$ if $g_t$ satisfies
	\begin{align*}
	\partial_tg_t&=-2\Ric_{g_t}+\mathcal{L}_{V(g_t,\Phi(g_t))}g_t,\qquad t> 1,\\
	\partial_tg_t&=-2\Ric_{g_t}+\mathcal{L}_{V(g_t,\hat{h})}g_t,\qquad t\in [0,1].
	\end{align*}
	In other words, for $t\in [0,1]$, $g_t$ evolves under the Ricci-de Turck flow with reference metric $\hat{h}$ while for $t>1$, $g_t$ evolves under the Ricci-de Turck flow with moving reference metric $\Phi(g_t)$.
\end{definition}

\begin{thm}\label{thm : fixed point}
Let $q\in (1,n)$ and $r\in (n,\infty)$ so large that $\frac{n}{2}\left(\frac{1}{q}-\frac{1}{r}\right) > \frac{1}{2}$ and $\frac{n}{2}\left(\frac{1}{q}-\frac{1}{r}\right) \neq 1$. 
Then for any $\epsilon>0$, we can choose $\delta>0$ so small that if a metric $g_0$ satisfies
\begin{align*}
\left\|g_0-\hat{h}\right\|_{L^{q}}+\left\|g_0-\hat{h}\right\|_{L^{\infty}}<\delta,
\end{align*}
the modified Ricci-de Turck flow $g_t$
starting at $g_0$
is well-defined, exists for all $t\geq0$ and such that for $t\geq1$, the tensors
$
h_t:=\Phi(g_t), k_t:=g_t-h_t
$
satisfy 
\begin{align*}
	\left\|\left(h_t-\hat{h},k_t\right)\right\|_{Y_{q,r}}<\epsilon.
\end{align*}
\end{thm}
\begin{proof} By Lemma \ref{lem : Ricci flow short-time estimates in Ck}, we can for any given $\epsilon_0>0$ pick $\delta>0$ so small such that the bound
\begin{align*}
\left\|g_0-\hat{h}\right\|_{L^{\infty}}<\delta
\end{align*}
implies 
\begin{align*}
\left\|g_t-\hat{h}\right\|_{L^{\infty}}<\epsilon_0
\end{align*}
for all $t\in [0,1]$.
	Therefore, for any given $\epsilon_1>0$, Lemma \ref{lem : Ricci flow short-time estimates} enables us to choose $\delta>0$ so small that the Ricci-de Turck flow $g_t$ with background metric $\hat{h}$ with initial data $g_0$ exists up to time $t=1$ and satisfies
	\begin{align*}
	\left\|g_1-\hat{h}\right\|_{W^{2,q}}+\left\|g_1-\hat{h}\right\|_{W^{2,\infty}}< \epsilon_1.
   \end{align*}
   Due to interpolation, this also implies
 	\begin{align*}
 \left\|g_1-\hat{h}\right\|_{W^{2,r}}< \epsilon_1
 \end{align*}  
for $r\in (q,\infty)$.
  For any given $\epsilon_2>0$, we may choose $\epsilon_1>0$ so small that the projection map 
  \begin{align*}
  \Phi:\mathcal{M}\supset \mathcal{U}\to\mathcal{F}
  \end{align*}
  from Subsection \ref{subsec: proj map} can be applied to $g_1$ and such that the tensors
  \begin{align*}
  h_1:=\Phi(g_1),\qquad k_1:=g_1-\Phi(g_1)
  \end{align*}
  satisfy
  \begin{align*}
  \left\|k_1\right\|_{W^{2,q}}+  \left\|k_1\right\|_{W^{2,r}}+  \left\|k_1\right\|_{W^{2,\infty}}+ \left\|h_1-\hat{h}\right\|_{L^q}<\epsilon_2.
  \end{align*}
  We now define for $t\geq 1$
  \begin{align*}
  h_t^{(1)}:=h_1,\qquad  k_t^{(1)}:=e^{-(t-1)\Delta_{L,h_1}}k_1.
  \end{align*}
  It follows from Lemma \ref{lem : linear part} (applied to the special case where $g=h=h_{\infty}$ are all equal to $h_1$) that
  \begin{align*}
  \left\|k_t^{(1)}\right\|_{X_{q,r}}
  	\leq C \left(\left\|k_1\right\|_{W^{2,q}}+ \left\|k_1\right\|_{W^{2,r}}\right)
  \end{align*}
  and it is clear from the definition of the norm that
 \begin{align*}
  \left\|h_t^{(1)}-\hat{h}\right\|_{Z_{q,r}}= \left\|h_1-\hat{h}\right\|_{Z_{q,r}}= \left\|h_1-\hat{h}\right\|_{L^q}.
  \end{align*}
  Therefore, for any given $\epsilon_3>0$, we may choose $\epsilon_2>0$ so small that
  \begin{align*}
   \left\|\left( h_t^{(1)}-\hat{h},k_t^{(1)} \right) \right\|_{Y_{q,r}}< \epsilon_3.
  \end{align*}
  Inductively, we define the tuple
  \begin{align*}
 	\left( h_t^{(i+1)}, k_t^{(i+1)}\right)
 		:=\psi\left(h^{(i)},k^{(i)}\right).
  \end{align*}
    Now we claim that we can choose $\epsilon_2$ and $\epsilon_3 $ so small that
\begin{align*}
	\left\|\left(h_t^{(i)}-\hat{h},k_t^{(i)}\right)\right\|_{Y_{q,r}}
		< \epsilon_3
\end{align*}  
for all $i\in\N$. We prove this by induction on $i$. The claim obviously holds for $i=1$. 
Recall that by definition of the norms in Subsection \ref{subsec : Banach space}, we have
\begin{align*}
\left\| k\right\|_{X_{q,r}}\leq \left\|(h-\hat{h},k)\right\|_{Y_{q,r}}.
\end{align*}
Thus by Theorem \ref{thm : psimapping}, there exists an $\epsilon_4$ such that the estimate
\begin{align}\label{eq : psimapping}
	\left\| \psi(h,k)-(\hat{h},0)\right\|_{Y_{q,r}}
		\leq C_1 \left( \left\|k_1\right\|_{W^{2,q}}+\left\|k_1\right\|_{W^{2,\infty}}+\left\|h_1-\hat{h}\right\|_{L^q}+\left\|(h-\hat{h},k)\right\|_{Y_{q,r}}^2 \right)
\end{align}
holds for some constant $C_1>0$, as long as $\left\|(h-\hat{h},k)\right\|_{Y_{q,r}}<\epsilon_4$. If we choose $\epsilon_2 $ and $\epsilon_3$ such that 
\begin{align*}
\epsilon_2<\frac{\epsilon_3}{2C_1},\qquad \epsilon_3< \min\left\{\frac{1}{2C_1},{\epsilon_4}\right\},
\end{align*}
then the induction assumption implies
\begin{align*}
	\left\|\left(h^{(i)}-\hat{h},k^{(i)}\right) \right\|_{Y_{q,r}}<\epsilon_3<\epsilon_4.
\end{align*}
Now \eqref{eq : psimapping} implies
\begin{align*}
	\left\| \left( h^{(i+1)}-\hat{h},k^{(i+1)}\right) \right\|_{Y_{q,r}}
		&=\left\|\psi \left( h^{(i)},k^{(i)}\right)-\left(\hat{h},0 \right) \right\|_{Y_{q,r}}\\
		&\leq C_1\left( \left\|k_1\right\|_{W^{2,q}}+\left\|k_1\right\|_{W^{2,\infty}}+\left\|h_1-\hat{h}\right\|_{L^q}+\left\| \left( h^{(i)}-\hat{h},k^{(i)} \right) \right\|_{Y_{q,r}}^2 \right) \\
		&\leq C_1(\epsilon_2+\epsilon_3^2)\leq C_1\epsilon_2+(C_1\epsilon_3)\epsilon_3\leq \frac{\epsilon_3}{2}+\frac{\epsilon_3}{2}=\epsilon_3
\end{align*}
and the claim is shown by induction.\\
Now due to Theorem \ref{thm : psicontracting}, there exists an $\epsilon_5>0$ such that
the map $\psi$ satisfies the estimate
\begin{align*}
	\left\|\psi(h,k)-\psi(\tilde{h},\tilde{k})\right\|_{Y_{q,r}}
		&\leq C_2\left(\left\|k_1\right\|_{W^{2,q}}+\left\|k_1\right\|_{W^{2,r}}\right)\left\|(h-\tilde{h},k-\tilde{k})\right\|_{Y_{q,r}} \\
		&\qquad +C_2\left(\left\|(h-\hat{h},k)\right\|_{Y_{q,r}}+\left\|(\tilde{h}-\hat{h},\tilde{k})\right\|_{Y_{q,r}}\right)\left\|(h-\tilde{h},k-\tilde{k})\right\|_{Y_{q,r}}
\end{align*}
as long as $\left\|(h-\hat{h},k)\right\|_{Y_{q,r}}+\left\|(\tilde{h}-\hat{h},\tilde{k})\right\|_{Y_{q,r}}<\epsilon_5$. If we now choose $\epsilon_2$ and $\epsilon_3$ so small that
\begin{align*}
C_2\epsilon_2+2C_2\epsilon_3\leq \frac{1}{2},\qquad {\epsilon_3}\leq \frac{\epsilon_5}{2},
\end{align*}
we obtain
\begin{align*}
	&\left\|\left(h^{(i+2)}-h^{(i+1)},k^{(i+2)}-k^{(i+1)}\right)\right\|_{Y_{q,r}} \\
	&\qquad\leq C_2\left(\left\|k_1\right\|_{W^{2,q}}+\left\|k_1\right\|_{W^{2,r}}
+\left\|\left(h^{(i+1)}-\hat{h},k^{(i+1)}\right)\right\|_{Y_{q,r}}+\left\|\left(h^{(i)}-\hat{h},k^{(i)}\right)\right\|_{Y_{q,r}}
\right)\\
	&\qquad\qquad\cdot\left\|\left(h^{(i+1)}-h^{(i)},k^{(i+1)}-k^{(i)}\right)\right\|_{Y_{q,r}}\\
	&\qquad\leq (C_2\epsilon_2+2C_2\epsilon_3)\left\|\left(h^{(i+1)}-h^{(i)},k^{(i+1)}-k^{(i)}\right)\right\|_{Y_{q,r}}\\
	&\qquad \leq \frac{1}{2}\left\|\left(h^{(i+1)}-h^{(i)},k^{(i+1)}-k^{(i)}\right)\right\|_{Y_{q,r}}
\end{align*}
for all $i\in\N$. 
Thus by induction, the sequence
\begin{align*}
\left\{\left(h^{(i)},k^{(i)}\right)\right\}_{i\in\N}
\end{align*}
is a Cauchy sequence in $Y_{q,r}$. By construction, it converges to an element $\left(h^{(\infty)},k^{(\infty)}\right)\in Y_{q,r}$ which satisfies
\begin{align*}
	\psi\left(h^{(\infty)},k^{(\infty)}\right)=\left(h^{(\infty)},k^{(\infty)}\right)
\end{align*}
and thus is by construction the (unique) fixed point of $\psi$. In addition, if for the $\epsilon>0$ given in the statement of the Theorem, $\epsilon_3$ is chosen so small that $\epsilon_3<\epsilon$, we get
\begin{align*}
	\left\|\left(h^{(\infty)}-\hat{h},k^{(\infty)}\right)\right\|_{Y_{q,r}}
		\leq \lim_{i\to\infty}\left\|\left(h^{(i)}-\hat{h},k^{(i)}\right)\right\|_{Y_{q,r}}
		\leq \epsilon_3<\epsilon.
\end{align*}
By the discussion in Section \ref{sec : RdT mixed evolution}, $\left(g^{(\infty)}_t=h^{(\infty)}_t+k^{(\infty)}_t\right)_{t\geq1}$ (which is for each fixed time an element in $W^{2,q}\cap W^{2,r}\subset W^{1,\infty}$) is a (weak) solution of the Ricci-de Turck flow with moving gauge, starting at $g_1$. On the other hand, a solution of the Ricci-de Turck flow with moving gauge starting at $g_1$ is uniquely obtained by solving the $\hat{h}$-gauged Ricci-de Turck flow and pulling back by a suitable family of diffeomorphisms. By construction, the resulting flow $(g_t)_{t\geq1}$ is $W^{2,q}\cap W^{2,r}$-close to $\hat{h}$ at least for small times $[t,t+\epsilon]$. By uniqueness, $g_t=g^{(\infty)}_t$ as long as $g_t$ does not leave a small neighbourhood. A bootstrapping argument then implies that $g_t=g^{(\infty)}_t$ for all time which finishes the proof.
\end{proof}

\subsection{Optimal convergence rates of the modified Ricci-de Turck flow}\label{subsection : optimal convergence}
Throughout this subsection, let $q\in (1,n)$, $r\in (n,\infty)$ be H\"{o}lder exponents satisfying 
\begin{align*}
\frac{n}{2}\left(\frac{1}{q}-\frac{1}{r}\right) > \frac{1}{2},\qquad
\frac{n}{2}\left(\frac{1}{q}-\frac{1}{r}\right) \neq 1.
\end{align*}
Fix $\epsilon>0$ and $\delta>0$ so small that Theorem \ref{thm : fixed point} applies: The modified Ricci-de Turck flow $g_t$ starting at an initial metric $g_0$ $\delta$-close to $\hat{h}$ in $L^{[q,\infty]}$ is well-defined, exists for all $t\geq0$ and for $t\geq1$ satisfies 
\begin{align*}
	\left\|\left(h_t-\hat{h},k_t\right)\right\|_{Y_{q,r}}<\epsilon,
\end{align*}
with $
h_t:=\Phi(g_t), k_t:=g_t-h_t
$.

\begin{prop}\label{prop : improved rates}
If we have
	\begin{align*}
	\left\|g_0-\hat{h}\right\|_{L^p}<\infty,
	\end{align*}
	for some $p\in (1,q)$, then  $(h-\hat{h},k)\in Y_{p,r'}$ for every $r'\in [r,\infty)$.
\end{prop}
\begin{proof}
We are first going to show that  $(h-\hat{h},k)\in Y_{p,r}$. For this purpose, let $h_1,k_1$ be as in the previous proof. Recall that we have
  \begin{align*}
\left\|k_1\right\|_{W^{2,q}}+  \left\|k_1\right\|_{W^{2,r}}+  \left\|k_1\right\|_{W^{2,\infty}}+ \left\|h_1-\hat{h}\right\|_{L^q}<\epsilon_2
\end{align*}
and that $\left\|g_t-\hat{h}\right\|_{L^{\infty}}\leq \epsilon_0$ for all $t\in [0,1]$.
Therefore, short-time estimates under the Ricci-de Turck flow (Lemma \ref{lem : Ricci flow short-time estimates}) yield
  \begin{align*}
\left\|g_1-\hat{h}\right\|_{W^{2,p}}<\infty
\end{align*}
and therefore,
\begin{align*}
\left\|k_1\right\|_{W^{2,p}}\leq \left\|g_1-\hat{h}\right\|_{W^{2,p}}+\left\|h_1-\hat{h}\right\|_{W^{2,p}}<\infty.
\end{align*}
By interpolation, $\left\|k_1\right\|_{W^{2,q'}}<\infty$ for all $q'\in [p,q]$.
	Let now $q=q_0>q_1>\ldots >q_{N}=p$ be a finite sequence of H\"{o}lder exponents satisfying
\begin{align*}
	\frac{n}{2}\left(\frac{1}{q_i}-\frac{1}{q_{i-1}}\right)
\leq \min\left\{\frac{n}{2r},n\left(\frac{1}{q}-\frac{1}{r}\right)-1,\frac{1}{2}\right\},
\end{align*}
for all $i\in\left\{1,\ldots,N\right\}$.
As $q\geq q_{i-1}$,
\begin{align*}
	\frac{n}{2}\left(\frac{1}{q_i}-\frac{1}{q_{i-1}}\right)
		\leq \min\left\{\frac{n}{2r},n\left(\frac{1}{q_{i-1}}-\frac{1}{r}\right)-1,\frac{1}{2}\right\}.
\end{align*}
A repetitive application of Proposition \ref{psi1mapping} and Proposition \ref{psi2mapping} applied to $h=\psi_1(h,k)$ and $k=\psi_2(h,k)$ yields
		\begin{align*}
	\left\|k\right\|_{X_{q_{i},r}}
		&\leq C
			\left[\left\|k_1 \right\|_{W^{2,q_{i}}}+\left\|k_1 \right\|_{W^{2,r}}
+\left(\left\|k\right\|_{X_{q_{i-1},r}}+\left\|h-\hat{h}\right\|_{Z_{q_{i-1},r}}\right)\left\|k\right\|_{X_{q_{i-1},r}}\right],\\
\left\|h-\hat{h}\right\|_{Z_{q_i,r}}&\leq C\left(\left\|h_1-\hat{h}\right\|_{L^{q_i}}+\left\|k\right\|_{X_{q_i,r}}^2\right)
\leq C\left(\left\|h_1-\hat{h}\right\|_{L^{q}}+\left\|k\right\|_{X_{q_i,r}}^2\right),
\end{align*}
and after a finite number of steps, we obtain
\begin{align*}
\left\|(h-\hat{h},k)\right\|_{Y_{p,r}}= \left\|h-\hat{h}\right\|_{Z_{p,r}}+\left\|k\right\|_{X_{p,r}}<\infty,
\end{align*} 
as desired.

Now we are going to show that $(h-\hat{h},k)\in Y_{p,r'}$ for any $r'\in [r,\infty)$. 
The argumentation is similar to the above but slightly more involved.
Pick a finite sequence $r=r_0<r_1<\ldots <r_N=r'$ such that
\begin{align*}
	\min\left\{n\left(\frac{1}{p}-\frac{1}{r}\right)-1,\frac{1}{2}\right\}>\frac{n}{2}\left(\frac{1}{r_{i-1}}-\frac{1}{r_i}\right).
\end{align*}
Note that this also implies 
\begin{align*}
	\min\left\{n\left(\frac{1}{p}-\frac{1}{r_{i-1}}\right)-1,\frac{1}{2}\right\}
		> \frac{n}{2}\left(\frac{1}{r_{i-1}}-\frac{1}{r_i}\right),
\end{align*}
as $r_{i-1}\geq r$.
By interpolation,
\begin{align*}
\left\|k_1 \right\|_{W^{2,r_i}}<\infty.
\end{align*}
Let us first show that
\begin{align*}
	\left\|k\right\|_{X_{p,r_i}}
		&\leq C
\left(\left\|k_1 \right\|_{W^{2,p}}+\left\|k_1 \right\|_{W^{2,r_i}}\right)\\
&\qquad 
+
C\left(\left\|k_1 \right\|_{W^{2,r_i}}+\left\|k\right\|_{X_{p,r_{i-1}}}+\left\|h-\hat{h}\right\|_{Z_{p,r_{i-1}}}\right)
\left(\left\|k_1 \right\|_{W^{2,r_i}}+ \left\|k\right\|_{X_{p,r_{i-1}}}\right).
\end{align*}
To do so we estimate for $k=\psi_2(h,k)$, using Lemma \ref{lem : projection and Banach space} and the triangle inequality,
\begin{align*}
	\left\|{\psi}_2(h,k)\right\|_{X_{p,r_i}}
		&\leq C\left\|\overline{\psi}_2(h,k)\right\|_{X_{p,r_i}} \\
		&\leq C\left\| P(g,h,h_{\infty})_{1\to t}(\Pi^{\perp}_{h_\infty}(k_1))\right\|_{X_{p,r_i}}\\
		&\qquad+ C\left\|\int_1^{\max\left\{1,t-2\right\}}  P(g,h,h_{\infty})_{s\to t}I(t,s,k,h,h_{\infty})ds\right\|_{X_{p,r_i}}\\
		&\qquad + C\left\|	\int_{\max\left\{1,t-2\right\}}^t P(g,h,h_{\infty})_{s\to t}I(t,s,k,h,h_{\infty})ds\right\|_{X_{p,r_i}}.
\end{align*}
Lemma \ref{lem : linear part} yields
\begin{align*}
\left\| P(g,h,h_{\infty})_{1\to t}(\Pi^{\perp}_{h_\infty}(k_1))\right\|_{X_{p,r_i}}
\leq C(\left\|k_1 \right\|_{W^{2,p}}+\left\|k_1 \right\|_{W^{2,r_i}})
\end{align*}
and Lemma \ref{lem : nonlinear part 1} yields
\begin{align*}
	&\left\|\int_1^{\max\left\{1,t-2\right\}} P(g,h,h_{\infty})_{s\to t}I(t,s,k,h,h_{\infty})ds\right\|_{X_{p,r_i}} \\
		&\qquad \qquad \leq C\left(\left\|k\right\|_{X_{p,r_{i-1}}}+\left\|h-\hat{h}\right\|_{Z_{p,r_{i-1}}}\right)\left\|k\right\|_{X_{p,r_{i-1}}}.
\end{align*}
Now due to the definition of the norms, we have for
\begin{align*}
(A):=	\int_{\max\left\{1,t-2\right\}}^t P(g,h,h_{\infty})_{s\to t}I(t,s,k,h,h_{\infty})ds
\end{align*}
the inequality
\begin{align*}
	\left\|(A)\right\|_{X_{p,r_i}}
		\leq \left\|	(A)\right\|_{X_{p,r_{i-1}}}+\sup_{t\geq 1}t^{\frac{n}{2}\left(\frac{1}{p}-\frac{1}{r_{i}}\right)}\left\|(A)\right\|_{W^{2,r_{i}}}.
\end{align*}
From Lemma \ref{lem : nonlinear part 2}, we know already that 
\begin{align*}
	\left\|(A)\right\|_{X_{p,r_{i-1}}}
		\leq C\left(\left\|k\right\|_{X_{p,r_{i-1}}}+\left\|h-\hat{h}\right\|_{Z_{p,r_{i-1}}}\right)\left\|k\right\|_{X_{p,r_{i-1}}},
\end{align*}
and thus it suffices to consider the second term. From Lemma \ref{lem : short-time estimates} (iii) and (iv), and Lemma \ref{lem : integral term}, we get
\begin{equation}\begin{split}\label{eq: estimate (A)}
	\left\|(A)\right\|_{W^{2,r_{i}}}
&\leq 	\int_{\max\left\{1,t-1\right\}}^t  \left\| P(g,h,h_{\infty})_{s\to t}I(t,s,k,h,h_{\infty})\right\|_{W^{2,r_{i}}} ds\\
&\qquad+\int_{\max\left\{1,t-2\right\}}^{\max\left\{1,t-1\right\}}  \left\| P(g,h,h_{\infty})_{s\to t}I(t,s,k,h,h_{\infty})\right\|_{W^{2,r_{i}}} ds\\
&\leq 	C\int_{\max\left\{1,t-1\right\}}^t  (t-s)^{-1/2}
\left\|I(t,s,k,h,h_{\infty})\right\|_{W^{1,r_{i}}}ds\\
&\qquad +C\int_{\max\left\{1,t-2\right\}}^{\max\left\{1,t-1\right\}}    (t-s)^{-1}
\left\|I(t,s,k,h,h_{\infty})\right\|_{L^{r_{i}}}ds\\
		&\leq C\sup_{s\in [\max\left\{1,t-1\right\},t]} \left\|I(t,s,k,h,h_{\infty})\right\|_{W^{1,r_{i}}}\\
		&\qquad + C\sup_{s\in [\max\left\{1,t-2\right\},\max\left\{1,t-1\right\}]} \left\|I(t,s,k,h,h_{\infty})\right\|_{L^{r_{i}}} \\
		&\leq C\sup_{s\in [\max\left\{1,t-2\right\},t]}(\left\|k\right\|_{W^{2,r_{i}}}+\left\|h-h_{\infty}\right\|_{L^p})\left\|k\right\|_{W^{2,r_{i}}}.
\end{split}
\end{equation}
Let us distinguish between large times and small times. 
For $t\leq 4$, the short-time estimates in Lemma \ref{lem : Ricci flow moving gauge short-time estimates} (ii) yield 
\begin{align*}
	&\sup_{s\in [\max\left\{1,t-2\right\},t]}(\left\|k\right\|_{W^{2,r_{i}}}+\left\|h-h_{\infty}\right\|_{L^p})\left\|k\right\|_{W^{2,r_{i}}}\\
&\qquad\qquad\leq 		
		C\left(\left\|k_1\right\|_{W^{2,r_{i}}}+\sup_{s\in [1,t]}\left\|k\right\|_{W^{2,r_{i-1}}}+ \left\|h-h_{\infty}\right\|_{L^p}\right) 	\left(\left\|k_1\right\|_{W^{2,r_{i}}}+\sup_{s\in [1,t]}\left\|k\right\|_{W^{2,r_{i-1}}}\right)	\\	
&\qquad\qquad\leq C\left( \left\|k_1\right\|_{W^{2,r_{i}}}+
\left\|k\right\|_{X_{p,r_{i-1}}} +
\left\|h-\hat{h}\right\|_{Z_{p,r_{i-1}}} \right)\left(
 \left\|k_1\right\|_{W^{2,r_{i}}}+
\left\|k\right\|_{X_{p,r_{i-1}}}\right).
\end{align*}
Now let us consider large times $t\geq 4$.
By  using Sobolev embedding and applying Lemma \ref{lem : Ricci flow moving gauge short-time estimates} (i), we have
\begin{equation}\begin{split}\label{eq : estimate (A) part 1}
	\sup_{s\in [t-2,t]}\left\|k\right\|_{W^{2,r_{i}}}
		&\leq	\sup_{s\in [t-2,t]} C\left\|k\right\|_{W^{3,r_{i-1}}}\\
		&\leq 	\sup_{s\in [t-3,t]} C\left\|k\right\|_{W^{2,r_{i-1}}}
 		\leq C t^{-\frac{n}{2}\left(\frac{1}{p}-\frac{1}{r_{i-1}}\right)}\left\|k\right\|_{X_{p,r_{i-1}}}.
\end{split}
\end{equation}
Since $p<r_{i-1}$, we conclude
\begin{align}\label{eq : estimate (A) part 2a}
	\sup_{s\in [t-2,t]}\left\|k\right\|_{W^{2,r_{i}}}
		\leq C t^{-\frac{n}{2}\left(\frac{1}{p}-\frac{1}{r_{i-1}}\right)} \left\|k\right\|_{X_{p,r_{i-1}}}
		\leq  C t^{-\frac{n}{2}\left(\frac{1}{r_{i-1}}-\frac{1}{r_{i}}\right)}\left\|k\right\|_{X_{p,r_{i-1}}}
\end{align}
and since $n\left(\frac{1}{p}-\frac{1}{r_{i-1}}\right)-1>\frac{n}{2}\left(\frac{1}{r_{i-1}}-\frac{1}{r_{i}}\right)$, we also get
\begin{align}\label{eq : estimate (A) part 2b}
\sup_{s\in[1,t]}	\left\|h-h_{\infty}\right\|_{L^p}
		\leq C t^{1-n\left(\frac{1}{p}-\frac{1}{r_{i-1}}\right)}\left\|h-\hat{h}\right\|_{Z_{p,r_{i-1}}}
		\leq C t^{-\frac{n}{2}\left(\frac{1}{r_{i-1}}-\frac{1}{r_{i}}\right)} \left\|h-\hat{h}\right\|_{Z_{p,r_{i-1}}}.
\end{align}
Applying the estimates \eqref{eq : estimate (A) part 1}, \eqref{eq : estimate (A) part 2a} and 
\eqref{eq : estimate (A) part 2b} to \eqref{eq: estimate (A)}, we get
\begin{align*}
	\left\|(A)\right\|_{W^{2,r_{i}}}
		\leq Ct^{-\frac{n}{2}\left(\frac{1}{p}-\frac{1}{r_{i}}\right)}\left( \left\|k\right\|_{X_{p,r_{i-1}}} + \left\|h-\hat{h}\right\|_{Z_{p,r_{i-1}}}\right) \left\|k\right\|_{X_{p,r_{i-1}}}.
\end{align*}
Thus we have shown
\begin{align*}
	\left\|k\right\|_{X_{p,r_i}}
		&\leq C\left( \left\|k_1\right\|_{W^{2,r_{i}}}+
\left\|k\right\|_{X_{p,r_{i-1}}} +
\left\|h-\hat{h}\right\|_{Z_{p,r_{i-1}}} \right)\left(
 \left\|k_1\right\|_{W^{2,r_{i}}}+
\left\|k\right\|_{X_{p,r_{i-1}}}\right)<\infty
\end{align*}
and from Proposition \ref{psi1mapping} and Lemma \ref{lem : elliptic regularity}, we have
\begin{align*}
\left\|h-\hat{h}\right\|_{Z_{p,r_{i}}}&\leq C\left(\left\|h-\hat{h}\right\|_{L^{p}}+\left\|k\right\|_{X_{p,r_{i}}}^2\right)
\leq C\left(\left\|h-\hat{h}\right\|_{L^{q}}+\left\|k\right\|_{X_{p,r_{i}}}^2\right)<\infty.
\end{align*}
Thus after a finite number of steps, we get
\begin{align*}
\left\|(h-\hat{h},k)\right\|_{Y_{p,r'}}= \left\|h-\hat{h}\right\|_{Z_{p,r'}}+\left\|k\right\|_{X_{p,r'}}<\infty,
\end{align*} 
as desired
\end{proof}
\begin{cor}\label{prop Ck convergence rates}
Suppose there exists $p\in (1,q)$ such that $\left\|g_0-\hat{h}\right\|_{L^p}<K$ for some  constant $K<\infty$. Then
for each $\tau>0$ there exists a constant $C=C(K,p,\epsilon,\tau)$ such that
\begin{align*}
	\left\|h_t-h_{\infty}\right\|_{C^m}\leq C\cdot t^{1-\frac{n}{p}+\tau},\qquad \left\|k_t\right\|_{C^m}\leq C\cdot t^{-\frac{n}{2p}+\tau}.
\end{align*}
\end{cor}
\begin{proof}
	We get from Proposition \ref{prop : improved rates} that 
	\begin{align*}
	\left\|(h_t-\hat{h},k_t)\right\|_{Y_{p,r'}}<\infty,
	\end{align*}
	for all $r'\in [r,\infty)$. Let $\tau>0$ be given and choose 
	 $r'$ so large that $\frac{n}{2r'}<\frac{n}{r'}<\tau$.
	Then we get
	\begin{align*}
	\left\|h_t-h_{\infty}\right\|_{C^m}
		\leq C \left\|h_t-h_{\infty}\right\|_{L^{p}}&\leq C\int_t^{\infty}\left\|\partial_sh\right\|_{L^{p}}ds\\
		&\leq C \int_t^{\infty}s^{-n\left(\frac{1}{p}-\frac{1}{r'}\right)}ds\left\|h-\hat{h}\right\|_{Z_{p,r'}}\\
		&\leq C t^{1-n\left(\frac{1}{p}-\frac{1}{r'}\right)}\left\|h-\hat{h}\right\|_{Z_{p,r'}}\\
		&\leq Ct^{1-\frac{n}{p}+\tau}\left\|h-\hat{h}\right\|_{Z_{p,r'}}.
	\end{align*}
For proving the second estimate, we distinguish between small and large times. Consider first $t\in [1,2]$.
Because $g_t$, $t\in[0,1]$ evolves under the Ricci-de Turck flow with fixed gauge, Lemma \ref{lem : Ricci flow short-time estimates} yields $\left\|k_{1}\right\|_{W^{m+1,r'}}<\infty$. Thus, Sobolev embedding and Lemma \ref{lem : Ricci flow moving gauge short-time estimates} (ii) yield
	\begin{align*}
	\left\|k_t\right\|_{C^m}\leq C\left\|k_{t}\right\|_{W^{m+1,r'}}&\leq C(\left\|k_{1}\right\|_{W^{m+1,r'}}+ \sup_{s\in [1,t]}\left\|k_{s}\right\|_{W^{2,r'}})\\
	&\leq  C(\left\|k_{1}\right\|_{W^{m+1,r'}}+ \left\|k\right\|_{X_{p,r'}}).
	\end{align*}	
For $t>2$	we use Sobolev embedding and Lemma \ref{lem : Ricci flow moving gauge short-time estimates} (i) to get
	\begin{align*}
	\left\|k_t\right\|_{C^m}\leq C\left\|k_{t}\right\|_{W^{m+1,r'}}\leq C\sup_{s\in [t-1,t]}\left\|k_{s}\right\|_{W^{2,r'}}\leq 
	C t^{-\frac{n}{2}\left(\frac{1}{p}-\frac{1}{r'}\right)}\left\|k\right\|_{X_{p,r'}}\leq Ct^{-\frac{n}{2p}+\tau}\left\|k\right\|_{X_{p,r'}},
	\end{align*}
	and the proof of the theorem is finished.
\end{proof}
For proving a refinement of the decay of $k_t$, let us fix some time $t\geq 1$.
Define
\begin{align*}
\overline{k}_s=\Pi^{\perp}_{h_s,h_t}(k_s),\qquad s\in [1,t].
\end{align*}
By Proposition \ref{prop : alternative mod RdTf}, the evolution on $\overline{k}$ can be written as
\begin{align*}\partial_s\overline{k}_s+\Delta_{L,h_t}\overline{k}_s=\Pi^{\perp}_{h_t}[(\Delta_{L,h_t}-\Delta_{L,h_s})(k_s)+
(1-D_{g_s}\Phi)(H_1(s))].
\end{align*}
Observe also that $\overline{k}_t=\Pi^{\perp}_{h_t,h_t}(k_t)=k_t$. Therefore, by the Duhamel principle
we get an alternative formula for $k_t$ which is
\begin{equation}\label{eq : modified expression k_t}
	\begin{split}
k_t=\overline{k}_t&=e^{-(t-1)\Delta_{L,h_t}}\overline{k}_1\\
&\qquad+
\int_1^{t}e^{-(t-s)\Delta_{L,h_t}}\Pi^{\perp}_{h_t}[(\Delta_{L,h_t}-\Delta_{L,h_s})(k_s)+
(1-D_{g_s}\Phi)(H_1(s))]ds.
\end{split}
\end{equation}
To obtain estimates for this expression, we have to derive estimates on the integrand. The estimates  in Lemma \ref{lem : nonlinear part 0.1} and Lemma \ref{lem : nonlinear part 0.2} will enable us to control the part of the integral from $1$ to $t-1$. To treat also the part from $t-1$ to $t$, we need another lemma.
\begin{lem}\label{lem : nonlinear part 0.3}
Suppose there exists $p\in (1,q)$ such that $\left\|g_0-\hat{h}\right\|_{L^p}<\infty$ .
Then for  $r'\in[r,\infty)$, $q'\in [p,r']$ and $i\in\N_0$, there exists a constant $C(p,q',r',i,\epsilon)$ such that for all $t\geq 3$, we have
\begin{align*}
	&\sup_{s\in [t-1,t]}\left\| (\Delta_{L,h_t}-\Delta_{L,h_s})(k_s)+(1-D_{g_s}\Phi)(H_1(s))\right\|_{W^{i,q'}}\\
		&\qquad\qquad\leq Ct^{-\frac{n}{2}\left(\frac{1}{p}-\frac{1}{r'}\right)-\frac{n}{2}\left(\frac{1}{p}-\frac{1}{q'}\right)} \left(\left\|h-\hat{h}\right\|_{Z_{p,r'}}+\left\|k\right\|_{X_{p,r'}}\right)\left\|k\right\|_{X_{p,r'}}.
\end{align*}
\end{lem}
\begin{proof}
By Lemma \ref{lem : projection lemma} (iii) and the H\"{o}lder inequality applied to \eqref{eq : difference LL} and \eqref{eq : RdT1}, we get
\begin{align*}
	&\sup_{s\in [t-1,t]}\left\| (\Delta_{L,h_t}-\Delta_{L,h_s})(k_s)+
(1-D_{g_s}\Phi)(H_1(s))\right\|_{W^{i,q'}}\\
&\qquad\qquad\leq C\sup_{s\in [t-1,t]} \left(
\left\| (\Delta_{L,h_t}-\Delta_{L,h_s})(k_s)\right\|_{W^{i,q'}}
+\left\|H_1(s)\right\|_{W^{i,q'}}\right)\\
		&\qquad\qquad\leq C\sup_{s\in [t-1,t]} \left(\left\|h_t-h_s\right\|_{W^{i+2,\infty}}+\left\|k_s\right\|_{W^{i+2,\infty}}\right)\left\|k_s\right\|_{W^{i+2,q'}}.
\end{align*}
From the end of the proof of Corollary \ref{prop Ck convergence rates}, we already know that
\begin{align*}
	\sup_{s\in [t-1,t]}\left\|k_s\right\|_{W^{i+2,\infty}}
		\leq Ct^{-\frac{n}{2}\left(\frac{1}{p}-\frac{1}{r'}\right)} \left\|k\right\|_{X_{p,r'}}.
\end{align*}
In addition,
\begin{align*}
	\sup_{s\in [\max\left\{1,t-1\right\},t]}\left\|h_t-h_s\right\|_{W^{i+2,\infty}}
		&\leq C\sup_{s\in [t-1,t]}\int_{s}^{t}\left\|\partial_{s'}h\right\|_{L^p}ds'\\
		&\leq C\sup_{s\in [t-1,t]}\int_s^t(s')^{-n\left(\frac{1}{p}-\frac{1}{r'}\right)}ds'\cdot\left\|h-\hat{h}\right\|_{Z_{p,r'}}\\
		&\leq Ct^{-n\left(\frac{1}{p}-\frac{1}{r'}\right)}\left\|h-\hat{h}\right\|_{Z_{p,r'}}\\
		&\leq Ct^{-\frac{n}{2}\left(\frac{1}{p}-\frac{1}{r'}\right)}\left\|h-\hat{h}\right\|_{Z_{p,r'}}.
\end{align*}
Let $\theta\in [0,1]$ be defined by the equation 
\begin{align*}
	\frac{1}{p}-\frac{1}{q'}
		=\theta \left(\frac{1}{p}-\frac{1}{r'}\right).
\end{align*}
By interpolation, the short-time estimates in Lemma \ref{lem : Ricci flow moving gauge short-time estimates} (i) and the definition of the $X_{p,r'}$-norm,
\begin{align*}
	&\sup_{s\in [t-1,t]}\left\|k_s\right\|_{W^{i+2,q'}}\leq C\sup_{s\in [t-1,t]}\left\|k_s\right\|_{W^{i+2,p}}^{1-\theta}\left\|k_s\right\|_{W^{i+2,r'}}^{\theta}\\
&\leq C\left(\left\|k_{t-2}\right\|_{L^{p}}+\sup_{s\in [t-2,t]}\left\|k_s\right\|_{W^{2,p}}\right)^{1-\theta}
\left(\left\|k_{t-2}\right\|_{L^{r'}}+\sup_{s\in [t-2,t]}\left\|k_s\right\|_{W^{2,r'}}\right)^{\theta}\\	
%	
%		&\leq C\left\|k_{t-2}\right\|_{L^{p}}^{1-\theta}\left\|k_{t-2}\right\|_{L^{r'}}^{\theta}\\
		&\leq C\left\|k\right\|_{X_{p,r'}}^{1-\theta}\left( (t-2)^{-\frac{n}{2}\left(\frac{1}{p}-\frac{1}{r'}\right) }\left\|k\right\|_{X_{p,r'}}\right)^{\theta}
		\\
		&\leq C t^{-\frac{n}{2}\left(\frac{1}{p}-\frac{1}{q'}\right) }\left\|k\right\|_{X_{p,r'}},
\end{align*}
and the statement follows from putting these estimates together.
\end{proof}
\begin{prop}\label{prop : L^p conv rates}
Suppose there exists $p\in (1,q)$ such that $\left\|g_0-\hat{h}\right\|_{L^p}<\infty$.
Let  $q'\in [p,\infty)$ and $i\in\N$.
\begin{itemize}
	\item[(i)] If $\frac{n}{2}\left(\frac{1}{p}-\frac{1}{q'}\right)+\frac{i}{2}<\frac{n}{2p}$ there exists a constant $C = C(p, q',i,K,\epsilon)$ such that for all $t\geq 1$, we have
	\begin{align}\label{eq : L^p conv rates 1}
	\norm{\nabla^i k_t}_{L^{q'}} \leq Ct^{-\frac n2 \left(\frac1p - \frac{1}{q'}\right)-\frac{i}{2}}
	\end{align}
	\item[(ii)] If $\frac{n}{2}\left(\frac{1}{p}-\frac{1}{q'}\right)+\frac{i}{2}\geq\frac{n}{2p}$ for each  $\tau>0$ there exists  a constant $C = C(p, q',i,\tau,K,\epsilon)$ such that for all $t\geq 1$, we have
	\begin{align}\label{eq : L^p conv rates 2}
	\norm{\nabla^i k_t}_{L^{q'}} \leq Ct^{-\frac{n}{2p}+\tau}.
	\end{align}
\end{itemize}
\end{prop}
\begin{proof}
For $t\in [1,3]$, the assertion follows from the short-time estimates in Lemma \ref{lem : Ricci flow moving gauge short-time estimates}. Therefore, we may assume $t\geq 3$ from now on.
	We have, using \eqref{eq : modified expression k_t} and the triangle inequality,		
	\begin{align*}
	&\left\|\nabla^ik_t\right\|_{L^{q'}}\leq \left\|\nabla^i\circ e^{-(t-1)\Delta_{L,h_t}}\overline{k}_1\right\|_{L^{q'}}\\
	&\qquad+
	\int_1^{t-1}\left\|\nabla^i\circ e^{-(t-s)\Delta_{L,h_t}}\Pi^{\perp}_{h_t}[(\Delta_{L,h_t}-\Delta_{L,h_s})(k_s)+
	(1-D_{g_s}\Phi)(H_1(s))]\right\|_{L^{q'}}ds\\
	&\qquad+
	\int_{t-1}^{t}\left\|\nabla^i\circ e^{-(t-s)\Delta_{L,h_t}}\Pi^{\perp}_{h_t}[(\Delta_{L,h_t}-\Delta_{L,h_s})(k_s)+
	(1-D_{g_s}\Phi)(H_1(s))]\right\|_{L^{q'}}ds .
	\end{align*}
	Let $\alpha=\alpha(p,q',i,\tau)$ be defined by
\begin{align*}
\alpha(p,q',i,\tau)=\begin{cases}
\frac n2 \left(\frac1p - \frac{1}{q'}\right)+\frac{i}{2},& \text{ if }\frac{n}{2}\left(\frac{1}{p}-\frac{1}{q'}\right)+\frac{i}{2}<\frac{n}{2p},\\
\frac{n}{2p}-\tau, & \text{ if }\frac{n}{2}\left(\frac{1}{p}-\frac{1}{q'}\right)+\frac{i}{2}\geq\frac{n}{2p},
\end{cases}
\end{align*}	
so that it equals the decay rates in \eqref{eq : L^p conv rates 1} and \eqref{eq : L^p conv rates 2} in the respective cases. Choose $p'\in (1,p)$ small and $r'\in [r,\infty)$ large (where $r$ is as in Theorem \ref{thm : fixed point}).
	Furthermore, we denote $\alpha'=\alpha(p',q',i,\tau)>0$.
	Note that $\alpha'-\alpha=\frac{n}{2}(\frac{1}{p'}-\frac{1}{p})$ in all cases.
	 By Corollary \ref{cor : linear estimates} and Lemma \ref{lem : projection lemma} (ii), we already know
	\begin{align*}
	\left\|\nabla^i\circ e^{-(t-1)\Delta_{L,h_t}}\overline{k}_1\right\|_{L^{q'}}\leq Ct^{-\alpha}\left\|\overline{k}_1\right\|_{L^{p}}\leq Ct^{-\alpha}\left\|k_1\right\|_{L^{p}}.
	\end{align*}
Therefore, by Lemma  \ref{lem : nonlinear part 0.1}, we get
	\begin{align*}
	&\int_1^{t-1}\left\|\nabla^i\circ e^{-(t-s)\Delta_{L,h_t}}\Pi^{\perp}_{h_t}[(\Delta_{L,h_t}-\Delta_{L,h_s})(k_s)]\right\|_{L^{q'}}ds\\
	&\leq \int_1^{t-1}\left\|\nabla^i\circ e^{-(t-s)\Delta_{L,h_t}}\Pi^{\perp}_{h_t}\right\|_{L^{p'},L^{q'}}
	\left\|(\Delta_{L,h_t}-\Delta_{L,h_s})(k_s)\right\|_{L^{p'}}ds\\
	&\leq C  \int_1^{t-1}(t-s)^{-\alpha'}s^{1-\frac{3n}{2}(\frac{1}{p}-\frac{1}{r'})}ds\cdot \left(\left\|h-\hat{h}\right\|_{Z_{p,r'}}+\left\|k\right\|_{X_{p,r'}}\right)\left\|k\right\|_{X_{p,r'}}
\end{align*}
We have
\begin{align*}
	\alpha'
		&=\frac{n}{2}\left(\frac{1}{p'}-\frac{1}{p}\right)+\alpha>\alpha \\
	\frac{3n}{2}\left(\frac{1}{p}-\frac{1}{r'}\right)-1
		&>\frac{n}{2}\left(\frac{1}{p}-\frac{1}{r'}\right) 
		>\alpha \\
	\alpha'+\frac{3n}{2}\left(\frac{1}{p}-\frac{1}{r'}\right)-2
	&=\alpha+\frac{n}{2}\left( \frac{1}{p'}-\frac{1}{p} \right)+\frac{3n}{2} \left( \frac{1}{p}-\frac{1}{r'} \right)-2\\
	&=\alpha+\frac{n}{2}\left(\frac{1}{p'}-\frac{1}{r'}\right)+n\left(\frac{1}{p}-\frac{1}{r'}\right)-2
	>\alpha
\end{align*}
so that Lemma \ref{important_technical_lemma} implies 
\begin{align*}
	 \int_1^{t-1}(t-s)^{-\alpha'}s^{1-\frac{3n}{2}(\frac{1}{p}-\frac{1}{r'})}ds\leq Ct^{-\alpha}.
\end{align*}
	From Corollary \ref{cor : linear estimates} and Lemma \ref{lem : nonlinear part 0.2}, we have (with $\beta=\min\left\{1,\frac{n}{2}\left(\frac{1}{p}-\frac{1}{r'}\right) \right\}$) that
	\begin{align*}
	&\int_1^{t-1}\left\|\nabla^i\circ e^{-(t-s)\Delta_{L,h_t}}\Pi^{\perp}_{h_t}[
(1-D_{g_s}\Phi)(H_1(s))]\right\|_{L^{q'}}ds\\
	&\leq \int_1^{t-1}\left\|\nabla^i\circ e^{-(t-s)\Delta_{L,h_t}}\Pi^{\perp}_{h_t}\right\|_{L^p,L^{q'}}\left\|
(1-D_{g_s}\Phi)(H_1(s))\right\|_{L^p}ds\\
	&\leq C  \int_1^{t-1}(t-s)^{-\alpha}s^{-\beta-\frac{n}{2}\left(\frac{1}{p}-\frac{1}{r'}\right)}ds\cdot \left(\left\|h-\hat{h}\right\|_{Z_{p,r'}}+\left\|k\right\|_{X_{p,r'}}\right)\left\|k\right\|_{X_{p,r'}},
	\end{align*}
	provided that we have chosen $r'$ so large that $n\left(\frac{1}{p}-\frac{1}{r'}\right) > 1$. 
	Since any of the given $\alpha$ satisfies $\alpha<\frac{n}{2p}$, we may always choose $r'$ so large that $\alpha<\frac{n}{2}\left(\frac{1}{p}-\frac{1}{r'}\right) \leq \frac{n}{2}\left(\frac{1}{p}-\frac{1}{r'}\right)+\beta $. 
	In addition, we always have $\frac{n}{2}\left(\frac{1}{p}-\frac{1}{r'}\right)+\beta> 1$. Then we get from Lemma \ref{important_technical_lemma} that
	\begin{align*}
	\int_1^{t-1}(t-s)^{-\alpha}s^{-\beta-\frac{n}{2}\left(\frac{1}{p}-\frac{1}{r'}\right)}ds \leq Ct^{-\alpha}.
	\end{align*}
	Finally, we get from Lemma \ref{lem : nonlinear part 0.3} that
	\begin{align*}
		&\int_{t-1}^{t}\left\|\nabla^i\circ e^{-(t-s)\Delta_{L,h_t}}\Pi^{\perp}_{h_t}[(\Delta_{L,h_t}-\Delta_{L,h_s})(k_s)+
(1-D_{g_s}\Phi)(H_1(s))]\right\|_{L^{q'}}ds\\
		&\qquad \leq C\sup_{s\in [t-1,t]}\left\| (\Delta_{L,h_t}-\Delta_{L,h_s})(k_s)+
(1-D_{g_s}\Phi)(H_1(s))\right\|_{W^{i,q'}}\\
		&\qquad\leq Ct^{-\frac{n}{2}\left(\frac{1}{p}-\frac{1}{r'}\right)-\frac{n}{2}\left(\frac{1}{p}-\frac{1}{q'}\right)}
\left(\left\|h-\hat{h}\right\|_{Z_{p,r'}}+\left\|k\right\|_{X_{p,r'}}\right)\left\|k\right\|_{X_{p,r'}}\\
		&\qquad\leq Ct^{-\alpha}
\left( \left\|h-\hat{h}\right\|_{Z_{p,r'}}+\left\|k\right\|_{X_{p,r'}} \right) \left\|k\right\|_{X_{p,r'}},	
	\end{align*}
	because for given $\alpha$, $r'$ was chosen so large that
	\begin{align*}
	\frac{n}{2}\left(\frac{1}{p}-\frac{1}{r'}\right)+\frac{n}{2}\left(\frac{1}{p}-\frac{1}{q'}\right)
		\geq \frac{n}{2}\left(\frac{1}{p}-\frac{1}{r'}\right)
		\geq\alpha.
	\end{align*}
Putting all the estimates together, we get
\begin{align*}
	\left\|\nabla^ik_t\right\|_{L^{q'}}
		&\leq Ct^{-\alpha} \left[ \left\|k_1\right\|_{L^p} + \left(\left\|h-\hat{h}\right\|_{Z_{p,r'}}+\left\|k\right\|_{X_{p,r'}}\right)\left\|k\right\|_{X_{p,r'}}\right],
\end{align*}	
which proves the proposition.
\end{proof}
\begin{proof}[Proof of Theorem \ref{thm : mainthm 1 introduction}]
	For an initial metric $g_0$ which is $L^{[q,\infty]}$-close to $\hat{h}$, we denote by $\tilde{g}_t$ the modified Ricci-de Turck flow starting at $g_0=\tilde{g}_0$. We are first going to show that all the assertions of Theorem \ref{thm : mainthm 1 introduction} hold with $\phi_t^*g_t$ replaced by $\tilde{g}_t$. Let $C_1>0$ be the constant of Lemma \ref{lem : normcontrol}.	
	 For the given neighbourhood $\U$, choose $\epsilon>0$ so small that 
	 \begin{align*}
	 	\mathcal{B}^{[q,\infty]}_{C_1\cdot \epsilon}(\hat{h}):=\left\{g\in\M\mid \left\|g-\hat{h}\right\|_{L^q}+\left\|g-\hat{h}\right\|_{L^{\infty}}<C_1\cdot \epsilon\right\}\subset \U.
	 	\end{align*}
 	Now if $g_0\in \V:=\mathcal{B}^{[q,\infty]}_{\delta}(\hat{h})$, for a $\delta>0$ chosen small enough, Lemma \ref{lem : Ricci flow short-time estimates in Ck} and Lemma \ref{lem : Ricci flow short-time estimates} imply that we have
 	$\tilde{g}_t\in 	\mathcal{B}^{[q,\infty]}_{C_1\cdot \epsilon}(\hat{h})$ for all $t\in [0,1]$.
 	Moreover, by Theorem \ref{thm : fixed point}, $\tilde{g}_t$ exists for all time and for $t\geq1$, the tensors $h_t=\Phi(\tilde{g}_t)$ and $k_t=\tilde{g}_t-h_t$ satisfy
      \begin{align*}
    	\left\|(h_t-\hat{h},k_t)\right\|_{Y_{q,r}}<\epsilon,
    \end{align*}
which by Lemma \ref{lem : normcontrol} implies $\tilde{g}_t\in 	\mathcal{B}^{[q,\infty]}_{C_1\cdot \epsilon}(\hat{h})\subset \U$ for all $t\geq 1$. The decay and convergence rates in (i)-(iii) for $h_t$ and $k_t$ follow from Corollary \ref{prop Ck convergence rates} and Proposition \ref{prop : L^p conv rates}. To finish the proof, it suffices to show that we can write $\tilde{g}_t=\phi_t^*g_t$, where $g_t$ is the standard Ricci flow starting at $g_0$. For this purpose, let
\begin{align}\label{eq : diff_generators}
	V_t=-\begin{cases} V(\tilde{g}_t,\hat{h}), & t\in (0,1),\\
		V(\tilde{g}_t,h_t), &t\in [1,\infty)
		\end{cases}
\end{align}
For $t\in (0,1)$, $V_t$ is of the form $V_t=(\tilde{g}_t)^{-1}*(\tilde{g}_t)^{-1}*\nabla^{\hat{h}}(\tilde{g}_t-\hat{h})$ while for $t\geq1$, we have $V_t=(\tilde{g}_t)^{-1}*(\tilde{g}_t)^{-1}*\nabla^{h_t} k_t$. Therefore, by Lemma \ref{lem : Ricci flow short-time estimates} and Corollary \ref{prop Ck convergence rates}, $V_t$ is bounded for $t>0$ and hence a family of complete vector fields.
Due to Lemma \ref{lem : Ricci flow short-time estimates in Ck}, we have 
\begin{align}\label{eq : blowup de Turck vector field origin}
\left\| V_t\right\|_{L^{\infty}}\leq C t^{-\frac{1}{2}}\left\|g_0-\hat{h}\right\|_{L^{\infty}}\in L^1([0,1])
\end{align}
for $t\leq 1$. Therefore, $V_t$ actually generates a family of diffeomorphisms $({\varphi}_t)_{t\geq0}$ with ${\varphi}_0=\id_{M}$. A standard computation shows that the family $g_t=\varphi_t^*\tilde{g}_t$ is a Ricci flow starting at $g_0$ and the proof is completed with $\phi_t:=\varphi_t^{-1}$.
\end{proof}
\begin{rem}\label{rem : Lipschitz bound}
	Note that the bound in \eqref{eq : blowup de Turck vector field origin} implies the following: For given $\epsilon>0$, there exists $\delta>0$ such that if $g_0\in \mathcal{B}^{\infty}_{\delta}(\hat{h})$, the Ricci-de Turck flow $\tilde{g}_t$ as well as the standard Ricci flow ${g}_t$ starting at $g_0$ exist up to time $1$ and stay in $\mathcal{B}^{\infty}_{\epsilon}(\hat{h})$ for all $t\in [0,1]$.
\end{rem}

\subsection{Decay of the de Turck vector field and the Ricci curvature}\label{subsec : de Turck vector field}
%Throughout this subsection, we assume the same as in Proposition \ref{prop : improved rates}. Let as in Theorem \ref{thm : fixed point} the family $g_t$ be the modified Ricci-de Turck flow starting at $g_0$ which for $t\geq1$ splits as
%\begin{align*}
%g_t=h_t+k_t,
%\end{align*}
%where $h_t$ is Ricci-flat and $k_t\in\ker_{L^2}(\Delta_{L,h_t})^{\perp}$. 
%Throughout this subsection, let $q\in (1,n)$, $r\in (n,\infty)$ be H\"{o}lder exponents satisfying 
%\begin{align*}
%\frac{n}{2}\left(\frac{1}{q}-\frac{1}{r}\right) > \frac{1}{2},\qquad
%\frac{n}{2}\left(\frac{1}{q}-\frac{1}{r}\right) \neq 1.
%\end{align*}
%Fix $\epsilon>0$ and $\delta>0$ so small that Theorem \ref{thm : fixed point} applies: The modified Ricci-de Turck flow $g_t$ starting at an initial metric $g_0$ $\delta$-close to $\hat{h}$ in $L^{[q,\infty]}$ is well-defined, exists for all $t\geq0$ and for $t\geq1$ satisfies 
%\begin{align*}
%	\left\|\left(h_t-\hat{h},k_t\right)\right\|_{Y_{q,r}}<\epsilon.
%\end{align*}
%with $
%h_t:=\Phi(g_t), k_t:=g_t-h_t
%$. 
Let $q,r,\delta,\epsilon,g_0, h_t$ and $k_t$ be exactly as at the beginning of Subsection \ref{subsection : optimal convergence}.
Assume in addition that $\left\|g_0-\hat{h}\right\|_{L^p}\leq K<\infty$ so that Proposition \ref{prop : improved rates} applies and $\left(h_t-\hat{h},k_t\right)\in Y_{p,r'}$ for every $r'\in [r,\infty)$.

The goal of this subsection is to get improved estimates for the de Turck vector field and the Ricci curvature which are as follows:
\begin{prop}\label{prop : de Turck vector field}
	For each $i\in\N_0$ and $\tau>0$, there exists a constant $C=C(i,p,\tau,K,\epsilon)$ such that
	\begin{align*}
		\left\| V(g_t,h_t)\right\|_{C^i}
			\leq \begin{cases} Ct^{-\frac{n}{2p}-\frac{1}{2}+\tau},
			& \text{ if  }
			p\in \left(1,\frac{2n}{3}\right), \\
		Ct^{1-\frac{3n}{2p}+\tau},
			&\text{ if  }
			p\in \left[\frac{2n}{3},n\right), \end{cases}	\\
		\left\| \Ric_{g_t}\right\|_{C^i}
			\leq \begin{cases} C\cdot t^{-\frac{n}{2p}-1+\tau},
			& \text{ if  }p\in \left(1,\frac{n}{2}\right), \\
		C\cdot t^{1-\frac{3n}{2p}+\tau},& \text{ if  }
			p\in \left[\frac{n}{2},n\right).
		\end{cases}
	\end{align*}
	for all $t\geq 1$.
\end{prop}
By Taylor expansion along the curve 
$[0,1]\ni s\mapsto h_t+s\cdot k_t$, we get
\begin{align*}
V(g_t,h_t)&=V(g_t,h_t)-V(h_t,h_t)=DV_{h_t}(k_t)+\frac{1}{2}\int_0^1(1-s)^2D^2V_{h_t+s\cdot k_t}(k_t,k_t)ds,\\
\Ric_{g_t}&=\Ric_{g_t}-\Ric_{h_t}=D\Ric_{h_t}(k_t)+\frac{1}{2}\int_0^1(1-s)^2D^2\Ric_{g_t+s\cdot k_t}(k_t,k_t)ds,
\end{align*}
where $D^i$ denotes the $i$th Fr\'{e}chet derivative (for $V$, just in the first variable).
The proposition now follows from analyzing the respective parts on the right hand side.
We first need some estimates on the pure linear part of the equations.
\begin{lem}\label{lem : heat flow linearized vector field}
Let $\mathcal{U}$ be a small $L^{[q,\infty]}$-neighbourhood of $\hat{h}$. For $h\in\mathcal{F}\cap\mathcal{U}$, $t\geq1$, $i\in\N_0$ and $1<p\leq r<\infty$, there exists a constant $C=C(i,p,r,\mathcal{U})$ such that 
	\begin{equation}
		\begin{split}
			\label{eq : heat flow linearized vector field 1}
			\left\|\nabla^i\circ DV_h\circ e^{-t\Delta_{L,h}}\right\|_{L^p,L^r}
				&\leq C t^{-\frac{n}{2}\left(\frac{1}{p}-\frac{1}{r}\right)-\frac{1}{2}},\\
			\left\|\nabla^i\circ D\Ric_h\circ e^{-t\Delta_{L,h}}\right\|_{L^p,L^r}
				&\leq C t^{-\frac{n}{2}\left(\frac{1}{p}-\frac{1}{r}\right)-1}.
		\end{split}
	\end{equation}
	For $t\in [0,1]$, $i\in\N_0$ and $p\in (1,\infty)$, there exists a constant $C=C(i,p,\mathcal{U})$ such that
	\begin{align}\label{eq : heat flow linearized vector field 2}
	\left\|DV_h\circ e^{-t\Delta_{L,h}}\right\|_{W^{i+1,p},W^{i,p}}\leq C,\qquad 	\left\|D\Ric_h\circ e^{-t\Delta_{L,h}}\right\|_{W^{i+2,p},W^{i,p}}\leq C.
	\end{align}
\end{lem}
\begin{proof}We consider the case of the de Turck vector field first.
We fix $h\in\mathcal{F}$ and drop it in the notation in the proof for convenience.
 By Theorem \ref{thm : linear estimates} and Theorem \ref{thm : special derivative estimates}.
\begin{align*}
	\left\|e^{-t\Delta_L}\circ\Pi^{\perp}\right\|_{L^p,L^r}
		\leq C t^{-\frac{n}{2}\left(\frac{1}{p}-\frac{1}{r}\right)},\qquad \left\| DV\circ e^{-t\Delta_L}\right\|_{L^r,L^r}\leq C t^{-\frac{1}{2}},
\end{align*}
for all $t\geq 0$. 
The connection Laplacian $\Delta_{VF}$  on vector fields satisfies the commutation formula $\Delta_{VF}\circ DV=DV\circ \Delta_L$, see e.g.\ \cite{KP2020}*{Section 6.3}.  Therefore, $k\in\ker_{L^2}(\Delta_L)$ implies $DV(k)\in  \ker_{L^2}(\Delta_{VF})$.
However, $\Delta_{VF}=\nabla^*\nabla$ has vanishing $L^2$ kernel by an integration by parts argument. Thus, $DV\circ \Pi^{\parallel}=0$ and we obtain
\begin{align*}
 DV\circ e^{-t\Delta_L}=  DV\circ \Pi^{\perp} \circ e^{-t\Delta_L} 
=  DV\circ e^{-t\Delta_L} \circ \Pi^{\perp} 
  = DV\circ e^{-\frac{t}{2}\Delta_L}\circ e^{-\frac{t}{2}\Delta_L}\circ \Pi^{\perp}.
\end{align*} 
  This implies \eqref{eq : heat flow linearized vector field 1} in the case $i=0$. 
For derivatives, we have
\begin{align*}
	\nabla^i\circ DV\circ e^{-t\Delta_L}
		=\nabla^i\circ e^{-\frac{1}{2}\Delta_{VF}}\circ DV\circ e^{-\left(t-\frac{1}{2}\right)\Delta_L}
\end{align*}
By standard short-time estimates for parabolic equations (similarly as in Lemma \ref{lem : short-time estimates} (i) and (ii)), $\nabla^i\circ e^{-\frac{1}{2}\Delta_{VF}}$ is a bounded map on $L^r$ and \eqref{eq : heat flow linearized vector field 1} follows from the case $i=0$. 
Also by short-time estimates, $e^{-t\Delta_L}$ is bounded on $W^{i,p}$ for $t\in [0,1]$. Because $DV$ is a linear first order operator, \eqref{eq : heat flow linearized vector field 2} is immediate. 

	The case of the Ricci curvature is very similar. By \cite{Besse07}*{Theorem 1.174}, we have the relation $D\Ric(k)=\frac{1}{2}(\Delta_L(k)+\mathcal{L}_{DV(k)}h)$. Since $k\in\ker_{L^2}(\Delta_L)$ implies $DV(k)\in  \ker_{L^2}(\Delta_{VF})=\left\{0\right\}$, we obtain $\ker_{L^2}(\Delta_L)\subset\ker_{L^2}(D\Ric)$.
Furthermore, we have the commutator formula $\Delta_L\circ D\Ric=D\Ric\circ \Delta_L$ 
From now on, we can proceed exactly as above, using the estimate
\begin{align*}
	\left\| D\Ric\circ e^{-t\Delta_L}\right\|_{L^r,L^r}\leq C t^{-1},
\end{align*}
which holds due to Theorem \ref{thm : special derivative estimates}.
\end{proof}
The next step is to apply these linearizations to $k_t$ instead of $e^{-t\Delta_L}$.
\begin{lem}\label{lem : linear de Turck}
		For each $i\in\N_0$ and $\tau>0$, there exists a constant $C=C(i,\tau,p,K,\epsilon)$ such that
	\begin{align*}
		\left\|DV_{h_t} (k_t)\right\|_{C^i}&\leq \begin{cases}  Ct^{-\frac{n}{2p}+\tau-\frac{1}{2}},& \text{ if  }p\in \left(1,\frac{2n}{3}\right), \\
			Ct^{1-\frac{3n}{2p}+\tau},&\text{ if  }
			p\in \left[\frac{2n}{3},n\right), \end{cases}	\\
		\left\|D\Ric_{h_t} (k_t)\right\|_{C^i}&\leq \begin{cases}  Ct^{-\frac{n}{2p}+\tau-1},& \text{ if  }p\in \left(1,\frac{n}{2}\right), \\
	Ct^{1-\frac{3n}{2p}+\tau},&\text{ if  }
	p\in \left[\frac{n}{2},n\right), \end{cases}		
	\end{align*}
	for all $t\geq 1$.
\end{lem}
\begin{proof}
We are going to establish the estimates for both operators simultaneously. For notational convenience, we define
$P_{1,h_t}:=DV_{h_t}$ and $P_{2,h_t}:=D\Ric_{h_t}$ . Then our goal is to prove
\begin{align*}
\left\|P_{j,h_t} (k_t)\right\|_{C^i}&\leq \begin{cases}  Ct^{-\frac{n}{2p}+\tau-\frac{j}{2}},& \text{ if  }p\in \left(1,\frac{2n}{2+j}\right), \\
			Ct^{1-\frac{3n}{2p}+\tau},&\text{ if  }
			p\in \left[\frac{2n}{2+j},n\right) \end{cases}	
\end{align*}
for $j=1,2$. 
	 For $t\in [1,3]$, the bounds follow immediately from Corollary \ref{prop Ck convergence rates} as
	\begin{align*}
		\left\|P_{j,h_t}  (k_t)\right\|_{C^i}\leq C\left\| k_t\right\|_{C^{i+j}}\leq C.
	\end{align*}
Therefore, we may assume $t>3$ from now on.
	By applying $P_{j,h_t}$ to \eqref{eq : modified expression k_t}, we write
	\begin{align*}
	P_{j,h_t} (k_t)&=P_{j,h_t}\circ e^{-(t-1)\Delta_{L,h_t}}\overline{k}_1\\
	&\qquad+
	\int_1^{t-1}P_{j,h_t}\circ e^{-(t-s)\Delta_{L,h_t}}\Pi^{\perp}_{h_t}[(\Delta_{L,h_t}-\Delta_{L,h_s})(k_s)+
	(1-D_{g_s}\Phi)(H_1(s))]ds\\
	&\qquad+
	\int_{t-1}^tP_{j,h_t}\circ e^{-(t-s)\Delta_{L,h_t}}\Pi^{\perp}_{h_t}[(\Delta_{L,h_t}-\Delta_{L,h_s})(k_s)+
	(1-D_{g_s}\Phi)(H_1(s))]ds,
	\end{align*}
	and we estimate these three terms separately. Choose a H\"{o}lder exponent $r\in (n,\infty)$ whose precise value is yet to be determined but which is so large that
	\begin{align*}
	\frac{n}{2}\left(\frac{1}{p}-\frac{1}{r}\right)
		>\frac{1}{2}.
	\end{align*}
	Note that the expression $e^{-(t-1)\Delta_{L,h_t}}$ means that we solve for each fixed $t$ an autonomous heat equation for $s\in [0,t-1]$ with respect to the metric $h_t\in\mathcal{F}$.
	Thus for the first term, Lemma \ref{lem : heat flow linearized vector field} (applied for each $t$ to the fixed metric $h_t\in\mathcal{F}$) and Lemma \ref{lem : projection lemma} (ii) yield
	\begin{align*}
	\left\|\nabla^i\circ P_{j,h_t}\circ e^{-(t-1)\Delta_{L,h_t}}\overline{k}_1\right\|_{L^r}&\leq 
	C(t-1)^{-\frac{n}{2}(\frac{1}{p}-\frac{1}{r})-\frac{j}{2}}\left\|\overline{k}_1\right\|_{L^p}\leq C t^{-\frac{n}{2}(\frac{1}{p}-\frac{1}{r})-\frac{j}{2}}\left\|{k}_1\right\|_{L^p}
	\end{align*}
	for $j=1,2$.
	To estimate the second term, we first deal with the integrands. For $1\leq s\leq t$, Lemma \ref{lem : nonlinear part 0.1} yields
	\begin{align}\label{eq : nonlinear part used 1}
	\left\| (\Delta_{L,h_t}-\Delta_{L,h_s})(k_s)\right\|_{L^p}
	&\leq Cs^{1-\frac{3n}{2}(\frac{1}{p}-\frac{1}{r})}\left\|h-\hat{h}\right\|_{Z_{p,r}}\left\|k\right\|_{X_{p,r}}
	\end{align}
	and Lemma \ref{lem : nonlinear part 0.2} yields
	\begin{align}\label{eq : nonlinear part used 2}
	\left\|(1-D_{g_s}\Phi)(H_1(s))\right\|_{L^p}
	&\leq C s^{-\beta-\frac{n}{2}(\frac{1}{p}-\frac{1}{r})}
	\left\|k\right\|_{X_{p,r}}^2,
	\end{align}
	where $\beta=\min\left\{1,\frac{n}{2}\left(\frac{1}{p}-\frac{1}{r}\right)\right\}$. 
	We now distinguish between the cases of small and large $p$.

	 If $p<\frac{2n}{2+j}$, we pick $r\in (n,\infty)$ so large that
 	 $\beta\geq\frac{j}{2}$ for $j=1,2$ and
	  $n\left(\frac{1}{p}-\frac{1}{r}\right)>1+\frac{j}{2}$. 
	Then we get
	\begin{align*}
	\beta+\frac{n}{2}\left(\frac{1}{p}-\frac{1}{r}\right)
		\geq\frac{n}{2}\left(\frac{1}{p}-\frac{1}{r}\right)+\frac{j}{2},
		\qquad 
		\frac{3n}{2}\left(\frac{1}{p}-\frac{1}{r}\right)-1
		>\frac{n}{2}\left(\frac{1}{p}-\frac{1}{r}\right)+\frac{j}{2}.
	\end{align*}
	By the triangle inequality, \eqref{eq : nonlinear part used 1} and \eqref{eq : nonlinear part used 2}, we thus get
	\begin{align*}
	&\left\| (\Delta_{L,h_t}-\Delta_{L,h_s})(k_s)+(1-D_{g_s}\Phi)(H_1(s))\right\|_{L^p} \\
	&\qquad \qquad \leq Cs^{-\frac{n}{2}\left(\frac{1}{p}-\frac{1}{r}\right)-\frac{1}{2}}
	\left\|k\right\|_{X_{p,r}}\left(\left\|k\right\|_{X_{p,r}}+\left\|h-\hat{h}\right\|_{Z_{p,r}}\right).
	\end{align*}
	Consequently, Lemma \ref{lem : heat flow linearized vector field} and Lemma \ref{important_technical_lemma} imply that
	\begin{align*}
	&\left\|\nabla^i\int_1^{t-1}P_{j,h_t}\circ e^{-(t-s)\Delta_{L,h_t}}\Pi^{\perp}_{h_t}[(\Delta_{L,h_t}-\Delta_{L,h_s})(k_s)+
	(1-D_{g_s}\Phi)(H_1(s))]ds\right\|_{L^r}\\
	&\leq C\int_1^{t}\left\| \nabla^i \circ P_{j,h_t}\circ e^{-(t-s)\Delta_{L,h_t}}\Pi^{\perp}_{h_t}\right\|_{L^p,L^r}
	\left\| (\Delta_{L,h_t}-\Delta_{L,h_s})(k_s)+
	(1-D_{g_s}\Phi)(H_1(s))\right\|_{L^p}ds\\
	&\leq C\int_1^{t-1}(t-s)^{-\frac{n}{2}\left(\frac{1}{p}-\frac{1}{r}\right)-\frac{j}{2}}s^{-\frac{n}{2}\left(\frac{1}{p}-\frac{1}{r}\right)-\frac{j}{2}}ds\cdot \left\|k\right\|_{X_{p,r}}\left(\left\|k\right\|_{X_{p,r}}+\left\|h-\hat{h}\right\|_{Z_{p,r}}\right)\\
	&\leq C t^{-\frac{n}{2}\left(\frac{1}{p}-\frac{1}{r}\right)-\frac{j}{2}}\left\|k\right\|_{X_{p,r}}\left(\left\|k\right\|_{X_{p,r}}+\left\|h-\hat{h}\right\|_{Z_{p,r}}\right),
	\end{align*}
    because $\frac{n}{2}\left(\frac{1}{p}-\frac{1}{r}\right)+\frac{j}{2}>1$ for $j=1,2$.

\noindent	Now if $p\in \left[\frac{2n}{2+j},n\right)$, we have
	\begin{align*}
		n\left(\frac{1}{p}-\frac{1}{r}\right)
			<\frac{2+j}{2}
	\end{align*}
	for any $r\in (n,\infty)$ so that
	\begin{align}\label{eq : p r comparison}
		\beta+\frac{n}{2}\left(\frac{1}{p}-\frac{1}{r}\right)
			>\frac{n}{2}\left(\frac{1}{p}-\frac{1}{r}\right)+\frac{j}{2}
			>\frac{3n}{2}\left(\frac{1}{p}-\frac{1}{r}\right)-1.
	\end{align}
By the triangle inequality, \eqref{eq : nonlinear part used 1} and \eqref{eq : nonlinear part used 2}, we thus get
	\begin{align*}
	&\left\| (\Delta_{L,h_t}-\Delta_{L,h_s})(k_s)+ (1-D_{g_s}\Phi)(H_1(s))\right\|_{L^p} \\
		&\qquad \leq Cs^{1-\frac{3n}{2}\left(\frac{1}{p}-\frac{1}{r}\right)} 
			\left\|k\right\|_{X_{p,r}}
			\left(\left\|k\right\|_{X_{p,r}}+\left\|h-\hat{h}\right\|_{Z_{p,r}}\right).
	\end{align*}
	Consequently, by Lemma \ref{important_technical_lemma},
	\begin{align*}
	&\left\|\nabla^i\int_1^{t-1}P_{j,h_t}\circ e^{-(t-s)\Delta_{L,h_t}}\Pi^{\perp}_{h_t}[(\Delta_{L,h_t}-\Delta_{L,h_s})(k_s)+
	(1-D_{g_s}\Phi)(H_1(s))]ds\right\|_{L^r}\\
	&\leq C\int_1^{t}\left\| \nabla^i \circ P_{j,h_t}\circ e^{-(t-s)\Delta_{L,h_t}}\Pi^{\perp}_{h_t}\right\|_{L^p,L^r}
	\left\| (\Delta_{L,h_t}-\Delta_{L,h_s})(k_s)+
	(1-D_{g_s}\Phi)(H_1(s))\right\|_{L^p}ds\\
	&\leq C\int_1^{t-1}(t-s)^{-\frac{n}{2}\left(\frac{1}{p}-\frac{1}{r}\right)-\frac{j}{2}}s^{1-\frac{3n}{2}\left(\frac{1}{p}-\frac{1}{r}\right)}ds	
		\cdot \left\|k\right\|_{X_{p,r}}\left(\left\|k\right\|_{X_{p,r}}+\left\|h-\hat{h}\right\|_{Z_{p,r}}\right)\\
	&\leq C t^{1-\frac{3n}{2}\left(\frac{1}{p}-\frac{1}{r}\right)}
		\left\|k\right\|_{X_{p,r}}\left(\left\|k\right\|_{X_{p,r}}+\left\|h-\hat{h}\right\|_{Z_{p,r}}\right)
	\end{align*}
due to \eqref{eq : p r comparison} and the inequality $\frac{n}{2}\left(\frac{1}{p}-\frac{1}{r}\right)+\frac{j}{2}>1$. Because $p<n$, the latter inequality holds provided that $r$ is sufficiently large.

	For the third term, we get, using \eqref{eq : heat flow linearized vector field 2} and Lemma \ref{lem : nonlinear part 0.3},
	\begin{align*}
	&\left\|\nabla^i\int_{t-1}^t P_{j,h_t}\circ e^{-(t-s)\Delta_{L,h_t}}\Pi^{\perp}_{h_t}[(\Delta_{L,h_t}-\Delta_{L,h_s})(k_s)+
	(1-D_{g_s}\Phi)(H_1(s))]ds\right\|_{L^r}\\
	&\qquad\leq C\sup_{s\in [t-1,t]}\left\| (\Delta_{L,h_t}-\Delta_{L,h_s})(k_s)+
	(1-D_{g_s}\Phi)(H_1(s))\right\|_{W^{i+2,r}}\\
	&\qquad\leq Ct^{-n\left(\frac{1}{p}-\frac{1}{r}\right)}\left(\left\|h-\hat{h}\right\|_{Z_{p,r}}+\left\|k\right\|_{X_{p,r}}\right).
	\end{align*}
	Note also that we have for $j=1,2$ that
	\begin{align*}
	-n\left(\frac{1}{p}-\frac{1}{r}\right)\leq \begin{cases}  -\frac{n}{2}\left(\frac{1}{p}-\frac{1}{r}\right)-\frac{j}{2},& \text{ if  }p\in \left(1,\frac{2n}{2+j}\right), \\
		1-\frac{3n}{2}\left(\frac{1}{p}-\frac{1}{r}\right),&\text{ if  }
		p\in \left[\frac{2n}{2+j},n\right), \end{cases}.
	\end{align*}
	provided that $r$ is chosen sufficiently large.
Thus by summing up the inequalities, we get for all $i\in\N_0$ and $j=1,2$ that
	\begin{align*}
	\left\| P_{j,h_t} (k_t)\right\|_{W^{i,r}}&\leq \begin{cases}  Ct^{-\frac{n}{2}\left(\frac{1}{p}-\frac{1}{r}\right)-\frac{j}{2}},& \text{ if  }p\in \left(1,\frac{2n}{2+j}\right), \\
		Ct^{1-\frac{3n}{2}\left(\frac{1}{p}-\frac{1}{r}\right)},&\text{ if  }
		p\in \left[\frac{2n}{2+j},n\right). \end{cases}	
\end{align*}
	Combining this with the Sobolev type inequality
	\begin{align*}
	\left\|   P_{j,h_t} (k_t)\right\|_{C^i}\leq C\left\|P_{j,h_t} (k_t)\right\|_{W^{i+1,r}}
	\end{align*}
	and choosing $r$ so large that
	\begin{align*}
	\frac{n}{2r}\leq \frac{n}{r}<\tau,
	\end{align*}
	we obtain the desired result.
\end{proof}
It remains to consider the error terms in the Taylor expansion.
\begin{lem}\label{lem : nonlinear de Turck}
	For each $i\in\N_0$ and $\tau>0$, there exists a constant $C=C(i,\tau,p,K,\epsilon)$ such that
	\begin{align*}
		\left\|\int_0^1(1-s)^2D^2V_{h_t+s\cdot k_t,h_t}(k_t,k_t)ds\right\|_{C^i}\leq C t^{-\frac{n}{p}+\tau},\\
		\left\|\int_0^1(1-s)^2D^2\Ric_{h_t+s\cdot k_t,h_t}(k_t,k_t)ds\right\|_{C^i}\leq C t^{-\frac{n}{p}+\tau}
	\end{align*}
	for all $t\geq 1$.
\end{lem}
\begin{proof}
	First note that due to short-time estimates $g_t$ is $C^i$-close to $\hat{h}$ for all $t\geq 1$. Therefore, all $C^i$-norms of the metrics $g_{t,s}:=h_t+s\cdot k_t$ are equivalent for $t\geq 1$ and $s\in [0,1]$ and we may suppress the dependence of the norms on the metric. 
	We have the schematic expressions
	\begin{align*}
		D^2V_{g_t+s\cdot k_t,h_t}(k_t,k_t)=\nabla k_t*k_t,\qquad 	D^2\Ric_{g_{t,s}}(k_t,k_t)=\nabla^2k_t*k_t+\nabla k_t*\nabla k_t+R_{g_{t,s}}*k_t*k_t,
	\end{align*}
	from which we conclude
	\begin{align*}
		\left\|\int_0^1(1-s)^2D^2V_{g_t+s\cdot k_t,h_t}(k_t,k_t)ds\right\|_{C^i}
			&\leq C \left\|k_t\right\|_{C^{i+1}}^2
			\leq C \left(t^{-\frac{n}{2p}+\frac{\tau}{2}}\right)^2
			=C t^{-\frac{n}{p}+\tau},\\
		\left\|\int_0^1(1-s)^2D^2\Ric_{g_t+s\cdot k_t}(k_t,k_t)ds\right\|_{C^i}
			&\leq C \left\|k_t\right\|_{C^{i+2}}^2
			\leq C \left(t^{-\frac{n}{2p}+\frac{\tau}{2}}\right)^2=C t^{-\frac{n}{p}+\tau}.	
	\end{align*}
	The inequalities on the right hand sides follow from Proposition \ref{prop Ck convergence rates} above.
\end{proof}
\begin{proof}[Proof of Proposition \ref{prop : de Turck vector field}]
	This follows now directly from Lemma \ref{lem : linear de Turck} and Lemma \ref{lem : nonlinear de Turck}. Note that the terms in Lemma \ref{lem : linear de Turck} are dominating for any $p\in (1,n)$.
\end{proof}
\subsection{Convergence of the Ricci flow}
\begin{proof}[Proof of Theorem \ref{thm : mainthm 2 introduction}]
Choose an arbitrary $\U$ and let $\V=\V(\U)$ as in Theorem \ref{thm : mainthm 1 introduction}. Then the modified Ricci-de Turck flow $\tilde{g}_t$ starting at $g_0$ exists for all time and converges to a Ricci-flat limit $\tilde{h}_{\infty}$ as $t\to\infty$.
 Moreover, the tensors $\tilde{h}_t=\Phi(\tilde{g}_t)$ and $\tilde{k}_t=\tilde{g}_t-\tilde{h}_t$ satisfy the convergence rates of the Propositions \ref{prop Ck convergence rates} and \ref{prop : L^p conv rates}. The family of vector fields $V(\tilde{g}_t,\tilde{h}_t)$, $t\geq1$ satisfies the decay rates of Proposition \ref{prop : de Turck vector field}. In particular since $p<\frac{3n}{4}$, we have
for all $i\in \N_0$ that
\begin{align*}
\left\|V(\tilde{g}_t,\tilde{h}_t)\right\|_{C^i(\hat{h})}\leq Ct^{-\alpha}
\end{align*}
for all $t\geq 1$  and some $\alpha>1$. Therefore, the family of diffeomorphisms $(\tilde{\varphi}_t)_{t\geq 0}$ with $\tilde{\varphi}_0=\id_M$ generated by $\tilde{V}_t$, $t\geq 0$ defined in \eqref{eq : diff_generators} converges for all $i\in \N_0$ in $C^i(\hat{h})$ to a limit diffeomorphism $\tilde{\varphi}_{\infty}$ as $t\to\infty$. Now let
\begin{align*}
	g_t=\tilde{\varphi}_t^*\tilde{g}_t,\qquad 	h_t=\tilde{\varphi}_t^*\tilde{h}_t,\qquad
		k_t=\tilde{\varphi}_t^*\tilde{k}_t,\qquad t\in [0,\infty).
\end{align*}
Observe that $(g_t)_{t\geq 0}$ is the standard Ricci flow starting at $g_0$ and $h_t$ is a family of Ricci-flat metrics.
Because $\tilde{h}_t\to\tilde{h}_{\infty}$ in $C^i(\hat{h})$, we also have $h_t=\tilde{\varphi}_t^*\tilde{h}_t\to \tilde{\varphi}_{\infty}^*\tilde{h}_{\infty}=:h_{\infty}$ in $C^i(\hat{h})$.
Recall that the $C^i$ norms of $\tilde{h}_{\infty}=(\tilde{\varphi}_{\infty}^{-1})^*h_{\infty}$ and $\hat{h}$ are equivalent as $\tilde{h}_{\infty},\hat{h}\in\U\cap\F$.
Therefore, the $C^i$-norms induced by the Ricci-flat metrics $(\tilde{\varphi}_t^{-1})^*h_{\infty}$ $t\in [1,\infty]$ and $\hat{h}$ are also equivalent. 
Thus we get for each $\tau>0$ a constant such that
	\begin{align*}
	\left\|g_t-h_t\right\|_{C^i(h_{\infty})}=\left\|k_t\right\|_{C^i(h_{\infty})}= 
	\left\|\tilde{k}_t\right\|_{C^i((\tilde{\varphi}_t^{-1})^*h_{\infty})}
	\leq 	C\left\| \tilde{k}_t\right\|_{C^i(\hat{h})}
	\leq C t^{-\frac{n}{2p}+\tau}
	\end{align*}
	due to Theorem \ref{thm : mainthm 1 introduction}. In particular, the Ricci flow $g_t=k_t+h_t$ converges to $h_{\infty}$ as $t\to\infty$.
	 To obtain the convergence rate of $h_t$, we compute
	\begin{align*}
	\partial_th_t=\tilde{\varphi}_t^*(\partial_t \tilde{h}_t)-\tilde{\varphi}_t^*\left(\mathcal{L}_{V(\tilde{g}_t,\tilde{h}_t)}\tilde{h}_t\right),
	\end{align*}
	which yields 
	\begin{align*}
	\left\|\partial_th_t\right\|_{C^i(h_{\infty})}&\leq \left\|\tilde{\varphi}_t^*(\partial_t \tilde{h}_t)  \right\|_{C^i(h_{\infty})}+ \left\|\tilde{\varphi}_t^*(\mathcal{L}_{V(\tilde{g}_t,\tilde{h}_t)}\tilde{h}_t) \right\|_{C^i(h_{\infty})}\\
	&= \left\|\partial_t \tilde{h}_t \right\|_{C^i((\tilde{\varphi}_t^{-1})^*h_{\infty})}+ \left\|\mathcal{L}_{V(\tilde{g}_t,\tilde{h}_t)}\tilde{h}_t   \right\|_{C^i((\tilde{\varphi}_t^{-1})^*h_{\infty})}\\
	&\leq C\left(\left\|\partial_t \tilde{h}_t \right\|_{C^i(\hat{h})}+ \left\|V(\tilde{g}_t,\tilde{h}_t) \right\|_{C^{i+1}(\hat{h})}\right).
	\end{align*}
   Proposition \ref{prop : improved rates} yields by definition of the $Y_{p,r'}$-norm that $\left\|\partial_t \tilde{h}_t \right\|_{C^i(\hat{h})}\leq C t^{-\frac{n}{p}+\tau}$, where $C=C(\tau)$ and $\tau>0$ can be chosen arbitrarily small. The convergence rate of $h_t$ now follows from  Proposition \ref{prop : de Turck vector field} and integrating in time.
\end{proof}

\subsection{Positive scalar curvature rigidity}
In this subsection, we will prove the scalar curvature rigidity statement using our stability result. We will  use that the Ricci curvature (and hence the scalar curvature as well) decay of order $\O(t^{-\frac{n}{2p}-1+\tau})$ for small $p$. On the other hand, because the scalar curvature satisfies the super heat equation
\begin{align*}
\partial_t\scal_{g_t}+\Delta_{g_t}\scal_{g_t}=2|\Ric_{g_t}|^2_{g_t}
\end{align*}
along the Ricci flow, we expect a decay rate of at most of order $\O(t^{-\frac{n}{2}})$, which is the $L^{\infty}$ decay rate of the heat kernel on ALE spaces. We will follow the same strategy as \cite{App18}.

Let $g$ be a metric satisfying the assumptions of Theorem \ref{thm : psc rigidity introduction}. Let $\tilde{g}_t$ be the $\hat{h}$-gauged Ricci-de Turck flow and $g_t$ the standard Ricci flow starting from $g=g_0$, both defined up to time $1$.
We need to understand the heat kernel of the evolving backgrounds. For $0\leq s<t$ and $x,y\in M$, let $K(x,t,y,s)$ be the heat kernel associated to $g_t$, i.e.\
\begin{align*}
u(t,x):=\int_M K(x,t,y,s) u_s(y)\dv_{g_s}
\end{align*}
is the solution of the initial value problem
\begin{align*}
\partial_tu+\Delta_{g_t}u=0,\qquad u(s,x)=u_s(x).
\end{align*}
Let $\tilde{K}(x,t,y,s)$ be the heat kernel associated to $\tilde{g}_t$, then we have the relation
\begin{align*}
K(x,t,y,s)=\tilde{K}(\phi_t(x),t,\phi_t(y),s),
\end{align*}
where $\phi_t$ are the diffeomorphisms such that $\phi_t^*\tilde{g}_t=g_t$.
For $0<s<t\leq 1$, \cite{Zhu16}*{Theorem 4.2} yields the Gaussian bounds
\begin{align*}
\tilde{K}(x,t,y,s)\leq C_1(t-s)^{-\frac{n}{2}}\exp(C_2\Lambda+C_3(t-s)\kappa+C_4\sqrt{(t-s)\kappa}) \exp\left(-\frac{d_{\tilde{g}_t}(x,y)^2}{8 \exp(4\kappa )(t-s)}\right),
\end{align*}
where $\Lambda=\int_s^t\left\|\Ric_{\tilde{g}_{t'}}\right\|_{C^0(\tilde{g_t})}dt'$ and $\Ric_{\tilde{g}_{t'}}\geq-\kappa$ for $t'\in [s,t]$. By Remark \ref{rem : Lipschitz bound}, $g_t$ stays $L^{\infty}$-close to $\hat{h}$ up to time $1$, so that the induced distance functions $d_{g_t}$, $d_{\tilde{g}_t}$ and $d_{\hat{h}}$ are all equivalent.
 By diffeomorphism invariance, we thus get
	\begin{equation}\begin{split}\label{eq : Gaussian bounds}
K(x,t,y,s)&\leq C_1(t-s)^{-\frac{n}{2}}\exp\left(C_2\Lambda+C_3(t-s)\kappa+C_4\sqrt{(t-s)\kappa}\right)\\
&\qquad \cdot\exp\left(-\frac{d_{\hat{h}}(x,y)^2}{C_5 \exp(4\kappa )(t-s)}\right),
\end{split}
\end{equation}
where $\Lambda=\int_s^t\left\|\Ric_{{g}_{t'}}\right\|_{C^0({g_t})}dt'$ and $\Ric_{{g}_{t'}}\geq-\kappa$ for $t'\in [s,t]$.
\begin{lem}
If $\scal_{g_0}\geq0$, then $\scal_{g_1}\geq0$.
	\end{lem}
\begin{proof}
This lemma has been shown in the case of $\R^n$ in \cite{App18}, based on the analysis in \cite{Bam16} and a parabolic scaling argument which does not work on general ALE manifolds. 
For this reason, we present the details here although the ideas are similar as in \cite{Bam16}. Let $\theta\in (0,1)$ and consider the sequence of times $t_i=\theta^i$, $i\in\N_0$. 
 Due to short-time estimates for the Ricci-de Turck flow,  $\left\|\Ric_{{g}_{t}}\right\|_{C^0({g_t})}\leq C_6 t^{-1}$ for $t\in (0,1]$. From \eqref{eq : Gaussian bounds}, we conclude
\begin{align*}
K(x,t_i,y,t_{i+1})\leq C_7 \theta^{-C_8\cdot i}\exp\left(-\frac{d_{\hat{h}}(x,y)^2}{C_9 \theta^i}\right).
\end{align*}
Now let $\beta > (\sqrt{\theta},1)$, $R>0$ and consider the sequence of radii
\begin{align*}
r_i=R\cdot (1-\beta^i).
\end{align*}
Fix a point $x\in M$ and set 
\begin{align*}
a_i:=\inf\left\{\scal_{g_{t_i}}(y) \mid y\in B(r_i,x)   \right\},
\end{align*}
where $B(r_i,x)$ is the ball of radius $r_i$ around $x$, defined with respect to $\hat{h}$. 
Standard regularity theory of the Ricci-de Turck flow (see e.g. \cite{Bam16}) shows that $g_t\in C^{2}_{loc}([0,1],\mathcal{M})$. Therefore, 
\begin{align*}
\liminf_{i\to\infty}a_i\geq \inf\left\{\scal_{g_0}(y) \mid y\in B(R,x)   \right\}\geq 0
\end{align*}
Then we have, for any $y\in M$,
\begin{align*}
\scal_{g_{t_i}}(y)&\geq
\int_M K(y,t_i,z,t_{i+1})\cdot \scal_{g_{t_{i+1}}}(z)	\dv_{g_{t_{i+1}}}\\
&\geq a_{i+1}\int_{B(r_{i+1}-r_i,y) } K(y,t_i,z,t_{i+1})	\dv_{g_{t_{i+1}}}\\
&\qquad -\frac{C_{10}}{t_{i+1}}\int_{M\setminus B(r_{i+1}-r_i,y) } K(y,t_i,z,t_{i+1})	\dv_{g_{t_{i+1}}}\\
&\geq a_{i+1}-\frac{C_{11}}{t_{i+1}}\theta^{-C_8\cdot i}\exp\left(-\frac{(r_{i+1}-r_i)^2}{C_8 \theta^i}\right)\\
&\geq a_{i+1}-C_{11}\cdot \theta^{-(C_8+1)\cdot i-1}
\exp\left(-\frac{R^2}{C_{12}}\frac{\beta^{2i}}{ \theta^i}\right).
\end{align*}
Because $\beta^2>\theta$, we may fix some $j\in\N$ such that 
\begin{align*}
\left(\frac{\theta}{\beta^2}\right)^j\leq \frac{1}{2}\theta^{C_8+1}.
\end{align*} 
Using $x^j\exp{(-x)}\leq C_{13}$ for $x>0$, we thus get
\begin{align*}
\scal_{g_{t_i}}(y)&\geq a_{i+1}- \frac{C_{14}}{R^{2j}\cdot 2^i}.
\end{align*}
We conclude
\begin{align*}
a_i\geq a_{i+1}- \frac{C_{14}}{R^{2j}\cdot 2^i},
\end{align*}
and therefore,
\begin{align*}
\scal_{g_{1}}(x)\geq a_0\geq \liminf_{i\to\infty}a_{i}- \frac{2\cdot C_{14}}{R^{2j}}\geq - \frac{2\cdot C_{14}}{R^{2j}}.
\end{align*}
Because $x\in M$ was taken arbitrarily, the result follows from letting $R\to\infty$.
	\end{proof}
Now, we continue with our analysis on large times. Let $(g_t)_{t\geq1}$ be the standard Ricci flow starting from $g_1$ and $(\overline{g}_t)_{t\geq 1}$ be the Ricci flow with moving gauge, also starting from $g_1$. Again, we have diffeomorphisms $(\overline{\varphi}_t)_{t\geq 1}$ such that 
$g_t=\overline{\varphi}^{*}_t\overline{g}_t$.
\begin{definition}
	Let $1\leq s<t$ and $x,y\in M$. Then the $\mathcal{L}$-length of a curve $\gamma:[s,t]\to M$ is
	\begin{align*}
	\mathcal{L}(\gamma):=\int_s^t\sqrt{t-t'}(\scal_{g_{t'}}(\gamma(t'))+|\gamma'(t')|_{g_{t'}}^2)dt'
	\end{align*}
	and the reduced distance between $(x,t)$ and $(y,s)$ is
	\begin{align*}
	\ell(x,t,y,s):=\frac{1}{2\sqrt{t-s}}\inf\left\{\mathcal{L}(\gamma)\mid \gamma:[s,t]\to M \text{ is a smooth curve with }\gamma(s)=y\text{ and }\gamma(t)=x \right\}.	\end{align*}
\end{definition}
\begin{lem}\label{lem : heat kernel reduced distance}
With the same notation as above, we have
\begin{align*}
K(x,t,y,s)\geq\frac{1}{(4\pi(t-s))^{\frac{n}{2}}}\exp({-\ell(x,t,y,s)} )
\end{align*}
\end{lem}
\begin{proof}
In the compact case, this result is \cite{chowetalII}*{Lemma 16.49}. The proof of this lemma is on the one hand based on \cite{chowetalII}*{Lemma 16.48} (whose proof in turn builds up on results in \cite{chowetalI} which do also hold for Ricci flows of complete manifolds of bounded curvature) and on the other hand on the weak maximum principle which does also hold in the present  situation (see e.g. \cite{chowetalII}*{Theorem 12.10}). Therefore, the assertion of \cite{chowetalII}*{Lemma 16.49} also holds. 
\end{proof}
\begin{lem}\label{lem : heatkernellowerbound}
	There exist constants $C_1,C_2>0$ such that
	\begin{align*}
	K(x,t,y,s)\geq\frac{C_1}{(4\pi(t-s))^{\frac{n}{2}}}
 	\exp\left({-\frac{(d_{\hat{h}}(x,y))^2}{C_2(t-s)}}\right)
	\end{align*}
\end{lem}
\begin{proof}
	For $x,y\in M$, let $\gamma_{x,y}:[s,t]\to M$  be a $\hat{h}$-geodesic joining $x$ and $y$. Due to the parametrization interval, $|\gamma'_{x,y}(t')|_{\hat{h}}=\frac{d_{\hat{h}}(x,y)}{(t-s)}$. Therefore,
	\begin{align*}
	\ell(x,t,y,s)\leq \frac{1}{\sqrt{t-s}}\mathcal{L}(\gamma_{x,y})&=\frac{1}{\sqrt{t-s}}\int_s^t\sqrt{t-t'}(\scal_{g_{t'}}(\gamma_{x,y}(t'))+|\gamma_{x,y}'(t')|_{g_{t'}}^2)dt'\\
	&\leq C_1+\frac{C_2}{\sqrt{t-s}}\int_s^t\sqrt{t-t'}|\gamma_{x,y}'(t')|_{\hat{h}}^2dt'
	= C_1+\frac{2C_2}{3}\frac{(d_{\hat{h}}(x,y))^2}{t-s}.
	\end{align*}
	Thus, we get 
	\begin{align*}
	\exp({-\ell(x,t,y,s)} )\geq \exp\left(- C_1-\frac{2C_2}{3}\frac{(d_{\hat{h}}(x,y))^2}{t-s}\right)=\exp(- C_1)\exp\left(-\frac{2C_2}{3}\frac{(d_{\hat{h}}(x,y))^2}{t-s}\right)
	\end{align*}
	and the result follows from Lemma \ref{lem : heat kernel reduced distance}.
\end{proof}

\begin{proof}[Proof of Theorem \ref{thm : psc rigidity introduction}] Note that due to interpolation, we may assume that $p>1$.
By the Duhamel principle, we have
	\begin{align*}
	\scal_{g_t}(x)=\int_M K(x,t,y,1) \scal_{g_1}(y)\dv_{g_1}+\int_1^t\int_M K(x,t,y,t') |\Ric_{g_{t'}}|_{g_{t'}}^2\dv_{g_{t'}}dt'.
	\end{align*}
	Now suppose that $\Ric_{g_0}\neq 0$. Then we also have $\Ric_{g_1}\neq 0$.
 Because $\scal_{g_1}\geq0$, we thus get $\scal_{g_t}(x)>0$ for all $t>1$ and $x\in M$.
  Now fix a point $x\in M$ and a ball $B_r(x)\subset M$ (defined with respect to $\hat{h}$) such that $\scal_{g_2}(y)\geq R>0$ for all $y\in B_r(x)$. 
		Then for $t>3$ using Lemma \ref{lem : heatkernellowerbound}, we get
		\begin{align*}
	\scal_{g_t}(x)&\geq \int_M K(x,t,y,2) \scal_{g_2}(y)\dv_{g_2}\\
	&\geq R\frac{C_1}{(4\pi(t-2))^{\frac{n}{2}}}\int_{B_r(x)} 
	\exp\left({-\frac{(d_{\hat{h}}(x,y))^2}{C_2(t-2)}}\right)\dv_{g_2}\geq C(t-2)^{-\frac{n}{2}}\geq Ct^{-\frac{n}{2}}.
		\end{align*}
		On the other hand, by Proposition \ref{prop : de Turck vector field}, we have for any $\tau>0$ constants such that 
		\begin{align}\label{eq : scalar curvature convergence}
	\scal_{g_t}(x)=\scal_{\overline{g}_t}(\overline{\varphi}_t(x))\leq C	\left\| \Ric_{\overline{g}_t}\right\|_{C^0}&\leq Ct^{-\frac{n}{2p}+\tau-1},
		\end{align}
		which leads to a contradiction since $p<\frac{n}{n-2}$.		
	\end{proof}
\begin{rem}
In \cite{App18}*{Lemma 6.6}, proves that under the present assumptions, $\scal_{g_t}\in L^1$ for $t>1$, if $p=\frac{n}{n-2}$. However, we are not able to reproduce this result because we do not have \eqref{eq : scalar curvature convergence} with $\tau=0$.
\end{rem}
\begin{proof}[Proof of Theorem \ref{thm: counterexample scalar curvature rigidity intro}]
	For $k\in\N_0$ and $\delta\in\R$, let
	\begin{align*}
	\mathrm{Conf}^{k,p}_{\delta}(\hat{h})=(1+W^{k,p}_{\delta}(M))\cdot \hat{h}=\left\{g\mid g\text{ conformal to }\hat{h}\text{ and }g-\hat{h}\in W^{k,p}_{\delta}(S^2M)\right\}
	\end{align*}
If $k>\frac{n}{p}$ and $\delta=-\frac{n}{p}$, we have a map
\begin{align*}
\scal: \mathrm{Conf}^{k,p}_{\delta}(\hat{h})\to W^{k-2,p}_{\delta-2}(M),\qquad   g\mapsto \scal_g.
\end{align*}
Its linearization at $\hat{h}$ is given by
\begin{align*}
(n-1)\cdot \Delta :W^{k,p}_{\delta}(M)\to W^{k-2,p}_{\delta-2}(M),
\end{align*}
see e.g. \cite{Besse07}*{Theorem 1.174}. 
Due to the condition on $p$, we have $\delta>2-n$ and this map is indeed an isomorphism (see e.g. \cite{Bartnik1986}*{Proposition 2.2}). Let now $f\in W^{k-2,p}_{\delta-2}(M)$
Due to the inverse function theorem for Banach manifolds, $\scal$ restricts to a diffeomorphism
\begin{align*}
\scal:\mathrm{Conf}^{k,p}_{\delta}(\hat{h})\supset\mathcal{U}\to \mathcal{V}\subset W^{k-2,p}_{\delta-2}(M),
\end{align*}
for some small neighbourhoods $\mathcal{U}$ of $\hat{h}$ and $\mathcal{V}$ of $0$, respectively. Therefore, we find for each sequence of positive functions $f_i\in \mathcal{V}$ converging to $0$ in $ W^{k-2,p}_{\delta-2}$ a sequence of metrics $g_i\in \mathrm{Conf}^{k,p}_{\delta}(\hat{h})$ with $\scal_{g_i}=f_i$ converging to $\hat{h}$ in  $ W^{k,p}_{\delta}$. By Sobolev embedding, we have
\begin{align*}
\left\| g_i-\hat{h}\right\|_{L^{[p,\infty]}}&=\left\| g_i-\hat{h}\right\|_{L^{p}}+\left\| g_i-\hat{h}\right\|_{L^{\infty}}\\
&\leq \left\| g_i-\hat{h}\right\|_{L^{p}_{\delta}}+C\left\| g_i-\hat{h}\right\|_{L^{\infty}_{\delta}}\leq C\left\| g_i-\hat{h}\right\|_{W^{k,p}_{\delta}}\to 0,
\end{align*}
which proves the result.
\end{proof}

\end{sloppypar}
\end{document}